\documentclass[11pt]{article}

\usepackage[slantedGreek]{mathpazo}
\usepackage[pdftex]{graphicx}
\usepackage{amsfonts}
\usepackage{setspace}
\usepackage{amssymb}
\usepackage{amsthm}
\usepackage{graphicx}
\usepackage{amsmath}
\usepackage{mathtools}
\usepackage{lscape}
\usepackage{subcaption}
\usepackage{suffix}
\usepackage{color}
\usepackage{bm}
\usepackage{verbatim}
\usepackage{rotating}
\usepackage{tabularx}
\usepackage{lscape}
\usepackage{enumitem}
\usepackage{afterpage}
\usepackage{pdflscape}
\usepackage{extarrows}

\DeclarePairedDelimiter\floor{\lfloor}{\rfloor}

\definecolor{darkblue}{rgb}{0.0,0.0,0.55}
\RequirePackage[colorlinks,citecolor=darkblue,urlcolor=darkblue,linkcolor=darkblue]{hyperref} 
\usepackage[normalem]{ulem}

\newtheorem{theorem}{\textsc{{Theorem}}}
\newtheorem{corollary}{\textsc{{Corollary}}}
\newtheorem{lemma}{\textsc{{Lemma}}}
\newtheorem{example}{\textsc{{Example}}}

\newtheorem{proposition}{\textsc{{Proposition}}}

\theoremstyle{remark}
\newtheorem{remark}{{\textsc{Remark}}}

\usepackage{booktabs}
\usepackage{array}
\usepackage{multirow}
\newcommand{\bfa}{\mathbf{a}}

\newcommand{\bfh}{\mathbf{h}}

\newcommand{\bfr}{\mathbf{r}}
\newcommand{\bfs}{\mathbf{s}}
\newcommand{\bft}{\mathbf{t}}
\newcommand{\bfu}{\mathbf{u}}
\newcommand{\bfv}{\mathbf{v}}
\newcommand{\bfw}{\mathbf{w}}
\newcommand{\bfx}{\mathbf{x}}
\newcommand{\bfy}{\mathbf{y}}

\newcommand{\bfA}{\mathbf{A}}
\newcommand{\bfB}{\mathbf{B}}

\newcommand{\bfD}{\mathbf{D}}

\newcommand{\bfI}{\mathbf{I}}

\newcommand{\bfO}{\mathbf{O}}

\newcommand{\bfR}{\mathbf{R}}
\newcommand{\bfS}{\mathbf{S}}

\newcommand{\bfU}{\mathbf{U}}

\newcommand{\bfY}{\mathbf{Y}}
\newcommand{\bfZ}{\mathbf{Z}}

\newcommand{\bftheta}{\boldsymbol \theta}

\newcommand{\bfomega}{\boldsymbol \omega}

\newcommand{\bfLambda}{\boldsymbol \Lambda}

\newcommand{\bfzero}{\mathbf{0}}

\newcommand{\diag}{\text{diag}}

\newcommand{\IG}{{\cal IG}}

\newcommand{\calK}{{\cal K}}
\newcommand{\calJ}{{\cal J}}

\usepackage[sort,comma]{natbib}

\setcitestyle{citesep={;}}

\renewcommand{\baselinestretch}{1.5}

\allowdisplaybreaks

\def\spacingset#1{\renewcommand{\baselinestretch}%
{#1}\small\normalsize} \spacingset{1}

\title{\vspace{-1cm} \baselineskip=20pt  \bf Beyond Mat\'ern: On A Class of Interpretable Confluent Hypergeometric Covariance Functions}

\date{}
\begin{document}
\maketitle
\baselineskip=15pt
\vspace{-1.75cm}
\begin{center}
Pulong Ma \\
  School of Mathematical and Statistical Sciences, Clemson University \\
  220 Parkway Dr., Clemson, SC 29634
  \\
  plma@clemson.edu\\
  and\\
 Anindya Bhadra\\ 
  Department of Statistics, Purdue University\\
  {250 N. University St., West Lafayette, IN 47907, USA} \\
  bhadra@purdue.edu\\
  \hskip 5mm\\
     \end{center}
\begin{abstract}

\noindent The Mat\'ern covariance function is a popular choice for prediction in spatial statistics and uncertainty quantification literature.  A key benefit of the Mat\'ern class is that it is possible to get precise control over the degree of mean-square differentiability of the random process. However, the Mat\'ern class possesses exponentially decaying tails, and thus may not be suitable for modeling polynomially decaying dependence. This problem can be remedied using polynomial covariances; however one loses control over the degree of mean-square differentiability of corresponding processes, in that  random processes with existing polynomial covariances are either infinitely mean-square differentiable or nowhere mean-square differentiable at all. We construct a new family of covariance functions called the \emph{Confluent Hypergeometric} (CH) class using a scale mixture representation of the Mat\'ern class where one obtains the benefits of both Mat\'ern and polynomial covariances. The resultant covariance contains two parameters: one controls the degree of mean-square differentiability near the origin and the other controls the tail heaviness, independently of each other. Using a spectral representation, we derive theoretical properties of this new covariance including equivalent measures and asymptotic behavior of the maximum likelihood estimators under infill asymptotics. The improved theoretical properties of the CH class are verified via extensive simulations. Application using NASA's Orbiting Carbon Observatory-2 satellite data confirms the advantage of the CH class over the Mat\'ern class, especially in extrapolative settings.  
\end{abstract}

\spacingset{1.25} 
\section{Introduction} \label{sec:intro}

Kriging, also known as spatial best linear unbiased prediction, is a term coined by \cite{Matheron1963} in honor of the South African mining engineer D. G. Krige   \citep{Cressie1990}. With origins in geostatistics, applications of kriging have permeated fields as diverse as spatial statistics \cite[e.g.,][]{Berger2001, Matern1960, journel1978, Cressie1993, Stein1999, Banerjee2014}, uncertainty quantification or UQ \citep[e.g.,][]{Santner2018, Sacks1989, Berger2019UQ} and machine learning \citep{Williams2006}. Suppose that $\{Z(\bfs)\in \mathbb{R}: \bfs \in \mathcal{D} \subset \mathbb{R}^d \}$ is a stochastic process with a covariance function $\text{cov}(Z(\bfs), Z(\bfs+\bfh)) = C(\bfh)$ that is solely a function of the increment $\bfh$. Then $C(\cdot)$ is said to be second-order stationary (or weakly stationary). Further, if $C(\cdot)$ is a function of $|\bfh|$ with $| \cdot|$ denoting the Euclidean norm, then $C(\cdot)$ is called isotropic. If the process $Z(\cdot)$ possesses a constant mean function and a weakly stationary (resp.~isotropic) covariance function, the process $Z(\cdot)$ is called weakly stationary (resp.~isotropic).  Further, $Z(\cdot)$ is a Gaussian process (GP) if every finite-dimensional realization $Z(\bfs_1),\ldots, Z(\bfs_n)$ jointly follows a multivariate normal distribution for $\bfs_i\in \mathcal{D}$ and every $n$. 

The Mat\'ern covariance function \citep{Matern1960} has been widely used in spatial statistics due to its flexible local behavior and nice theoretical properties \citep{Stein1999} with increasing popularity in the UQ and machine learning literature \citep{Guttorp2006, Gu2018}. The Mat\'ern covariance function is of the form:
\begin{align} \label{eqn: Matern}
\mathcal{M}(h; \nu, \phi, \sigma^2) = \sigma^2\frac{2^{1-\nu}}{\Gamma(\nu)} \left(\frac{\sqrt{2\nu}}{\phi}h\right)^\nu \mathcal{K}_\nu \left(\frac{\sqrt{2\nu}}{\phi}h\right),
\end{align}
where $\sigma^2>0$ is the variance parameter, $\phi>0$ is the range parameter, and $\nu>0$ is the smoothness parameter that controls the mean-square differentiability of associated random processes. We denote by $\mathcal{K}_\nu(\cdot)$ the modified Bessel function of the second kind with the asymptotic expansion $\mathcal{K}_\nu(h) \asymp (\pi/(2h))^{1/2} \exp(-h)$ as $h\to\infty$, where $f(x) \asymp g(x)$ denotes $\lim_{x\to \infty} f(x) / g(x) = c \in (0,\infty)$. Further, we use the notation $f(x) \sim g(x)$ if $c=1$.  Thus, using this asymptotic expression of $\mathcal{K}_\nu(h)$ for large $h$ from Section 6 of \cite{barndorff1982normal}, the tail behavior of the Mat\'ern covariance function is given by: 
$$
\mathcal{M}(h; \nu, \phi, \sigma^2) \asymp h^{\nu -1/2} \exp\left(-\frac{\sqrt{2\nu}}{\phi} h\right), \quad h\to \infty.
$$
Eventually, the $\exp(-\sqrt{2\nu}h/\phi)$ term dominates, and the covariance decays exponentially for large $h$. This exponential decay may make it unsuitable for capturing polynomially decaying dependence. This problem with the Mat\'ern covariance can be remedied by using covariance functions that decay polynomially, such as the generalized Wendland \citep{Gneiting2002} and generalized Cauchy covariance functions \citep{Gneiting2000}, but in using these polynomial covariance functions one loses a key benefit of the Mat\'ern class: that of the degree of mean-square differentiability of the process. Random processes with a Mat\'ern covariance function are exactly $\floor{\nu}$ times mean-square differentiable, whereas the random processes with a generalized Cauchy covariance function are either non-differentiable (very rough) or infinitely differentiable (very smooth) in the mean-square sense, without any middle ground \citep{Stein2005}. The generalized Wendland covariance family also has limited flexibility near the origin compared to the Mat\'ern class and has compact support \citep{Gneiting2002}.

Stochastic processes with polynomial covariances are ubiquitous in many scientific disciplines including geophysics, meteorology, hydrology, astronomy, agriculture and engineering; see \citet{Beran1992} for a survey. In UQ, Gaussian stochastic processes have been often used for computer model emulation and calibration \citep{Santner2018}. In some applications, certain inputs may have little impact on output from a computer model, and these inputs are called \emph{inert inputs}; see Chapter 7 of \cite{Santner2018} for detailed discussion. Power-law covariance functions can allow for large correlations among distant observations and hence are more suitable for modeling these inert inputs. Most often, computer model outputs can have different smoothness properties due to the behavior of the physical process to be modeled. Thus, power-law covariances with the possibility of controlling the mean-square differentiability of stochastic processes are very desirable for modeling such output.  

In spatial statistics, polynomial covariances have been studied in a limited number of works \citep[e.g.,][]{Haslett1989, Gay1990, Gneiting2000}. In the rest of the paper, we focus on investigation of polynomial covariances in spatial settings. For spatial modeling, polynomial covariance can improve prediction accuracy over large missing regions. A covariance function with polynomially decaying tail can be useful to model highly correlated observations. As a motivating example, Figure~\ref{fig: Global OCO-2 data} shows a 16-day repeat cycle of NASA's Level 3 data product of the column-averaged carbon dioxide dry air mole fraction (XCO2) at $0.25^\circ$ and $0.25^\circ$ collected from the Orbiting Carbon Observatory-2 (OCO-2) satellite. The XCO2 data are collected over longitude bands and have large missing gaps between them. Predicting the true process over these large missing gaps based on a spatial process model is challenging. If the covariance function only allows exponentially decaying dependence, the predicted true process will be dominated by the mean function in the spatial process model with the covariance function having negligible impact over these large missing gaps. However, if the covariance function can model polynomially decaying dependence, the predicted true process over these missing gaps will carry more information from distant locations where observations are available, presumably resulting in better prediction. Thus, it is of fundamental and practical interest to develop a covariance function with polynomially decaying tails, without sacrificing the control over the smoothness behavior of the process realizations.

\begin{figure}[htbp] 
\makebox[\textwidth][c]{ \includegraphics[width=.8\linewidth, height=0.25\textheight]{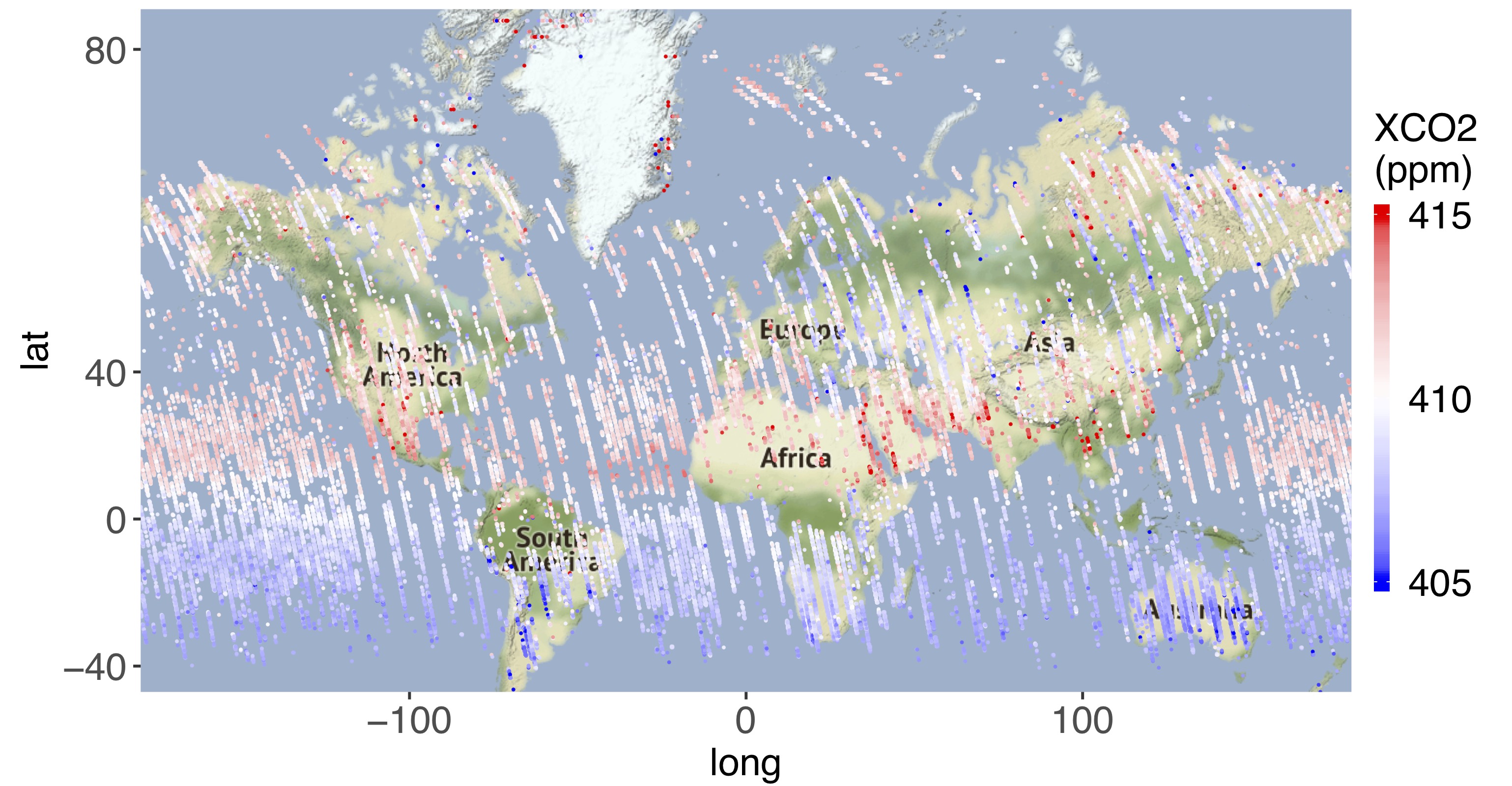}}
\caption{XCO2 data from June 1 to June 16, 2019. The units are parts per millions (ppm).}
\label{fig: Global OCO-2 data}
\end{figure}

In this paper we propose a new family of \emph{interpretable} covariance functions called the \emph{Confluent Hypergeometric} (CH) class that bridges this gap between the Mat\'ern covariance and polynomial covariances. The proposed covariance class is obtained by mixing the Mat\'ern covariance over its range parameter $\phi$. This is done by recognizing the Bessel function in the Mat\'ern  covariance function as proportional to the normalizing constant of the generalized inverse Gaussian distribution \citep[e.g.,][]{barn77}, which then allows analytically tractable calculations with respect to a range of choices for mixing densities, resulting in valid covariance functions with varied features. Apart from this technical innovation, the key benefit is that this mixing does not affect the origin behavior and thus allows one to retain the precise control over the smoothness of process realizations as in Mat\'ern. However, the tail is inflated due to mixing, and, in fact, the mixing distribution can be chosen in a way so that the tail of the resultant covariance function displays regular variation, with precise control over the tail decay parameter $\alpha$. A function $f(\cdot)$ is said to have a regularly decaying right tail with index $\alpha$ if it satisfies $f(x) \asymp x^{-\alpha} L(x)$ as $x\to\infty$ for some $\alpha >0$ where $L(\cdot)$ is a slowly varying function at infinity with the property $\lim_{x\to \infty} L(tx)/L(x)=1$ for all $t\in (0,\infty)$ \citep{bingham1989regular}. Unlike a generalized Cauchy covariance function, this CH class is obtained without sacrificing the control over the degree of mean-square differentiability of the process, which is still controlled solely by $\nu$, and the resulting process is still exactly $\floor{\nu}$ times mean-square differentiable, independent of $\alpha$. Moreover, regular variation is preserved under several commonly used transformations, such as sums or products. Thus, it is possible to exploit these properties of regular variation to derive CH covariances with similar features from the original covariance function that is obtained via a mixture of the Mat\'ern class.

The rest of the paper is organized as follows. Section~\ref{sec:classes} begins with the construction of the proposed covariance function as a mixture of the Mat\'ern  covariance function over its range parameter. We verify that such construction indeed results in a valid covariance function. Moreover, we demonstrate that the behaviors of this covariance function near the origin and in the tails are characterized by two distinct parameters, which in turn control the smoothness and the degree of polynomially decaying dependence, respectively. Section~\ref{sec:theory} presents the main theoretical results for the CH class. We first derive the spectral representation of this CH class and characterize its high-frequency behavior, and then show theoretical properties concerning equivalence classes under Gaussian measures and asymptotic properties related to parameter estimation and prediction. The resultant theory is extensively verified via simulations in Section~\ref{sec:numerical}. In Section~\ref{sec:real}, we use this CH covariance to analyze NASA's OCO-2 data, and demonstrate better prediction results over the Mat\'ern covariance. Section~\ref{sec:conclusions} concludes with some discussions for future investigations. All the technical proofs can be found in the Supplementary Material.  

\section{The CH Class as a Mixture of the Mat\'ern Class} \label{sec:classes}
Our starting point in mixing over the range parameter $\phi$ in the Mat\'ern covariance function is the correspondence between the form of the Mat\'ern covariance function and the normalizing constant of the generalized inverse Gaussian distribution \citep[e.g.,][]{barn77}. The generalized inverse Gaussian distribution has density on $(0,\infty)$ given by: 
$$
\pi_{GIG}(x) ={\frac {(a/b)^{p/2}}{2\mathcal{K}_{p}({\sqrt {ab}})}}x^{(p-1)}\exp\{-(ax+b/x)/2\} ; \;\; a,b>0, p\in \mathbb{R}.
$$
Thus,
$$
\calK_{p}({\sqrt {ab}}) = \frac{1}{2} (a/b)^{p/2} \int_{0}^{\infty} x^{(p-1)}\exp\{-(ax+b/x)/2\} dx.
$$
Take $a=\phi^{-2}, b=2\nu h^2$ and $p=\nu$. Then we have the following representation of the Mat\'ern covariance function with range parameter $\phi$ and smoothness parameter $\nu$: 
\begin{align*}  
\mathcal{M}(h; \nu, \phi, \sigma^2) &=\sigma^2 \frac{2^{1-\nu}}{\Gamma(\nu)} \left(\frac{\sqrt{2\nu}}{\phi}h \right)^\nu \mathcal{K}_\nu \left(\frac{\sqrt{2\nu}}{\phi}h\right)\\
&= \sigma^2 \frac{2^{1-\nu}}{\Gamma(\nu)}  \left(\frac{\sqrt{2\nu}h}{\phi} \right)^\nu  \frac{1}{2} \left(\frac{1} {\sqrt{2\nu} h\phi}\right)^\nu \int_{0}^{\infty} x^{(\nu-1)}\exp\{-(x/\phi^2+2\nu h^2/x)/2\}  dx\\
&= \frac{\sigma^2}{2^{\nu}\phi^{2\nu}\Gamma(\nu)} \int_{0}^{\infty} x^{(\nu-1)}\exp\{-(x/\phi^2+ 2\nu h^2/x)/2\}  dx. 
\end{align*} 
Thus, the mixture over $\phi^2$ with respect to a mixing measure $G(\phi^2)$ on $(0, \infty)$ can be written as \vspace{-10pt}
\begin{equation} \vspace{-4pt}
\begin{split}
C(h):&= \int_{0}^{\infty} \mathcal{M}(h; \nu, \phi, \sigma^2)  dG(\phi^2) \\
&= \int_{0}^{\infty} \left[\frac{\sigma^2}{2^{\nu}\phi^{2\nu}\Gamma(\nu)} \int_{0}^{\infty} x^{(\nu-1)}  \exp\{-(x/\phi^2+ 2\nu h^2/x)/2\}  dx\right] dG(\phi^2)  \\
& = \frac{\sigma^2}{2^{\nu}\Gamma(\nu)}  \int_{0}^{\infty} x^{(\nu-1)}\left[ \int_{0}^{\infty} \phi^{-2\nu} \exp\{-x/(2\phi^2)\}  dG(\phi^2) \right] \exp{(- \nu h^2/x)}  dx. \label{eq:mm}
\end{split}
\end{equation}
The resultant covariance via this mixture is quite general with different choices for the mixing measure $G(\phi^2)$. When the mixing measure $G(\phi^2)$ admits a probability density function, say $\pi(\phi^2)$, the inner integral may be recognized as a mixture of gamma integrals (by change of variable $u=\phi^{-2}$), which is analytically tractable for many choices of $\pi(\phi^2)$; see for example the chapter on gamma integrals in \citet{Abramowitz1965}. More importantly, as we show below, the mixing density $\pi(\phi^2)$ can be chosen to achieve precise control over certain features of the resulting covariance function.  

\begin{theorem} \label{lem: matern positive function}
Let $X \sim \IG(a,b)$ denote an inverse gamma random variable using the shape--scale parameterization with density $\pi_{IG}(x) = \{b^{a}/\Gamma(a)\} x^{-a-1} \exp(-b/x);\; a,b>0$. Assume that $\phi^2 \sim \IG(\alpha, \beta^2/2)$ and that $\mathcal{M}(h; \nu, \phi, \sigma^2)$ is the Mat\'ern covariance function in Equation~\eqref{eqn: Matern}. Then $C(h; \nu, \alpha, \beta, \sigma^2) := \int_{0}^{\infty} \mathcal{M}(h; \nu, \phi, \sigma^2) \pi(\phi^2; \alpha, \beta) d\phi^2$ is a positive-definite covariance function on $\mathbb{R}^d$ with the following form:
\begin{equation} \label{eqn: new kernel}
C(h;  \nu, \alpha, \beta, \sigma^2)
=  \frac{\sigma^2\beta^{2\alpha}\Gamma(\nu+\alpha)}{\Gamma(\nu)\Gamma(\alpha)}  \int_{0}^{\infty} x^{(\nu-1)} (x+\beta^2)^{-(\nu+\alpha)} \exp{(- \nu h^2/x)}  dx, 
\end{equation}
where $\sigma^2>0$ is the variance parameter, $\alpha>0$ is called the tail decay parameter, $\beta>0$ is called the scale parameter, and $\nu>0$ is called the smoothness parameter.   
\end{theorem}

\begin{remark}
The Mat\'ern covariance is sometimes parameterized differently. The mixing density can be chosen accordingly to arrive at results identical to ours. For instance, with parameterization of the Mat\'ern class given in \cite{Stein1999}, a gamma mixing density with shape parameter $\alpha$ and rate parameter $\beta^2/2$ would lead to an alternative route to the same representation of the CH class. The limiting case of the Mat\'ern class is the squared exponential (or Gaussian) covariance when its smoothness parameter $\nu$ goes to $\infty$. In this case, mixing over the inverse gamma distribution in Theorem~\ref{lem: matern positive function} yields the Cauchy covariance. 
\end{remark}

\begin{remark}The Mat\'ern covariance arises as a limiting case of the proposed covariance in Theorem~\ref{lem: matern positive function} when the mixing distribution on $\phi^2$ is a point mass. Indeed, standard calculations show that the mode of the inverse gamma distribution $\mathcal{IG}(\phi^2 \mid \alpha, \beta^2/2)$ is $\beta^2/\{2(\alpha+1)\}$ and its variance is $\beta^4/(\{4(\alpha-1)^2(\alpha-2)\}$. Thus, if one takes $\beta^2=2(\alpha+1) \gamma^2$ for fixed $\gamma>0$ and allows $\alpha$ to be large, the entire mass of the distribution $\mathcal{IG}(\alpha, (\alpha+1)\gamma^2)$ is concentrated at the fixed quantity $\gamma^2$ as $\alpha\to \infty$, which gives the Mat\'ern covariance $ \mathcal{M}(h; \nu, \gamma, \sigma^2)$  as the limiting case of the covariance function $C(h; \nu, \alpha, \sqrt{2(\alpha+1)}\gamma, \sigma^2)$ as $\alpha\to \infty$. 

\end{remark}

Having established  in Theorem~\ref{lem: matern positive function} the resultant mixture as a valid covariance function, one may take a closer look at its properties. To begin, although the final form of the CH class involves an integral, and thus may not appear to be in closed form at a first glance, the situation is indeed not too different from that of Mat\'ern, where the associated Bessel function is available in an algebraically closed form only for certain special cases; otherwise it is available as an integral. In addition, this representation of CH class is sufficient for numerically evaluating the covariance function as a function of $h$ via either quadrature or Monte Carlo methods. Additionally, with a certain change of variable, the above integral can be identified as belonging to a certain class of special functions that can be computed efficiently.  More precisely, we have the following elegant representation of the CH class, justifying its name. 

\begin{corollary}\label{lem: conhyper} The proposed covariance function in Equation~\eqref{eqn: new kernel} can also be represented in terms of the confluent hypergeometric function of the second kind: \vspace{-4pt}
\begin{equation} \label{eqn: hyperGeo representation}
C(h; \nu, \alpha, \beta, \sigma^2) = 
\frac{\sigma^2 \Gamma(\nu+\alpha)}{\Gamma(\nu)} \mathcal{U}\left(\alpha, 1-\nu, \nu \left(\frac{h}{\beta}\right)^2\right), \vspace{-4pt}
\end{equation}
where $\sigma^2>0, \alpha>0, \beta>0$, and $\nu>0$. We name the proposed covariance class as the {Confluent Hypergeometric} (CH) class after the confluent hypergeometric function.
\end{corollary}
\begin{proof}
By making the change of variable $x=\beta^2/t$, standard calculation yields that \vspace{-2pt}
\begin{equation*}
C(h; \nu, \alpha, \beta, \sigma^2)  = \frac{\sigma^2\Gamma(\nu+\alpha)}{\Gamma(\nu)\Gamma(\alpha)}  \int_0^{\infty} t^{\alpha-1} (t+1)^{-(\nu+\alpha)} \exp(-\nu h^2 t/\beta^2  ) dt.
\vspace{-2pt}
\end{equation*}
Thus, the conclusion follows by recognizing the form of the confluent hypergeometric function of the second kind $\mathcal{U}(a,b,c)$ from Chapter 13.2 of \citet{Abramowitz1965}.
\end{proof}

Equation~\eqref{eqn: hyperGeo representation} provides a convenient way to evaluate the CH covariance function, since efficient numerical calculation of the  confluent hypergeometric function is implemented in various libraries such as the GNU scientific library \citep{GSL}  and softwares including \textsf{R} and \textsc{MATLAB}, facilitating its practical deployment; see Section~\ref{sec: timing} of the Supplementary Material for an illustration of computing times for Bessel function and confluent hypergeometric function. For certain special parameter values, the evaluation of the confluent hypergeometric covariance function can be as easy as the Mat\'ern covariance function; see Chapter 13.6 of \citet{Abramowitz1965}. Besides the computational convenience, the CH covariance function in Equation~\eqref{eqn: new kernel} also allows us to make precise statements concerning the origin and tail behaviors of the resultant mixture. The next theorem makes the origin and tail behaviors explicit.

\begin{theorem} \label{thm: new horseshoe class}
The CH class has the following two properties:
\begin{itemize}[itemsep=-2pt, topsep=0pt, partopsep=0pt]
    \item[(a)] \textbf{Origin behavior}: The differentiability of the CH class is solely controlled by $\nu$ in the same way as the Mat\'ern class given in Equation~\eqref{eqn: Matern}.
    \item[(b)] \textbf{Tail behavior}: $C(h; \nu, \alpha, \beta, \sigma^2) \sim  \frac{\sigma^2 \beta^{2\alpha}  \Gamma(\nu+\alpha) }{\nu^{\alpha}\Gamma(\nu)}    |h|^{-2\alpha} L(h^2) $ as $h\to \infty$, where $L(x)$ is a slowly varying function at $\infty$ of the form $L(x)=\{x/(x+\beta^2/(2\nu)) \}^{\nu+\alpha}$. 
\end{itemize}
\end{theorem}

Theorem~\ref{thm: new horseshoe class} indicates that random processes with the CH covariance function are $\lfloor \nu \rfloor$ times mean-square differentiable. The local behavior of this CH covariance function is very flexible in the sense that the parameter $\nu$ can allow for any degree of mean-square differentiability of a weakly stationary process in the same way as the Mat\'ern class. However, its tail behavior is quite different from that of the Mat\'ern class, since the CH covariance has a polynomial tail that decays slower than the exponential tail in the Mat\'ern covariance. This is natural since mixture inflates the tails in general, and in our particular case, changes the exponential decay to a polynomial one. The rate of tail decay is controlled by the parameter $\alpha$. Thus, the CH class is more suitable for modeling polynomially decaying dependence which any exponentially decaying covariance function fails to capture. Moreover, the control over the degree of smoothness of process realizations is not lost.  Theorem~\ref{thm: new horseshoe class} also establishes a very desirable property that the degrees of differentiability near origin and the rate of decay of the tail for the CH covariance are controlled by two different parameters, $\nu$ and $\alpha$, independently of each other. Each of these parameters can allow any degrees of flexibility.

\begin{remark}
{\citet{Porcu2012} point out that it is possible to obtain a covariance function with flexible origin behavior and polynomial tails by simply taking a sum of a Mat\'ern and a Cauchy covariance, which is again a valid covariance function. There are three major difficulties in this approach compared to ours: (a) the individual covariances in such a finite sum are not identifiable and hence practical interpretation becomes difficult, although prediction may still be feasible, (b) this summed covariance has five parameters, hindering its practical use in both frequentist and Bayesian settings, since numerical optimization of the likelihood function is costly and judicious prior elicitation is likely to be difficult. In contrast, our covariance has four parameters, each of which has a well-defined role. Finally, (c) the microergodic parameter under such a summed covariance is not likely available in closed form, in contrast to ours, as derived later in Section~\ref{sec:theory}.}
\end{remark}

\begin{remark}
{Our approach to constructing the CH covariance by mixing over $\phi^2$ leads to a well-defined covariance class that can be used for Gaussian process modeling. The resulting covariance has four parameters and inference can be performed either via maximum likelihood or Bayesian approaches, although we solely focus on the former in the current work.} This construction should not be confused with Bayesian spatial modeling where a standard practice is to put a prior on the spatial range parameter in the Mat\'ern covariance. More importantly, the likelihood under the CH covariance is fundamentally different from the posterior that is proportional to the product of the likelihood under the Mat\'ern covariance and the prior on the spatial range parameter, where the prior could either be discrete or inverse gamma. 
\end{remark}

\begin{remark}
It is also worthing noting that our construction yields a covariance that is fundamentally different from a finite sum of Mat\'ern covariances where the range parameter is assigned a discrete prior, since the latter does not possess polynomially decaying dependence and is undesirable for modeling spatial data in practice due to costly computation and lack of practical motivation. Moreover, individual covariances in the finite sum are not identifiable. 
\end{remark}

\begin{example}
This example visualizes the difference between the CH class and the Mat\'ern class. We fix the \emph{effective range} (ER) at $200$ and $500$, where ER is defined as the distance at which a correlation function has value approximately $0.05$. For the CH class, we find the corresponding scale parameter $\beta$ such that the ER corresponds to $200$ and $500$ under different smoothness parameters $\nu \in \{0.5, 2.5\}$ and different tail decay parameters $\alpha\in \{0.3, 0.5, 1\}$. For the Mat\'ern class, we find the corresponding range parameter $\phi$ such that the ER corresponds to $200$ and $500$ under smoothness parameters $\nu\in \{0.5, 2.5\}$. These correlation functions are visualized in Figure~\ref{fig: correlations with different ERs}. As the CH correlation has a polynomial tail, it drops much faster than the Mat\'ern correlation in order to reach the same correlation 0.05 at the same ER. As $\alpha$ becomes smaller, the new correlation has a heavier tail, and hence it drops more quickly to reach the correlation 0.05 at the same effective range. After the ER, the CH correlation with a smaller $\alpha$ decays slower than those with larger $\alpha$. The faster decay of the tail in the Mat\'ern class is indicated by the behavior after the ER. Corresponding 1-dimensional process realizations can be found in Section~\ref{app: 1D Realization} of the Supplementary Material. 
\end{example}

\captionsetup[figure]{belowskip=12pt}
\begin{figure}[!t]  
\captionsetup[subfigure]{aboveskip=5pt}

\begin{subfigure}{.25\textwidth}
\makebox[\textwidth][c]{ \includegraphics[width=1.0\linewidth, height=0.18\textheight]{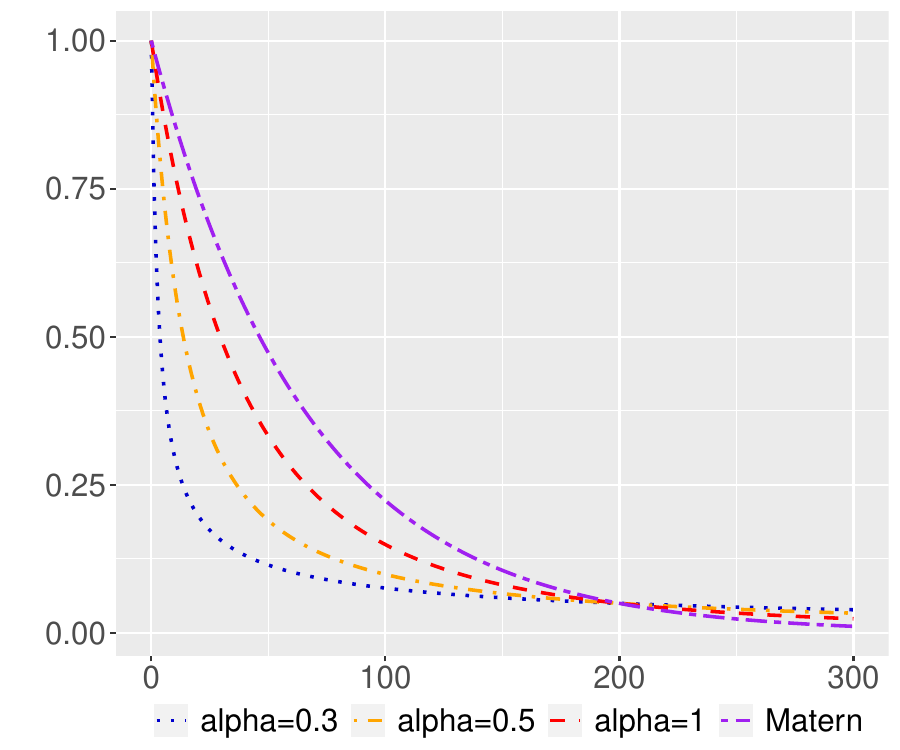}}
\caption{$\nu=0.5, ER=200$}
\end{subfigure}%
\begin{subfigure}{.25\textwidth}
\makebox[\textwidth][c]{ \includegraphics[width=1.0\linewidth, height=0.18\textheight]{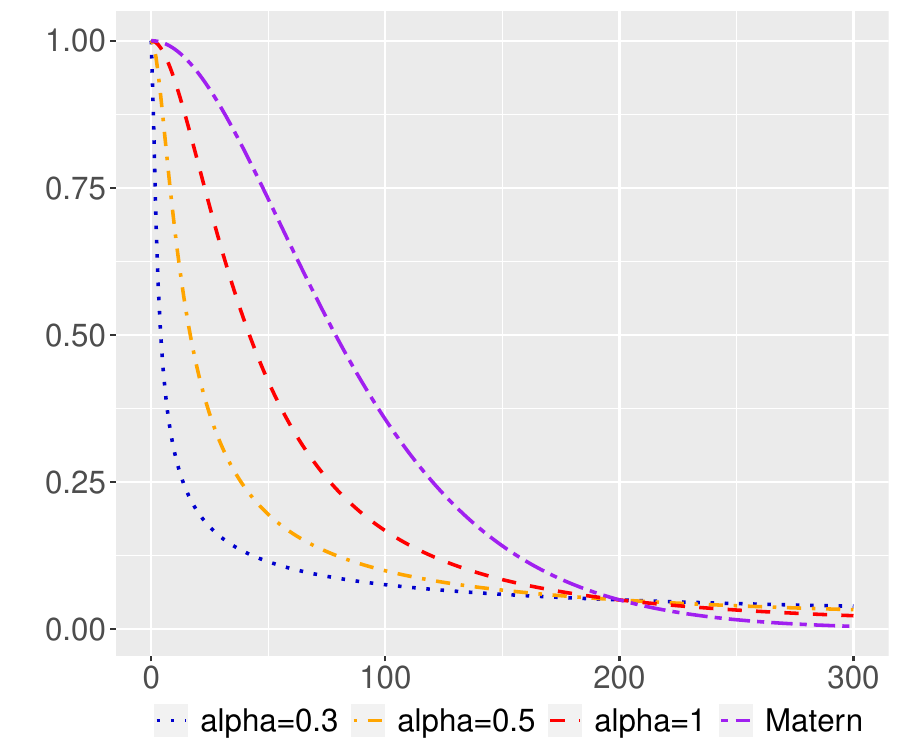}}
\caption{$\nu=2.5, ER=200$}
\end{subfigure}%
\begin{subfigure}{.25\textwidth}
\makebox[\textwidth][c]{ \includegraphics[width=1.0\linewidth, height=0.18\textheight]{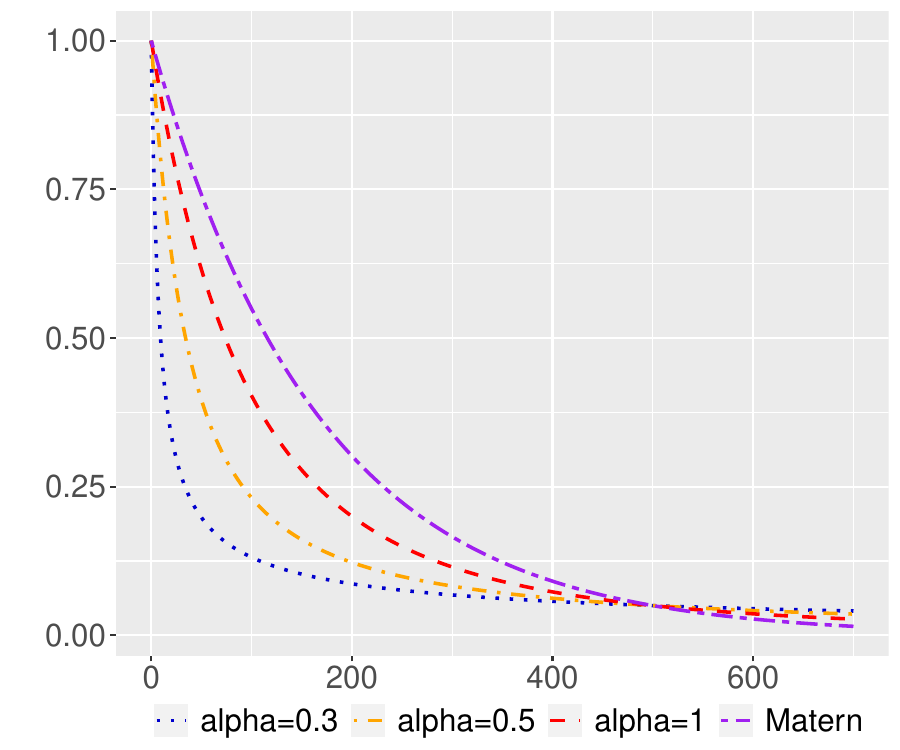}}
\caption{$\nu=0.5, ER=500$}
\end{subfigure}%
\begin{subfigure}{.25\textwidth}
\makebox[\textwidth][c]{ \includegraphics[width=1.0\linewidth, height=0.18\textheight]{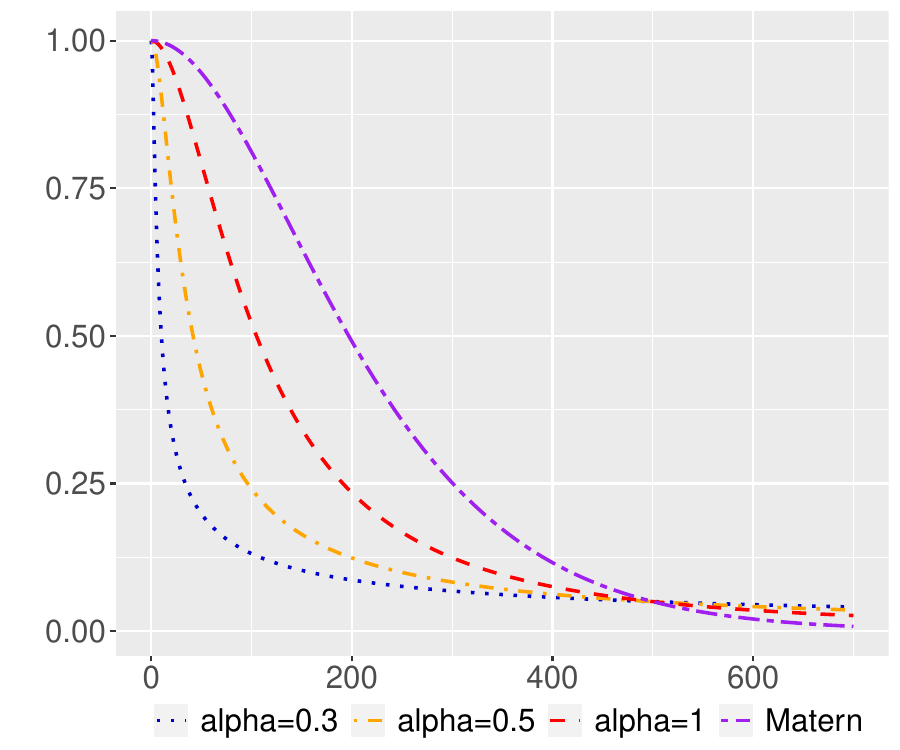}}
\caption{$\nu=2.5,ER=500$}
\end{subfigure}
\caption{Correlation functions for the CH class and the Mat\'ern class. The panels (a) and (b) show the correlation functions with the effective range (ER) at $200$. The panels (c) and (d) show the correlation functions with the effective range (ER) at $500$. ER is defined as the distance at which correlation is approximately $0.05$.}
\label{fig: correlations with different ERs}
\end{figure}
\captionsetup[figure]{belowskip=0pt}

\section{Theoretical Properties of the CH Class} \label{sec:theory}
For an isotropic random field, the properties of a covariance function can be characterized by its spectral density. The tail behavior of the spectral density can be used to derive properties of the theoretical results shown in later sections. The following proposition characterizes the tail behavior of the spectral density for the CH covariance function in Equation~\eqref{eqn: new kernel}.

\begin{proposition}[\textbf{Tail behavior of the spectral density}] \label{thm: spectral density}
The spectral density of the CH covariance function in Equation~\eqref{eqn: new kernel} admits the following tail behavior:
\vspace{-4pt}
\begin{equation*}
f(\omega) \sim \frac{\sigma^2 2^{2\nu} \nu^{\nu}  \Gamma(\nu+\alpha)}{\pi^{d/2}\beta^{2\nu}\Gamma(\alpha)} \omega^{-(2\nu+d)} L(\omega^2), \quad  \omega \to \infty, \vspace{-6pt}
\end{equation*}
where $L(x) = \{ {x}/(x + \beta^2/(2\nu))\}^{\nu+d/2}$ is a slowly varying function at $\infty$.
\end{proposition}

Recall that the spectral density of the Mat\'ern class is proportional to $\omega^{-(2\nu+d)}$ for large $\omega$. By mixing over the range parameter with an inverse gamma mixing density,  the high-frequency behavior of the CH class differs from that of the Mat\'ern class by a slowly varying function $L(\omega^2)$ up to a constant that does not depend on any frequency.

\subsection{{Equivalence Results}}
Let $(\Omega, \mathcal{F})$ be a measurable space with sample space $\Omega$ and $\sigma$-algebra $\mathcal{F}$. Two probability measures $\mathcal{P}_1, \mathcal{P}_2$ defined on the same measurable space $(\Omega, \mathcal{F})$ are said to be \emph{equivalent} if $\mathcal{P}_1$ is absolutely continuous with respect to $\mathcal{P}_2$ and $\mathcal{P}_2$ is absolutely continuous with respect to $\mathcal{P}_1$. Suppose that the $\sigma$-algebra $\mathcal{F}$ is generated by a random process $\{Z(\bfs): \bfs \in \mathcal{D}\}$. If $\mathcal{P}_1$  is equivalent to $\mathcal{P}_2$ on the $\sigma$-algebra $\mathcal{F}$, the probability measures $\mathcal{P}_1$ and $\mathcal{P}_2$ are then said to be equivalent on the realizations of the random process $\{Z(\bfs): \bfs \in \mathcal{D}\}$. It follows immediately that the equivalence of two probability measures on the $\sigma$-algebra $\mathcal{F}$ implies that their equivalence on any $\sigma$-algebra $\mathcal{F}' \subset \mathcal{F}$. The equivalence of probability measures has important applications to statistical inferences on parameter estimation and prediction according to \cite{Zhang2004}. The equivalence between $\mathcal{P}_1$ and $\mathcal{P}_2$ implies that $\mathcal{P}_1$ cannot be correctly distinguished from $\mathcal{P}_2$ with probability 1 under measure $\mathcal{P}_1$ for any realizations. Let $\{\mathcal{P}_{\theta}: \theta \in \Theta\}$ be a collection of equivalent measures indexed by $\theta$ in the parameter space $\Theta$.  Let $\hat{\theta}_n$ be an estimator for $\theta$ based on $n$ observations. Then $\hat{\theta}_n$ cannot converge to $\theta$ in probability regardless of what is observed \citep{Zhang2004}. Otherwise, for any fixed $\theta \in \Theta$, there exists a subsequence $\{\hat{\theta}_{n_k}\}_{k\geq 1}$ such that $\hat{\theta}_{n_k}$ converges to $\theta$ with probability 1 under measure $\mathcal{P}_{\theta}$ \citep[see, e.g.,][p.~288]{Dudley2002}. For any $\theta'\in \Theta$ with $\theta'\neq \theta$, it follows from the property of equivalent measures that $\hat{\theta}_{n_k}$ also converges to $\theta$ with probability 1 under measure $\mathcal{P}_{\theta'}$. This further implies that there exists a subsubsequence $\{\hat{\theta}_{n_{k_r}}\}_{r\geq 1}$ that converges to $\theta'$ with probability 1 under measure $\mathcal{P}_{\theta'}$. Hence, the subsequence $\{\hat{\theta}_{n_k}\}_{k\geq 1}$ and its subsubsequence $\{\hat{\theta}_{n_{k_r}}\}_{r\geq 1}$ converge to two different values under the same measure $\mathcal{P}_{\theta'}$. By contradiction, $\hat{\theta}_n$ cannot converge to $\theta$ in probability. This implies that individual parameters cannot be estimated consistently under equivalent measures. The second application of equivalent measures concerns the asymptotic efficiency of predictors that is discussed in Section~\ref{sec: APE}.

The tail behavior of the spectral densities in Proposition~\ref{thm: spectral density} can be used to check the equivalence of probability measures generated by stationary Gaussian random fields. The details of equivalence of Gaussian measures and the condition for equivalence are given in Section~\ref{app: ancilary} of the Supplementary Material. Any zero-mean Gaussian process with a covariance function defines a corresponding Gaussian probability measure. In what follows, we say that the Gaussian probability measure defined under a covariance function implies that such Gaussian measure is defined through a Gaussian process on a bounded domain with mean zero and a covariance function. Our first result on equivalence of two Gaussian measures under the CH class is given in Theorem~\ref{thm: equivalence}.
 
\begin{theorem} \label{thm: equivalence}
Let $\mathcal{P}_i$ be the Gaussian probability measure corresponding to the covariance $C(h; \nu, \alpha_i,$ $\beta_i, \sigma^2_i)$ with $\alpha_i>d/2$ for $i=1,2$. Then $\mathcal{P}_1$ and $\mathcal{P}_2$ are equivalent on the realizations of $\{ Z(\bfs): \bfs \in \mathcal{D} \}$ for any fixed and bounded set $\mathcal{D} \subset \mathbb{R}^d$ with infinite locations in $\mathcal{D}$ and $d=1,2,3$ if and only if  \vspace{-6pt}
    \begin{equation} \label{eqn: consistency for new covariance} \vspace{-4pt}
    	\frac{\sigma^2_1\Gamma(\nu+\alpha_1)}{\beta_1^{2\nu} \Gamma(\alpha_1)} = \frac{\sigma^2_2 \Gamma(\nu+\alpha_2)}{\beta_2^{2\nu}\Gamma(\alpha_2)}. 
    \end{equation} 
\end{theorem}

An immediate consequence of Theorem~\ref{thm: equivalence} is that for fixed $\nu$; the tail decay parameter $\alpha$, the scale parameter $\beta$ and the variance parameter $\sigma^2$ cannot be estimated consistently under the infill asymptotics. Instead, the quantity ${\sigma^2\beta^{-2\nu} \Gamma(\nu+\alpha)}/{\Gamma(\alpha)}$ is consistently estimable and has been referred to as the \emph{microergodic parameter}. We refer the readers to page 163 of \cite{Stein1999} for the definition of microergodicity. 

Theorem~\ref{thm: equivalence} gives the result on equivalent measures within the CH class. The CH class can allow the same smoothness behavior as the Mat\'ern class, but it has a polynomially decaying tail that is quite different from the Mat\'ern class. One may ask whether there is an analogous result on the Gaussian measures under the CH class and the Mat\'ern class.    Theorem~\ref{thm: equivalence with matern} provides an answer to this question.

\begin{theorem} \label{thm: equivalence with matern}
Let $\mathcal{P}_1$ be the Gaussian probability measure under the CH covariance $C(h; \nu, \alpha, \beta, \sigma^2_1)$ with $\alpha>d/2$ and  $\mathcal{P}_2$ be the Gaussian probability measure under the Mat\'ern covariance function $\mathcal{M}(h; \nu, \phi, \sigma^2_2)$. If \vspace{-2pt}
    \begin{equation} \label{eqn: equivalence with matern} \vspace{-2pt}
   \sigma^2_1 (\beta^2/2)^{-\nu} \Gamma(\nu+\alpha) / \Gamma(\alpha) = \sigma^2_2 \phi^{-2\nu}, 
    \end{equation} 
 then
    $\mathcal{P}_1$ and $\mathcal{P}_2$ are equivalent on the realizations of $\{ Z(\bfs): \bfs \in \mathcal{D} \}$ for any fixed and bounded set $\mathcal{D} \subset \mathbb{R}^d$ with $d=1,2,3$.
\end{theorem}

Theorem~\ref{thm: equivalence with matern} gives the conditions under which the Gaussian measures under the CH class and the Mat\'ern class are equivalent. If the condition in Equation~\eqref{eqn: equivalence with matern} is satisfied, the Gaussian measure under the CH class cannot be distinguished from the Gaussian measure under the Mat\'ern class, regardless of what is observed. This shows the robustness property for statistical inference under the CH class when the underlying true covariance model is the Mat\'ern class.

In Section~\ref{subsec: asymptotic normality}, the microergodic parameter of the CH class can be shown to be consistently estimated under infill asymptotics for a Gaussian process under the CH model with fixed and known $\nu$. Moreover, one can show that the maximum likelihood estimator of this microergodic parameter converges to a normal distribution. 

\subsection{Asymptotic Normality} \label{subsec: asymptotic normality}
Let $\{ Z(\bfs): \bfs \in \mathcal{D} \}$ be a zero mean Gaussian process with the covariance function $C(h; \nu, \alpha, \beta,$ $\sigma^2)$, where $\mathcal{D} \subset \mathbb{R}^d$ is a bounded subset of $\mathbb{R}^d$ with $d=1,2,3$. Let $\bfZ_n :=(Z(\bfs_1), \ldots, Z(\bfs_n))^\top$ be a partially observed realization of the process $Z(\cdot)$ at $n$ distinct locations in $\mathcal{D}$, denoted by  $\mathcal{D}_n: = \{ \bfs_1, \ldots, \bfs_n\}$. Then the log-likelihood function is 
\vspace{-4pt}
\begin{equation} \label{eqn: loglikelihood}
\ell_n(\sigma^2, \bftheta) = - \frac{1}{2} \left\{ n \log(2\pi \sigma^2) + \log|\bfR_n(\bftheta)| + \frac{1}{\sigma^2} \bfZ_n^\top \bfR^{-1}_n(\bftheta) \bfZ_n  \right\},
\vspace{-4pt}
\end{equation}
where $\bftheta := \{ \alpha, \beta \}$ and $\bfR_n(\bftheta) = [R(|\bfs_i- \bfs_j|; \bftheta)]_{i,j=1, \ldots, n}$ is an $n\times n$ correlation matrix with  the correlation function $R(h) := C(h)/\sigma^2$.

In what follows, $\nu$ is assumed to be known and fixed. Let $\hat{\sigma}^2_n$ and $\hat{\bftheta}_n$ be the maximum likelihood estimators (MLE) for $\sigma^2$ and $\bftheta$ by maximizing the log-likelihood function in Equation~\eqref{eqn: loglikelihood}. To show the consistency and asymptotic normality results for the microergodic parameter, we first obtain an estimator for $\sigma^2$ when $\bftheta$ is fixed:
$ 
\hat{\sigma}^2_n  = \bfZ_n^\top \bfR_n^{-1}(\bftheta) \bfZ_n / n. 
$
Then, let  $\hat{c}_n(\bftheta)$ be the maximum likelihood estimator of $c(\bftheta):={\sigma^2\beta^{-2\nu} \Gamma(\nu+\alpha)}/{\Gamma(\alpha)}$, as a function of $\bftheta$, given by 
\begin{equation*} 
\hat{c}_n(\bftheta) = \hat{c}_n(\alpha, \beta) = \frac{\hat{\sigma}^2_n\Gamma(\nu+\alpha)}{\beta^{2\nu} \Gamma(\alpha)} =  \frac{\bfZ_n^\top \bfR_n^{-1}(\bftheta) \bfZ_n \Gamma(\nu+\alpha)}{n \beta^{2\nu}\Gamma(\alpha)}. 
\end{equation*}

For notational convenience, we use $c(\alpha, \beta)$ instead of $c(\bftheta)$ to denote the microergodic parameter in what follows. We discuss three situations. In the first situation, we consider joint estimation of $\beta$ and $\sigma^2$ 
for fixed $\alpha$. 
The MLE of $\beta$ will be denoted by $\hat{\beta}_n$, and the MLE of the microergodic parameter is $\hat{c}_n(\alpha, \hat{\beta}_n)$. In the second situation, we consider joint estimation of $\alpha$ and $\sigma^2$ 
for fixed $\beta$. The MLE of $\alpha$ will be denoted by $\hat{\alpha}_n$ and the MLE of the microergodic parameter is $\hat{c}_n(\hat{\alpha}_n, \beta)$. In the third situation, we consider joint estimation of all parameters $\alpha, \beta, \sigma^2$. 
The corresponding MLE for $c(\bftheta)$ is denoted by $\hat{c}_n(\hat{\bftheta}_n)$, where $\hat{\bftheta}_n:=\{\hat{\alpha}_n, \hat{\beta}_n\}$. Note that the MLEs of either $\alpha$ or $\beta$ (or both) are typically computed numerically, since there is no closed-form expression. We have the following results on the asymptotic properties of $\hat{c}_n(\bftheta)$ for various scenarios of $\alpha$ and $\beta$ under the infill asymptotics.

\begin{theorem}[\textbf{Asymptotics of the MLE}] \label{thm: ML consistency2}
Let $\mathcal{P}_0$ be the Gaussian measure defined under the covariance function $C(h; \nu, \alpha_0, \beta_0, \sigma^2_0)$ with $\sigma^2_0>0$, and let $\bftheta_0:=\{\alpha_0, \beta_0\}$.
Let $Z_n$ be the set of observations generated under $\mathcal{P}_0$. Then the following results can be established: 
\begin{enumerate}[itemsep=-8pt, topsep=0pt, partopsep=0pt]
\item[(a)] Suppose that $\alpha_0>d/2$ and $\beta_0 \in [\beta_L, \beta_U]$, where $\beta_L, \beta_U$ are fixed constants such that $0<\beta_L<\beta_U$. For any fixed $\alpha>d/2$, if $(\hat{\sigma}^2_n, \hat{\beta}_n)$ maximizes the log-likelihood function~\eqref{eqn: loglikelihood} over $(0, \infty) \times [\beta_L, \beta_U]$, then as $n\to \infty$, $\hat{c}_n(\alpha, \hat{\beta}_n) \stackrel{a.s.}{\longrightarrow} c(\bftheta_0)$ under $\mathcal{P}_0$ and  $\sqrt{n} \left  \{ \hat{c}_n(\alpha, \hat{\beta}_n)- c(\bftheta_0) \right\} \stackrel{\mathcal{L}}{\longrightarrow} \mathcal{N}\left (0, 2[c(\bftheta_0) ]^2   \right)$.
\item[(b)] Suppose that  $\alpha_0\in [\alpha_L, \alpha_U]$ and $\beta_0>0$, where $\alpha_L, \alpha_U$ are fixed constants such that $d/2<\alpha_L<\alpha_U$.  For any fixed $\beta>0$, if $(\hat{\sigma}^2_n, \hat{\alpha}_n)$ maximizes the log-likelihood function~\eqref{eqn: loglikelihood} over $(0, \infty) \times [\alpha_L, \alpha_U]$, then as $n\to \infty$,  $\hat{c}_n(\hat{\alpha}_n, {\beta}) \stackrel{a.s.}{\longrightarrow} c(\bftheta_0)$ under $\mathcal{P}_0$ and $\sqrt{n} \left  \{ \hat{c}_n(\hat{\alpha}_n, {\beta})- c(\bftheta_0) \right\} \stackrel{\mathcal{L}}{\longrightarrow} \mathcal{N}\left (0, 2[c(\bftheta_0) ]^2   \right)$.
\item[(c)] Suppose that $\alpha_0\in [\alpha_L, \alpha_U]$ and $\beta_0 \in [\beta_L, \beta_U]$ where $\alpha_L, \alpha_U, \beta_L, \beta_U$ are fixed constants such that $d/2<\alpha_L<\alpha_U$ and $0<\beta_L<\beta_U$. If $(\hat{\sigma}^2_n, \hat{\alpha}_n, \hat{\beta}_n)$ maximizes the log-likelihood function~\eqref{eqn: loglikelihood} over $(0, \infty) \times [\alpha_L, \alpha_U] \times [\beta_L, \beta_U]$, then as $n\to \infty$, $\hat{c}_n(\hat{\bftheta}_n) \stackrel{a.s.}{\longrightarrow} c(\bftheta_0)$ under $\mathcal{P}_0$ and $\sqrt{n} \left  \{ \hat{c}_n(\hat{\bftheta}_n)- c(\bftheta_0) \right\} \stackrel{\mathcal{L}}{\longrightarrow} \mathcal{N}\left (0, 2[c(\bftheta_0) ]^2   \right)$.
\end{enumerate}
\end{theorem}

The first two results of Theorem~\ref{thm: ML consistency2} imply that the microergodic parameter can be estimated consistently by fixing $\alpha>d/2$ or $\beta>0$ in compact sets. In practice, fixing either $\alpha$ or $\beta$ may be too restrictive for modeling spatial processes. 
We would expect that the finite sample prediction performance can be improved by jointly estimating all covariance parameters for the CH class, which should be the preferred approach for practical purposes. 

The third result of Theorem~\ref{thm: ML consistency2} establishes that the microergodic parameter can be consistently estimated by jointly maximizing the log-likelihood~\eqref{eqn: loglikelihood} over $\alpha$ and $\beta$. However, the current result requires that $\alpha>d/2$. This means that the CH covariance cannot decay too slowly in its tail in order to establish the consistency result. Nevertheless, this result shows a significant improvement over existing asymptotic normality results for other types of polynomially decaying covariance functions. For instance, it was shown by \cite{Bevilacqua2019} that the microergodic parameter in the generalized Cauchy class can be estimated consistently under infill asymptotics. However, their results assume that the parameter that controls the tail behavior is fixed. This is similar to the first result of Theorem~\ref{thm: ML consistency2}.  Unlike their results, a theoretical improvement in Theorem~\ref{thm: ML consistency2} is that the asymptotic results for the microergodic parameter $c(\bftheta)$ can be obtained for the joint estimation of all three parameters, including the parameter that controls the decay of the tail. We provide extensive numerical evidence in support of Theorem~\ref{thm: ML consistency2} in Section~\ref{sec: mle} of the Supplementary Material.


\subsection{Asymptotic Prediction Efficiency} \label{sec: APE}
This section is focused on studying the prediction problem of Gaussian process at a new location $\mathbf{s}_0 \in \mathcal{D} \cap \mathcal{D}_n^c$. This problem has been studied extensively when an incorrect covariance model is used. Our focus here is to show the asymptotic efficiency and asymptotically correct estimation of prediction variance in the context of the CH class. \cite{Stein1988} shows that both of these two properties hold when the Gaussian measure under a misspecified covariance model is equivalent to the Gaussian measure under the true covariance model. In the case of the CH class, Theorem~\ref{thm: equivalence} gives the conditions for equivalence of two Gaussian measures in the light of the microergodic parameter $c(\bftheta)={\sigma^2\beta^{-2\nu} \Gamma(\nu+\alpha)}/{\Gamma(\alpha)}$.  As in Section~\ref{subsec: asymptotic normality}, $\nu$ will be assumed to be fixed. 

With observations generated under the CH model $C(h; \nu, \alpha, \beta, \sigma^2)$, we define the best linear unbiased predictor for $Z(\bfs_0)$ to be
\vspace{-12pt}
\begin{equation} \label{eqn: blup}
\hat{Z}_n(\bftheta) = \bfr_n^\top(\bftheta) \bfR^{-1}_n(\bftheta) \bfZ_n,
\vspace{-12pt}
\end{equation} 
where $\bfr_n(\bftheta) :=[R(|\bfs_0-\bfs_i|; \bftheta)]_{i=1,\ldots, n}$ is an $n$-dimensional vector. This predictor depends only on correlation parameters $\{\alpha, \beta\}$. If the true covariance is $C(h; \nu, \alpha_0, \beta_0, \sigma^2_0)$, the mean squared error of the predictor in Equation~\eqref{eqn: blup} is given by 
$
\text{Var}_{\nu, \bftheta_0, \sigma^2_0} \{ \hat{Z}_n(\bftheta) - Z(\bfs_0) \} =
\sigma^2_0\left \{ 1 -2 \bfr_n^\top( \bftheta) \bfR_n^{-1}(\bftheta) \bfr_n(\bftheta_0) \right. \\
\left. + \bfr_n^\top(\bftheta) \bfR_n^{-1}(\bftheta) \bfR_n(\bftheta_0) \bfR_n^{-1}(\bftheta) \bfr_n(\bftheta) 
\right\}. 
$
If $\bftheta=\bftheta_0$, i.e., $\alpha=\alpha_0$ and $\beta=\beta_0$, the above expression simplifies to 
\vspace{-10pt}
\begin{equation}\label{eqn: MSE}
\text{Var}_{\nu, \bftheta_0, \sigma^2_0} \{ \hat{Z}_n(\bftheta_0) - Z(\bfs_0) \} = \sigma^2_0\left\{1 -\bfr_n^\top(\bftheta_0) \bfR_n^{-1}(\bftheta_0) \bfr_n(\bftheta_0)   \right\}. \vspace{-10pt}
\end{equation}
If the true model is $\mathcal{M}(h; \nu, \phi, \sigma^2)$, analogous expressions can be derived for $\text{Var}_{\nu, \phi_0, \sigma^2_0} \{\hat{Z}_n(\bftheta) - Z(\bfs_0) \}$.

Let $\mathcal{P}_0$ be the Gaussian measure defined under the true covariance model and $\mathcal{P}_1$ be the Gaussian measure defined under the misspecified covariance model.  The following results concern the asymptotic equivalence between the best linear predictor (BLP) under a misspecified probability measure $\mathcal{P}_1$ and the BLP under the the true measure $\mathcal{P}_0$.

\begin{theorem}\label{thm: prediction efficiency with new covariance}
 Suppose that $\mathcal{P}_0, \mathcal{P}_1$ are two Gaussian probability measures defined by a zero mean Gaussian process with the CH class $C(h; \nu, \alpha_i, \beta_i, \sigma^2_i)$ for $i=1,2$ on $\mathcal{D}$. The following results hold true:
\begin{itemize}[itemsep=-2pt, topsep=0pt, partopsep=0pt]
\item[(a)] For any fixed $\bftheta_1$, under $\mathcal{P}_0$, as $n\to \infty$, 
$$
\frac{\text{Var}_{\nu, \bftheta_0, \sigma^2_0} \{ \hat{Z}_n(\bftheta_1) - Z(\bfs_0) \} }{\text{Var}_{\nu, \bftheta_0, \sigma^2_0} \{ \hat{Z}_n(\bftheta_0) - Z(\bfs_0) \} }  \to 1.
$$ 
\item[(b)] Moreover, if $\sigma^2_0\beta_0^{-2\nu} \Gamma(\nu+\alpha_0)/\Gamma(\alpha_0) = \sigma^2_1\beta_1^{-2\nu} \Gamma(\nu+\alpha_1)/\Gamma(\alpha_1)$, then under $\mathcal{P}_0$, as $n\to \infty$, \\
$$
\frac{\text{Var}_{\nu, \bftheta_1, \sigma^2_1} \{ \hat{Z}_n(\bftheta_1) - Z(\bfs_0) \} }{\text{Var}_{\nu, \bftheta_0, \sigma^2_0} \{ \hat{Z}_n(\bftheta_1) - Z(\bfs_0) \} }  \to 1.
$$
\item[(c)] Let $\hat{\sigma}^2_n = \bfZ_n^\top \bfR_n^{-1}(\bftheta_1)\bfZ_n/n$. It follows that almost surely under $\mathcal{P}_0$, as $n\to \infty$,
\vspace{-6pt}
$$
\frac{\text{Var}_{\nu, \bftheta_1, \hat{\sigma}^2_n} \{ \hat{Z}_n(\bftheta_1) - Z(\bfs_0) \} }{\text{Var}_{\nu, \bftheta_0, \sigma^2_0} \{ \hat{Z}_n(\bftheta_1) - Z(\bfs_0) \} }  \to 1. \vspace{-4pt}
$$
\end{itemize}
\end{theorem} 

Part (a) of Theorem~\ref{thm: prediction efficiency with new covariance} implies that if the smoothness parameter $\nu$ is correctly specified, any values for $\alpha$ and $\beta$ will result in asymptotically efficient predictors. The condition $\sigma^2_0\beta_0^{-2\nu} \Gamma(\nu+\alpha_0)/\Gamma(\alpha_0) = \sigma^2_1\beta_1^{-2\nu} \Gamma(\nu+\alpha_1)/\Gamma(\alpha_1)$ is not necessary for asymptotic efficiency, but it provides asymptotically correct estimate of the mean squared prediction error (MSPE). The quantity $\text{Var}_{\nu, \bftheta_1, \sigma^2_1}$ $\{ \hat{Z}_n(\bftheta_1)$ $- Z(\bfs_0) \}$ is the MSPE for $\hat{Z}_n(\bftheta_1)$ under the model $C(h; \nu, \alpha_1, \beta_1, \sigma^2_1)$, while the quantity $\text{Var}_{\nu, \bftheta_0, \sigma^2_0} \{ \hat{Z}_n(\bftheta_1) - Z(\bfs_0) \}$ is the true MSPE for $\hat{Z}_n(\bftheta_1)$ under the true model $C(h; \nu, \alpha_0, \beta_0, \sigma^2_0)$. In practice, it is common to estimate model parameters and then prediction is made by plugging these estimates into Equations \eqref{eqn: blup} and \eqref{eqn: MSE}. Part (c) shows the same convergence results when $\bftheta$ is fixed at $\bftheta_1$, but $\sigma^2$ is estimated via the maximum likelihood method. 


One can conjecture that the result in Part (c) of Theorem~\ref{thm: prediction efficiency with new covariance} still holds if $\bftheta_1$ is replaced by its maximum likelihood estimator, but its proof seems elusive. Theorem~\ref{thm: prediction efficiency with new covariance} demonstrates the asymptotic prediction efficiency for the CH class. The following results are established to show the asymptotic efficiency of the best linear predictor under the CH class when the true Gaussian measure is defined by a zero-mean Gaussian process under the Mat\'ern class.

\begin{theorem}\label{thm: PE with Matern covariance}
Let $\mathcal{P}_0$ be the Gaussian probability measure under the Mat\'ern covariance $\mathcal{M}(h; \nu, \phi, \sigma^2_0)$ and $\mathcal{P}_1$ be the Gaussian probability measure under the CH covariance $C(h; \nu, \alpha, \beta, \sigma^2_1)$ on $\mathcal{D}$. $\hat{Z}_n(\alpha, \beta)$ be the kriging predictor under $C(h; \nu, \alpha, \beta, \sigma^2_1)$ and $\hat{Z}_n(\phi)$ be the kriging predictor under $\mathcal{M}(h; \nu, \phi, \sigma^2_0)$.  If the condition in Equation~\eqref{eqn: equivalence with matern} is satisfied, then it follows that under the Gaussian measure $\mathcal{P}_0$,  as $n\to \infty$, 
\vspace{-1pt}
\begin{equation*}
\frac{\text{Var}_{\nu, \alpha, \beta, \sigma^2_1} \{ \hat{Z}_n(\alpha, \beta) - Z(\bfs_0) \} }{\text{Var}_{\nu, \phi, \sigma^2_0} \{ \hat{Z}_n(\phi) - Z(\bfs_0) \} }  \to 1,
\end{equation*} 
for any fixed $\alpha>0$ and  $\beta>0$.
\end{theorem}

A key consequence of Theorem~\ref{thm: PE with Matern covariance}  is that when a true Gaussian process is generated by the Mat\'ern covariance model, the CH covariance model~\eqref{eqn: new kernel} can yield an asymptotically equivalent BLP. The practical implication is when the true model is generated from the Mat\'ern class, the predictive performance under the CH class is indistinguishable from that under the Mat\'ern class as the number of observations gets larger in a fixed domain. Both Theorem~\ref{thm: prediction efficiency with new covariance} and Theorem~\ref{thm: PE with Matern covariance} imply that the kriging predictor under the CH class can allow robust prediction property even if the underlying true covariance model is misspecified.

\section{Numerical Illustrations}\label{sec:numerical}
In this section, we use simulated examples to study the properties of the CH class and compare with alternative covariance models. In what follows, we compare the CH model with the other two covariance models:  the Mat\'ern class and the generalized Cauchy class. The predictive performance is evaluated based on root mean-squared prediction errors (RMSPE), coverage probability of the 95\% percentile confidence intervals (CVG), and the average length of the predictive confidence intervals (ALCI) at held-out locations.

The goal of this section is to study the finite sample predictive performance under the CH model in interpolative settings. Specifically, we consider three different cases, where the true covariance model is specified as the Mat\'ern covariance (Case 1), the CH covariance (Case 2) and the generalized Cauchy (GC) covariance (Case 3), respectively. The Mat\'ern class is very flexible near origin and has an exponentially decaying tail, the CH class is also very flexible near origin but has a polynomially decaying tail, and the GC class is either non-differentiable or infinitely differentiable and has a polynomially decaying tail. The GC covariance has the form $C(h)=\sigma^2\left\{ 1 + (h/\phi)^{\delta} \right\}^{-\lambda/\delta}$, where $\sigma^2>0$ is the variance parameter, $\phi>0$ is the range parameter, $\lambda \in (0, d]$ is the parameter controlling the degree of polynomial decay, and $\delta \in (0, 2]$ is the smoothness parameter. When $\delta\in (0, 2)$, the corresponding process is nowhere mean-square differentiable. When $\delta=2$, it corresponds to the Cauchy covariance, whose process is infinitely mean-square differentiable. For each case, predictive performance is compared at held-out locations with estimated covariance structures.

We simulate data in the square domain $\mathcal{D}=[0, 2000]\times [0, 2000]$ from mean zero Gaussian processes with three different covariance models: the Mat\'ern covariance (Case 1), the CH covariance (Case 2), and the GC covariance (Case 3) for a variety of settings. We simulate $n=2000$ data points via maximin Latin hypercube design \citep{Stein1987} for parameter estimation and evaluate predictive performance at 10-by-10 regular grid points in $\mathcal{D}$. 
We fix the variance parameter at 1 and consider moderate spatial dependence with effective range (ER) at 200 and 500 for the underlying true covariances. For each of these simulation settings, we use 30 different random number seeds to generate the realizations. We always choose the same smoothness parameter for the Mat\'ern class and the CH class. For the GC covariance, we fix its smoothness parameter to be $\delta=\min\{2\nu, d\}$, since the Gaussian measure with the Mat\'ern class could be equivalent to that with the GC class as pointed by \cite{Bevilacqua2019}. However, the smoothness parameter $\delta$ in the GC class cannot be greater than 2, otherwise the GC class is no longer a valid covariance function. 

\subsection{Case 1: Examples with the Mat\'ern Class as Truth} \label{sec: case1}

In Case 1, we simulate Gaussian process realizations from the Mat\'ern model with smoothness parameter $\nu$ fixed at 0.5 and 2.5 and effective range at 200 and 500.  The parameters in each covariance model are estimated based on profile likelihood as described in Section~\ref{subsec: asymptotic normality}. Figure~\ref{fig: simulation setting under case 1} shows the estimated covariance structures and summary of prediction results. Regardless of the smoothness behavior and strength of dependence in the underlying true process, there is no clear difference between the CH class and the Mat\'ern class in terms of estimated covariance structures and prediction performance. In contrast, the estimated GC covariance structure only performs as accurately as the Mat\'ern class when $\nu=0.5$. When the process is twice mean-square differentiable ($\nu=2.5$), as expected, the GC class cannot mimic such behavior, and hence, yields worse estimates of the covariance structures and prediction results compared to both the Mat\'ern class and the CH class. The CH covariance is able to capture the true covariance structure as implied by Theorem~\ref{thm: PE with Matern covariance}. In terms of RMSPE, there is no clear difference between the estimated CH covariance and the estimated Mat\'ern covariance. However, the CVG and ALCI based on the CH class are slightly larger than those based on the estimated Mat\'ern covariance.

\begin{figure}[!t] 
\begin{subfigure}{.35\textwidth}
  \centering
\makebox[\textwidth][c]{ \includegraphics[width=1.0\linewidth, height=0.12\textheight]{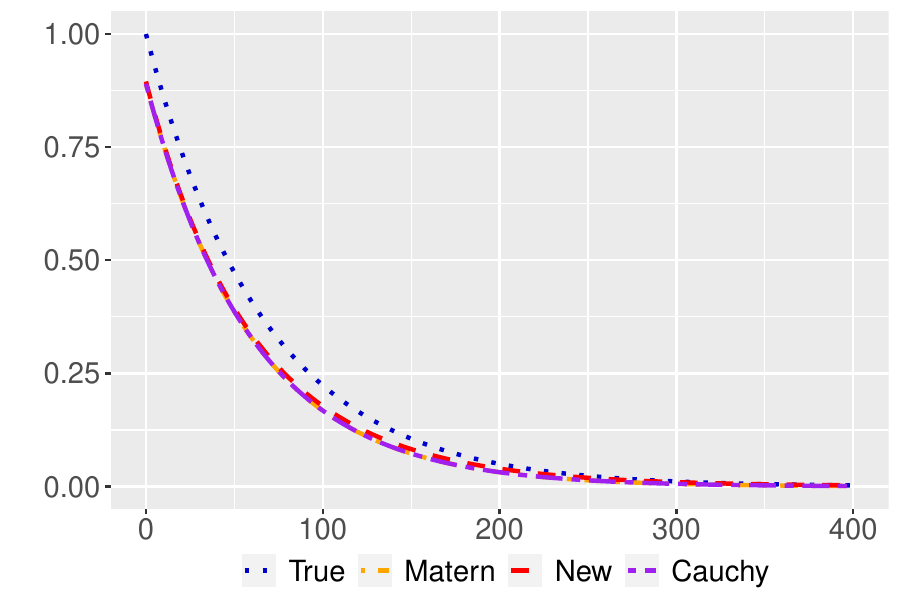}}
\end{subfigure}%
\begin{subfigure}{.65\textwidth}
  \centering
\makebox[\textwidth][c]{ \includegraphics[width=1.0\linewidth, height=0.12\textheight]{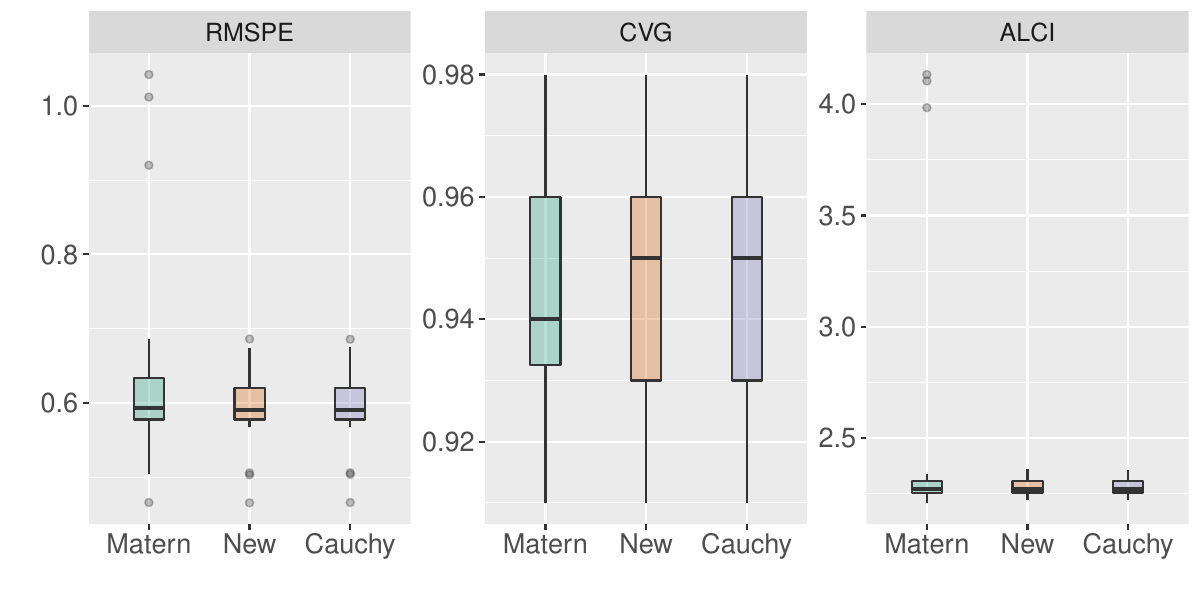}}
\end{subfigure}
\caption*{$\nu=0.5, ER=200$}

\begin{subfigure}{.35\textwidth}
  \centering
\makebox[\textwidth][c]{ \includegraphics[width=1.0\linewidth, height=0.12\textheight]{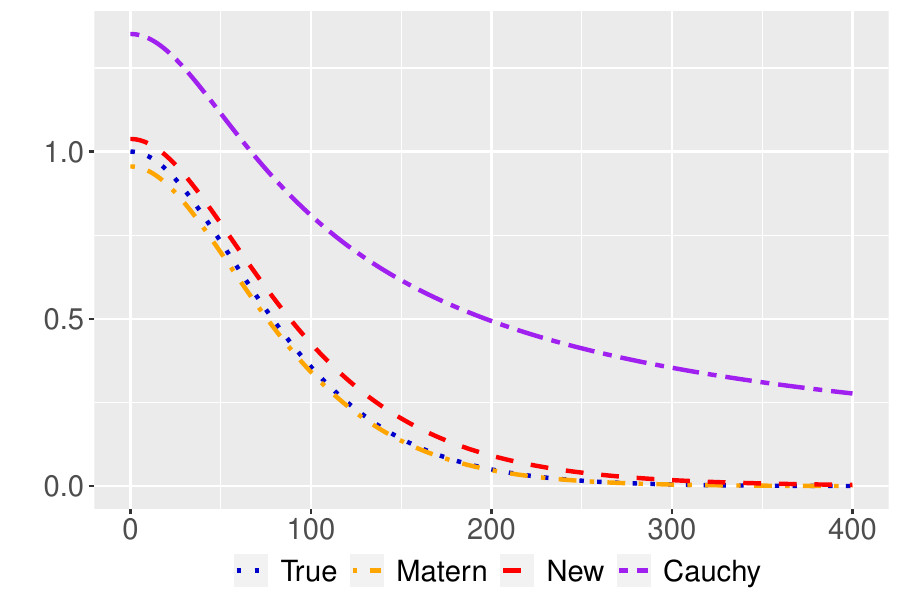}}
\end{subfigure}%
\begin{subfigure}{.65\textwidth}
  \centering
\makebox[\textwidth][c]{ \includegraphics[width=1.0\linewidth, height=0.12\textheight]{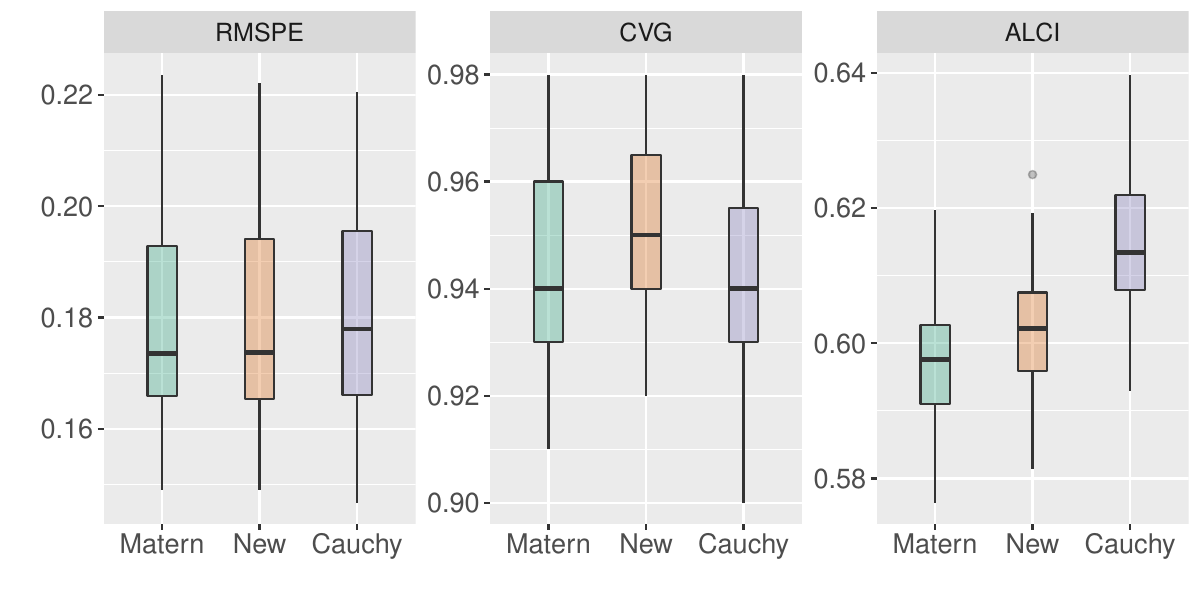}}
\end{subfigure}
\caption*{$\nu=2.5, ER=200$}

\begin{subfigure}{.35\textwidth}
  \centering
\makebox[\textwidth][c]{ \includegraphics[width=1.0\linewidth, height=0.12\textheight]{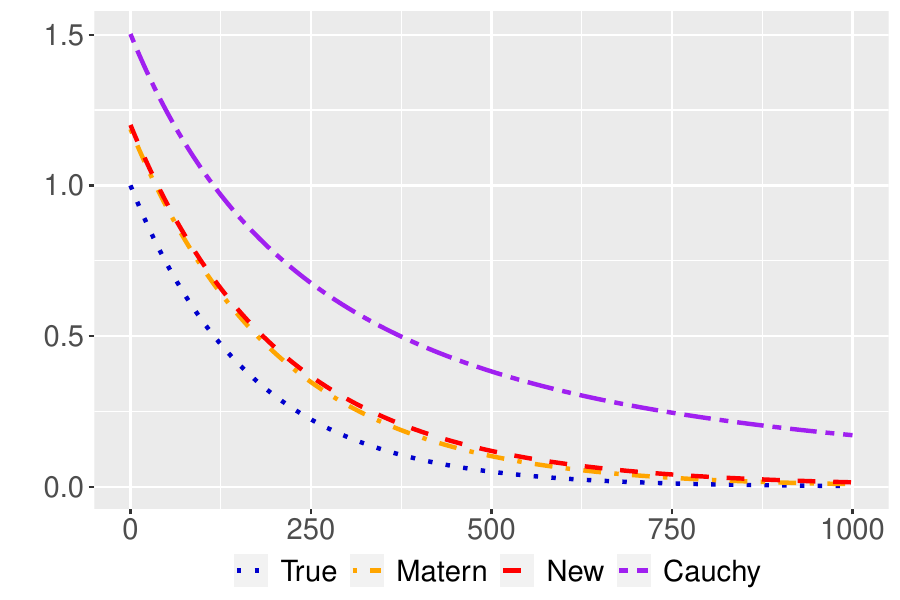}}
\end{subfigure}%
\begin{subfigure}{.65\textwidth}
  \centering
\makebox[\textwidth][c]{ \includegraphics[width=1.0\linewidth, height=0.12\textheight]{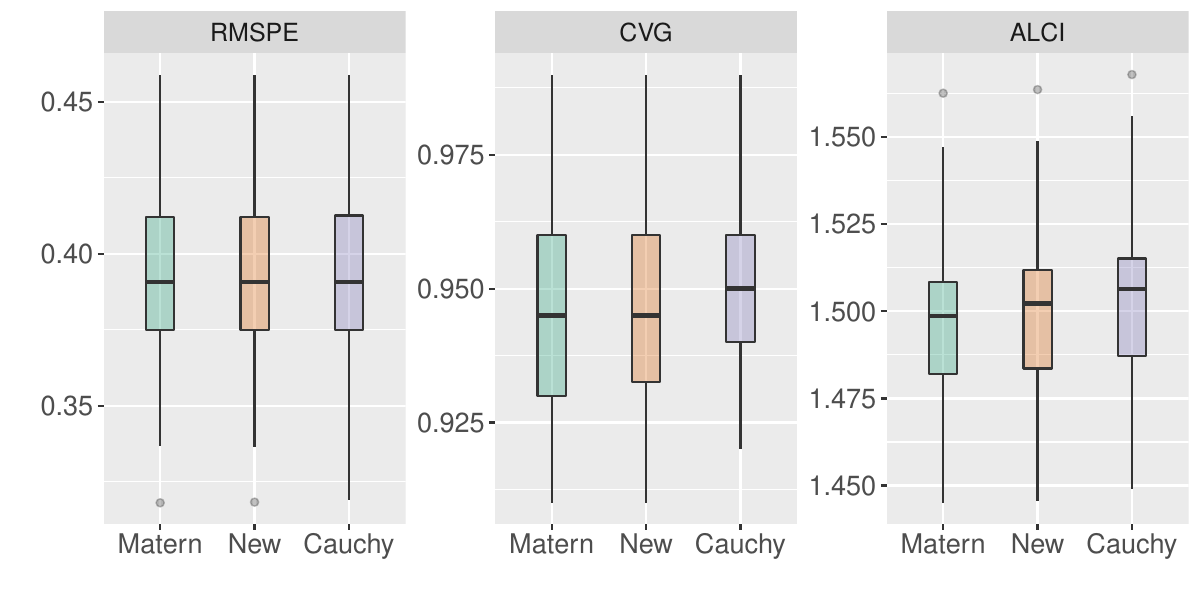}}
\end{subfigure}
\caption*{$\nu=0.5, ER=500$}

\begin{subfigure}{.35\textwidth}
  \centering
\makebox[\textwidth][c]{ \includegraphics[width=1.0\linewidth, height=0.15\textheight]{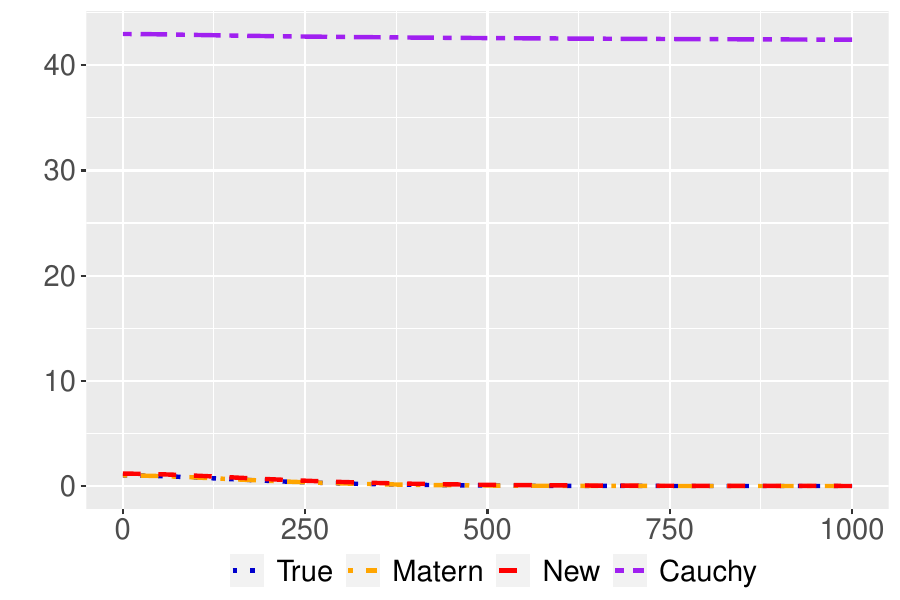}}
\end{subfigure}%
\begin{subfigure}{.65\textwidth}
  \centering
\makebox[\textwidth][c]{ \includegraphics[width=1.0\linewidth, height=0.15\textheight]{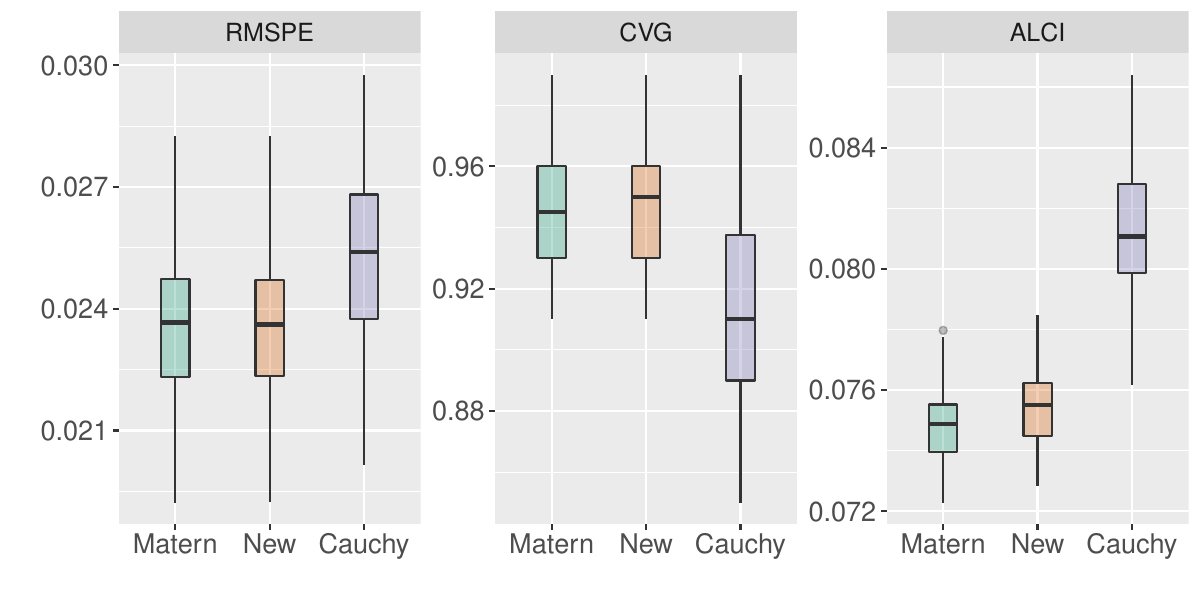}}
\end{subfigure}
\caption*{$\nu=2.5, ER=500$}

\caption{Case 1: Comparison of predictive performance and estimated covariance structures when the true covariance is the Mat\'ern class with 2000 observations. The predictive performance is evaluated at 10-by-10 regular grids in the square domain. These figures summarize the predictive measures based on RMSPE, CVG and ALCI under 30 simulated realizations.}
\label{fig: simulation setting under case 1}
\end{figure}

\subsection{Case 2: Examples with the CH Class as Truth} \label{sec: case 2}

In Case 2, we simulate Gaussian process realizations from the CH covariance model with smoothness parameter $\nu$ fixed at 0.5 and 2.5, tail decay parameter fixed at 0.5, and effective range fixed at 200 and 500. Figure~\ref{fig: simulation setting under case 2} shows the estimated covariance structures and summary of prediction results. As expected, when the underlying true process is simulated from a process with polynomially decaying dependence, the Mat\'ern class cannot be expected to capture such behavior. The prediction results also indicate that the Mat\'ern class performs much worse than the other two covariance models. When the underlying true process is not differentiable ($\nu=0.5$), there is no clear difference between the estimates under the GC covariance structure and the estimates under the CH covariance structure. However, when the underlying true process is twice differentiable ($\nu=2.5$), it is obvious that the estimated GC covariance structure is not as accurate as the estimated CH covariance structure. This makes sense because the GC class is either non-differentiable or infinitely differentiable. In terms of prediction performance, the CH covariance class performs better than the GC class in terms of coverage probability. 

\begin{figure}[!t] 
\begin{subfigure}{.35\textwidth}
  \centering
\makebox[\textwidth][c]{ \includegraphics[width=1.0\linewidth, height=0.12\textheight]{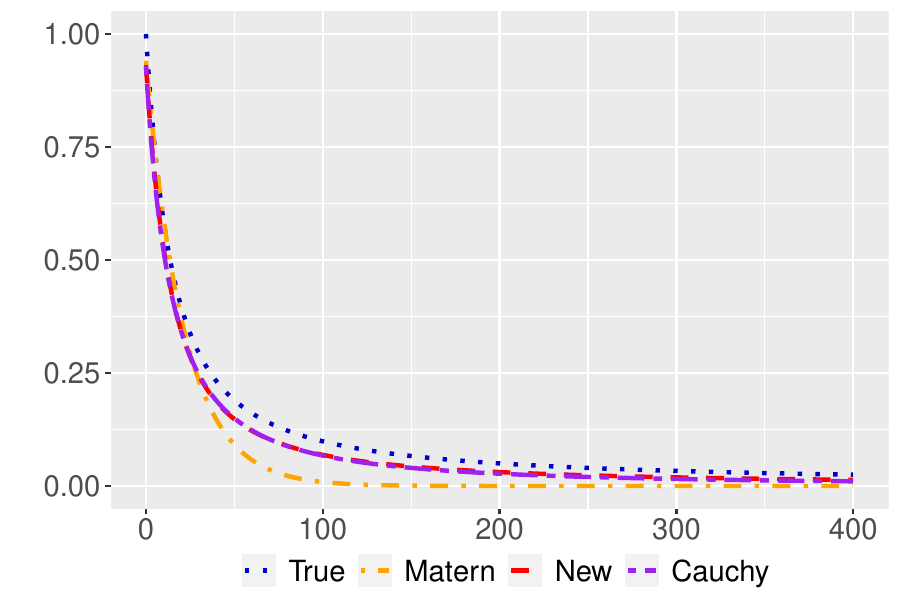}}
\end{subfigure}%
\begin{subfigure}{.65\textwidth}
  \centering
\makebox[\textwidth][c]{ \includegraphics[width=1.0\linewidth, height=0.12\textheight]{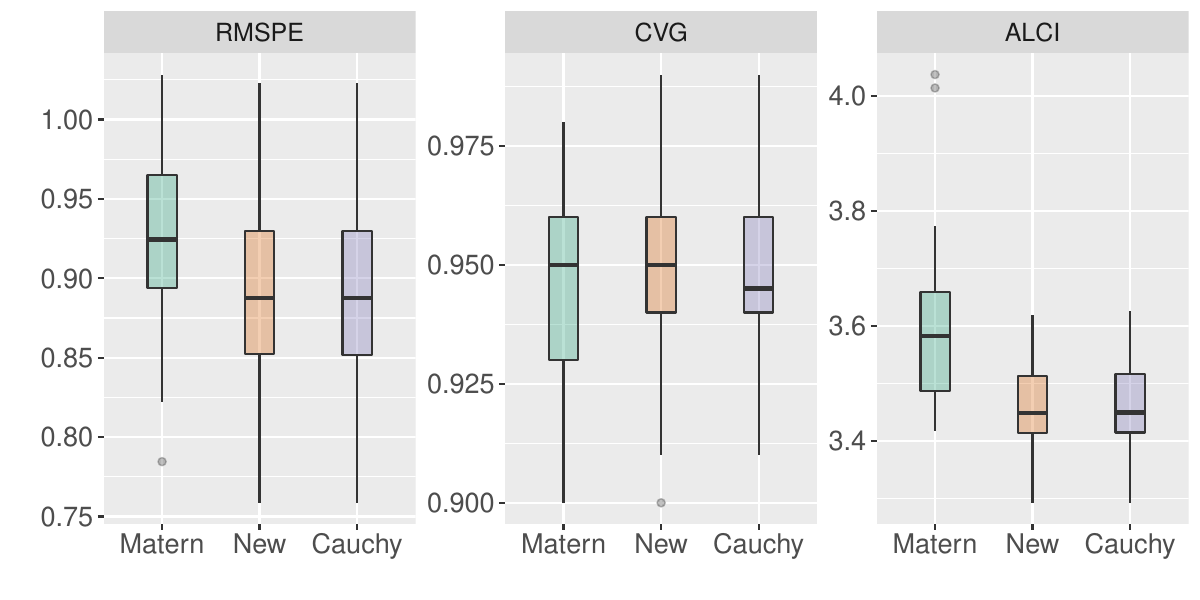}}
\end{subfigure}
\caption*{$\nu=0.5, ER=200$}

\begin{subfigure}{.35\textwidth}
  \centering
\makebox[\textwidth][c]{ \includegraphics[width=1.0\linewidth, height=0.12\textheight]{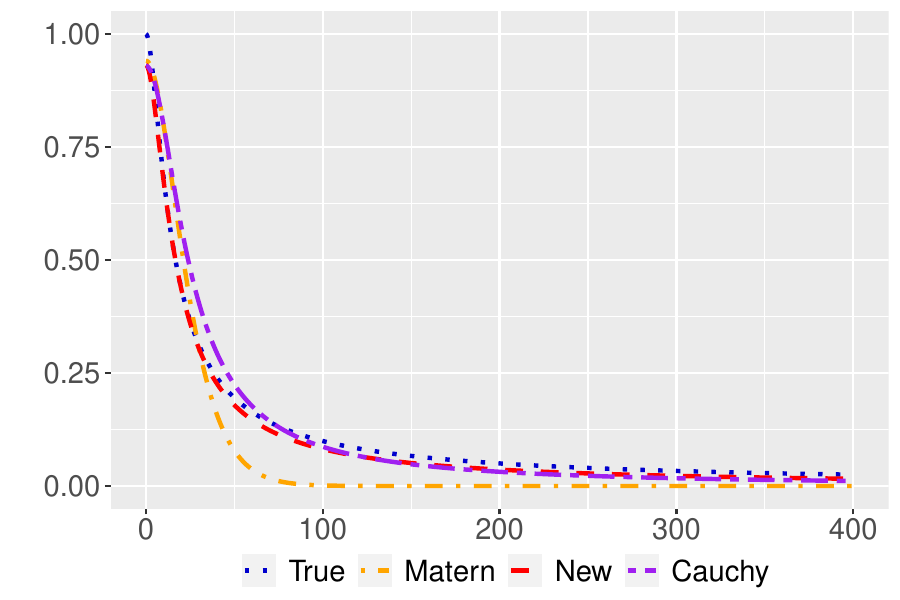}}
\end{subfigure}%
\begin{subfigure}{.65\textwidth}
  \centering
\makebox[\textwidth][c]{ \includegraphics[width=1.0\linewidth, height=0.12\textheight]{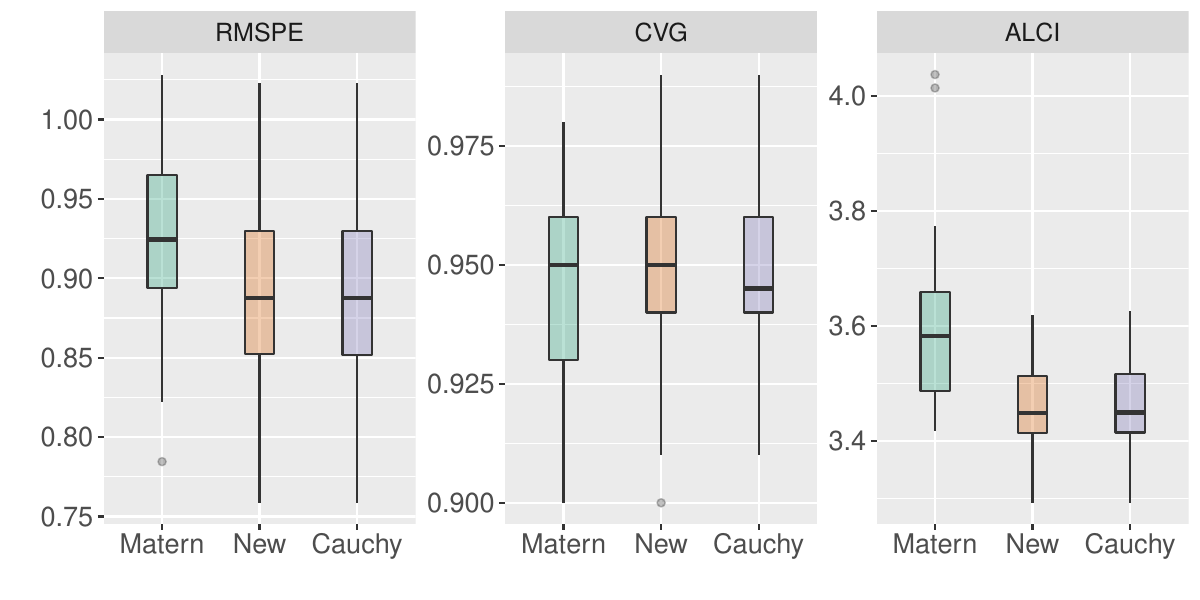}}
\end{subfigure}
\caption*{$\nu=2.5, ER=200$}

\begin{subfigure}{.35\textwidth}
  \centering
\makebox[\textwidth][c]{ \includegraphics[width=1.0\linewidth, height=0.12\textheight]{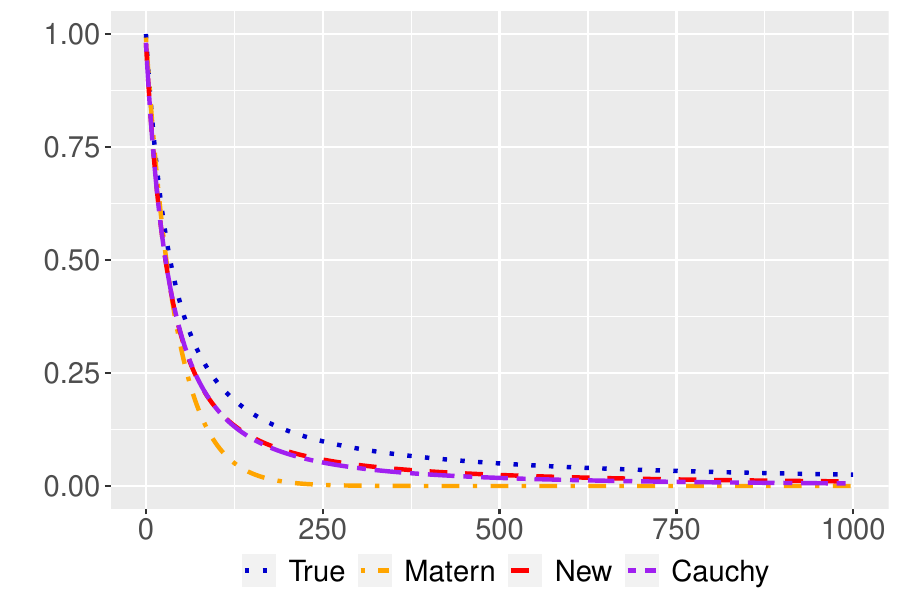}}
\end{subfigure}%
\begin{subfigure}{.65\textwidth}
  \centering
\makebox[\textwidth][c]{ \includegraphics[width=1.0\linewidth, height=0.12\textheight]{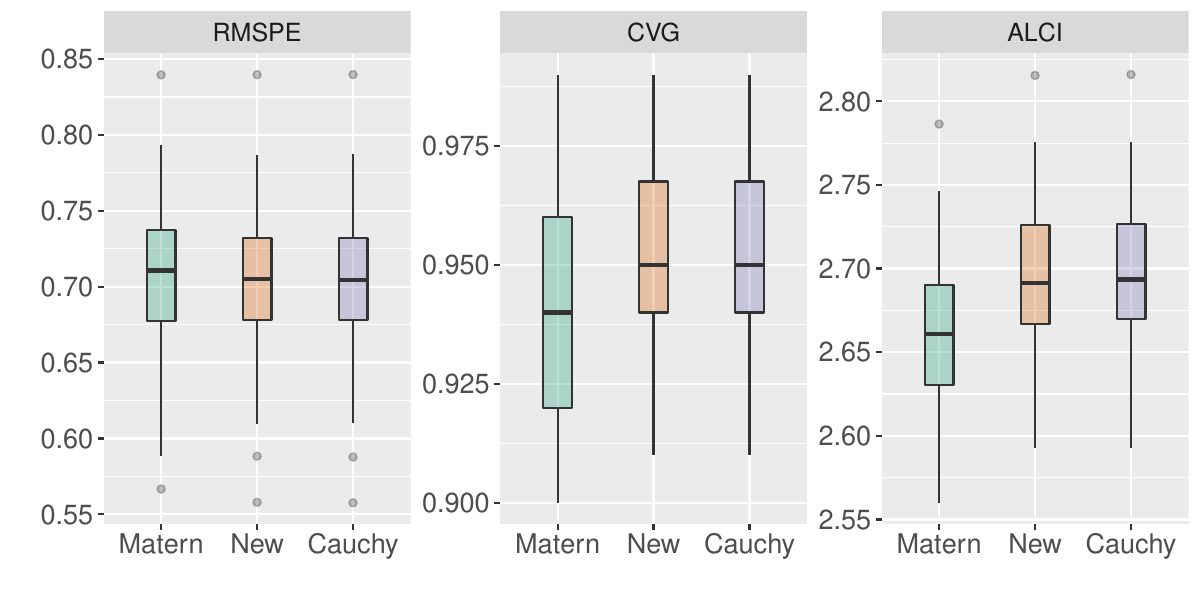}}
\end{subfigure}
\caption*{$\nu=0.5,  ER=500$}

\begin{subfigure}{.35\textwidth}
  \centering
\makebox[\textwidth][c]{ \includegraphics[width=1.0\linewidth, height=0.12\textheight]{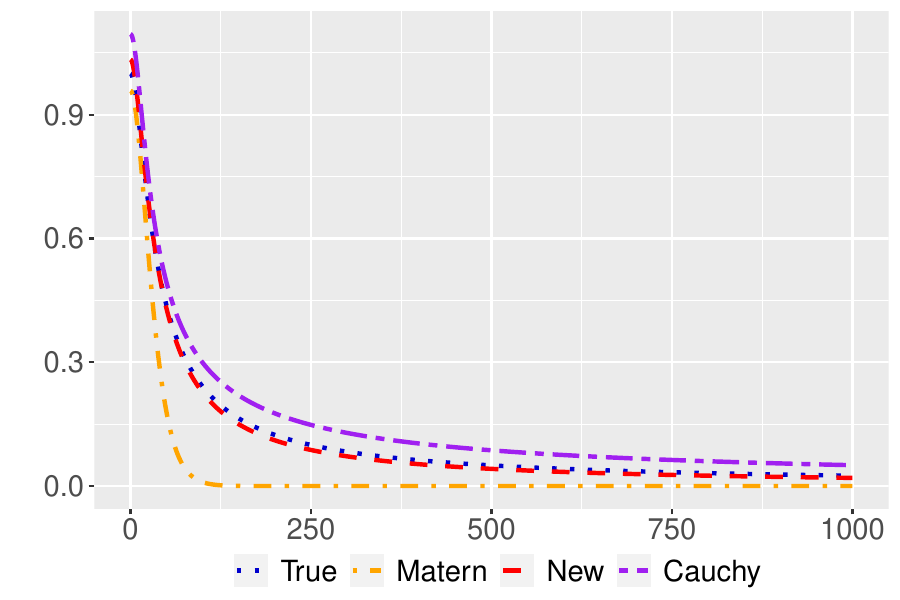}}
\end{subfigure}%
\begin{subfigure}{.65\textwidth}
  \centering
\makebox[\textwidth][c]{ \includegraphics[width=1.0\linewidth, height=0.12\textheight]{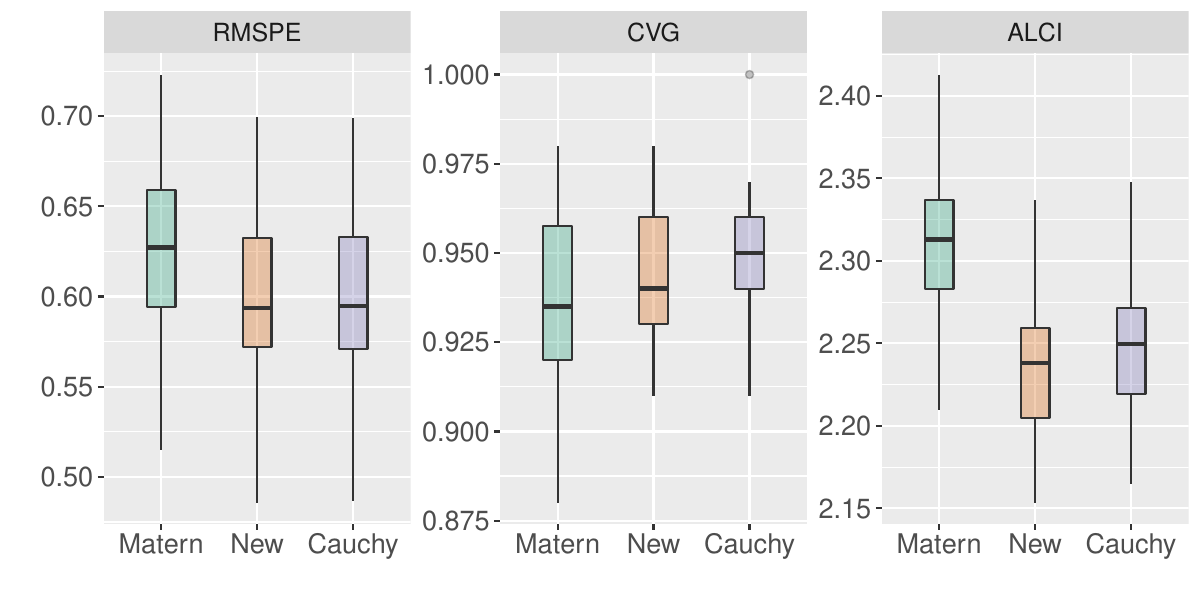}}
\end{subfigure}
\caption*{$\nu=2.5,  ER=500$}

\caption{Case 2: Comparison of predictive performance and estimated covariance structures when the true covariance is the CH class with 2000 observations. The predictive performance is evaluated at 10-by-10 regular grids in the square domain. These figures summarize the predictive measures based on RMSPE, CVG and ALCI under 30 simulated realizations.}
\label{fig: simulation setting under case 2}
\end{figure}

\subsection{Case 3: Examples with the GC Class as Truth} \label{sec: case3}

In Case 3, we simulate Gaussian process realizations from the GC class with the smoothness parameter $\delta=1$ and $\lambda=1$ under ER=200 and 500.  The corresponding process is non-differentiable and corresponds to the smoothness parameter $\nu=0.5$ in both the Mat\'ern class and the CH class. The parameter $\lambda$ in the GC class is fixed at 1 so that it corresponds to the tail parameter $\alpha=0.5$ in the CH class. We did not consider Gaussian processes that are infinitely differentiable, since such processes are unrealistic for environmental processes. Figure~\ref{fig: simulation setting under case 3} shows the estimated covariance structures and prediction results. As expected, the Mat\'ern class performs much worse than the CH class and the GC class for the same reason as in Case 2. Between the CH class and the GC class, no difference is seen in terms of estimated covariance structures and predictive performances. This is not surprising, since the CH class has a tail decay parameter $\alpha$ that is able to capture the tail behavior in the GC class. 

\begin{figure}[!t] 
\begin{subfigure}{0.35\textwidth}
\centering
\makebox[\textwidth][c]{ \includegraphics[width=\linewidth, height=0.12\textheight]{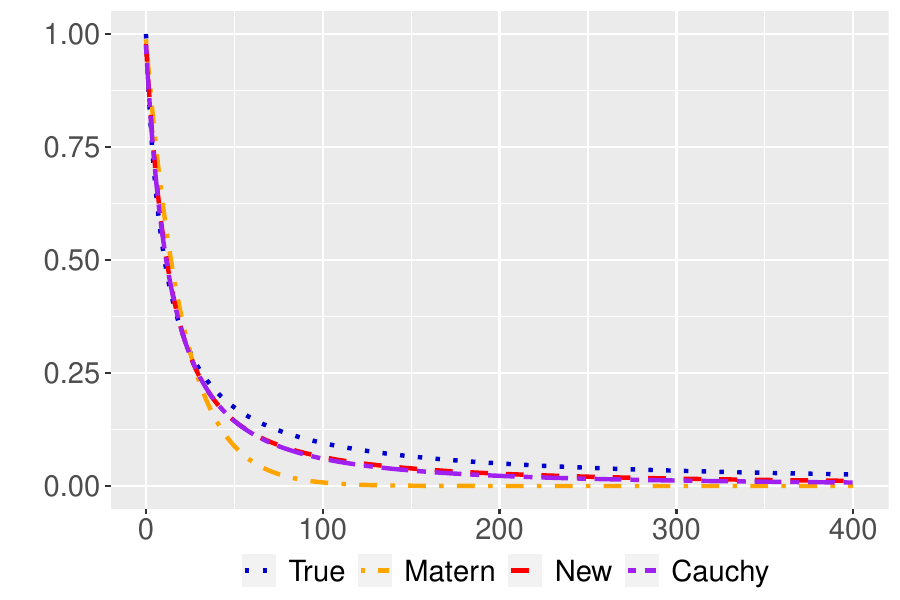}}
\end{subfigure}%
\begin{subfigure}{0.65\textwidth}
\makebox[\textwidth][c]{ \includegraphics[width=\linewidth, height=0.12\textheight]{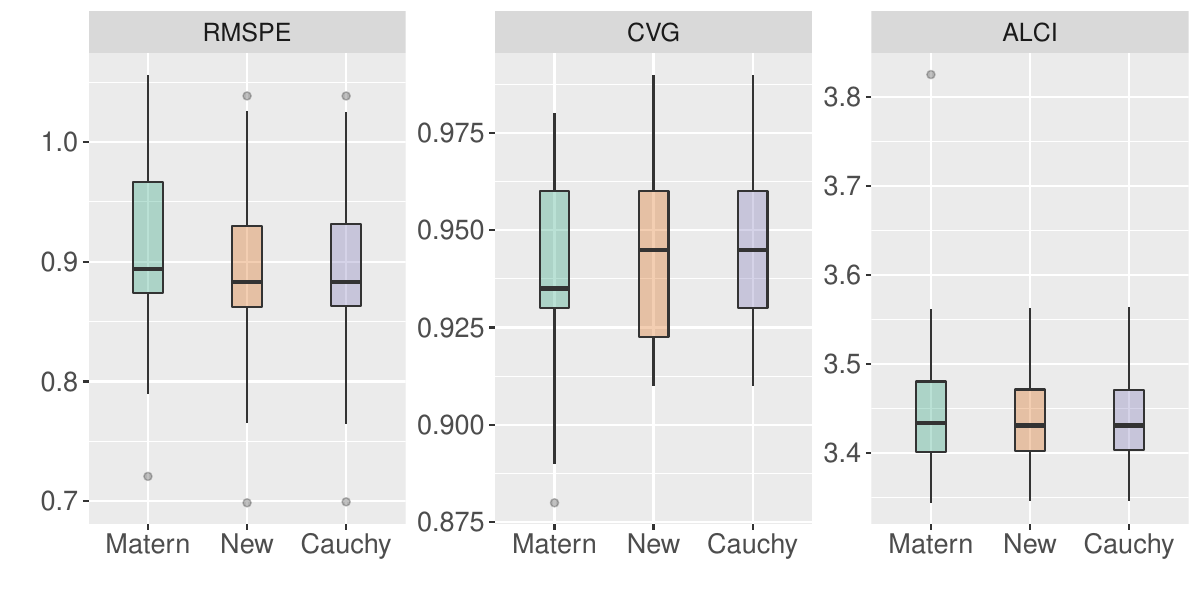}}
\end{subfigure}
\caption*{$\delta=1,  ER=200$}

\begin{subfigure}{0.35\textwidth}
  \centering
\makebox[\textwidth][c]{ \includegraphics[width=\linewidth, height=0.12\textheight]{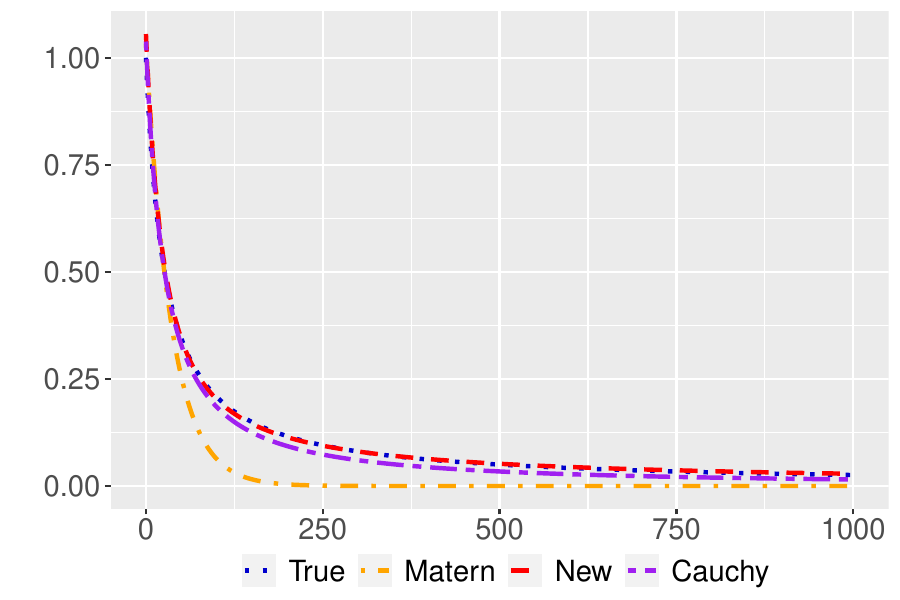}}
\end{subfigure}%
\begin{subfigure}{0.65\textwidth}
\makebox[\textwidth][c]{ \includegraphics[width=\linewidth, height=0.12\textheight]{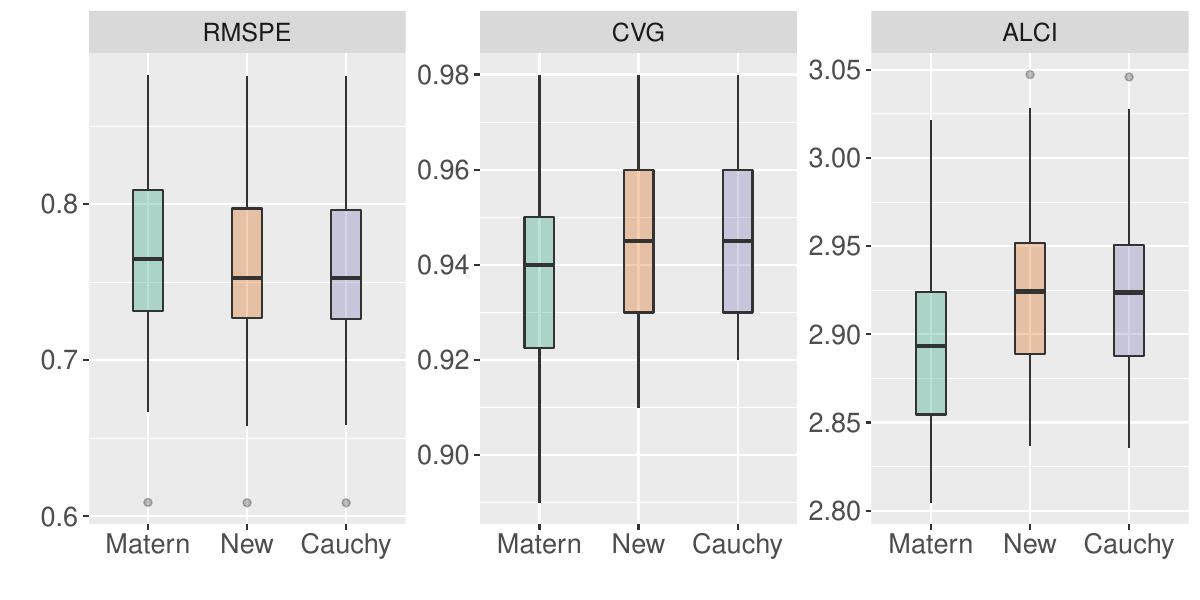}}
\end{subfigure}
\caption*{$\delta=1,  ER=500$}

\caption{Case 3: Comparison of predictive performance and estimated covariance structures when the true covariance is the GC class with 2000 observations. The predictive performance is evaluated at 10-by-10 regular grids in the square domain. These figures summarize the predictive measures based on RMSPE, CVG and ALCI under 30 simulated realizations.}
\label{fig: simulation setting under case 3}
\end{figure}

\section{Application to the OCO-2 Data}\label{sec:real}
In this section, the proposed CH class is used to model spatial data collected from NASA's Orbiting Carbon Observatory-2 (OCO-2) satellite and comparisons are made in kriging performances with alternative covariances. The OCO-2 satellite is NASA's first dedicated remote sensing earth satellite to study atmospheric carbon dioxide from space with the primary objective to estimate the global geographic distribution of CO$_2$ sources and sinks at Earth's surface; see \cite{Cressie2017, Wunch2011} for detailed discussions. The OCO-2 satellite carries three high-resolution grating spectrometers designed to measure the near-infrared absorption of reflected sunlight by carbon dioxide and molecular oxygen and orbits over a 16-day repeat cycle. In this application, we consider NASA's Level 3 data product of the XCO2 at $0.25^\circ \times 0.25^\circ$ spatial resolution over one repeat cycle from June 1 to June 16, 2019. These gridded data were processed based on Level 2 data product by the OCO-2 project at the Jet Propulsion Laboratory, California Technology, and obtained from the OCO-2 data archive maintained at the NASA Goddard Earth Science Data and Information Services Center. They can be downloaded at \url{https://co2.jpl.nasa.gov/#mission=OCO-2}. 

This Level 3 data product consists of 43,698 measurements. We focus on the study region that covers the entire United States with longitudes between 140W and 50W and latitudes between 15N and 60N. This region includes 3,682 measurements; see panel (a) of Figure~\ref{fig: OCO-2 data}. These data points are very sparse in space. As the OCO-2 satellite has swath width 10.6 kilometers, large missing gaps can be observed between swaths. Predicting the underlying geophysical process based on data with such patterns requires the statistical model not only to interpolate in space (prediction near observed locations) but also to extrapolate in space (prediction away from observed locations). 

\begin{figure}[!t] 

\begin{subfigure}{.5\textwidth}
  \centering
\makebox[\textwidth][c]{ \includegraphics[width=1.0\linewidth, height=0.15\textheight]{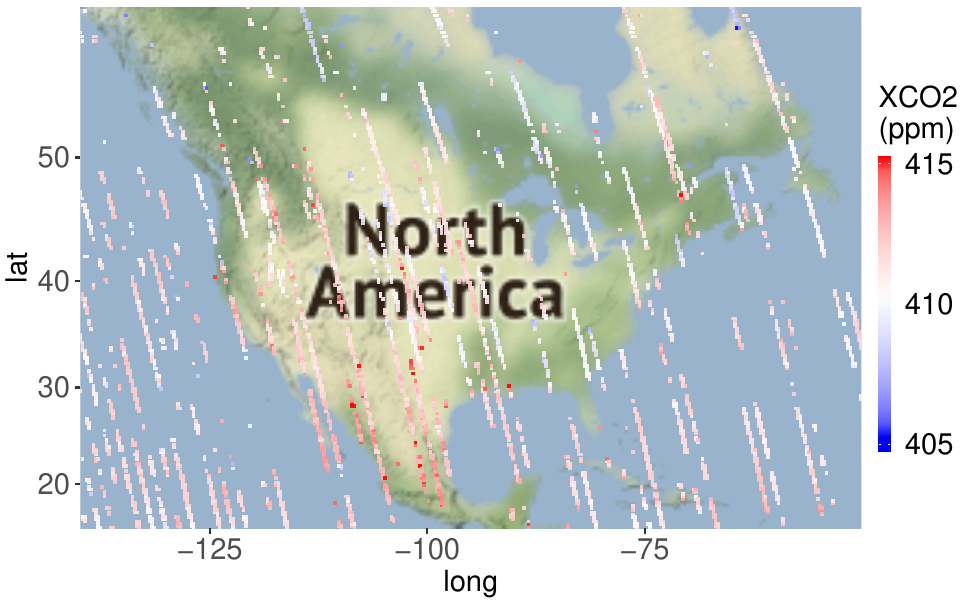}}
\caption{XCO2 data in the study region.}
\end{subfigure}%
\begin{subfigure}{.5\textwidth}
\makebox[\textwidth][c]{ \includegraphics[width=1\linewidth, height=0.15\textheight]{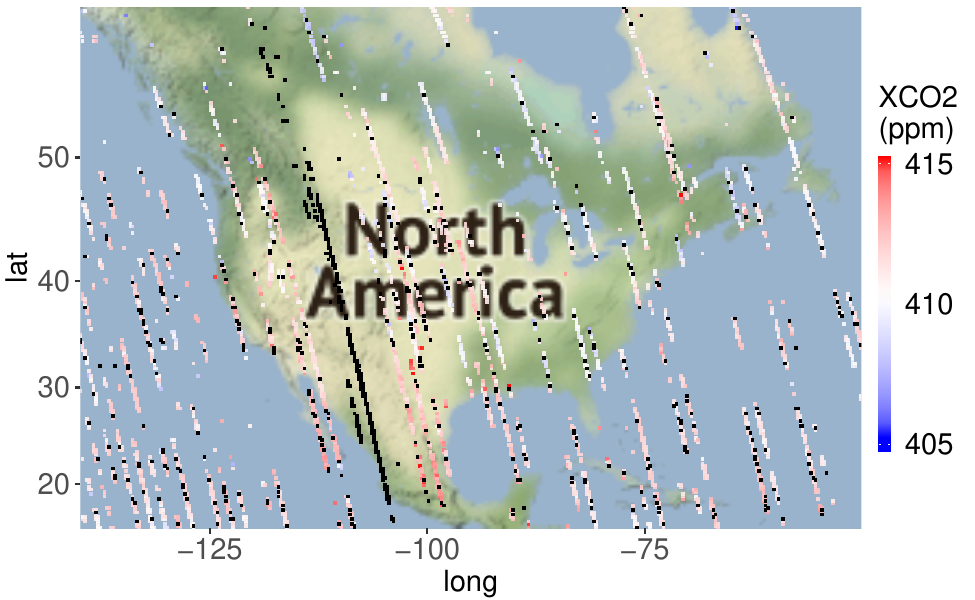}}
\caption{XCO2 testing data in black.}
\end{subfigure}

\caption{XCO2 measurements from June 1 to June 16, 2019 in the study region.}
\label{fig: OCO-2 data}
\end{figure}

Given the data $\bfZ:=(Z(\bfs_1), \ldots, Z(\bfs_n))^\top$, we assume a typical spatial process model:
\vspace{-12pt}
\begin{equation*} 
Z(\mathbf{s}) = Y(\mathbf{s}) + \epsilon(\mathbf{s}), \quad \bfs \in \mathcal{D}, \vspace{-12pt}
\end{equation*}
where $Y(\cdot)$ is assumed to be a Gaussian process with mean function $\mu(\cdot)$ and covariance function $C(\cdot, \cdot)$. The term $\epsilon(\cdot)$ is assumed to be a spatial white-noise process accounting for the nugget effect with $\text{var}(\epsilon(\mathbf{s}))=\tau^2>0$. The goal of this analysis is to predict the process $Y(\mathbf{s}_0)$ for any $\mathbf{s}_0\in \mathcal{D}$ based on the data $\bfZ$. Exploratory analysis indicates no clear trend, so we assume a constant trend for the mean function $\mu(\bfs)=b$. For this particular dataset, the assumption of an isotropic covariance function seems to be reasonable based on directional semivariograms in Figure~\ref{fig: isotropy} of the Supplementary Material.  For the covariance function $C(\cdot, \cdot)$, we assume the CH model with parameters $\{\sigma^2, \alpha, \beta, \nu \}$, where the smoothness parameter $\nu$ is fixed at $0.5$ and $1.5$, indicating the resulting process is non-differentiable or once differentiable, respectively. Here we fix $\nu$ in the Mat\'ern  and CH classes over a grid of values, {since (a) estimating $\nu$ requires intensive computations and it has often been fixed in practice \citep[e.g.,][]{Berger2001,Banerjee2014} and (b) the likelihood can be nearly flat (\citeauthor{Berger2001}, \citeyear{Berger2001}; \citeauthor{Gu2018}, \citeyear{Gu2018}; \citeauthor{Stein1999}, \citeyear{Stein1999}, p.~173; \citeauthor{Zhang2004}, \citeyear{Zhang2004}) and hence it is notoriously difficult to estimate covariance parameters including $\nu$ with either profile or integrated likelihood functions \citep{Gu2018}. However, estimating $\nu$ may improve prediction accuracy in practice, which is left for future investigation.} 

To evaluate the performance of the CH class, we perform cross-validation and make comparisons with the Mat\'ern class. The testing dataset consists of (1) a complete longitude band across the United States, which will be referred to as missing by design (MBD) and (2) randomly selected 15\% of remaining XCO2 measurements, which will be referred to as missing at random (MAR). Panel (b) of Figure~\ref{fig: OCO-2 data} highlights these testing data with black grid points. This dataset is used for evaluating out-of-sample predictive performance in interpolative and extrapolative settings. The remaining data points are used for parameter estimation under the Mat\'ern covariance and the CH covariance. The parameters are estimated based on the restricted maximum likelihood \citep{Harville1974}. Table~\ref{table: CV for OCO2 data} shows the predictive measures and estimated nugget parameters. The CH model with the smoothness parameter $\nu=0.5$ yields the smallest estimated nugget parameter among all the models. This suggests that the CH model with $\nu=0.5$ best captures the spatial dependence structure among all the models. The kriging predictions under the CH model show lots of fine-scale or micro-scale variations, which are more desired for accurate spatial prediction. In an interpolative setting, the Mat\'ern covariance yields slightly smaller (but indistinguishable) RMSPE and ALCI  over randomly selected locations than the CH covariance, which indicates that the Mat\'ern covariance has slightly better interpolative prediction skill than the CH model in this application. The empirical coverage probability is closer to the nominal value of 0.95 under the Mat\'ern covariance model. In contrast, in an extrapolative setting, the CH model yields much smaller RMSPE and ALCI than the Mat\'ern covariance model with indistinguishable empirical coverage probabilities, which indicates that the CH model has a better extrapolative prediction skill than the Mat\'ern covariance model. These prediction results are not surprising, since the Mat\'ern class can only model exponentially decaying dependence while the CH class can offer considerable benefits for extrapolative predictions while maintains the same interpolative prediction skill as the Mat\'ern class. The difference in interpolative prediction performance between the CH class and the Mat\'ern class is negligible, in part because the CH class can yield asymptotically equivalent best linear predictors as the Mat\'ern class under conditions established in Theorem~\ref{thm: PE with Matern covariance}. Notice that the empirical coverage probabilities under all the models are less than the nominal coverage probability 0.95, this is partly because uncertainties due to parameter estimation are not accounted for in the predictive distribution. A fully Bayesian analysis may remedy this issue. 

For other model parameters shown in Table~\ref{table: CV for OCO2 data parameter estimation} of the Supplementary Material, we notice that the estimates of the regression parameters under the two different covariance models are very similar. As expected, the estimated variance parameter (partial sill) is larger under the CH class than the one estimated under the Mat\'ern class. Perhaps the most interesting parameter is the tail decay parameter in the CH class, which is estimated to be around 0.38. This clearly indicates that the underlying true process has a polynomially decaying dependence structure.  As \cite{Gneiting2013} points out, the Mat\'ern class is positive definite on sphere only if $\nu\leq 0.5$ with great circle distance. To avoid this technical difficulty, we use chordal distance for modeling spatial data on sphere when $\nu>0.5$, since it was  pointed out on pages 71-77 of \cite{Yadrenko1983} that chordal distance can guarantee the positive definiteness of a covariance function on $\mathbb{S}^d \times \mathbb{S}^d$ when the original covariance function is positive definite on $\mathbb{R}^{d+1}\times \mathbb{R}^{d+1}$.

\begin{table}[htbp]
\centering
   \caption{Cross-validation results on the XCO2 data based on the Mat\'ern covariance model and the CH covariance model. The measures in the first coordinate correspond to those based on MAR locations for interpolative prediction, and the measures in the second coordinate correspond to those based on MBD locations for extrapolative prediction.}
  {\resizebox{1.0\textwidth}{!}{%
  \setlength{\tabcolsep}{2.5em}
   \begin{tabular}{@{} l  c  c  c  c c @{}} 
   \toprule \noalign{\vskip 1.5pt} 
		& \multicolumn{2}{c}{Mat\'ern class}   &\multicolumn{2}{c}{CH class}  \\  \noalign{\vskip 1.5pt}  
	  \noalign{\vskip 1.5pt} 
 & $\nu = 0.5$ &   $\nu = 1.5$ & $\nu = 0.5$ & $\nu=1.5$ \\ \noalign{\vskip 2.5pt} 
 \noalign{\vskip 2.5pt} \noalign{\vskip 3.5pt} 
 $\tau^2$ (nugget) & 0.0642 &   0.2215    &  \textbf{0.0038}&     0.1478       \\  \noalign{\vskip 1.5pt}  \noalign{\vskip 2.5pt}
{RMSPE}    &0.672, 1.478    & 0.675, 1.599   &0.676, {1.263}   & 0.735, 1.227  \\ \noalign{\vskip 1.5pt}  \noalign{\vskip 2.5pt}
{CVG(95\%)}  &0.952, 0.929   & 0.952, 0.951   &0.944, {0.921}   & 0.878, 0.937  \\ \noalign{\vskip 1.5pt}  \noalign{\vskip 2.5pt}
{ALCI(95\%)}  &2.533, 5.095   & 2.536, 5.044  & 2.543, {4.722}   & 2.098, 4.855  \\ \noalign{\vskip 1.5pt}  
\noalign{\vskip 1.5pt} \bottomrule
   \end{tabular}%
   }}
   \label{table: CV for OCO2 data}
\end{table}

Next, we predict the process $Y(\cdot)$ at $0.25^\circ \times 0.25^\circ$ grid in the study region. The parameters are estimated based on all the data points under the CH class and the Mat\'ern class with the smoothness parameter fixed at 0.5. In Figure~\ref{fig: comparison for OCO2 data} of the Supplementary Material, we observe that the optimal kriging predictors over these grid points under the CH covariance model generally yield smaller values than those under the Mat\'ern covariance function model in large missing gaps except for certain regions such as the Gulf of Mexico. More importantly, we also observe that the CH covariance model yields $10\%$ to $20\%$ smaller kriging standard errors than the Mat\'ern covariance model in the observed spatial locations and contiguous missing regions. This indicates that the CH covariance model has an advantage over the Mat\'ern covariance in terms of in-sample prediction skills and in an extrapolative setting (such as large missing gaps). Prediction in an interpolative setting (such as locations near the observed locations) shows that the CH class yields indistinguishable (no more than 2\%) kriging standard errors compared to the Mat\'ern class. It is clear to see that the CH class is able to show lots of fine-scale variations in the kriging map, which is a desirable property for prediction accuracy. This is partly because the nugget parameter under the CH covariance is estimated to be much smaller than that under the Mat\'ern covariance and partly because the polynomially decaying dependence exhibited in the CH class can better utilize information at both nearby locations and distant locations to infer such fine-scale variations. Finally, Figure~\ref{fig: prediction for OCO2 data} shows the optimal kriging predictors and associated kriging standard errors at $0.25^\circ \times 0.25^\circ$ grid in the study region. These kriging maps help create a complete NASA Level 3 data product with associated uncertainties, which can be further used for downstream applications such as CO2 flux inversion.

\begin{figure}[!t] 

\begin{subfigure}{.33\textwidth}
  \centering
\makebox[\textwidth][c]{ \includegraphics[width=1.0\linewidth, height=0.15\textheight]{OCO2_data.pdf}}
\caption{XCO2 data.}
\end{subfigure}%
\begin{subfigure}{.33\textwidth}
\makebox[\textwidth][c]{ \includegraphics[width=1\linewidth, height=0.15\textheight]{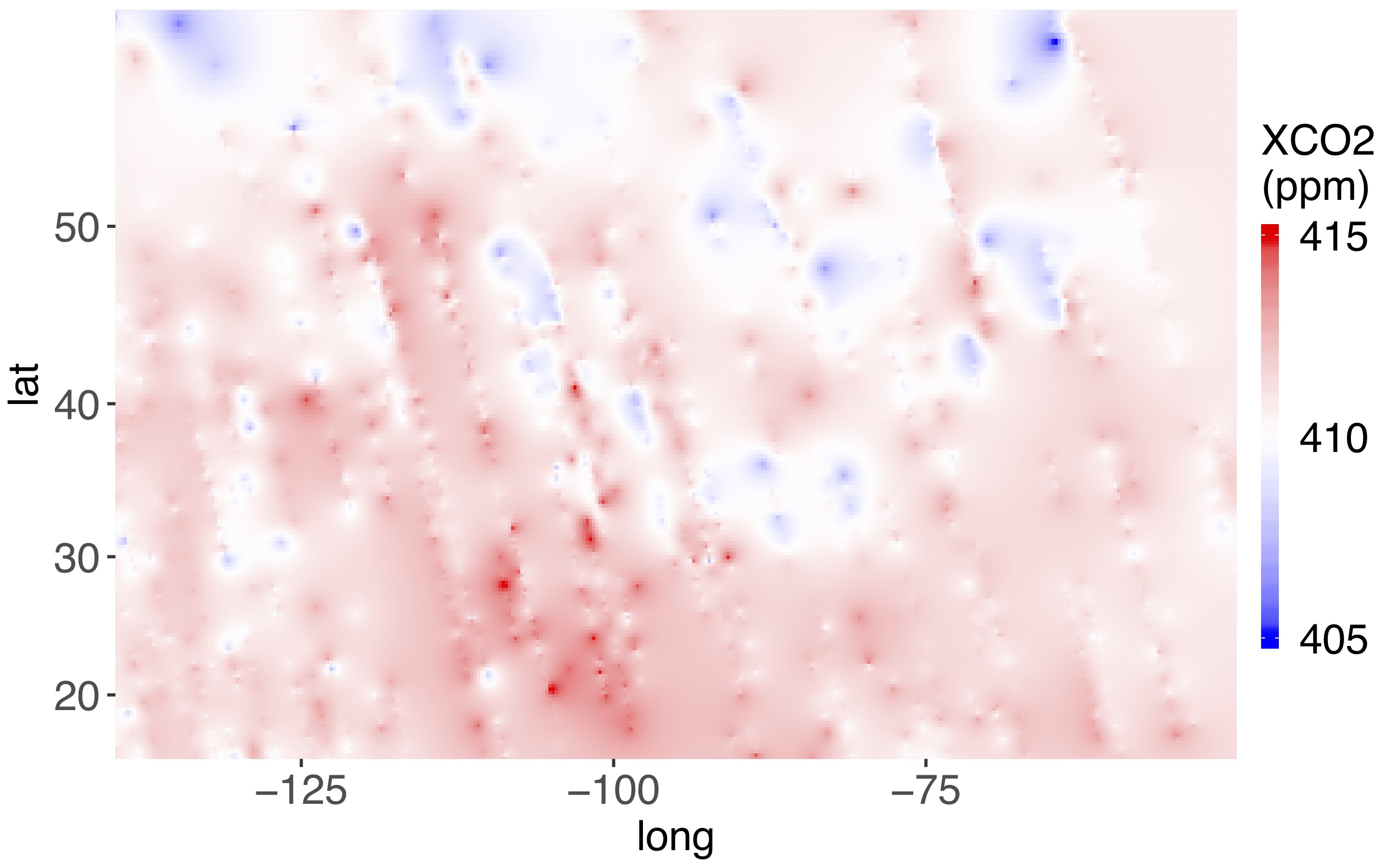}}
\caption{Kriging predictors.}
\end{subfigure}%
\begin{subfigure}{.33\textwidth}
\makebox[\textwidth][c]{ \includegraphics[width=1\linewidth, height=0.15\textheight]{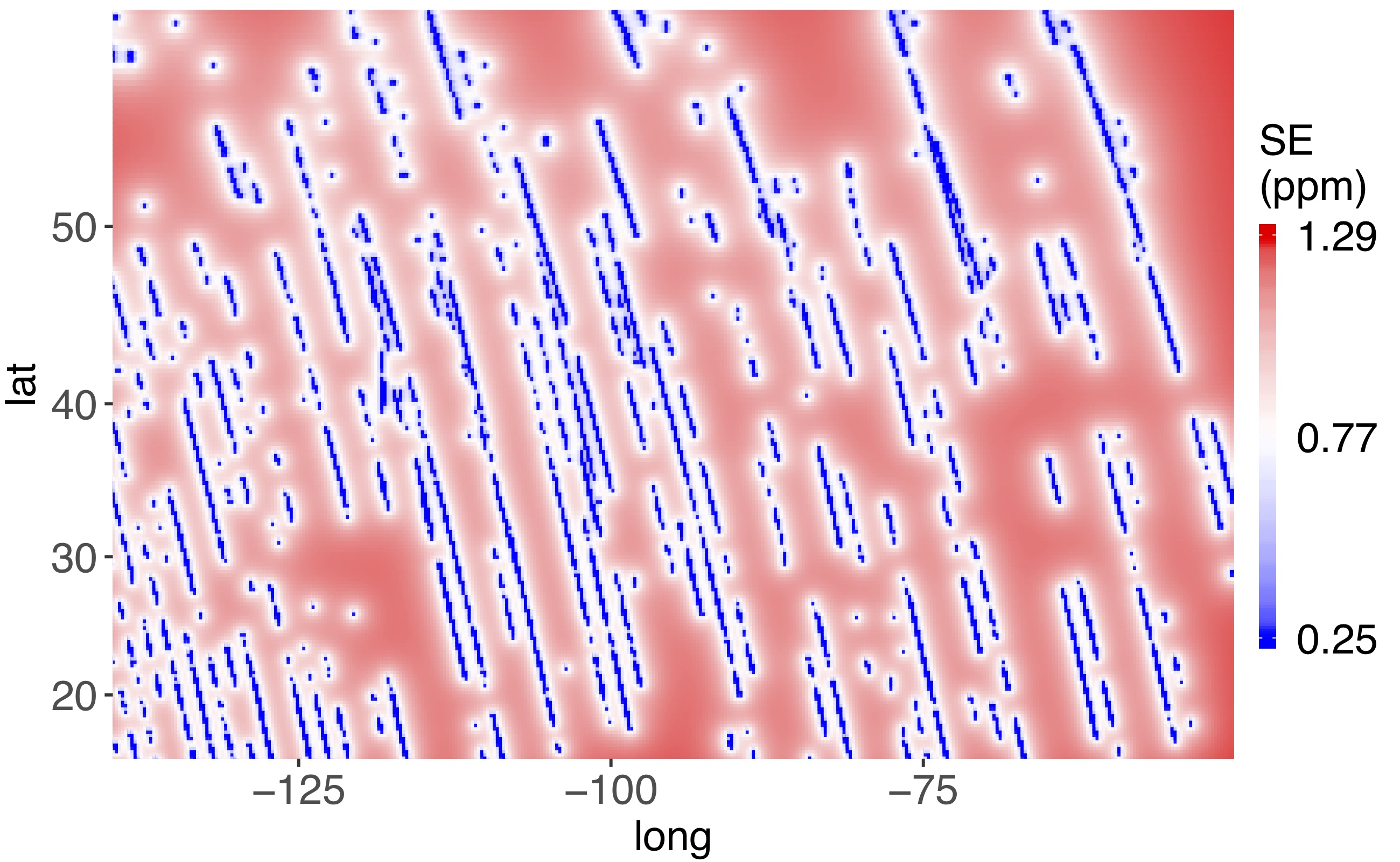}}
\caption{Kriging standard errors.}
\end{subfigure}

\caption{XCO2 data and kriging predictions based on the CH model.}
\label{fig: prediction for OCO2 data}
\end{figure}

\section{Concluding Remarks}\label{sec:conclusions}

This paper introduces a new class of interpretable covariance functions called the \emph{Confluent Hypergeometric} class that can allow precise and simultaneous control of the origin and tail behaviors with well-defined roles for each covariance parameter. Our approach in constructing the CH class is to mix over the range parameter of the Mat\'ern class. As expected,  the origin behavior of this CH class is as flexible as the Mat\'ern class. The high-frequency behavior of the CH class is also similar to that of the Mat\'ern class, since they differ by a slowly varying function up to a multiplicative constant. Unlike the Mat\'ern class, however, this CH class has a polynomially decaying tail, which allows for modeling power-law stochastic processes. 

The advantage of the CH class is examined in theory and numerical examples. Conditions for equivalence of two Gaussian measures based on the CH class are established. We derive the conditions on the asymptotic efficiency of kriging predictors based on an increasing number of observations in a bounded region when the CH covariance is misspecified. We also show that the CH class can yield an asymptotically efficient kriging predictor under the infill asymptotics framework when the true covariance belongs to the Mat\'ern class. {It is worth noting that the CH class itself is valid and can allow any degrees of decaying tail, while the asymptotic results of the MLE for the microergodic parameters are proven for $\alpha>d/2$. Investigation of the similar theoretical result on the MLE is elusive for the case $\alpha \in (0, d/2]$}. Extensive simulation results show that when the underlying true process is generated from either the Mat\'ern covariance or the GC covariance, the CH covariance can allow robust prediction property. We also noticed in simulation study that the Mat\'ern class gives worse performance than the CH class when the underlying true covariance has a polynomially decaying tail. In the real data analysis, we found significant advantages of the CH class when prediction is made in an extrapolative setting while the difference in terms of interpolative prediction is indistinguishable, which is implied by our theoretical results. This feature is practically important for spatial modeling especially with large missing patterns. Future work along the theoretical side is to establish theoretical results of the CH class under the increasing domain asymptotics.

This paper mainly focuses on theoretical contributions and practical advantage of the CH class. Common challenges in spatial statistics include modeling large spatial data and spatial nonstationarity, which are often tackled based on the Mat\'ern class in recent developments \citep[e.g.,][]{Lindgren2011, Ma2020FGP}. The proposed CH class can be used as a substantially improved starting point over the Mat\'ern class to develop more complicated covariance models to tackle these challenges. Several extensions can be pursued. It is interesting to extend the proposed CH class for modeling dependence on sphere, space-time dependence, and/or multivariate dependence \citep[e.g.,][]{Apanasovich2012,Ma2019AAGP, Ma2019DFGP}. Prior elicitation for the CH class could be challenging. It is also interesting to develop objective priors such as reference prior to facilitate default Bayesian analysis for analyzing spatial data or computer experiments \citep{Berger2001, Ma2019Obayes}.

The CH class not only plays an important role in spatial statistics, but also is of particular interest in UQ. In the UQ community, a covariance function that is of a product form \citep[e.g.,][]{Sacks1989, Santner2018} has been widely used to model dependence structures for computer model output to allow for different physical interpretations in each input dimension.  The product form of this CH covariance can not only control the smoothness of the process realizations in each direction but also allow polynomially decaying dependence in each direction. The simulation example in Section~\ref{app: UQ simulation} of the Supplementary Material shows significant improvement of the CH class over the Mat\'ern class and the GC class. Predicting real-world processes often relies on computer models whose output can have different smoothness properties and can be insensitive to certain inputs. This CH class can not only allow flexible control over the smoothness of the physical process of interest, but also allow near constant behavior along these inert inputs. Most often, predicting the real-world process involves extrapolation away from the original input space. The CH covariance should be useful in dealing with such challenging applications.

\section*{Supplementary Material}
The Supplementary Material contains seven parts: (1) illustration of timing for Bessel function and confluent hypergeometric function, (2) 1-dimensional process realizations for the Mat\'ern class and CH class, (3) two lemmas that are used to prove the main theorems, (4) technical proofs omitted in the main text, (5) simulation results that verify asymptotic normality, (6) additional simulation examples referenced in Section~\ref{sec:numerical}, and (7) parameter estimation results and figures referenced in Section~\ref{sec:real}. Computer code for the real data analysis is also available as a \textbf{.zip} archive. 

\section*{Acknowledgements}
Ma carried out this work partly as a postdoctoral fellow at the Statistical and Applied Mathematical Sciences Institute (SAMSI) and Duke University and gratefully acknowledges the postdoctoral fellowship awarded by the National Science Foundation under Grant DMS-1638521 to SAMSI. Bhadra gratefully acknowledges a visiting fellowship at SAMSI where part of this research was conducted. Ma is grateful to Professor James Berger for his insightful comments on an early draft of the manuscript that substantially improved the work. The authors would like to thank Professor Michael Stein for pointing out a reference that improved this work.

\begin{singlespace}
\bibliographystyle{biom}
\bibliography{hs-review}
\end{singlespace}

\clearpage\pagebreak\newpage
\thispagestyle{empty}
\begin{center}
{\LARGE{\bf Supplementary Material: {Beyond Mat\'ern: On A Class of Interpretable Confluent Hypergeometric Covariance Functions}} } 
\vspace{1cm}
\\ \LARGE {\bf by}
\end{center}

\vskip 1cm
\baselineskip=15pt
\begin{center}
Pulong Ma\\
  School of Mathematical and Statistical Sciences, Clemson University \\
 220 Parkway Dr., Clemson, SC 29634, USA \\
  plma@clemson.edu\\
  and\\
 Anindya Bhadra\\ 
  Department of Statistics, Purdue University\\
  {250 N. University St., West Lafayette, IN 47907, USA} \\
  bhadra@purdue.edu\\
  \hskip 5mm\\
     \end{center}

\setcounter{equation}{0}
\setcounter{page}{0}
\setcounter{table}{0}
\setcounter{section}{0}
\setcounter{figure}{0}
\renewcommand{\theequation}{S.\arabic{equation}}
\renewcommand{\thesection}{S.\arabic{section}}
\renewcommand{\thesubsection}{S.\arabic{section}.\arabic{subsection}}
\renewcommand{\thepage}{S.\arabic{page}}
\renewcommand{\thetable}{S.\arabic{table}}
\renewcommand{\thefigure}{S.\arabic{figure}}

\clearpage\pagebreak\newpage

\renewcommand{\theequation}{S.\arabic{equation}}
\renewcommand{\thesubsection}{S.\arabic{section}.\arabic{subsection}}

\spacingset{1.25} 

This supplement contains seven sections. Section~\ref{sec: timing} gives an illustration to compare the timing to evaluate Bessel function and confluent hypergeometric function. Section~\ref{app: 1D Realization} shows 1-dimensional process realizations under different parameter values for the Mat\'ern class and the CH class. Section~\ref{app: ancilary} shows some theoretical results that are used to prove main theorems in the main text. Section~\ref{app: proof} contains technical proofs that are omitted in the main text. Section~\ref{sec: mle} contains simulation results that verify the theoretical results in Section~\ref{sec:theory}. Section~\ref{app: numerical examples} contains additional simulation results referenced in Section~\ref{sec:numerical}. Section~\ref{app: numerical results} contains parameter estimation results and figures referenced in Section~\ref{sec:real}. 

\section{Illustration of Timing to Evaluate Bessel Function and Confluent Hypergeometric Function} \label{sec: timing}
When the confluent hypergeometric function is evaluated by calling the GNU scientific library GSL via \texttt{Rcpp} package and the modified Bessel function of second kind  is evaluated with the \texttt{base} package in \textsf{R},  with 10000 times repeated evaluations, on average, the confluent hypergeometric function takes about 10.7  microseconds for each evaluation, while the Bessel function takes about 7.8 microseconds for each evaluation. The timings are recorded using the \textsf{R} package \texttt{microbenchmark} with results shown in Figure~\ref{fig: timing}.\\

\begin{figure}[htbp]
\begin{center}
\makebox[\textwidth][c]{ \includegraphics[width=1.0\textwidth, height=0.6\textheight]{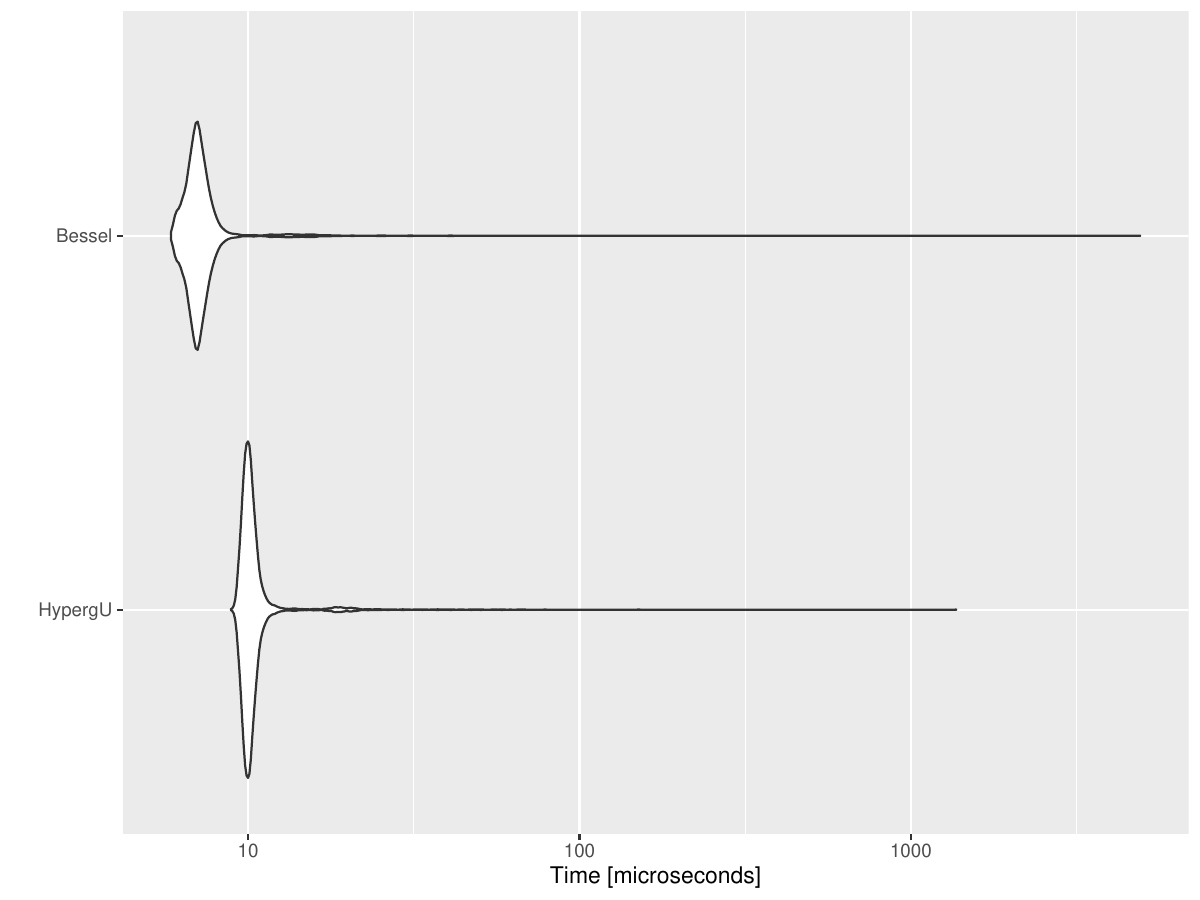}}
\caption{Benchmark of the computing time to evaluate the confluent hypergeometric function and the Bessel function in \textsf{R}. ``Bessel'' refers to the timing for evaluating the modified Bessel function of the second kind and ``HypergU'' refers to the timing for evaluating the confluent hypergeometric function of the second kind.}
\label{fig: timing}
\end{center}
\end{figure}

\section{1-D Process Realizations} \label{app: 1D Realization}
In Figure~\ref{fig: 1D realizations}, we show the realizations from zero mean Gaussian processes with the CH class and the Mat\'ern class under different parameter settings. When the distance is within the effective range, the Mat\'ern covariance function results in more large correlations than the CH covariance function. This makes the process realizations from the Mat\'ern class smoother even though the smoothness parameter is fixed at the same value for both the Mat\'ern class and the CH class. For the CH class, if $\alpha$ has a smaller value, the corresponding correlation function has more small values within the effective range. This makes the process realizations under the CH class look rougher. As we expect, when the effective range and the tail decay parameter are fixed, the process realizations under the CH class look smoother for a larger value of the smoothness parameter.

\begin{figure}[htbp] 
\captionsetup[subfigure]{aboveskip=5pt}

\begin{subfigure}{.25\textwidth}
\makebox[\textwidth][c]{ \includegraphics[width=1.0\linewidth, height=0.18\textheight]{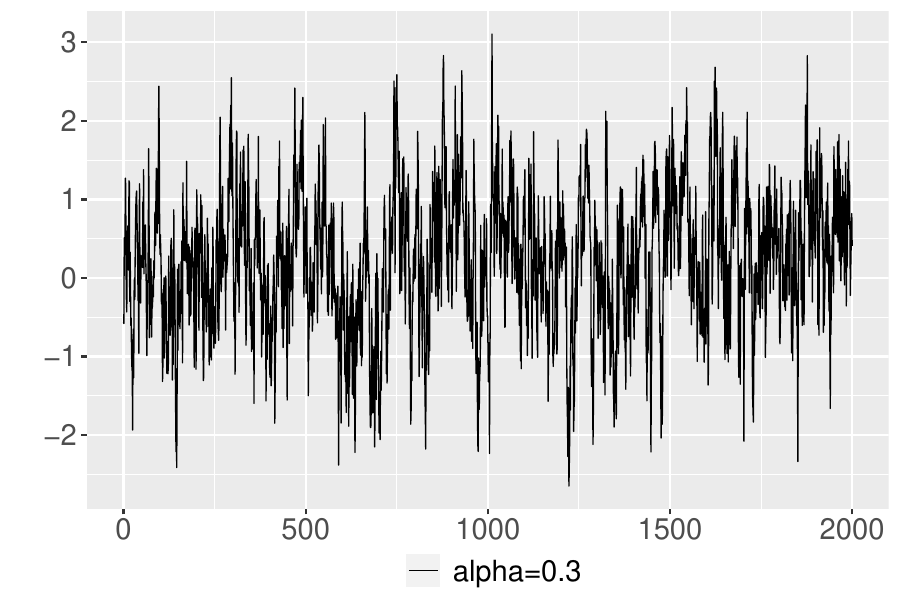}}
\end{subfigure}%
\begin{subfigure}{.25\textwidth}
\makebox[\textwidth][c]{ \includegraphics[width=1.0\linewidth, height=0.18\textheight]{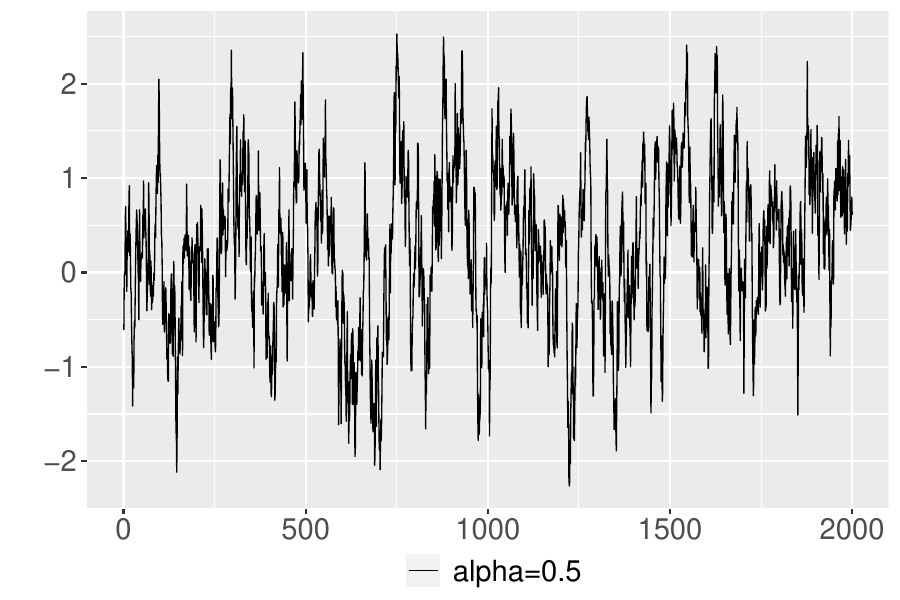}}
\end{subfigure}%
\begin{subfigure}{.25\textwidth}
\makebox[\textwidth][c]{ \includegraphics[width=1.0\linewidth, height=0.18\textheight]{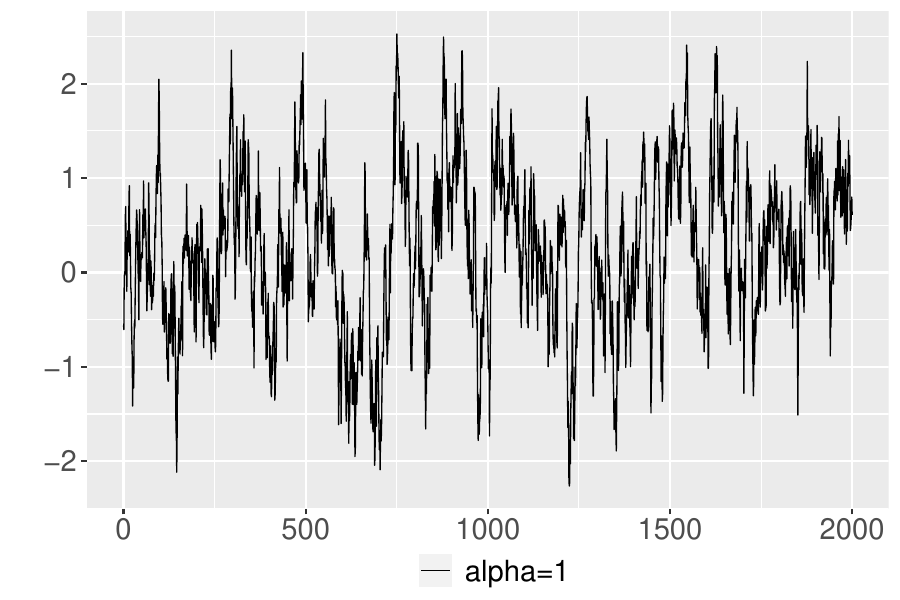}}
\end{subfigure}%
\begin{subfigure}{.25\textwidth}
\makebox[\textwidth][c]{ \includegraphics[width=1.0\linewidth, height=0.18\textheight]{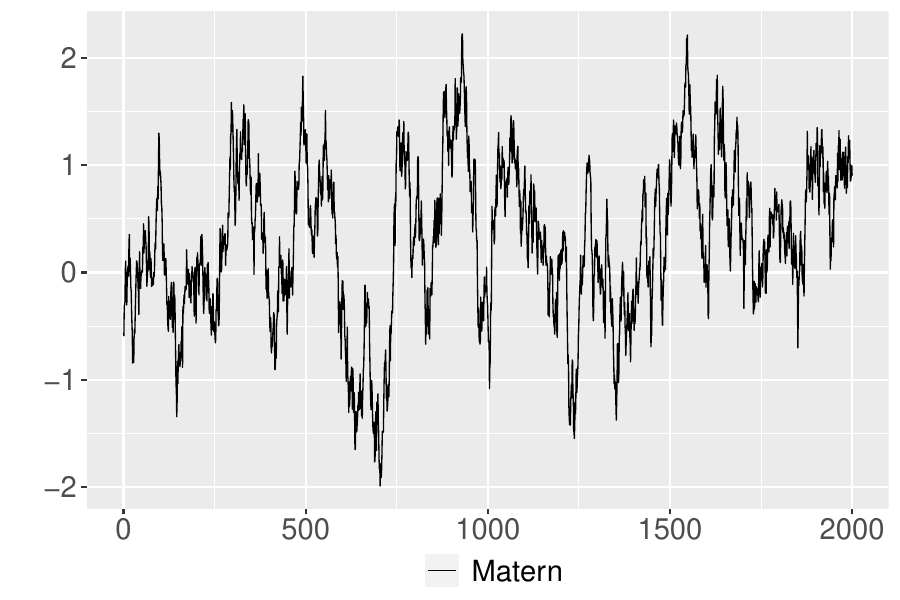}}
\end{subfigure}
{\caption*{$\nu=0.5,ER=200$}}

\begin{subfigure}{.25\textwidth}
\makebox[\textwidth][c]{ \includegraphics[width=1.0\linewidth, height=0.18\textheight]{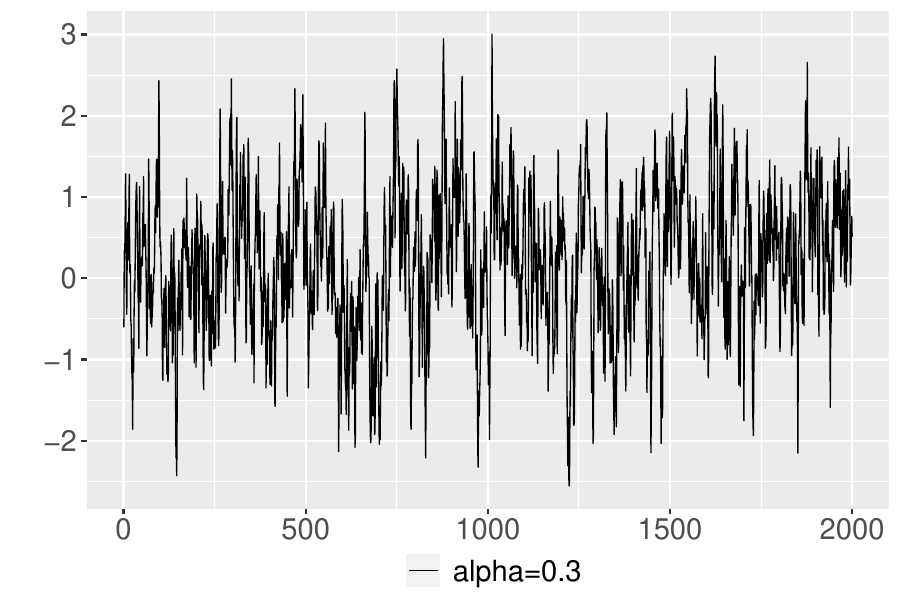}}
\end{subfigure}%
\begin{subfigure}{.25\textwidth}
\makebox[\textwidth][c]{ \includegraphics[width=1.0\linewidth, height=0.18\textheight]{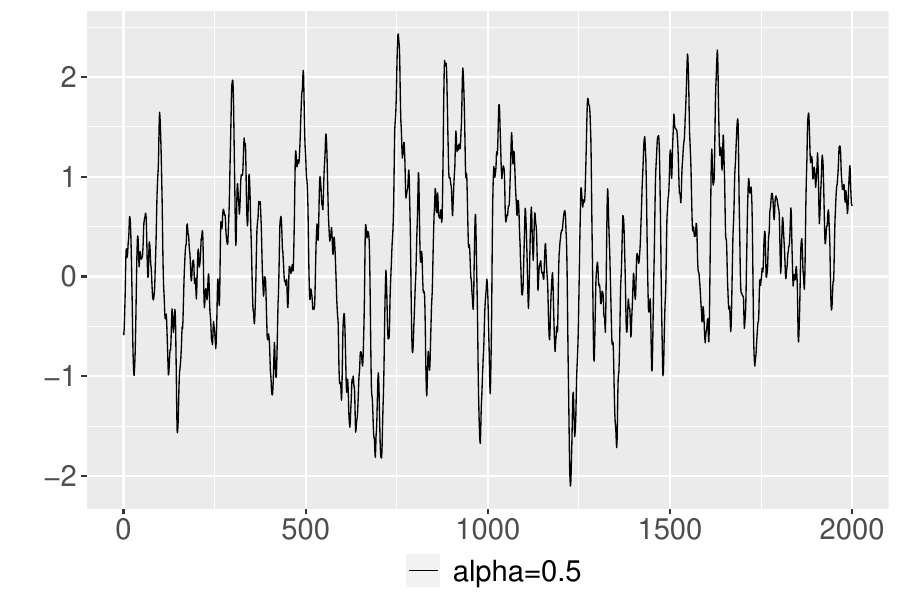}}
\end{subfigure}%
\begin{subfigure}{.25\textwidth}
\makebox[\textwidth][c]{ \includegraphics[width=1.0\linewidth, height=0.18\textheight]{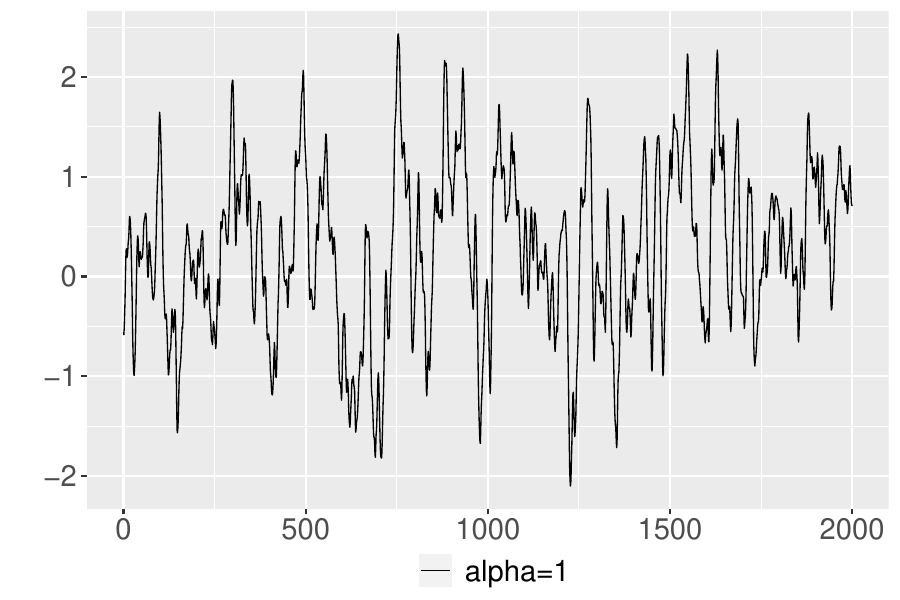}}
\end{subfigure}%
\begin{subfigure}{.25\textwidth}
\makebox[\textwidth][c]{ \includegraphics[width=1.0\linewidth, height=0.18\textheight]{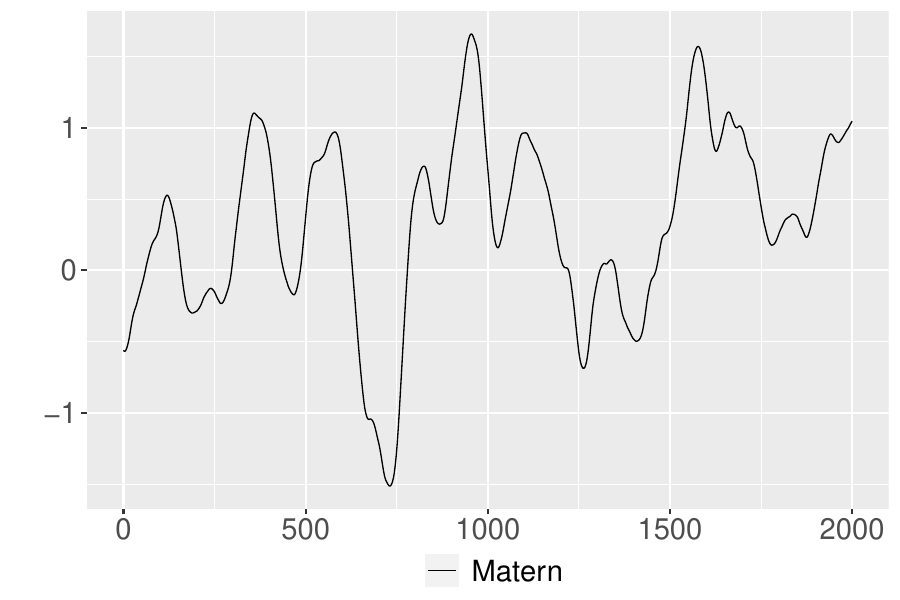}}
\end{subfigure}
\caption*{$\nu=2.5, ER=200$}

\begin{subfigure}{.25\textwidth}
\makebox[\textwidth][c]{ \includegraphics[width=1.0\linewidth, height=0.18\textheight]{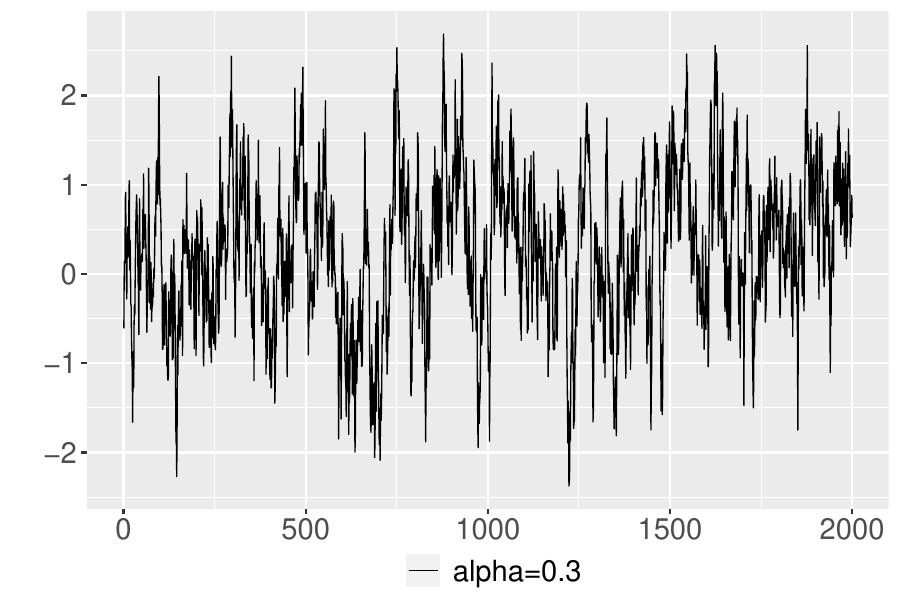}}
\end{subfigure}%
\begin{subfigure}{.25\textwidth}
\makebox[\textwidth][c]{ \includegraphics[width=1.0\linewidth, height=0.18\textheight]{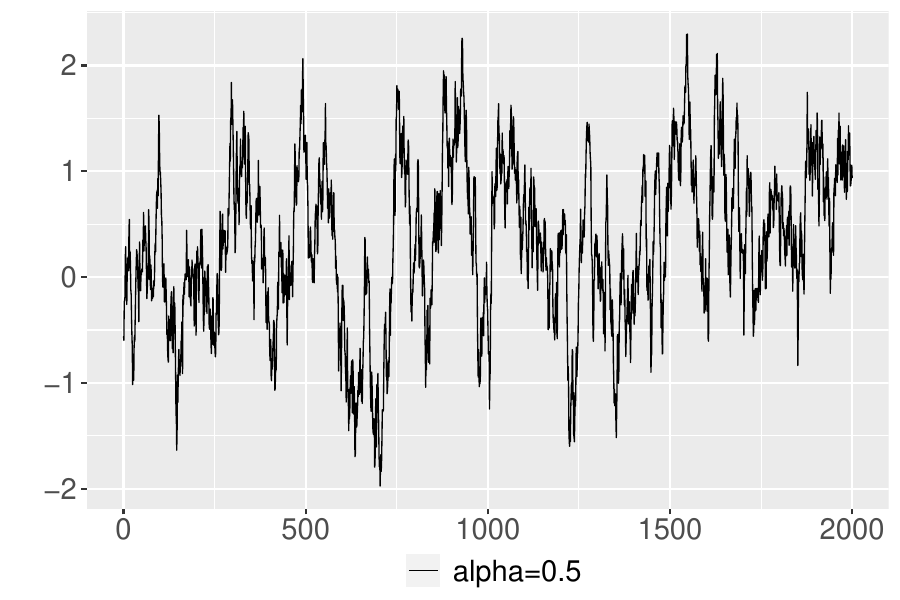}}
\end{subfigure}%
\begin{subfigure}{.25\textwidth}
\makebox[\textwidth][c]{ \includegraphics[width=1.0\linewidth, height=0.18\textheight]{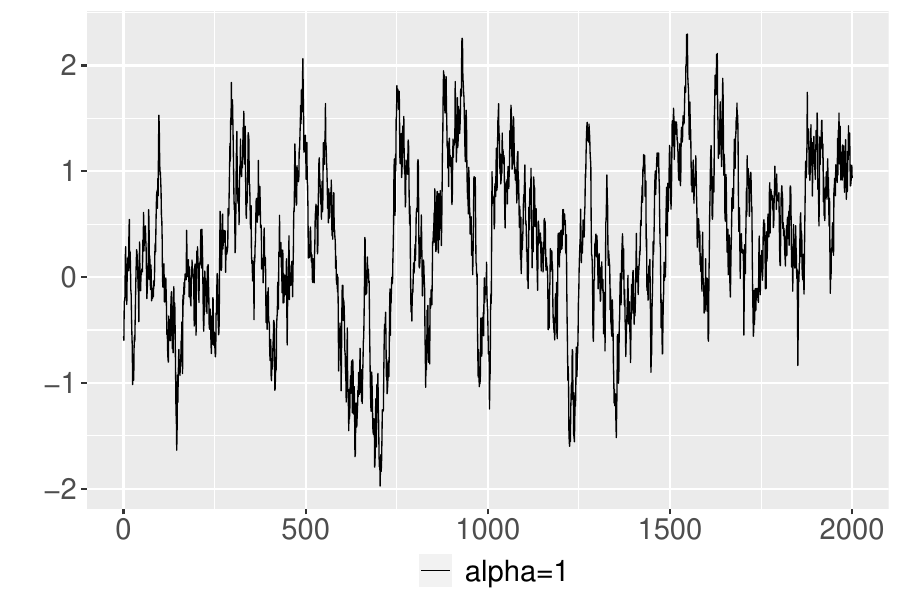}}
\end{subfigure}%
\begin{subfigure}{.25\textwidth}
\makebox[\textwidth][c]{ \includegraphics[width=1.0\linewidth, height=0.18\textheight]{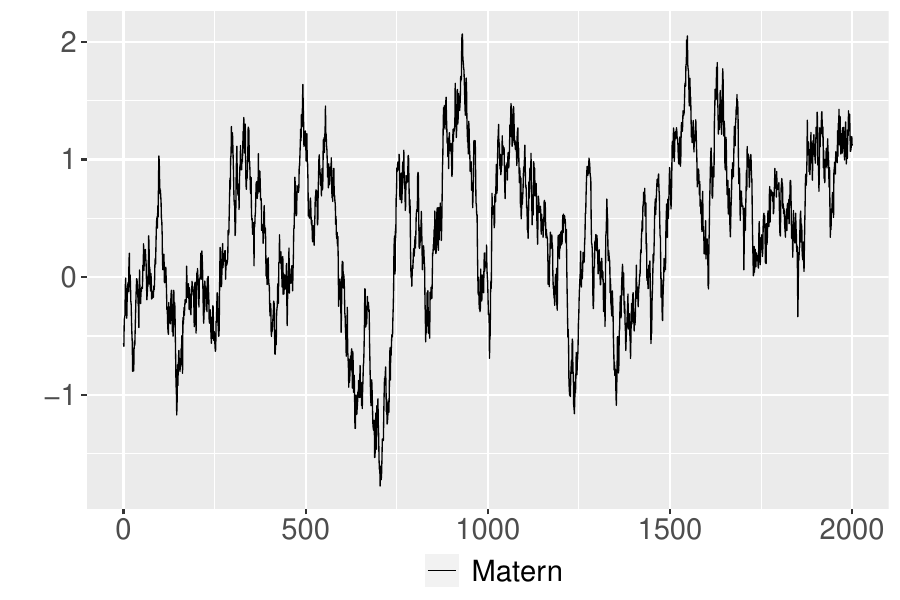}}
\end{subfigure}
\caption*{$\nu=0.5, ER=500$}

\begin{subfigure}{.25\textwidth}
\makebox[\textwidth][c]{ \includegraphics[width=1.0\linewidth, height=0.18\textheight]{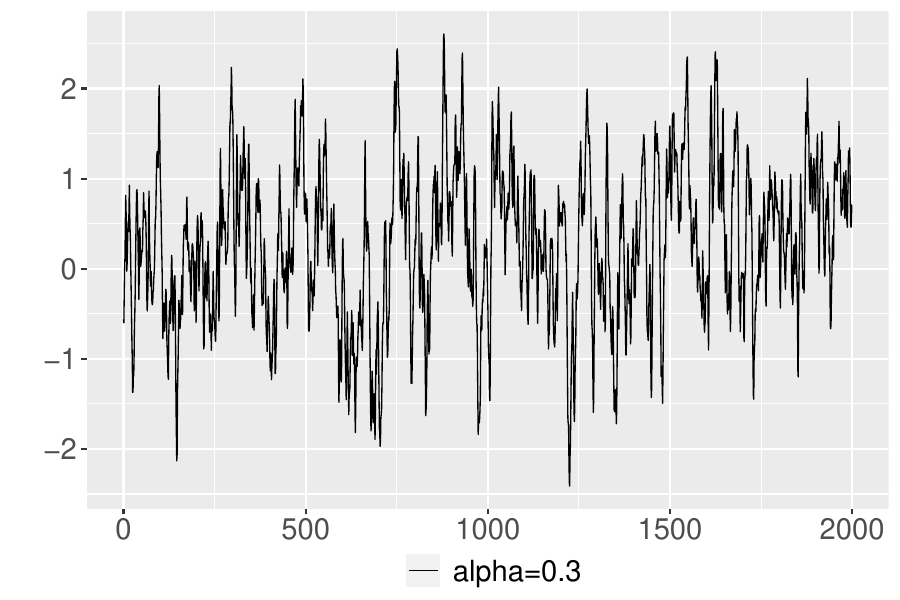}}
\end{subfigure}%
\begin{subfigure}{.25\textwidth}
\makebox[\textwidth][c]{ \includegraphics[width=1.0\linewidth, height=0.18\textheight]{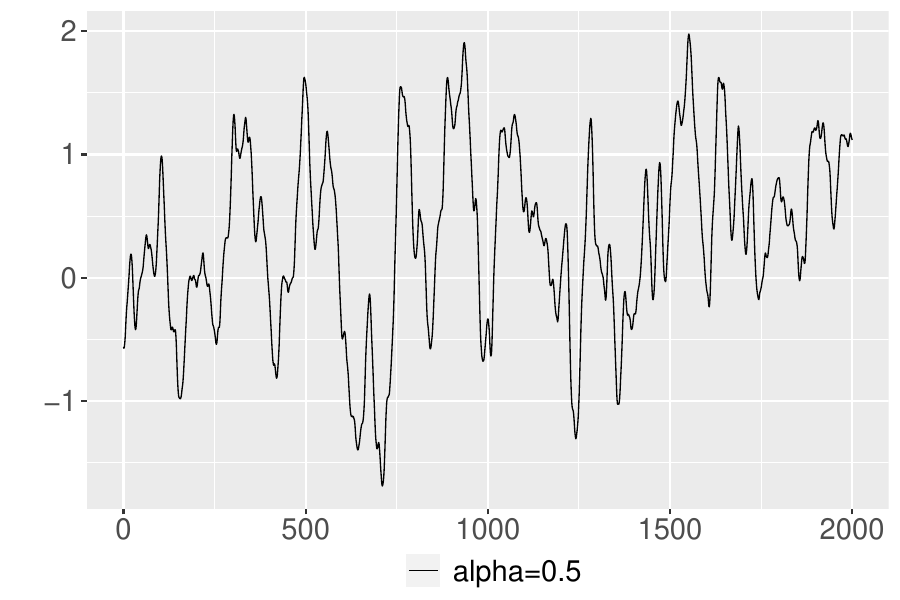}}
\end{subfigure}%
\begin{subfigure}{.25\textwidth}
\makebox[\textwidth][c]{ \includegraphics[width=1.0\linewidth, height=0.18\textheight]{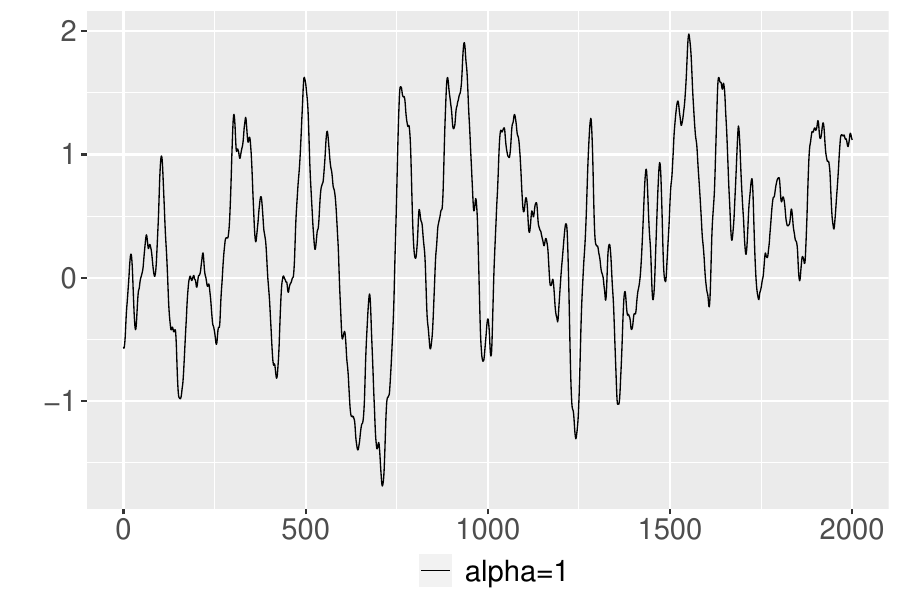}}
\end{subfigure}%
\begin{subfigure}{.25\textwidth}
\makebox[\textwidth][c]{ \includegraphics[width=1.0\linewidth, height=0.18\textheight]{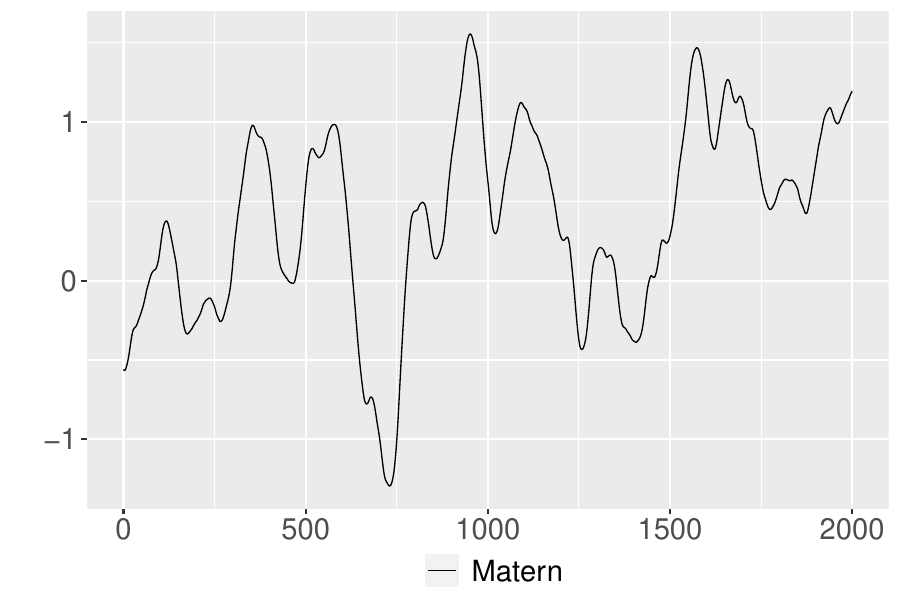}}
\end{subfigure}
\caption*{$\nu=2.5, ER=500$}

\caption{Realizations over 2000 regular grid points in the domain $[0, 2000]$ from zero mean Gaussian processes with the CH covariance model and the Mat\'ern covariance model under different parameter settings. The realizations from the CH covariance are shown in the first three columns and those from the Mat\'ern covariance are shown in the last column. For the first two rows, the effective range (ER) is fixed at 200. For the last two rows, the effective range is fixed at 500. ER is defined as the distance at which correlation is approximately $0.05$.}
\label{fig: 1D realizations}
\end{figure}

\newpage

\section{Ancillary Results} \label{app: ancilary}

To show the asymptotic behavior of the MLE of the microergodic parameter, we need some results in terms of spectral densities of covariance functions. More precisely, the tail behavior of the spectral densities can be used to check the equivalence of probability measures generated by stationary Gaussian random fields. Equivalence of Gaussian measures defined by Gaussian processes has been studied in probability and statistics with sufficient conditions given in Theorem 17 of Chapter III of \cite{Ibragimov1978} for $d=1$ and given on page 156 of \cite{Yadrenko1983} and page 120 of \cite{Stein1999} for $d>1$. In particular, the following sufficient conditions can be used to check the equivalence of Gaussian probability measures defined by covariance functions. If for some $\lambda>0$ and for some finite $c\in \mathbb{R}$, one has

\begin{align} \label{eqn: boundedness}
    0<f_1(\bfomega)|\bfomega|^{\lambda}<\infty \quad \text{ as } \quad |\bfomega| \to \infty, \quad \text{ and } \\ 
    \int_{|\bfomega|>c} \left\{ \frac{f_1(\bfomega) - f_2(\bfomega)}{f_1(\bfomega)} \right\}^2 d \bfomega < \infty, \label{eqn: equivalence for SGRF} 
\end{align}
then the two corresponding Gaussian measures $\mathcal{P}_1$ and $\mathcal{P}_2$ are equivalent. For isotropic Gaussian random fields, the condition~\eqref{eqn: equivalence for SGRF} can be expressed as 
\begin{align}\label{eqn: equivalence for IGRF}
    \int_c^{\infty} \omega^{d-1} \left\{ \frac{f_1(\omega) - f_2(\omega)}{f_1(\omega)} \right\}^2 \, d\omega <\infty,
\end{align}
where $\omega:=|\bfomega|$ with $|\cdot|$ denoting the Euclidean norm. The detailed discussion on equivalence of Gaussian measures and the condition for equivalence can be found in Chapter 4 of \cite{Stein1999} and references \citep[e.g.,][]{Stein1989, Stein1988, Stein1993}.

In what follows, we will introduce a few useful lemmas. Lemma~\ref{lem: diagonalization} is used to diagonalize two covariance matrices and it is needed in Lemma~\ref{lem: inequality}, which gives an important result on the behavior of eigenvalues of a correlation matrix constructed from the CH correlation function. 

\begin{lemma} \label{lem: diagonalization}
Let $\bfA$ and $\bfB$ be two $n\times n$ symmetric positive definite matrices. Then there exists a non-singular matrix $\bfU$ such that $\bfU^\top \bfA \bfU=\mathbf{I}_{n\times n}$ and $\bfU^\top \bfB \bfU = \bfD$, where $\bfD$ is an $n\times n$ diagonal matrix with positive diagonal entries. 
\end{lemma}

\begin{proof} 
For the symmetric matrix $\bfA$,  it follows from the Schur Decomposition Theorem \citep[e.g.,][p.~17]{Magnus1999} that there exists an orthogonal $n\times n$ matrix $\bfS$ consisting of eigenvectors of $\bfA$ and a diagonal matrix $\bfLambda:=\text{diag}\{\lambda_1, \ldots, \lambda_n\}$ such that $\bfS^\top \bfA \bfS = \bfLambda$. Since $\bfA$ is positive definite,  the diagonal entries of $\bfLambda$ are all positive. Let $\bfLambda^{1/2}: = \text{diag}\{\sqrt{\lambda_1}, \ldots, \sqrt{\lambda_n}\}$ be the ``square root'' of $\bfLambda$. Then we call $\bfA^{1/2}:=   \bfLambda^{1/2} \bfS$ a square root of $\bfA$ satisfying $\bfA = (\bfA^{\top})^{1/2} \bfA^{1/2}$. As the matrix $\bfA^{1/2}$ is invertible, the symmetric matrix $(\bfA^\top)^{-1/2} \bfB \bfA^{-1/2}$ is well-defined. Note that $(\bfA^\top)^{-1/2} \bfB \bfA^{-1/2} $ is positive definite since for all $\bfx \in \mathbb{R}^n$, $\bfx^\top (\bfA^\top)^{-1/2} \bfB \bfA^{-1/2} \bfx = \|\bfB^{1/2} \bfA^{-1/2} \bfx \|^2 \geq 0$ with the inequality becoming an equality only if $\bfx=\bfzero$. Hence $(\bfA^\top)^{-1/2} \bfB \bfA^{-1/2}$ is also a symmetric and positive definite matrix.  According to the Schur Decomposition Theorem, there exists an orthogonal matrix $\bfO$ and diagonal matrix $\bfD$ with positive diagonal entries such that $\bfO^\top (\bfA^\top)^{-1/2} \bfB \bfA^{-1/2} \bfO = \bfD$. Now we define the non-singular matrix $\bfU:=\bfA^{-1/2}\bfO$, which satisfies $\bfU^\top \bfA \bfU=\bfI_{n\times n}$, as to be established. 
\end{proof}

\begin{lemma} \label{lem: inequality}
Suppose that $\nu>0$ is fixed. Given a set of $n$ observation locations in a bounded domain $\mathcal{D}$,  let $\sigma^2_0\bfR_n(\bftheta_0)$ be the $n\times n$ covariance matrix defined by the CH covariance function $C(h, \nu, \alpha_0, \beta_0, \sigma^2_0)$ with $\bftheta_0:=\{\alpha_0, \beta_0\}$ and $\sigma^2\bfR_n(\bftheta)$ be the $n\times n$ covariance matrix defined by the CH covariance function $C(h, \nu, \alpha, \beta, \sigma^2)$ with $\bftheta:=\{\alpha, \beta\}$. Assume that $\alpha_0, \alpha>d/2$. Let $\bfLambda:=\diag\{\lambda_{1,n}, \ldots, \lambda_{n,n}\}$ be an $n\times n$ diagonal matrix with diagonal elements $\lambda_{k,n}>0$ for $k=1, \ldots, n$ such that $\mathbf{U}^\top \sigma^2_0\bfR_n(\bftheta_0) \mathbf{U}=\mathbf{I}_n$ and $\mathbf{U}^\top \sigma^2\bfR_n(\bftheta) \mathbf{U}=\bfLambda$ for some non-singular matrix $\bfU$.  Then it can be established that for any $\epsilon>0$, as $n\to \infty$,
\begin{align*}
\frac{1}{\epsilon\sqrt{n}}  \max_{1\leq i \leq n} \sum_{k=1}^n\{ \lambda_{i,n}^{-1}\} |\lambda_{k,n} - 1| \to 0.
\end{align*}

\end{lemma}

\begin{proof}
Note that the existence of the matrix $\mathbf{U}$ is true according to Lemma~\ref{lem: diagonalization} and thus $\lambda_{k,n}, k=1, \ldots, n$ are well-defined. 

Let $\xi_0: \mathbb{R}^d \to \mathbb{R}$ be a function of the form $\xi_0(\bfomega) = \int_{\mathbb{R}^d} \exp\{- i\bfx^\top \bfomega \} c_0(\bfx) d\bfx$ for any $\bfomega \in \mathbb{R}^d$, where $c_0(\bfx) = |\bfx|^{\kappa - d} I(|\bfx|\leq 1)$ for any $\bfx \in \mathbb{R}^d$,  $\kappa=(\nu+d/2)/(2m)$ with $|\cdot|$ denoting the Euclidean norm, and $m=\floor{\nu+d/2} +1$ with $\floor{x}$ denoting the largest integer less than or equal to $x$. As $d \in \{1,2,3\}$, $\kappa \in (0, 1/2)$, it follows from Lemma 2.3 of \cite{Wang2010} and proof of Theorem~8 of \citet{Bevilacqua2019AOS}  that the function $\xi_0$ is a continuous, isotropic, and strictly positive function with $\xi_0(\bfomega) \asymp |\bfomega|^{-\kappa}$ when $|\bfomega| \to \infty$. 

Let $c_1=c_0 * \ldots * c_0$ denote the $2m$-fold convolution of the function $c_0$ with itself, and let $\xi_1(\bfomega) = \int_{\mathbb{R}^d} \exp\{- i\bfx^\top \bfomega \} c_1(\bfx) d\bfx$. Then $\xi_1(\bfomega) = \xi_0(\bfomega)^{2m}$ for all $\bfomega \in \mathbb{R}^d$. This implies that $\xi_1$ is also a continuous, isotropic, and strictly positive function. By Proposition 1, the spectral density $f(|\bfomega|)$ of the CH covariance function satisfies $f(|\bfomega|) \asymp |\bfomega|^{-(2\nu + d)}$, and hence we have ${f(|\bfomega|) }/{\xi_1(\bfomega) } \asymp 1$ as $|\bfomega| \to \infty$. Note that this ratio (as a function of $|\bfomega|$) is a well-defined and continuous function on arbitrary compact interval of the positive real line with $\xi_1>0$. Thus, there exist two positive constants (not depending on $\bfomega$) such that 
\begin{align} \label{eqn: f over xi} 
c_{\xi_1} \leq \frac{f(|\bfomega|) }{\xi_1(\bfomega) } \leq C_{\xi_1},\quad \text{as } |\bfomega|\to \infty.
\end{align}

For any fixed $\nu>0$, let $f_{\sigma, \alpha, \beta}(|\bfomega|)$ denote the spectral density of the CH covariance $C(h; \nu, \alpha, \beta, \sigma^2)$ and  let $f_{\sigma_0, \alpha_0, \beta_0}(|\bfomega|)$ denote the spectral density of the CH covariance $C(h; \nu,$ $\alpha_0, \beta_0, \sigma^2_0)$. Then we define 
\begin{align*}
\eta(\bfomega) := \frac{f_{\sigma, \alpha, \beta}(|\bfomega|) - f_{\sigma_0, \alpha_0, \beta_0}(|\bfomega|)}{\xi_1(\bfomega)}.
\end{align*}
It follows from direct calculation that for a constant $C_{\eta}>0$, 
\begin{align} \label{eqn: eta} 
\begin{split}
\int_{\mathbb{R}^d} \eta(\bfomega)^2 d \bfomega = \frac{2\pi^{d/2}}{\Gamma(d/2)} \left\{ \int_0^{C_{\eta}} r^{d-1} \left( \frac{f_{\sigma, \alpha, \beta}(r) - f_{\sigma_0, \alpha_0, \beta_0}(r)}{\xi_1(\mathbf{r})} \right)^2 dr \right. \\
\left. + \int_{C_{\eta}}^{\infty} r^{d-1} \left(\frac{f_{\sigma, \alpha, \beta}(r) - f_{\sigma_0, \alpha_0, \beta_0}(r)}{\xi_1(\mathbf{r})}\right)^2 dr\right\},
\end{split}
\end{align}
where $r = |\mathbf{r}|$ and $\mathbf{r} \in \mathbb{R}^d$. 

As shown in Theorem 3, for any fixed $\nu>0$, the spectral density of the CH class satisfies the conditions~\eqref{eqn: boundedness} and~\eqref{eqn: equivalence for SGRF} when $\frac{\sigma^2 \Gamma(\nu+\alpha)}{\beta^{2\nu} \Gamma(\alpha)} = \frac{\sigma^2_0 \Gamma(\nu+\alpha_0)}{\beta_0^{2\nu} \Gamma(\alpha_0)}$. This implies that there exist a constant 
$C_{\eta}^0$ (not depending on $\bfomega$) such that 
$$|\eta(\bfomega)| \leq \frac{C_{\eta}^0}{ (1+|\bfomega|^2)}, \quad \forall \bfomega \in \mathbb{R}^d.$$
It follows immediately that the two integrals in the righ-hand side of Equation~\eqref{eqn: eta} are hence finite for $d=1,2,3$. Thus, $\eta$ is square integrable, i.e., $\eta\in L^2(\mathbb{R}^d)$. From classic Fourier theory (see Chapter 1 of \cite{Stein1971}), an immediate consequence of the square integrability of $\eta$ is that there exists a square-integrable function $g: \mathbb{R}^d \to \mathbb{C}$ such that 
\begin{align*}
\int_{\mathbb{R}^d} | \eta(\bfomega) - \hat{g}_{k}(\bfomega)|^2 d \bfomega, \text{ as } k \to \infty,
\end{align*}
where $\hat{g}_{k}(\bfomega) = \int_{\mathbb{R}^d} \exp\{-i\bfx^\top \bfomega \} g(\bfx) I(|\bfx|_{\max} \leq k) d \bfx$ for all $\bfomega \in \mathbb{R}^d$ and $|\bfx|_{\max} = \max_{1\leq j \leq d} |x_j|$ for $\bfx=(x_1, \ldots, x_d) \in \mathbb{R}^d$. 

Let $a>0$, $m_a:=\floor{a+d/2}+1, a_0:=(a+d/2)/(2m_a)$. Define 
\begin{align*}
\tilde{c}_0(\bfx) &:= |\bfx|^{a_0-d} I(|\bfx|\leq 1), \quad \forall \bfx \in \mathbb{R}^d, \\
\tilde{\xi}_0(\bfomega) &:= \int_{\mathbb{R}^d} \exp\{-i\bfx^\top \bfomega\} \tilde{c}_0(\bfx) d\bfx,\quad \forall \bfomega \in \mathbb{R}^d.
\end{align*}
Let $\tilde{c}_1:=\tilde{c}_0*\cdots*\tilde{c}_0$ denote the $2m_a$-fold convolution of $\tilde{c}_0$ with itself. 
Let $\{\epsilon_n: n=1,2,\ldots\}$ be a sequence of real numbers such that $\epsilon_n\in(0, 1)$ for all $n$ and $\lim_{n\to\infty} \epsilon_n=0$. Then we define 
\begin{align*} 
e_n(\bfx) := \frac{1}{C_e \epsilon_n^d} \tilde{c}_1\left( \frac{\bfx}{\epsilon_n}\right), 
\end{align*}
where $C_e:=\int_{\mathbb{R}^d} \tilde{c}_1(\bfx) d \bfx$. Then we obtain the Fourier transform of $e_n$: 
\begin{align*}
\hat{e}_n(\bfx) = \int_{\mathbb{R}^d} \exp\{-i\bfx^\top \bfomega\}  e_n(\bfx) d\bfx = \frac{\tilde{\xi}_1(\epsilon_n \bfomega)}{C_e},
\end{align*}
where $\tilde{\xi}_1(\bfomega) := \int_{\mathbb{R}^d} \exp\{-i\bfx^\top \bfomega\} \tilde{c}_1(\bfx) d\bfx = \tilde{\xi}_0^{2m_a}(\bfomega)$. This implies that there exists a constant $C_{\hat{e}}$ (not depending on $\bfomega$ and $n$) such that 
\begin{align} \label{eqn: bound for e_n_hat} 
|\hat{e}_n(\bfomega)| \leq \frac{C_{\hat{e}} }{ (1+\epsilon_n|\bfomega|)^{a+d/2} },\quad \forall \bfomega  \in \mathbb{R}^d.
\end{align}

Note that it follows from Plancherel's theorem that 
\begin{align*}
\int_{\mathbb{R}^d}  |g(\bfx-\bfy) - g(\bfy)|^2 d\bfx &= \frac{1}{(2\pi)^d} \int_{\mathbb{R}^d} |(\exp\{-i\bfw^\top \bfy \} -1) \eta(\bfomega)|^2 d \bfomega\\
& \leq \frac{2^{2-\ell_0} |\bfy|^{\ell_0}}{(2\pi)^d} \int_{\mathbb{R}^d} |\bfomega|^{\ell_0} |\eta(\bfomega)|^2 d\bfomega,
\end{align*}
and it follows from Minkowski's equality that 
\begin{align*}
\left\{ \int_{\mathbb{R}^d} | e_n*g(\bfx) - g(\bfx) |^2 d\bfx \right\}^{1/2} &= \left\{ \int_{\mathbb{R}^d} \int_{|\bfy|\leq 2m_a\epsilon_n} | (g(\bfx-\bfy) - g(\bfx) )e_n(\bfy) d\bfy|^2 d\bfx  \right\}^{1/2} \\
&\leq \frac{2^{1-\ell_0/2} (2m_a\epsilon_n)^{\ell_0/2}}{(2\pi)^{d/2}} \left\{ \int_{\mathbb{R}^d} |\bfomega|^{\ell_0} |\eta(\bfomega)|^2 d\bfomega \right\}^{1/2} \\
&\leq \frac{2^{1-\ell_0/2} (2m_a\epsilon_n)^{\ell_0/2}}{(2\pi)^{d/2}} C_{\eta}^0 \left\{ \int_{\mathbb{R}^d} \frac{|\bfomega|^{\ell_0}}{(1+|\bfomega|^2)^2}  d \bfomega \right\}^{1/2},
\end{align*}
where the integral 
$ \int_{\mathbb{R}^d} {|\bfomega|^{\ell_0}}{(1+|\bfomega|^2)^{-2}} d\bfomega$ 
is finite for $\ell_0<\min\{2, 4-d\}$. 
Thus, there exists a constant $C_{\ell_0}$ such that 
\begin{align} \label{eqn: egg} 
\int_{\mathbb{R}^d} | e_n*g(\bfx) - g(\bfx) |^2 d\bfx \leq C_{\ell_0} \epsilon_n^{\ell_0} 
\end{align}
for $\ell_0<\min\{2, 4-d\}$.

Next, we will show some useful bounds on eigenvalues of covariance matrices based on results from spectral theory.
  
Let $b(\bfs, \bfu) :=  E_{f_{\sigma, \alpha, \beta}}[Z(\bfs)Z(\bfu)] - E_{f_{\sigma_0, \alpha_0, \beta_0}}[Z(\bfs)Z(\bfu)]$ for all $\bfs, \bfu \in \mathcal{D}=[0, L]^d$. It follows from Equation (2.24) of \cite{Wang2010} and the fact that supp($c_1$) $\subset [-2m, 2m]^d$ that for all $\bfs, \bfu \in \mathcal{D}$, 
\begin{align} \label{eqn: b fun} 
\begin{split}
b(\bfs, \bfu) &=\int_{\mathbb{R}^d} \exp\{ - i(\bfs-\bfu)^\top \bfomega \} \{ f_{\sigma, \alpha, \beta}(|\bfomega|) - f_{\sigma_0, \alpha_0, \beta_0}(|\bfomega|) \} d\bfomega \\
&= \int_{\mathbb{R}^d} \exp\{ - i(\bfs-\bfu)^\top \bfomega \} \eta(\bfomega) \xi_1(\bfomega) d\bfomega  \\
&= (2\pi)^d  \int_{\mathbb{R}^d} \int_{\mathbb{R}^d}  g(\bfx - \bfy) c_1(\bfs - \bfx) c_1(\bfu - \bfy) d\bfx d \bfy \\
&= (2\pi)^d \int_{\mathbb{R}^d}  \int_{\mathbb{R}^d} e_n * g(\bfx - \bfy) c_1(\bfs - \bfx) c_1(\bfu - \bfy) d\bfx d \bfy \\ 
& \quad + (2\pi)^d  \int_{\mathbb{R}^d} \int_{\mathbb{R}^d}  h_n^*(\bfx, \bfy) c_1(\bfs-\bfx) c_1(\bfu - \bfy) d\bfx d \bfy,
\end{split}
\end{align}
where $h_n^*(\bfx, \bfy) =\{g(\bfx-\bfy) - e_n*g(\bfx-\bfy)\} I(|\bfx+\bfy|_{\max} \leq 4m + 2L)$ for all $\bfx, \bfy \in \mathbb{R}^d$ and $h_n^*$ is square integrable. 

Define 
\begin{align*}
h_n^{**}(\bfx, \bfy) := \int_{|\bfu|_{\max}\leq 2m + 2m_a + L} e_n(\bfx - \bfu) g(\bfu- \bfy) d\bfu, \quad \forall \bfx, \bfy \in \mathbb{R}^d.
\end{align*}
The function $h_n^{**}: \mathbb{R}^{2d} \to \mathbb{C}$ is again square integrable. Direct calculation yields that the first part of Equation~\eqref{eqn: b fun} can be re-expressed as 
\begin{align} \label{eqn: b1} 
\begin{split}
& (2\pi)^d \int_{\mathbb{R}^d}  \int_{\mathbb{R}^d} e_n * g(\bfx - \bfy) c_1(\bfs - \bfx) c_2(\bfu - \bfy) d\bfx d \bfy \\
&=  (2\pi)^d \int_{\mathbb{R}^d}  \int_{\mathbb{R}^d} h_n^{**}(\bfx, \bfy) c_1(\bfs - \bfx) c_1(\bfu - \bfy) d\bfx d\bfy \\
&= (2\pi)^{-d} \int_{\mathbb{R}^d}  \int_{\mathbb{R}^d} \exp\{i(\bfomega^\top \bfs - \bfv^\top \bfu) \} \xi_1(\bfomega) \xi_1(\bfv) \\
&\quad \times \left\{ \int_{|\bft|_{\max} \leq 2m + 2m_a +L} \exp\{-i(\bfomega^\top \bft - \bfv^\top \bft) \} \hat{e}_n(\bfomega) \eta(\bfv) d \bft\right\} d\bfv d\bfomega. 
\end{split}
\end{align} 

Let $\eta_n^*: \mathbb{R}^d \to \mathbb{C}$ be the Fourier transform of $g - e_n * g$ and define 
$$\hat{g}_{n,k}(\bfomega):=\int_{\mathbb{R}^d} \exp\{-i\bfomega^\top\bfx \} [g(\bfx) - e_n*g(\bfx)] I(|\bfx|_{\max} \leq k) d\bfx.$$ This implies that 
\begin{align} \label{eqn: eta_g} 
\int_{\mathbb{R}^d} |\eta_n^*(\bfomega) - \hat{g}_{n,k}(\bfomega)  |^2 d \bfomega \to 0, \text{ as } k\to \infty.
\end{align}

Let $\theta(\bfomega) = 2^{-d} \int_{\mathbb{R}^d} \exp\{ -i\bft^\top\bfomega \} I(|\bft|_{\max} \leq 4m+2L) d\bft$. Then $\theta$ is continuous and square integrable with
\begin{align} \label{eqn: integrability of theta}
\int_{\mathbb{R}^d} \theta(\bfomega)^2 d \bfomega <\infty.
\end{align}
Direct calculation yields that the second part of Equation~\eqref{eqn: b fun} can be re-expressed as 
\begin{align} \label{eqn: b2} 
\begin{split}
&(2\pi)^d  \int_{\mathbb{R}^d} \int_{\mathbb{R}^d}  h_n^*(\bfx, \bfy) c_1(\bfs-\bfx) c_1(\bfu - \bfy) d\bfx d \bfy \\
&= (2\pi)^{-d} \int_{\mathbb{R}^d} \int_{\mathbb{R}^d} \exp\{ i (\bfomega^\top \bfs - \bfv^\top\bfu) \} \eta_n^*\left(\frac{\bfomega + \bfv}{2} \right) \theta\left(\frac{\bfomega - \bfv}{2} \right) \xi_1(\bfomega) \xi_1(\bfv) d\bfomega d\bfv.
\end{split} 
\end{align}

Combining Equations~\eqref{eqn: b1} and~\eqref{eqn: b2} allows us to write $b(\bfs,\bfu)$ as a sum of two parts: 
\begin{align*}
\begin{split}
b(\bfs, \bfu) &= (2\pi)^{-d} \int_{\mathbb{R}^d}  \int_{\mathbb{R}^d} \exp\{i(\bfomega^\top \bf\bfs - \bfv^\top \bfu) \} \xi_1(\bfomega) \xi_1(\bfv) \\
&\quad \times \left\{ \int_{|\bft|_{\max} \leq 2m + 2m_a +L} \exp\{-i(\bfomega^\top \bft - \bfv^\top \bft) \} \hat{e}_n(\bfomega) \eta(\bfv) d \bft\right\} d\bfv d\bfomega \\
&\quad + (2\pi)^{-d} \int_{\mathbb{R}^d} \int_{\mathbb{R}^d} \exp\{ i (\bfomega^\top \bfs - \bfv^\top \bfu) \} \eta_n^*\left(\frac{\bfomega + \bfv}{2} \right) \theta\left(\frac{\bfomega - \bfv}{2} \right) \xi_1(\bfomega) \xi_1(\bfv) d\bfomega d\bfv. 
\end{split} 
\end{align*}

In the rest of the proof, we will relate $b(\bfs, \bfu)$ with the eigenvalues of the CH covariance matrix and give bounds on these eigenvalues. Let $\{ \psi_1, \ldots, \psi_n\}$ be as in Equation (2.15) of \cite{Wang2010}. Then it follows from Equations~(2.16) and (2.60) of \cite{Wang2010} that 
\begin{align*}
\langle \psi_k, \psi_k\rangle_{f_{\sigma, \alpha, \beta}} - \langle \psi_k, \psi_k\rangle_{f_{\sigma_0, \alpha_0, \beta_0}} = \lambda_{k,n} - 1 =: \tilde{\nu}_{k,n}^* + \tilde{\nu}_{k,n}^{\dagger},
\end{align*} 
where 
\begin{align*}
 \tilde{\nu}_{k,n}^* &:= \frac{1}{(2\pi)^d} \int_{\mathbb{R}^d} \int_{\mathbb{R}^d}  \psi_k(\bfomega) \overline{\psi_k(\bfv)} \eta_n^*\left(\frac{\bfomega + \bfv}{2} \right) \theta\left(\frac{\bfomega - \bfv}{2} \right) \xi_1(\bfomega) \xi_1(\bfv) d\bfomega d\bfv, \\
 \tilde{\nu}_{k,n}^{\dagger} &:=  \frac{1}{(2\pi)^d} \int_{\mathbb{R}^d} \int_{\mathbb{R}^d}  \psi_k(\bfomega) \overline{\psi_k(\bfv)} \xi_1(\bfomega) \xi_1(\bfv)\\
 & \quad \times \left\{ \int_{|\bft|_{\max} \leq 2m + 2m_a +L} \exp\{-i(\bfomega^\top \bft - \bfv^\top \bft) \} \hat{e}_n(\bfomega) \eta(\bfv) d \bft\right\} d\bfv d \bfomega,
\end{align*}
with $\bar{x}$ denote the complex conjugate of $x$. 
It then follows from Bessel's inequality that 
\begin{align*}
\sum_{k=1}^n |\tilde{\nu}_{k,n}^*|^2 &\leq 2^{-d-1} \pi^{-d} \left\{ \sup_{\bfs \in \mathbb{R}^d} \frac{\xi_1(\bfs)^2}{f_{\sigma_0, \alpha_0, \beta_0}(\bfs)} \right\} \int_{\mathbb{R}^d} |\eta_n^*(\bfomega)|^2 d \bfomega \int_{\mathbb{R}^d} |\bftheta(\bfv)|^2 d\bfv, \\
\sum_{k=1}^n |\tilde{\nu}_{k,n}^\dagger| &\leq 2^{-d-1} \pi^{-d} \left\{ \sup_{\bfs \in \mathbb{R}^d} \frac{\xi_1(\bfs)^2}{f_{\sigma_0, \alpha_0, \beta_0}(\bfs)} \right\}  \int_{|\bft|_{\max} \leq 2m + 2m_a +L} d \bft \\
&\quad \times \left\{ \int_{\mathbb{R}^d} |\hat{e}_n(\bfomega) |^2 d \bfomega + \int_{\mathbb{R}^d} \eta(\bfv)^2 d\bfv \right\}. 
\end{align*} 

Combining the above inequalities with Equations~\eqref{eqn: f over xi}, \eqref{eqn: bound for e_n_hat}, \eqref{eqn: egg}, \eqref{eqn: eta_g}, and~\eqref{eqn: integrability of theta}, we observe that there exist constants $C, C_1, C_2$ (not depending on $n$) such that 
\begin{align*}
\sum_{k=1}^n |\tilde{\nu}_{k,n}^*|^2 \leq C \epsilon_n^{\ell_0}, \\
 \sum_{k=1}^n |\tilde{\nu}_{k,n}^*| \leq \left\{n\sum_{k=1}^n |\tilde{\nu}_{k,n}^*|^2\right\}^{1/2} \leq \sqrt{Cn\epsilon_n^{\ell_0}}, \\
\sum_{k=1}^n | \tilde{\nu}_{k,n}^\dagger| \leq \left(\frac{C_1}{\epsilon_n^{d}} + C_2 E\right), 
\end{align*}
where $E:=\int_{\mathbb{R}^d} \eta(\bfv)^2 d \bfv$ is finite. Thus it follows that 
\begin{align} \label{eqn: bound for lambda-1}
\sum_{k=1}^n |\lambda_{k,n}-1| \leq \sqrt{Cn\epsilon_n^{\ell_0}} + \frac{C_1}{\epsilon_n^{d}} + C_2 E. 
\end{align}

When $\alpha, \alpha_0>d/2$, the spectral density of the CH covariance function is well-defined and is finite for any frequency. Thus, the ratio 
$$ \frac{f_{\sigma, \alpha, \beta}(|\bfomega|)}{f_{\sigma_0, \alpha_0, \beta_0}(|\bfomega|)}$$
is well-defined for all $\bfomega \in \mathbb{R}^d$.   We also observe that there exist constants $\tilde{c}_f>0$ and $\tilde{C}_f>0$ (not depending on $\bfomega$) such that 
\begin{align*}
\tilde{c}_f \leq \frac{f_{\sigma, \alpha, \beta}(|\bfomega|)}{f_{\sigma_0, \alpha_0, \beta_0}(|\bfomega|)} \leq \tilde{C}_f,\quad \forall \bfomega \in \mathbb{R}^d.
\end{align*}
It follows immediately that $\tilde{c}_f \leq \lambda_{k,n} \leq \tilde{C}_f$ for all $k=1, \ldots, n$.

Finally, using \eqref{eqn: bound for lambda-1} yields that for any $\epsilon>0$, 
\begin{align} \label{eqn: bound for prob} 
\frac{1}{\epsilon\sqrt{n}}  \left\{ \max_{1\leq i \leq n}  \{ \lambda_{i,n}^{-1}\} \right\} \sum_{k=1}^n|\lambda_{k,n} - 1| \leq \frac{C^{1/2} \epsilon_n^{\ell_0/2}}{\tilde{c}_f \epsilon } + \frac{1}{\tilde{c}_f \epsilon n^{1/2} }  \left(\frac{C_1}{\epsilon_n^{d}} + C_2 E\right).
\end{align}
Therefore, choosing $\epsilon_n$ such that $\epsilon_n \to 0$ and $n^{1/2}\epsilon_n^d\to \infty$ as $n\to \infty$ yields that the right-hand side of Equation~\eqref{eqn: bound for prob} tends to 0 as $n\to \infty$, as desired. 

\end{proof}

Based on Lemma~\ref{lem: diagonalization} and Lemma~\ref{lem: inequality}, we can study the asymptotic behavior of the MLE for the microergodic parameter of the CH covariance function. 

\begin{lemma} \label{thm: ML consistency}
Let $\{\mathcal{D}_n\}_{n\geq 1}$ be an increasing sequence of subsets of a bounded domain $\mathcal{D}$ such that $\cup_{n=1}^{\infty} \mathcal{D}_n$ is bounded. Assume that $\nu$ is fixed. Let $\mathcal{P}_0$ be the Gaussian probability measure defined under $C(h; \nu, \alpha_0, \beta_0, \sigma^2_0 )$. Let $c(\bftheta_0) := {\sigma^2_0\beta_0^{-2\nu} \Gamma(\nu+\alpha_0)}/{\Gamma(\alpha_0)}$. Assume that $\alpha_0>d/2, \beta_0>0, \sigma^2_0>0$. The following results can be established.
\begin{enumerate}[itemsep=-2pt, topsep=0pt, partopsep=0pt]
\item[(a)] As $n \to \infty$, $\hat{c}_n(\bftheta) \stackrel{a.s.}{\longrightarrow} c(\bftheta_0)$ under measure $\mathcal{P}_0$ for any fixed $\alpha>d/2$ and $\beta>0$;

\item[(b)] As $n \to \infty$, $\sqrt{n} \left  \{ \hat{c}_n(\bftheta) - c(\bftheta_0) \right\} \stackrel{\mathcal{L}}{\longrightarrow} \mathcal{N}\left (0, 2[c(\bftheta_0)]^2  \right) $ for any fixed $\alpha>d/2$ and $\beta>0$.
\end{enumerate}
\end{lemma}

\begin{proof}
The proof of Part (a) follows from the same arguments as in the proof of Theorem 3 in \cite{Zhang2004} and is omitted. For the proof of Part (b), we follow the arguments in \cite{Wang2010, Wang2011} and \cite{Bevilacqua2019AOS}. Without loss of generality, we assume $\mathcal{D}=[0, L]^d, 0<L<\infty$ is a bounded subset of $\mathbb{R}^d$ with $d=1,2,3$. Let $\sigma^2$ be a positive constant such that $\sigma^2 \beta^{-2\nu} \Gamma(\nu+\alpha) / \Gamma(\alpha) = \sigma^2_0 \beta_0^{-2\nu} \Gamma(\nu+\alpha_0) / \Gamma(\alpha_0)$. Let $c(\bftheta)=\sigma^2 \beta^{-2\nu} \Gamma(\nu+\alpha) / \Gamma(\alpha)$ and $\hat{c}_n(\bftheta)=\hat{\sigma}^2_n \beta^{-2\nu} \Gamma(\nu+\alpha) / \Gamma(\alpha)$. Then we have 
\begin{align*}
\sqrt{n}\left\{\hat{c}_n(\bftheta) - c(\bftheta_0) \right\} &= \frac{c(\bftheta_0)}{\sqrt{n}} \left\{   \frac{1}{\sigma^2} \bfZ_n^\top \bfR_n^{-1}(\bftheta) \bfZ_n - \frac{1}{\sigma^2_0} \bfZ_n^\top \bfR_n^{-1}(\bftheta_0)\bfZ_n \right\} \\
&\quad+ \frac{c(\bftheta_0)}{\sqrt{n}} \left\{ \frac{1}{\sigma^2_0} \bfZ_n^\top \bfR_n^{-1}(\bftheta_0)\bfZ_n -n  \right\}.
\end{align*}
Under Gaussian measure $\mathcal{P}_0$ defined by the covariance function $C(h; \nu, \alpha_0, \beta_0, \sigma^2_0)$, we have  
\begin{align*}
\bfZ_n^\top \bfR_n^{-1}(\bftheta_0)\bfZ_n/\sigma^2_0 \sim \chi_n^2  \quad \text{ and} \quad
 \frac{c(\bftheta_0)}{\sqrt{n}} \left\{ \frac{1}{\sigma^2_0} \bfZ_n^\top \bfR_n^{-1}(\bftheta_0)\bfZ_n -n  \right\} \stackrel{\mathcal{L}}{\longrightarrow} \mathcal{N}(0,\, 2[c(\bftheta_0)]^2),
\end{align*}
as $n\to \infty$. To prove the result, it suffices to show that 
\begin{align} \label{eqn: converge in prob}
\frac{1}{\sqrt{n}}\left\{  \frac{1}{\sigma^2} \bfZ_n^\top \bfR_n^{-1}(\bftheta) \bfZ_n - \frac{1}{\sigma^2_0} \bfZ_n^\top \bfR_n^{-1}(\bftheta_0)\bfZ_n \right\} \stackrel{\mathcal{P}_0}{\longrightarrow} 0, \quad \text{as}\quad n\to \infty,
\end{align}
under Gaussian measure $\mathcal{P}_0$. 

According to Lemma~\ref{lem: diagonalization}, there exists an $n\times n$ non-singular matrix $\bfU$ such that 
\begin{align*}
\sigma^2_0 \bfU^\top \bfR_n(\bftheta_0) \bfU = \mathbf{I}_n,\quad \sigma^2 \bfU^\top \bfR_n(\bftheta) \bfU = \bfLambda,
\end{align*}
where $\bfLambda:=\diag\{\lambda_{1,n}, \ldots, \lambda_{n,n}\}$ is an $n\times n$ diagonal matrix with diagonal elements satisfying $\lambda_{k,n}>0$ for $k=1, \ldots, n$. Now we define the random vector $\bfY:=(Y_1, \ldots, Y_n)^\top = \bfU^\top \bfZ_n$. It is easy to check that $\bfY \sim \mathcal{N}_n(\mathbf{0}, \mathbf{I}_n)$ for $\bfZ_n$ generated under the measure $\mathcal{P}_0$. Thus, the assertion~\eqref{eqn: converge in prob}  is true if for any $\epsilon>0$,
\begin{align} \label{eqn: converge in prob 2}
\begin{split}
&\mathcal{P}_0\left( \frac{1}{\sqrt{n}} \left| \frac{1}{\sigma^2} \bfZ_n^\top \bfR_n^{-1}(\bftheta) \bfZ_n - \frac{1}{\sigma^2_0} \bfZ_n^\top \bfR_n^{-1}(\bftheta_0)\bfZ_n \right| > \epsilon \right) \\
&=  \mathcal{P}_0\left( \frac{1}{\sqrt{n}}  \left| \sum_{k=1}^n (\lambda_{k,n}^{-1} - 1) Y_k^2 \right | > \epsilon \right) \to 0, \quad \text{as} \quad n\to \infty. 
\end{split}
\end{align}
By Markov's inequality, the probability in the assertion~\eqref{eqn: converge in prob 2} can be bounded as  
\begin{align*}
\mathcal{P}_0\left( \frac{1}{\sqrt{n}}  \left| \sum_{k=1}^n (\lambda_{k,n}^{-1} - 1) Y_k^2 \right | > \epsilon \right) \leq \frac{1}{\epsilon\sqrt{n}} \sum_{k=1}^n |\lambda_{k,n}^{-1} - 1| \leq \frac{1}{\epsilon\sqrt{n}}  \left\{\max_{1\leq i \leq n} \{ \lambda_{i,n}^{-1}\} \right\} \sum_{k=1}^n |\lambda_{k,n} - 1|.
\end{align*}
The rest of the proof is to show that for any $\epsilon>0$, the term
$$\frac{1}{\epsilon\sqrt{n}}  \left\{\max_{1\leq i \leq n} \{ \lambda_{i,n}^{-1}\} \right\} \sum_{k=1}^n |\lambda_{k,n} - 1| $$ 
goes to 0 as $n\to \infty$.  This is true according to Lemma~\ref{lem: inequality}. 
\end{proof}

Lemma~\ref{thm: ML consistency} implies that the estimator  $\hat{c}_n(\bftheta)$ of the microergodic parameter  converges to the true microergodic parameter, almost surely, when the number of observations tends to infinity in a fixed and bounded domain. This result holds true for any value of $\bftheta$. As will be shown, if one replaces $\bftheta$ with its maximum likelihood estimator in $\hat{c}_n(\bftheta)$, this conclusion is true as well. The second statement of Lemma~\ref{thm: ML consistency} indicates that $\hat{c}_n(\bftheta)$ converges to a normal distribution. 

A key fact is that the above lemma holds true for arbitrarily fixed $\bftheta$. A more practical situation is to estimate $\bftheta$ and $\sigma^2$ by maximizing the log-likelihood $\eqref{eqn: loglikelihood}$. The following lemma is needed to prove the asymptotic behavior of $\hat{c}_n(\alpha, \hat{\beta}_n)$ and $\hat{c}_n(\hat{\alpha}_n, \beta)$ under infill asymptotics.

\begin{lemma} \label{lem: monotonicity}
Suppose that $d$ is the dimension of the domain $\mathcal{D}$ and $\bfZ_n$ is a vector of $n$ observations in $\mathcal{D}$. For any $\alpha_1, \alpha_2$ such that $d/2<\alpha_1<\alpha_2$ and any $\beta_1, \beta_2$ such that $0<\beta_1<\beta_2$, we have the following results: 
\begin{itemize}[itemsep=-2pt, topsep=0pt, partopsep=0pt]
\item[(a)] $\hat{c}_n(\alpha, \beta_1) \leq \hat{c}_n(\alpha, \beta_2)$ for any fixed $\alpha>d/2$.
\item[(b)] $\hat{c}_n(\alpha_1, \beta) \geq \hat{c}_n(\alpha_2, \beta)$ for any fixed $\beta>0$.   
\end{itemize} 
\end{lemma}

\begin{proof}
The difference
\begin{align*}
\hat{c}_n(\bftheta_1) - \hat{c}_n(\bftheta_2) = \bfZ_n^\top \left\{ \frac{ \Gamma(\nu+\alpha_1)}{\beta_1^{2\nu}\Gamma(\alpha_1)} \bfR_n^{-1}(\bftheta_1)  - \frac{ \Gamma(\nu+\alpha_2)}{\beta_2^{2\nu}\Gamma(\alpha_2)} \bfR_n^{-1}(\bftheta_2)  \right\} \bfZ_n  /n
\end{align*}
is nonnegative for any $\bfZ_n$ if the matrix $\bfA:= \frac{ \Gamma(\nu+\alpha_1)}{\beta_1^{2\nu}\Gamma(\alpha_1)} \bfR_n^{-1}(\bftheta_1)  - \frac{ \Gamma(\nu+\alpha_2)}{\beta_2^{2\nu}\Gamma(\alpha_2)} \bfR_n^{-1}(\bftheta_2)$ is positive semidefinite. Notice that $\bfA$ is positive semidefinite if and only if $\bfB:= \frac{\beta_2^{2\nu} \Gamma(\alpha_2)}{\Gamma(\nu+\alpha_2)} \bfR_n(\bftheta_2) - \frac{\beta_1^{2\nu}\Gamma(\alpha_1)}{ \Gamma(\nu+\alpha_1)} \bfR_n(\bftheta_1)$ is positive semidefinite. The entries of $\bfB$ can be expressed in terms of a function $K_B: \mathbb{R}^d \rightarrow \mathbb{R}$, with 
\begin{align*}
B_{ij} = K_B(\bfs_i - \bfs_j) = \frac{\beta_2^{2\nu} \Gamma(\alpha_2) }{ \Gamma(\nu+ \alpha_2)} R(|\bfs_i-\bfs_j|; \alpha_2, \beta_2, \nu) -  \frac{\beta_1^{2\nu} \Gamma(\alpha_1) }{ \Gamma(\nu+ \alpha_1)} R(|\bfs_i - \bfs_j|; \alpha_1, \beta_1, \nu), 
\end{align*}
where $|\cdot|$ denotes the Euclidean norm. 
The matrix $\bfB$ is positive semidefinite if $K_B$ is a positive definite function. Define the Fourier transform of $K_B$ by 
\begin{align*} 
\begin{split}
f_B(\bfomega) &:= \frac{1}{(2\pi)^d} \int_{\mathbb{R}^d} \exp\{-i\bfomega^\top \bfx\} K_B(\bfx) d\bfx  \\
&=   \frac{\beta_2^{2\nu} \Gamma(\alpha_2) }{ \Gamma(\nu+ \alpha_2)} \left\{ \frac{1}{(2\pi)^d} \int_{\mathbb{R}^d} \exp\{-i\bfomega^\top \bfx\}  R(|\bfx|; \alpha_2, \beta_2, \nu) d\bfx\right\} \\
&\quad- \frac{\beta_1^{2\nu} \Gamma(\alpha_1) }{ \Gamma(\nu+ \alpha_1)} \left\{\frac{1}{(2\pi)^d} \int_{\mathbb{R}^d} \exp\{-i\bfomega^\top \bfx\}   R(|\bfx|; \alpha_1, \beta_1, \nu) d\bfx \right\}.
\end{split}
\end{align*}
The integrals in $f_B(\bfomega)$ are finite for $\alpha_1, \alpha_2>d/2$. Let 
$
g(\bfomega)
$ 
be the spectral density of the CH correlation function with parameters $\alpha, \beta, \nu$: 
\begin{align*}
g(\bfomega) &:= \frac{1}{(2\pi)^d} \int_{\mathbb{R}^d} \exp\{-i\bfomega^\top \bfx\}  R(|\bfx|; \alpha, \beta, \nu) d\bfx \\
&\, = \frac{2^{2\nu} \nu^{\nu}}{\pi^{d/2} \beta^{2\nu} \Gamma(\alpha)} \int_0^{\infty} \{4\nu/(\beta^2 t) + |\bfomega|^2\}^{-(\nu+d/2)} t^{-(\nu+\alpha+1)} \exp\{-1/t  \} d t.
\end{align*}
Thus, $K_B$ is positive definite if $f_B$ is positive for all $\bfomega \in \mathbb{R}^d$.  
Notice that $f_B$ is given by 
\begin{align*}
f_B(\bfomega) &= \frac{(4\nu)^{\nu}}{\pi^{d/2} \Gamma(\nu+\alpha_2)} \int_0^{\infty} \{4\nu/(\beta_2^2 t) + |\bfomega|^2\}^{-(\nu+d/2)} t^{-(\nu+\alpha_2+1)} \exp\{-1/t  \} d t \\
&\quad - \frac{(4\nu)^{\nu}}{\pi^{d/2} \Gamma(\nu+\alpha_1)} \int_0^{\infty} \{4\nu/(\beta_1^2 t) + |\bfomega|^2\}^{-(\nu+d/2)} t^{-(\nu+\alpha_1+1)} \exp\{-1/t  \} d t.
\end{align*}
It is straightforward to check that when $\alpha:=\alpha_1=\alpha_2>d/2$, 
\begin{align*}
\beta_1 < \beta_2 \implies f_B(\bfomega) >0, &\quad  \forall\bfomega \in \mathbb{R}^d. 
\end{align*}
Thus, if $\beta_1<\beta_2$, then $\hat{c}_n(\alpha, \beta_1) \leq \hat{c}_n(\alpha, \beta_2)$, as claimed in Part (a). 

The proof of Part (b) is as follows. Note that $f_B(\bfomega)$ can be expressed as $f_B(\bfomega) = \{(4\nu)^\nu/\pi^{d/2}\} (I(\alpha_2) - I(\alpha_1))$, where
\begin{align*}
I(\alpha) &:=  \int_0^{\infty} {\frac{1}{\Gamma(\nu+\alpha)}} \{4\nu/(\beta^2 t) + |\bfomega|^2\}^{-(\nu+d/2)} t^{-(\nu+\alpha+1)} \exp\{-1/t  \} d t \\
  &= \int_0^{\infty} \frac{u^{\nu+\alpha-1}}{\Gamma(\nu+\alpha)} \exp\{-u \} \left( 4\nu u/\beta^2 + |\bfomega|^2 \right)^{-(\nu+d/2)} d u\\
 &= E_{U}\left(4\nu U/\beta^2 + |\bfomega|^2 \right)^{-(\nu+d/2)},
\end{align*}
with $U\sim \text{Gamma}(\nu+\alpha, 1)$. This expectation is finite if $\alpha>d/2$. Suppose that $\alpha_1 <\alpha_2$ and $\beta:=\beta_1=\beta_2$. To show $f_B(\bfomega)$ is negative for all $\bfomega \in \mathbb{R}^d$, it suffices to show that $I(\alpha_2) - I(\alpha_1)\leq 0$. Let $U_1\sim \text{Gamma}(\nu+ \alpha_1, 1)$ and $U_2\sim \text{Gamma}(\nu+\alpha_2, 1)$. Then $U_2 \stackrel{\mathcal{L}}= U_1 + U_0$, where $U_0\sim \text{Gamma}(\alpha_2-\alpha_1, 1)$ and $U_0$ is independent of $U_1$.  Thus, the quantify $I(\alpha_2)$ can be upper bounded by $I(\alpha_1)$, since,
\begin{align*}
I(\alpha_2) &= E_{U_1, U_0}\left\{ \frac{4\nu}{\beta^2} (U_1 + U_0)  +|\bfomega|^2  \right\} ^{-(\nu+d/2)}  \\
                 & =E_{U_1, U_0}\left\{\frac{4\nu}{\beta^2} U_1  + |\bfomega|^2 +  \frac{4\nu}{\beta^2} U_0  \right\} ^{-(\nu+d/2)}  \\
       & \leq E_{U_1} \left\{\frac{4\nu}{\beta^2} U_1  + |\bfomega|^2  \right\} ^{-(\nu+d/2)} = I(\alpha_1).
\end{align*}

\end{proof}

This lemma indicates that the MLE of the microergodic parameter is monotone when one of its parameters is fixed. This property is used to prove the asymptotics of the MLE for the microergodic parameter. Based on Lemma~\ref{thm: ML consistency} and Lemma~\ref{lem: monotonicity} , one can show that $\hat{c}_n(\hat{\bftheta}_n)$ has the same asymptotic properties as $\hat{c}_n(\bftheta)$ for any fixed $\bftheta$.

\section{Technical Proofs} \label{app: proof}
This section contains all the proofs that are not given in the main text. For notational convenience, we drop the parameters of the covariance function when there is no scope for ambiguity.

\subsection{Proof of Theorem~\ref{lem: matern positive function}} \label{app: matern positive function}

\begin{proof}

As $C(0)=\sigma^2>0$, it remains to verify the positive definiteness of the function $C(\cdot)$. For any $n$, all sequences $\{a_i\in \mathbb{R}: i=1, \ldots, n \}$ and all sequences of spatial locations $\{\bfs_i \in \mathbb{R}^d: i=1, \ldots, n\}$, it follows that 
\begin{align*}
    \sum_{i=1}^n\sum_{j=1}^n a_i a_j C(h_{ij}; \nu, \alpha, \beta, \sigma^2) &= \sum_{i=1}^n\sum_{j=1}^n a_i a_j \int_{0}^{\infty} \mathcal{M}(h_{ij}; \nu, \phi, \sigma^2) \pi(\phi^2; \alpha, \beta) d \phi^2 \\
    &= \int_{0}^{\infty}  \bfa^\top \bfA \bfa \pi(\phi^2; \alpha, \beta) d \phi^2\geq 0,
\end{align*}
where $h_{ij}=|\bfs_i - \bfs_j|$ with $|\cdot|$ denoting the Euclidean norm and $\bfa:=(a_1, \ldots, a_n)^\top$. The matrix $\bfA:=[\mathcal{M}(h_{ij})]_{i,j=1,\ldots,n}$ is a covariance matrix constructed via a Mat\'ern covariance function that is positive definite in $\mathbb{R}^d$ for all $d$, and hence $\bfA$ is a positive definite matrix, which yields that $\bfa^\top \bfA \bfa\geq 0$ for any $\bfa$. This implies that the resultant integral is nonnegative for any $\bfa$, and it is strictly positive for $\bfa\neq \bfzero$. Thus, the function $C(h)$ is positive definite in $\mathbb{R}^d$ for any all $d$.

To derive the form of Equation \eqref{eqn: new kernel}, we start with the gamma mixture representation in Equation~\eqref{eq:mm}, and substitute for $\pi(\phi^2)$ the required inverse gamma density.
\begin{align*}
C(h; \nu, \alpha, \beta, \sigma^2) &= \frac{\sigma^2}{2^{\nu}\Gamma(\nu) }  \int_{0}^{\infty} x^{(\nu-1)}\left[ \int_{0}^{\infty} \phi^{-2\nu} \exp\{-x/(2\phi^2)\} \pi(\phi^2) d\phi^2\right] \exp{(- \nu h^2/x)}  dx\\
&= \frac{\sigma^2\beta^{2\alpha}}{2^{\nu+\alpha}\Gamma(\nu) \Gamma(\alpha)} \int_{0}^{\infty} x^{(\nu-1)}\left[ \int_{0}^{\infty} \phi^{-2\nu} \exp\{-x/(2\phi^2)\} \phi^{-2(\alpha+1)} \right. \\
&\quad \left. \times \exp\{-\beta^2/(2\phi^2)\}d\phi^2\right]  \exp{(- \nu h^2/x)}  dx\\
&= \frac{\sigma^2\beta^{2\alpha}}{2^{\nu+\alpha}\Gamma(\nu) \Gamma(\alpha)}  \int_{0}^{\infty} x^{(\nu-1)}\left[ \int_{0}^{\infty} \phi^{-2(\nu+\alpha+1)} \exp\{-(\beta^2+x)/(2\phi^2)\}d\phi^2\right] \\
&\quad \times \exp\{- \nu h^2/x\}  dx\\
&=  \frac{\sigma^2\beta^{2\alpha}\Gamma(\nu+\alpha)}{\Gamma(\nu)\Gamma(\alpha)}  \int_{0}^{\infty} x^{(\nu-1)} (x+\beta^2)^{-(\nu+\alpha)} \exp{(- \nu h^2/x)}  dx.
\end{align*}

\end{proof}

\subsection{Proof of Theorem~\ref{thm: new horseshoe class}} \label{app: new horseshoe class}
\begin{proof}

\begin{itemize}

\item[(a)] Using the property of modified Bessel function \citep[see][p.~375]{Abramowitz1965}, as $|h| \to 0$, we can express the Mat\'ern covariance function as 
\begin{align*}
\mathcal{M}(h) =  \begin{cases}
a_1(h) + a_2(\phi, \nu, \sigma^2) |h|^{2\nu} \log |h| + O(|h|^{2\nu}); & \text{when }\quad  \nu =0, 1, 2, \ldots, \\
a_3(h) + a_4(\phi, \nu, \sigma^2) |h|^{2\nu} + O(|h|^{2\lceil \nu \rceil}); & \text{otherwise},
\end{cases}
\end{align*}
where $a_i(h), i=1,3$ are of the form $\sum_{k=0}^{\lfloor \nu \rfloor} c_k(\phi, \nu, \sigma^2) h^{2k}$ with $c_k(\phi,\nu, \sigma^2)$ being the coefficients that depend on parameters $\phi, \nu, \sigma^2$. The terms $a_2(\phi, \nu,\sigma^2) = \frac{(-1)^{\nu+1}\sigma^2}{2^{\nu-1}\Gamma(\nu)\Gamma(\nu+1) \phi^{2\nu}}$ and $a_4(\phi, \nu, \sigma^2)=\frac{-\pi\sigma^2}{2^{\nu} \sin(\nu\pi) \Gamma(\nu) \Gamma(\nu+1) \phi^{2\nu}}$. The terms $a_2(\phi, \nu,\sigma^2) |h|^{2\nu} \log |h|$ and $a_4(\phi, \nu,$ $\sigma^2) |h|^{2\nu}$ are called \emph{principal irregular} terms that determine the differentiability of a random field \citep[see][p.~32]{Stein1999}. This implies that the Mat\'ern covariance function is $2m$ times differentiable if and only if $\nu >m$ for an integer $m$. By mixing the parameter $\phi^2$ over an inverse gamma distribution $\mathcal{IG}(\alpha, \beta^2/2)$, when $h\to 0$, the covariance function $C(h)$ can be written as 
\begin{align*}
C(h) =  \begin{cases}
\int_{0}^{\infty} a_1(h) \pi(\phi^2) d \phi^2 +   \tilde{a}_2(\nu,\sigma^2)  |h|^{2\nu} \log |h| + O(|h|^{2\nu}); & \text{when }\quad  \nu =0, 1, 2, \ldots, \\
\int_0^{\infty} a_3(h) \pi(\phi^2) d \phi^2 +  \tilde{a}_4(\nu,\sigma^2) |h|^{2\nu} + O(|h|^{2\lceil \nu \rceil}); & \text{otherwise},
\end{cases}
\end{align*}
where 
\begin{align*}
\tilde{a}_2(\nu,\sigma^2) &:=\int_0^{\infty} a_2(\phi, \nu,\sigma^2) \pi(\phi^2) d\phi^2 \\
&= \frac{2(-1)^{\nu+1} \sigma^2}{\Gamma(\nu) \Gamma(\nu+1)}  \frac{\Gamma(\nu+\alpha)}{\beta^{2\nu} \Gamma(\alpha)},
\end{align*}
and  
\begin{align*}
\tilde{a}_4(\nu,\sigma^2) &:=\int_0^{\infty} a_4(\phi, \nu,\sigma^2) \pi(\phi^2) d\phi^2 \\
&= \frac{-\pi \sigma^2}{\sin(\nu\pi) \Gamma(\nu) \Gamma(\nu+1)} \frac{\Gamma(\nu+\alpha)}{\beta^{2\nu} \Gamma(\alpha)} \\
&=\frac{-\sigma^2 \Gamma(1-\nu)}{ \Gamma(\nu+1)}  \frac{\Gamma(\nu+\alpha)}{\beta^{2\nu} \Gamma(\alpha)}.
\end{align*}
Note that $\tilde{a}_2(\nu,\sigma^2)$ is finite for any positive integer $\nu$ and any fixed $\alpha>0, \beta>0$,  and $\tilde{a}_4(\nu,\sigma^2)$ is  finite for $\nu \in (0, \infty) \setminus \mathbb{Z}$ and $\alpha>0, \beta>0$. Thus, the covariance $C(h)$ has the same differentiability as the Mat\'ern covariance. 

\item[(b)]
It follows from Theorem~\ref{lem: matern positive function} that 
\begin{align*}
C(h; \nu, \alpha, \beta, \sigma^2) 
&= \frac{\sigma^2\beta^{2\alpha}\Gamma(\nu+\alpha)}{\Gamma(\nu)\Gamma(\alpha)}   \int_{0}^{\infty} \left(\frac{x}{x+\beta^2}\right)^{\nu+\alpha} x^{-\alpha-1} \exp{(- \nu h^2/x)}  dx\\
& \xlongequal{t=x/(2\nu)} \frac{\sigma^2\Gamma(\nu+\alpha)}{(2\nu/\beta^2)^{\alpha}\Gamma(\nu)\Gamma(\alpha)} \int_0^{\infty} t^{\nu-1} (t + \beta^2/(2\nu))^{-(\nu+\alpha)}  \\
&\quad\quad \quad \quad \times \exp\{- h^2/(2t) \}  \, dt \\
& = \frac{\sigma^2\sqrt{2\pi}\Gamma(\nu+\alpha)}{(2\nu/\beta^2)^{\alpha}\Gamma(\nu)\Gamma(\alpha)} \int_0^{\infty} \left( \frac{t}{t+\beta^2/(2\nu)} \right)^{\nu+\alpha}  t^{-\alpha-1/2} \\
&\quad \times \frac{1}{\sqrt{2\pi t}} \exp\{- h^2/(2t) \}  \, dt. 
\end{align*}
Let $L(x) = \left(\frac{x}{x+\beta^2/(2\nu)}\right)^{\nu + \alpha }$. Then $L(x)$ is a slowly varying function at $\infty$. Viewed as a function of $h$, the above integral is a Gaussian scale mixture with respect to $t$. Thus, an application of Theorem 6.1 of \cite{barndorff1982normal} yields
\begin{align*}
\begin{split}
C(h; \nu, \alpha, \beta, \sigma^2) &\sim \frac{\sigma^2\sqrt{2\pi}\Gamma(\nu+\alpha)}{(2\nu/\beta^2)^{\alpha}\Gamma(\nu)\Gamma(\alpha)}  (2\pi)^{-1/2} 2^{\alpha} \Gamma(\alpha) |h|^{-2\alpha} L(h^2),  \quad \text{ as }\quad h \to \infty, \\
       &\sim   \frac{\sigma^2 2^{\alpha}  \Gamma(\nu+\alpha) }{(2\nu/\beta^2)^{\alpha}\Gamma(\nu)}    |h|^{-2\alpha} L(h^2),   \quad \text{ as } \quad h \to \infty \\
       &\sim \frac{\sigma^2 \beta^{2\alpha}  \Gamma(\nu+\alpha) }{\nu^{\alpha}\Gamma(\nu)}    |h|^{-2\alpha} L(h^2),   \quad \text{ as } \quad h \to \infty.
\end{split}
\end{align*}
Thus, the tail decays as $|h|^{-2\alpha} L(h^2)$ when $\alpha>0$. 

\end{itemize}

\end{proof}

\subsection{Proof of Proposition~\ref{thm: spectral density}} \label{app: spectral density}

\begin{proof}
Let $\Phi_d$ denote the family of the continuous functions from $[0, \infty)$ to $\mathbb{R}$ that represent correlation functions of stationary and isotropic random processes on $\mathbb{R}^d$. Then the family $\Phi_d$ is nested satisfying $\Phi_1\supset \Phi_2 \supset \cdots \supset \Phi_{\infty}$, where $\Phi_{\infty}:=\cap_{d\geq 1} \Phi_d$ is the family of radial functions that are positive definite on any number of dimensions in Euclidean space.

The proof consists of two parts. We first show that the CH correlation function belongs to $\Phi_d$, from $[0, \infty)$ to $\mathbb{R}$. Then we use Theorem 6.1 of \cite{barndorff1982normal} to derive the tail behavior of the spectral density. 

Note that \cite{Schoenberg1938} shows that any member $\psi$ that is in the family $\Phi_d$ can be written as a scale mixture with a probability measure $F$ on $[0, \infty)$:
\begin{align*}
\psi(h) = \int_0^{\infty} h^{-(d-2)/2} \calJ_{(d-2)/2}(\omega h) dF(\omega), \quad h\geq 0,
\end{align*}
where $\calJ_{\nu}(\cdot)$ is the ordinary Bessel function \citep[see 9.1.20 of][]{Abramowitz1965}. It is well-known \citep[see Chapter 2 of][]{Matern1960} that the Mat\'ern correlation is positive definite in any number of dimensions in Euclidean space and  it  is a member of $\Phi_{\infty}$ for any positive values of $\phi$ and $\nu$. The CH correlation function as a scale mixture of the Mat\'ern correlation is also a member of $\Phi_{\infty}$ (see the proof in Theorem~\ref{lem: matern positive function}). The Fourier transform of $f\in \Phi_d$, denoted by $\mathcal{F}(f)$, is available in a convenient form \citep{Yaglom1987} with
\begin{align*}
\mathcal{F}(f)(\omega) & =(2\pi)^{-d/2} \int_0^{\infty} (u\omega)^{-(d-2)/2} \calJ_{(d-2)/2}(u\omega) u^{d-1} f(u) du, \quad \omega\geq 0. 
\end{align*}
Notice that the Mat\'ern covariance function~\eqref{eqn: Matern} has spectral density 
\begin{align*}
f_{\mathcal{M}}(\omega) &= (2\pi)^{-d/2} \int_0^{\infty} (\omega h)^{-(d-2)/2} \calJ_{(d-2)/2}(\omega h) h^{d-1} \mathcal{M}(h) \, dh, \\
                    &= \frac{\sigma^2 (\sqrt{2\nu}/\phi)^{2\nu} }{\pi^{d/2} ((\sqrt{2\nu}/\phi)^2 + \omega^2)^{\nu+d/2}}.
\end{align*}
Thus, the spectral density of the covariance function $C(h)$ is 
\begin{align*}
f(\omega) &= (2\pi)^{-d/2}  \int_0^{\infty} (\omega h)^{-(d-2)/2} \calJ_{(d-2)/2}(\omega h) h^{d-1}  \int_0^{\infty}\mathcal{M}(h; \nu, \phi, \sigma^2)  \pi(\phi^2) d \phi^2 dh \\
&= \frac{\sigma^2 2^{\nu} \nu^{\nu} (\beta^2/2)^{\alpha}}{\Gamma(\alpha)} \int_{0}^{\infty} \frac{ \phi^{-2\nu} }{\pi^{d/2}({2\nu}\phi^{-2} + \omega^2)^{\nu+d/2}}  \phi^{-2(\alpha+1)} \exp\{-\beta^2/(2\phi^2) \}d\phi^2 \\
&= \frac{\sigma^2 2^{\nu-\alpha} \nu^{\nu} \beta^{2\alpha}}{\pi^{d/2}\Gamma(\alpha)} \int_0^{\infty} (2\nu\phi^{-2} + \omega^2)^{-\nu-d/2} \phi^{-2(\nu+\alpha+1)} \exp\{ - \beta^2/(2\phi^2) \} d\phi^2.
\end{align*}
where the above spectral density is finite for $\alpha>d/2$ and is infinite for $\alpha \in (0, d/2]$. 

To derive the tail behavior, we make the change of variable $\phi^2 = \beta^2 t / \omega^2$. The spectral density above can be expressed as 
\begin{align*}
f(\omega) &= \frac{\sigma^2 2^{\nu-\alpha} \nu^{\nu} \beta^{2\alpha}}{\pi^{d/2}\Gamma(\alpha)} \omega^{2\alpha - d}  \int_0^{\infty} ((2\nu/\beta^2) t^{-1} + 1)^{-(\nu+d/2)} t^{-(\nu+\alpha+1)} \exp\{ - \omega^2/(2t) \} dt \\
&= \frac{\sigma^2 2^{\nu-\alpha} \nu^{\nu} }{\pi^{d/2}\beta^{2\nu}\Gamma(\alpha)} (2\pi)^{1/2} \omega^{2\alpha - d}   \int_0^{\infty} \left(\frac{t}{2\nu/\beta^2 + t}  \right)^{(\nu+d/2)} t^{(-\nu-\alpha+1/2)-1} \frac{1}{\sqrt{2\pi t}} \\
&\quad \times \exp\{ - \omega^2/(2t) \} dt.
\end{align*}
We define  
$$L(x): = \left\{ \frac{x}{x + \beta^2/(2\nu)}\right\}^{\nu+d/2}.$$ 
Then $L(x)$ is a slowly varying function at $\infty$. The above integral is also a Gaussian scale mixture.  Thus, an application of Theorem 6.1 of \cite{barndorff1982normal} yields that as $|\omega| \to \infty$, 
\begin{align*}
f(\omega) & \sim \frac{\sigma^2 2^{\nu-\alpha} \nu^{\nu} }{\pi^{d/2}\beta^{2\nu}\Gamma(\alpha)} (2\pi)^{1/2} \omega^{2\alpha - d}  (2\pi)^{-1/2} 2^{1/2 +(\nu+\alpha-1/2))} |\omega|^{-2(\nu+\alpha-1/2)-1} L(\omega^2) \\
 &\sim \frac{\sigma^2 2^{2\nu} \nu^{\nu}  \Gamma(\nu+\alpha)}{\pi^{d/2}\beta^{2\nu}\Gamma(\alpha)}   \omega^{-(2\nu+d)} L(\omega^2). 
\end{align*}
\end{proof}

\subsection{Proof of Theorem~\ref{thm: equivalence}} \label{app: equivalence}

\begin{proof} 

 Let $f_i(\omega), i=1,2$ be the spectral densities with parameters $\{\sigma^2_i,  \beta_i, \alpha_i, \nu\}$ for two covariance functions $C_1(\cdot), C_2(\cdot)$. The condition~\eqref{eqn: boundedness} says the spectral density $f_i(\omega)$ is bounded at zero and $\infty$ when $\omega \to \infty$. In fact, the boundness of $f_i$ near zero follows from the assumption that $\alpha_i>d/2$. Let $\lambda = 2\nu + d$. Then, one can show that 
 \begin{align*} 
\lim_{\omega\to \infty} f_1(\omega) |\omega|^{2\nu+d} = \frac{\sigma^2_1(\beta_1^2/2)^{-\nu} (2 \nu)^{\nu} \Gamma(\nu+\alpha_1)}{\pi^{d/2}\Gamma(\alpha_1)}.
\end{align*}
Thus, the condition~\eqref{eqn: boundedness} is satisfied. 

We first show the sufficiency. Assume that the condition in Equation~\eqref{eqn: consistency for new covariance} holds. To prove the equivalence of two  measures, it suffices to show that the condition~\eqref{eqn: equivalence for SGRF} is satisfied. Notice that as $\omega \to \infty$,
\begin{align*}
\left|\frac{f_1(\omega) - f_2(\omega)}{f_1(\omega)} \right| &=\left | \frac{\{\omega^2 + \beta_2^2/(2\nu)\}^{-(\nu+d/2)}}{\{\omega^2 + \beta_1^2/(2\nu)\}^{-(\nu+d/2)}} - 1\right | \\
& \leq \omega^{-(2\nu+d)} \left |\{\omega^2+\beta_2^2/(2\nu)\}^{\nu+d/2} -  \{\omega^2+\beta_1^2/(2\nu)\}^{\nu+d/2} \right| \\
& \leq \left |\{1+ (\beta_2^2/2\nu) \omega^{-2}\}^{\nu + d/2} - \{1+ (\beta_1^2/2\nu) \omega^{-2}\}^{\nu + d/2} \right| \\
& \leq \left | \{1 + (\nu+d/2) (\beta_2^2/2\nu) \omega^{-2} + O(\omega^{-4})  \}  \right. \\
& \quad - \left.  \{ 1 + (\nu+d/2) (\beta_1^2/2\nu) \omega^{-2} 
 + O(\omega^{-4}) \}   \right | \\
& \leq  |\beta_1^2-\beta_2^2|(\nu+d/2) / (2\nu) \omega^{-2} + O(\omega^{-4}).
\end{align*}
The integral in~\eqref{eqn: equivalence for SGRF} is finite for $d=1,2,3$. Therefore, the two measures are equivalent.

It remains to show the necessary condition. Suppose that 
$$\frac{\sigma^2_1\beta_1^{-2\nu} \Gamma(\nu+\alpha_1)}{\Gamma(\alpha_1)} \neq \frac{\sigma^2_2\beta_2^{-2\nu} \Gamma(\nu+\alpha_2)}{\Gamma(\alpha_2)}.$$ 
Let
\begin{align*}
\sigma^2_0 := \sigma^2_2 \frac{\beta_2^{-2\nu} \Gamma(\alpha_1) \Gamma(\nu+\alpha_2)}{\beta_1^{-2\nu} \Gamma(\alpha_2) \Gamma(\nu+\alpha_1)}. 
\end{align*}
Then 
$$\frac{\sigma^2_0 \beta_1^{-2\nu} \Gamma(\nu+\alpha_1)}{\Gamma(\alpha_1)}  = \frac{\sigma^2_2\beta_2^{-2\nu} \Gamma(\nu+\alpha_2)}{\Gamma(\alpha_2)}.$$
Thus, the two covariances $C(h; \nu, \alpha_1, \beta_1, \sigma^2_0)$ and $ C(h; \nu, \alpha_1, \beta_1, \sigma^2_1)$ define two equivalent measures. It remains to show that $ C(h; \nu, \alpha_1, \beta_1, \sigma^2_0)$ and $ C(h; \nu, \alpha_2, \beta_2, \sigma^2_2)$ define two equivalent Gaussian measures, which follows from the proof in Theorem 2 of \cite{Zhang2004}.
\end{proof}

\subsection{Proof of Theorem~\ref{thm: equivalence with matern}} \label{app: equivalence with matern}
\begin{proof}
Let $k_1 = \sigma^2_1 \frac{2^{2\nu} \nu^{\nu} \Gamma(\nu+\alpha)}{\pi^{d/2}\beta^{2\nu} \Gamma(\alpha)}  $ and 
$k_2 = \sigma^2_2 (2\nu)^{\nu} \phi^{-2\nu}/\pi^{d/2}$. Then the condition in Equation~\eqref{eqn: equivalence with matern} implies that $k_1=k_2$. 
It follows that as $|\omega| \to \infty$,
\begin{align*}
\left | \frac{f_1(\omega)  - f_2(\omega)}{f_1(\omega)}  \right| &= \left | \frac{k_2}{k_1} (\omega^2 + 2\nu/\phi^2)^{-(\nu+d/2)} (\omega^2 + \beta^2 / (2\nu))^{(\nu+d/2)} -1  \right| \\
& = (\omega^2 + 2\nu/\phi^2)^{-(\nu+d/2)} \times \left | k_2/k_1 (\omega^2 + 2\nu/\phi^2)^{\nu+d/2} \right. \\
& \quad \left. - (\omega^2 + \beta^2 / (2\nu))^{(\nu+d/2)}   \right| \\
& \leq \omega^{-(2\nu+d)} \times \left |  (\omega^2 + 2\nu/\phi^2)^{\nu+d/2} - (\omega^2 + \beta^2 / (2\nu))^{(\nu+d/2)}   \right| \\
&\leq \left |   \{1+ (2\nu/\phi^2)\omega^{-2} \}^{-(\nu+d/2)} - \{1  + \beta^2 / (2\nu) \omega^{-2}\}^{(\nu+d/2)} \right | \\
& \leq \left |  \{1 + (2\nu/\phi^2)(\nu+d/2)\omega^{-2} + O(\omega^{-4})   \}  - \{ 1+(\beta^2 / (2\nu)) (\nu+d/2)\omega^{-2} \right.\\
&\left. \quad + O(\omega^{-4})   \}  \right|. \\
& \leq | 2\nu/\phi^2 - \beta^2/ (2\nu) | (\nu+d/2) \omega^{-2} + O(\omega^{-4}). 
\end{align*}
The integral in~\eqref{eqn: equivalence for SGRF} is finite for $d=1,2,3$. Therefore, these two measures are equivalent. 
\end{proof}

\subsection{Proof of Theorem~\ref{thm: ML consistency2}} \label{app: ML consistency2}
\begin{proof} Note that the CH covariance function $C(h; \nu, \alpha, \beta, \sigma^2)$ is a continuous function of the covariance parameters $\alpha, \beta, \sigma^2$ over their natural parameter space $\{(\sigma^2, \alpha, \beta): \sigma^2>0, \alpha>0, \beta>0\}$, and hence the likelihood function is also a continuous function over this natural parameter space. 

For case (a), it follows from the continuity of the likelihood function and the assumption in case (a) that $\hat{\beta}_n \in [\beta_L, \beta_U]$ for all $n$.  Applying Lemma~\ref{lem: monotonicity} yields that $\hat{c}_n(\alpha, \beta_L) \leq \hat{c}_n(\alpha, \hat{\beta}_n) \leq \hat{c}_n(\alpha, \hat{\beta}_U)$. The result thus follows from Lemma~\ref{thm: ML consistency} immediately. 

For case (b),  it follows from the continuity of the likelihood function and the assumption in case (b) that $\hat{\alpha}_n \in [\alpha_L, \alpha_U]$ for all $n$. Applying Lemma~\ref{lem: monotonicity} yields that $\hat{c}_n({\alpha}_U, \beta) \leq \hat{c}_n(\hat{\alpha}_n, \beta)\leq \hat{c}_n(\alpha_U, \beta)$. The result thus follows from Lemma~\ref{thm: ML consistency} immediately. 

For case (c),  it follows from the continuity of the likelihood function and the assumption in case (c) that $\hat{\alpha}_n \in [\alpha_L, \alpha_U]$ and $\hat{\beta}_n \in [\beta_L, \beta_U]$ for all $n$. According to Lemma~\ref{lem: monotonicity}, $\hat{c}_n(\alpha_U, \beta_L) \leq \hat{c}_n(\alpha_U, \hat{\beta}_n) \leq \hat{c}_n(\hat{\alpha}_n, \hat{\beta}_n)$ and $\hat{c}_n(\hat{\alpha}_n, $ $\hat{\beta}_n) \leq \hat{c}_n({\alpha}_L, \hat{\beta}_n) \leq \hat{c}_n(\alpha_L, \beta_U)$. The result thus follows from Lemma~\ref{thm: ML consistency} immediately. 
\end{proof}

\subsection{Proof of Theorem~\ref{thm: prediction efficiency with new covariance}} \label{app: prediction efficiency with new covariance}
\begin{proof}
Part (a) and Part (b) can be proven by applying Theorem 1 and Theorem 2 of \cite{Stein1993}. Let $f_i(\omega)$ be the spectral density of the CH class $C(h; \nu, \alpha_i, \beta_i, \sigma^2_i)$ with $i=1,2$. Note that $\lim_{\omega \to \infty} f_i(\omega) |\omega|^{2\nu+d}$ is finite. If the condition in Equation~\eqref{eqn: consistency for new covariance} is satisfied, then,
\begin{align*}
\lim_{\omega \to \infty} \frac{f_2(\omega)}{f_1(\omega)} = \lim_{\omega \to \infty} \frac{f_2(\omega) |\omega|^{2\nu+d}}{f_1(\omega) |\omega|^{2\nu+d}} = 1.
\end{align*}


The proof of Part (c) is analogous to the proof of Theorem 4 in \cite{Kaufman2013}. Let 
\begin{align*}
\sigma^2_1 :=\sigma_0^2 (\beta_1/\beta_0)^{2\nu} \frac{\Gamma(\nu+\alpha_0) \Gamma(\alpha_1)}{\Gamma(\nu+\alpha_1) \Gamma(\alpha_0)}.
\end{align*}
Then $\mathcal{P}_0$ and $\mathcal{P}_1$ define two equivalent measures. We write 
\begin{align*}
\frac{\text{Var}_{\nu, \bftheta_1, \hat{\sigma}^2_n} \{ \hat{Z}_n(\bftheta_1) - Z(\bfs_0) \} }{\text{Var}_{\nu, \bftheta_0, \sigma^2_0} \{ \hat{Z}_n(\bftheta_1) - Z(\bfs_0) \} } = \frac{\text{Var}_{\nu, \bftheta_1, \hat{\sigma}^2_n} \{ \hat{Z}_n(\bftheta_1) - Z(\bfs_0) \} }{\text{Var}_{\nu, \bftheta_1, \sigma^2_1} \{ \hat{Z}_n(\bftheta_1) - Z(\bfs_0) \} }   
\frac{\text{Var}_{\nu, \bftheta_1, {\sigma}^2_1} \{ \hat{Z}_n(\bftheta_1) - Z(\bfs_0) \} }{\text{Var}_{\nu, \bftheta_0, \sigma^2_0} \{ \hat{Z}_n(\bftheta_1) - Z(\bfs_0) \} }. 
\end{align*}
According to Part (b) of Theorem~\ref{thm: prediction efficiency with new covariance}, it suffices to show that  almost surely under $\mathcal{P}_1$,
\begin{align*}
\frac{\text{Var}_{\nu, \bftheta_1, \hat{\sigma}^2_n} \{ \hat{Z}_n(\bftheta_1) - Z(\bfs_0) \} }{\text{Var}_{\nu, \bftheta_1, \sigma^2_1} \{ \hat{Z}_n(\bftheta_1) - Z(\bfs_0) \} }   \to 1. 
\end{align*}
By Equation~\eqref{eqn: MSE}, 
\begin{align*}
\frac{\text{Var}_{\nu, \bftheta_1, \hat{\sigma}^2_n} \{ \hat{Z}_n(\bftheta_1) - Z(\bfs_0) \} }{\text{Var}_{\nu, \bftheta_1, \sigma^2_1} \{ \hat{Z}_n(\bftheta_1) - Z(\bfs_0) \} } = \frac{\hat{\sigma}^2_n}{\sigma^2_1}.
\end{align*}
Note that under $\mathcal{P}_1$, we have $\hat{\sigma}^2_n \sim (\sigma^2_0/n) \chi_n$, and hence $\hat{\sigma}^2_n$ converges almost surely to $\sigma^2_0$ as $n\to \infty$. As $\mathcal{P}_0$ is equivalent to $\mathcal{P}_1$, It follows from Lemma~\ref{thm: ML consistency} that $\hat{\sigma}^2_n\to \sigma^2_1$, almost surely under $\mathcal{P}_0$.
\end{proof}

\subsection{Proof of Theorem~\ref{thm: PE with Matern covariance}} \label{app: PE with Matern covariance}

\begin{proof}
Let $f_0(\omega)$ be the spectral density of the Mat\'ern covariance function $\mathcal{M}(h; \nu, \phi, \sigma^2_0)$ and $f_1(\omega)$ be the spectral density of the covariance function $C(h; \nu, \alpha, \beta, \sigma^2_1)$. Notice that the spectral density of the Mat\'ern covariance satisfies the condition~\eqref{eqn: boundedness}. It suffices to show that $\lim_{\omega\to\infty} f_1(\omega)/f_0(\omega) = 1$.  Let $k_0 = \sigma^2_0 \phi^{-2\nu}$ and  $k_1 = \sigma^2_1 $ $(\beta^2/2)^{-\nu} \Gamma(\nu+\alpha) /\Gamma(\alpha)$. If $k_0=k_1$, it follows that
\begin{align*}
\lim_{\omega \to \infty} \frac{f_1(\omega)}{f_0(\omega)} &= \lim_{\omega \to \infty} \frac{f_1(\omega) |\omega|^{2\nu+d}}{f_0(\omega) |\omega|^{2\nu+d}} 
 = \lim_{\omega \to \infty} \frac{k_1}{k_0} \left( 2\nu \phi^{-2} \omega^{-2} + 1  \right)^{\nu + d/2} = k_1/k_0=1.
\end{align*}
Thus, the covariance function $C(h; \nu, \alpha, \beta, \sigma^2_1)$ yields an asymptotically equivalent BLP as the Mat\'ern covariance $\mathcal{M}(h; \nu, \phi, \sigma^2_0)$.
\end{proof}

\newpage

\section{Examples to Illustrate Asymptotic Normality} \label{sec: mle}
As shown in Section~\ref{subsec: asymptotic normality}, each individual parameter in the CH model cannot be estimated consistently, however, the microergodic parameter can be estimated consistently. 

To study the finite sample performance of the asymptotic properties of MLE for the microergodic parameter, we simulate 1000 realizations from a zero-mean Gaussian process with the CH class over 100-by-100 regular grid in the unit domain $\mathcal{D}=[0, 1]\times [0, 1]$. As there are no clear guidelines to pick the sample sizes such that the finite sample performances can appropriately reflect the asymptotic results,  we randomly select $n=4000, 5000, 6000$ locations from these 10,000 grid points. The variance parameter is fixed at 1 for all realizations. We consider two different values for the smoothness parameter $\nu$ at 0.5 and 1.5, three different values for the tail decay parameter $\alpha$ at 0.5, 2 and 5. The scale parameter $\beta$ is chosen such that the effective range is 0.6 or 0.9. {Although all the theoretical results in Section~\ref{sec:theory} are valid for $\alpha>d/2$, we also run the simulation setting with $\alpha=0.5$ to see whether there is any interesting numerical results compared to cases where $\alpha>d/2$.}

Let $C(h; \nu, \alpha_0, \beta_0, \sigma^2_0)$ be the true covariance. We use $\hat{c}_n({\bftheta})$ to denote the maximum likelihood estimator of the microergodic parameter $c(\bftheta_0)=\sigma^2_0\beta_0^{-2\nu} \Gamma(\nu+\alpha_0) / \Gamma(\alpha_0)$ for any $\bftheta$. Then the 95\% confidence interval for ${c}({\bftheta}_0)$ is given by $\hat{c}_n({\bftheta}) \pm 1.96 \sqrt{2 \hat{c}_n({\bftheta})^2 /n}$. Lemma~\ref{thm: ML consistency} and Theorem~\ref{thm: ML consistency2} show that this interval is asymptotically valid when $n$ is large and $\alpha>d/2$ for (1) arbitrarily fixed $\bftheta$, (2) $\bftheta= (\alpha, \hat{\beta}_n)$, (3) $\bftheta=(\hat{\alpha}_n, \beta)$ and (4) $\bftheta=(\hat{\alpha}_n, \hat{\beta}_n)$. In this simulation study, we primarily focus on the finite sample performance of $\hat{c}_n({\bftheta})$, where $\bftheta=(\alpha_0, \sqrt{0.5}\beta_0)$, $\bftheta=(\alpha_0, \beta_0)$, $\bftheta=(\alpha_0, \sqrt{2}\beta_0)$, $\bftheta=(\alpha_0, \hat{\beta}_n)$, and $\bftheta=(\hat{\alpha}_n, \hat{\beta}_n)$. Exhaustive simulations with all other settings of $\bftheta$ is considered future work. Let 
\begin{align*}
\xi: = \frac{ \sqrt{n}\{ \hat{c}_n({\bftheta}) - c(\bftheta_0) \}}{\sqrt{2} c(\bftheta_0)}. 
\end{align*}
Then $\xi$ should asymptotically follow the standard normal distribution. Based on these 1000 realizations, we compute the empirical coverage probability of the 95\% percentile confidence interval, bias and root-mean-square error (RMSE) for $c(\bftheta_0)$ and compare the quantiles of $\xi$ with the standard normal quantiles. 

The results are  reported in Table~\ref{table: mle}, Table~\ref{table: mle0} and Table~\ref{table: mle1} of the Supplementary Material. They can be summarized as follows. When the true parameters are used, i.e., $\bftheta=\bftheta_0$,  as expected, the sampling distribution of $\hat{c}_n(\bftheta_0)$ gives the best normal approximation and converges to the asymptotic distribution in Lemma~\ref{thm: ML consistency} when $n$ increases. The sampling distribution of $\hat{c}_n(\bftheta)$ can be highly biased and approach to the truth can be very slow with increase in $n$. Fixing $\beta$ at a larger value gives better empirical results than fixing $\beta$ at a small value. When the scale parameter is chosen to be its maximum likelihood estimator, i.e., $\beta=\hat{\beta}_n$, the sampling distribution of $\hat{c}_n(\alpha, \hat{\beta}_n)$ converges to the asymptotic distribution given in Theorem~\ref{thm: ML consistency2} as $n$ increases.  When $\alpha$ is small, e.g., $\alpha=0.5$, the sampling distributions of $\hat{c}_n(\bftheta)$, with $(\alpha_0, \sqrt{0.5}\beta_0)$, $(\alpha_0, \sqrt{2}\beta_0)$, $(\alpha_0, \hat{\beta}_n)$ and $(\hat{\alpha}_n, \hat{\beta}_n)$ substituted for $\bftheta$, has noticeable biases. As the tail decay parameter or the effective range increases, the sampling distributions of $\hat{c}_n(\bftheta)$ have smaller biases. As $\nu$ becomes smaller, the sampling distributions of $\hat{c}_n(\bftheta)$  approaches the truth better with increase in $n$. When $\nu=0.5$ and $\alpha \in \{2, 5\}$, these sampling distributions have negligible biases as $n$ increases. When both $\alpha$ and $\beta$ are substituted by their maximum likelihood estimator, the sampling distribution of $\hat{c}_n({\bftheta})$ has smaller bias and gives better approximation to the true asymptotic distribution given in Theorem~\ref{thm: ML consistency2} as $n$ increases for $\alpha>d/2=1$.

When $\alpha$ is fixed at its true value and $\beta$ is estimated by maximum likelihood method, the MLE of the microergodic parameter, $\hat{c}_n(\alpha, \hat{\beta}_n)$, gives better finite sample performance than the cases where $\beta$ is misspecified. When both $\alpha$ and $\beta$ are estimated by maximum likelihood method, the
MLE of the microergodic parameter, $\hat{c}_n(\hat{\alpha}_n, \hat{\beta}_n)$, also gives better finite sample performance than the cases where $\beta$ is misspecified and $\alpha$ is fixed at its true value. One would also expect that this is true when either $\alpha$ or $\beta$ is misspecified at incorrect values. In general, the MLE of the microergodic parameter has better finite sampler performance than those with any individual parameter fixed at an incorrect value in the microergodic parameter. {Theorem~\ref{thm: ML consistency2} requires $\alpha>d/2$ in order to derive asymptotic results for $\hat{c}_n(\hat{\alpha}_n, \hat{\beta}_n)$. However, it is interesting to observe from these simulation results that $\hat{c}_n(\bftheta)$ seems to converge to a normal distribution even when $\alpha<d/2$, i.e., when $\alpha=0.5$. It is an open problem to determine the exact distribution that the maximum likelihood estimator $\hat{c}_n(\hat{\alpha}_n, \hat{\beta}_n)$ of the microergodic parameter converges to asymptotically when $\alpha$ and $\beta$ are substituted with their maximum likelihood estimators for true $\alpha_0 \in (0, d/2]$. }

\afterpage{%
\clearpage%
\thispagestyle{empty}%
\begin{table}[htbp]
   \caption{Percentiles of $\xi$ and CVG, bias, and RMSE of $\hat{c}_n(\bftheta)$ when $\alpha_0=0.5$.}
   
   
  {\resizebox{1.0\textwidth}{!}{%
  \setlength{\tabcolsep}{2.25em}
   \begin{tabular}{ l c c c c  c c  c  c c} 
   \toprule \noalign{\vskip 1.5pt} 
	\multicolumn{2}{c}{Settings} 	& 5\%  & 25\% & 50\% & 75\%  & 95\%  & CVG & bias & RMSE  \\  \noalign{\vskip 1.5pt}  
	 \midrule \noalign{\vskip 1.5pt} 
\multicolumn{2}{c}{$\mathcal{N}(0, 1)$}  & -1.6449 & -0.6749 & 0 & 0.6749 & 1.6449 &0.95 &0 & \\ \noalign{\vskip 1.5pt} \hline \noalign{\vskip 2.5pt}

  \multicolumn{10}{c}{$ER=0.6,   \nu=0.5$}  \\  \noalign{\vskip 2.5pt}   
 \multicolumn{1}{c}{$\bftheta$}\\ \cline{1-1} \noalign{\vskip 2.5pt}   \noalign{\vskip 2.5pt}
\multirow{3}{*}{$\alpha=\alpha_0, \beta=\beta_0$} & $n=4000$ &-1.449  &-0.542  &0.009  &0.686  &1.767  &0.955  &0.020 &0.327 \\ \noalign{\vskip 2.5pt} 
 & $n=5000$ &-1.469  & -0.665 &-0.077  & 0.696 &1.573  &0.965  &-0.003 & 0.289 \\ \noalign{\vskip 2.5pt}
& $n=6000$ &-1.705  &-0.618  &0.056  &0.662 &1.847  &0.929  &0.010  &0.280 \\ \noalign{\vskip 2.5pt}   \noalign{\vskip 2.5pt}

\multirow{3}{*}{$\alpha=\alpha_0, \beta=\sqrt{0.5}\beta_0$} & $n=4000$ &2.113  &3.098  &3.730  &4.424  &5.549  &0.044  &1.259 & 1.308 \\ \noalign{\vskip 2.5pt} 
 & $n=5000$ &2.129  &2.930  &3.578  &4.439  &5.347  &0.040  & 1.097 &1.140 \\ \noalign{\vskip 2.5pt}
& $n=6000$ &1.798  &3.015  &3.705  &4.397  &5.606  &0.071  &1.005 & 1.049 \\ \noalign{\vskip 2.5pt}   \noalign{\vskip 2.5pt}

\multirow{3}{*}{$\alpha=\alpha_0, \beta=\sqrt{2}\beta_0$} & $n=4000$ &-3.471  &-2.608  & -2.073 &-1.420  &-0.394  &0.415  &-0.676 &0.746 \\ \noalign{\vskip 2.5pt} 
 & $n=5000$ &-3.480  &-2.693  &-2.114 &-1.395  & -0.519 &0.404  &-0.611 & 0.676 \\ \noalign{\vskip 2.5pt}
& $n=6000$ &-3.688  &-2.623  &-1.967  &-1.376  &-0.214  &0.462  &-0.543 &0.606 \\ \noalign{\vskip 2.5pt}   \noalign{\vskip 2.5pt}

\multirow{3}{*}{$\alpha=\alpha_0, \beta=\hat{\beta}_n$} & $n=4000$ &-1.871  & -0.711 &0.095  &1.000  &2.244  &0.889  & 0.047 & 0.428 \\ \noalign{\vskip 2.5pt} 
 & $n=5000$ &-1.912  &-0.767  &0.022  &0.881  & 2.134 &0.881 &0.013 & 0.371 \\ \noalign{\vskip 2.5pt}
& $n=6000$ &-2.016  &-0.760  &0.096  &0.862  &2.097  &0.879  &0.019 &0.343 \\ \noalign{\vskip 2.5pt}   \noalign{\vskip 2.5pt}

\multirow{3}{*}{$\alpha=\hat{\alpha}_n, \beta=\hat{\beta}_n$} & $n=4000$ & -1.778  &-0.875  &0.000  &0.925  &2.382  &0.887  & 0.030 & 0.446 \\ \noalign{\vskip 2.5pt} 
 & $n=5000$ &-2.129  &-0.816  &0.026  &0.893  &2.227  &{0.870}  & 0.019 & 0.395 \\ \noalign{\vskip 2.5pt}
& $n=6000$ &-2.268  &  -0.911&-0.015  &0.865  &2.117  &0.875  &-0.006 & 0.363 \\ \noalign{\vskip 2.5pt}  \hline \noalign{\vskip 2.5pt}

\multicolumn{10}{c}{$ER=0.6,  \nu=1.5$}  \\ \noalign{\vskip 2.5pt} \noalign{\vskip 2.5pt} 
 \multicolumn{1}{c}{$\bftheta$}\\ \cline{1-1} \noalign{\vskip 2.5pt}   \noalign{\vskip 2.5pt}

\multirow{3}{*}{$\alpha=\alpha_0, \beta=\beta_0$} & $n=4000$ &-1.654  &-0.604  &-0.014  &0.701  &1.776  &0.949  &12.650 & 370.7 \\ \noalign{\vskip 2.5pt} 
 & $n=5000$ &-1.430  &-0.687  &-0.046  &0.672  &1.576  &0.969  & 0.283 & 312.0 \\ \noalign{\vskip 2.5pt}
& $n=6000$ &-1.731  &-0.649  &0.070  & 0.710 &1.740  &0.929  &7.182 & 307.5 \\ \noalign{\vskip 2.5pt}   \noalign{\vskip 2.5pt}

\multirow{3}{*}{$\alpha=\alpha_0, \beta=\sqrt{0.5}\beta_0$} & $n=4000$ &26.20  &27.77  &28.82 &30.03  &31.76  &0.000  &10495 & 10513 \\ \noalign{\vskip 2.5pt} 
 & $n=5000$ &27.09  &28.35  &29.37  & 30.55 & 31.99 & 0.000 &9567 & 9581 \\ \noalign{\vskip 2.5pt}
& $n=6000$ &27.22  &28.79  & 29.85 & 30.89 &32.68  &0.000  &8860 & 8874 \\ \noalign{\vskip 2.5pt}   \noalign{\vskip 2.5pt}

\multirow{3}{*}{$\alpha=\alpha_0, \beta=\sqrt{2}\beta_0$} & $n=4000$ &-13.70  &-12.93  &-12.48  &-11.99 &-11.21  & 0.000 &-4526 & 4534 \\ \noalign{\vskip 2.5pt} 
 & $n=5000$ &-13.99  &-13.35  & -12.90 &-12.36  &-11.62  &0.000  &-4177 & 4184 \\ \noalign{\vskip 2.5pt}
& $n=6000$ &-14.47  &-13.61  &-13.11  &-12.58  & -11.80 &0.000  & -3886 &3893 \\ \noalign{\vskip 2.5pt}   \noalign{\vskip 2.5pt}

\multirow{3}{*}{$\alpha=\alpha_0, \beta=\hat{\beta}_n$} & $n=4000$ & -2.993 &-1.121  &0.172  &1.515  &3.505  &0.670  & 72.52 & 732.5 \\ \noalign{\vskip 2.5pt} 
 & $n=5000$ &-2.823  &-1.155  &0.146  &1.452  &3.398  &0.700  & 49.49 & 624.8 \\ \noalign{\vskip 2.5pt}
& $n=6000$ &-3.068  &-1.090  &0.235  &1.433  &3.093  &0.733  & 44.59 &543.1 \\ \noalign{\vskip 2.5pt}  \noalign{\vskip 2.5pt}

\multirow{3}{*}{$\alpha=\hat{\alpha}_n, \beta=\hat{\beta}_n$} & $n=4000$ &-3.887  &-1.656  &0.061  &1.681  &4.142  &0.565  & 9.059 & 895.5 \\ \noalign{\vskip 2.5pt} 
 & $n=5000$ &-3.607  &-1.643  & 0.055 & 1.497 &4.142   & 0.592 & 16.27 & 772.1 \\ \noalign{\vskip 2.5pt}
& $n=6000$ &-4.107  &-1.678  &-0.206  &1.488  &3.774  &0.592  &-70.59 & 832.3 \\ \noalign{\vskip 2.5pt}  \hline \noalign{\vskip 2.5pt}

\multicolumn{10}{c}{$ER=0.9,   \nu=0.5$}  \\ \noalign{\vskip 2.5pt}   \noalign{\vskip 2.5pt}
 \multicolumn{1}{c}{$\bftheta$}\\ \cline{1-1} \noalign{\vskip 2.5pt}   \noalign{\vskip 2.5pt}

\multirow{3}{*}{$\alpha=\alpha_0, \beta=\beta_0$} & $n=4000$ &-1.589  &-0.557  &-0.013  &0.669  &1.748  &0.955  & 0.007 & 0.220 \\ \noalign{\vskip 2.5pt} 
 & $n=5000$ &-1.429  &-0.654  &0.065  &0.759  &1.683  &0.978  &0.012 & 0.190 \\ \noalign{\vskip 2.5pt}
& $n=6000$ &-1.512  &-0.591  &0.004  &0.702  &1.768  &0.943  &0.009 &0.179 \\ \noalign{\vskip 2.5pt}   \noalign{\vskip 2.5pt}

\multirow{3}{*}{$\alpha=\alpha_0, \beta=\sqrt{0.5}\beta_0$} & $n=4000$ &0.628  & 1.679  &2.278  &2.967  &4.052  &0.399  & 0.513 & 0.563 \\ \noalign{\vskip 2.5pt} 
 & $n=5000$ &0.727  &1.546  &2.306  &2.994  &3.958  &0.420  & 0.454 & 0.496 \\ \noalign{\vskip 2.5pt}
& $n=6000$ & 0.548 & 1.543 &2.200  &2.888  & 4.013 &0.445 & 0.403 &0.444 \\ \noalign{\vskip 2.5pt}   \noalign{\vskip 2.5pt}

\multirow{3}{*}{$\alpha=\alpha_0, \beta=\sqrt{2}\beta_0$} & $n=4000$ & -2.812 &-1.808  & -1.259 & -0.595  & 0.440 & 0.769  & -0.272 & 0.346 \\ \noalign{\vskip 2.5pt} 
 & $n=5000$ &-2.620  &-1.846  & -1.170 & -0.478 &0.379  &0.760  & -0.229 & 0.295 \\ \noalign{\vskip 2.5pt}
& $n=6000$ &-2.635  &-1.756  &-1.187  &-0.477  &0.586  &0.799  &-0.205 &0.270 \\ \noalign{\vskip 2.5pt}   \noalign{\vskip 2.5pt}

\multirow{3}{*}{$\alpha=\alpha_0, \beta=\hat{\beta}_n$} & $n=4000$ & -1.856 & -0.688 & 0.087  & 0.911  & 1.946  &  0.905 & 0.021 & 0.262 \\ \noalign{\vskip 2.5pt} 
 & $n=5000$ &-1.587  &-0.696  &0.062   &0.822  &1.930  & 0.926 & 0.018 & 0.220 \\ \noalign{\vskip 2.5pt}
& $n=6000$ &-1.646  &-0.589  &0.045  & 0.833 & 2.008 &0.918  &0.020 &0.202 \\ \noalign{\vskip 2.5pt}  \noalign{\vskip 2.5pt}

\multirow{3}{*}{$\alpha=\hat{\alpha}_n, \beta=\hat{\beta}_n$} & $n=4000$ &-1.876  &-0.748  &0.082  &0.882  &2.157  &0.887  &0.016 &0.276 \\ \noalign{\vskip 2.5pt} 
 & $n=5000$ &-1.865  & -0.692 &0.023  &0.853  &1.994  &0.902  &0.014 & 0.233 \\ \noalign{\vskip 2.5pt}
& $n=6000$ & -1.884 &-0.744  &0.006  &0.901  & 1.978 &0.904  &0.008 &0.213 \\ \noalign{\vskip 2.5pt}  \hline \noalign{\vskip 2.5pt}

\multicolumn{10}{c}{$ER=0.9,  \nu=1.5$}  \\ \noalign{\vskip 2.5pt} \noalign{\vskip 2.5pt} 
 \multicolumn{1}{c}{$\bftheta$}\\ \cline{1-1} \noalign{\vskip 2.5pt}   \noalign{\vskip 2.5pt}

\multirow{3}{*}{$\alpha=\alpha_0, \beta=\beta_0$} & $n=4000$ &-1.598  & -0.618  &-0.015  &0.663  & 1.747  &  0.958 & 2.284 & 106.8 \\ \noalign{\vskip 2.5pt} 
 & $n=5000$ &-1.426  &-0.647  & 0.063 &0.774  &1.711  &0.977  & 6.598 & 92.45 \\ \noalign{\vskip 2.5pt}
& $n=6000$ &-1.660  &-0.623  &0.059  &0.685  &1.743  & 0.945 & 2.287 & 87.41 \\ \noalign{\vskip 2.5pt}   \noalign{\vskip 2.5pt}

\multirow{3}{*}{$\alpha=\alpha_0, \beta=\sqrt{0.5}\beta_0$} & $n=4000$ &17.15  &18.51  &19.42  &20.44  &21.94  &0.000  &2096 & 2102 \\ \noalign{\vskip 2.5pt} 
 & $n=5000$ &17.34  & 18.50  &  19.50 & 20.42  & 21.83  &0.000  & 1877 & 1881 \\ \noalign{\vskip 2.5pt}
& $n=6000$ &17.06  & 18.45 &19.36  &20.24  &21.70  &0.000  & 1699 &1703 \\ \noalign{\vskip 2.5pt}   \noalign{\vskip 2.5pt}

\multirow{3}{*}{$\alpha=\alpha_0, \beta=\sqrt{2}\beta_0$} & $n=4000$ &-10.02  &-9.207  &-8.722  &-8.151  & -7.284  & 0.000  & -934.2 & 938.3 \\ \noalign{\vskip 2.5pt} 
 & $n=5000$ &-9.964  & -9.263 & -8.691  & -8.092  & -7.331  & 0.000  & -8.835 & 839.5 \\ \noalign{\vskip 2.5pt}
& $n=6000$ &-10.13  &-9.280  &-8.717  &-8.159  &-7.26  &0.000  &-766.2 &769.8 \\ \noalign{\vskip 2.5pt}   \noalign{\vskip 2.5pt}

\multirow{3}{*}{$\alpha=\alpha_0, \beta=\hat{\beta}_n$} & $n=4000$ &-2.455  &-0.993  &0.029  &1.282    & 2.953 & 0.771 &15.62 & 180.4 \\ \noalign{\vskip 2.5pt} 
 & $n=5000$ &-2.224  &-1.038  &-0.031  &1.140  & 2.706 & 0.789  & 8.661 & 151.8 \\ \noalign{\vskip 2.5pt}
& $n=6000$ &-2.259  & -0.918 & 0.082 &1.152  &2.586  &0.803  &11.20 & 131.2 \\ \noalign{\vskip 2.5pt}  \noalign{\vskip 2.5pt}

\multirow{3}{*}{$\alpha=\hat{\alpha}_n, \beta=\hat{\beta}_n$} & $n=4000$ & -3.178 &-1.316  &-0.002  &1.215  & 3.289 &0.691  & 1.887 & 208.8 \\ \noalign{\vskip 2.5pt} 

 & $n=5000$ &-3.055  & -1.211 &-0.003  &1.229  &3.129  & 0.708 & 2.440 & 178.3 \\ \noalign{\vskip 2.5pt}

& $n=6000$ & -2.820 & -1.285 & 0.006  & 1.236 & 3.051  & 0.710  & 0.875 & 156.7 \\ 

\noalign{\vskip 1.5pt} \bottomrule
   \end{tabular}%
   }}

   \label{table: mle}
\end{table}
\clearpage%
}

\afterpage{%
\clearpage%
\thispagestyle{empty}

\begin{table}[htbp]
\centering
   \caption{Percentiles of $\xi$ and CVG, bias, and RMSE of $\hat{c}_n({\bftheta})$ when $\alpha_0=2$.}
   
   
  {\resizebox{1.0\textwidth}{!}{%
  \setlength{\tabcolsep}{2.26em}
   \begin{tabular}{ l c c c c  c c  c  c c} 
   \toprule \noalign{\vskip 1.5pt} 
	\multicolumn{2}{c}{Settings} 	& 5\%  & 25\% & 50\% & 75\%  & 95\%  & CVG & bias & RMSE  \\  \noalign{\vskip 1.5pt}  
	 \midrule \noalign{\vskip 1.5pt} 
\multicolumn{2}{c}{$\mathcal{N}(0, 1)$}  & -1.6449 & -0.6749 & 0 & 0.6749 & 1.6449 & 0.95 & 0 & \\ \noalign{\vskip 1.5pt} \hline \noalign{\vskip 2.5pt}

\multicolumn{10}{c}{$ER=0.6,  \nu=0.5$}  \\ \noalign{\vskip 2.5pt}  \noalign{\vskip 2.5pt} 
 \multicolumn{1}{c}{$\bftheta$}\\ \cline{1-1} \noalign{\vskip 2.5pt}   \noalign{\vskip 2.5pt}

\multirow{3}{*}{$\alpha=\alpha_0, \beta=\beta_0$} & $n=4000$ &-1.557 &-0.604 &-0.013 &0.691 &1.759 &0.954 &0.004 &0.099  \\ \noalign{\vskip 2.5pt} 
 & $n=5000$ &-1.442 &-0.614 &0.003 &0.723 &1.575 &0.962 & 0.002 &0.086  \\ \noalign{\vskip 2.5pt}
& $n=6000$ &-1.689 &-0.462 &0.093 &0.728 &1.970 &0.947 &0.003 & 0.084  \\ \noalign{\vskip 2.5pt}   \noalign{\vskip 2.5pt}

\multirow{3}{*}{$\alpha=\alpha_0, \beta=\sqrt{0.5}\beta_0$}& $n=4000$ &-1.072 &-0.128 &0.486 &1.179 &2.264 &0.921 &0.052 &0.113  \\ \noalign{\vskip 2.5pt} 
 & $n=5000$ &-0.999 &-0.179 &0.493 &1.183 &2.047 &0.939 & 0.043 & 0.097  \\ \noalign{\vskip 2.5pt}
& $n=6000$ &-1.315 & -0.010 & 0.556 & 1.184&2.396 & 0.929 & 0.041 & 0.094  \\ \noalign{\vskip 2.5pt}   \noalign{\vskip 2.5pt}

\multirow{3}{*}{$\alpha=\alpha_0, \beta=\sqrt{2}\beta_0$} & $n=4000$ &-1.801 &-0.860 &-0.258 &0.440 &1.505 &0.949 & -0.021 &0.101 \\ \noalign{\vskip 2.5pt} 
 & $n=5000$ &-1.680 &-0.840 & -0.232& 0.479 &1.347 & 0.946 & -0.018 & 0.088  \\ \noalign{\vskip 2.5pt}
& $n=6000$ &-1.880 &-0.681 &-0.136 &0.517 &1.756 &0.931 & -0.012 & 0.084  \\ \noalign{\vskip 2.5pt}   \noalign{\vskip 2.5pt}

\multirow{3}{*}{$\alpha=\alpha_0, \beta=\hat{\beta}_n$} & $n=4000$ &-1.616 &-0.600 &0.040 &0.758 &1.796 &0.954 & 0.008 & 0.103  \\ \noalign{\vskip 2.5pt} 
 & $n=5000$ &-1.443 &-0.583 &0.070 &0.752 &1.705 &0.962 &0.006 & 0.088  \\ \noalign{\vskip 2.5pt}
& $n=6000$ &-1.564 &-0.505 &0.171 &0.774 & 1.941 &0.938 & 0.008 & 0.087  \\ \noalign{\vskip 2.5pt}  \noalign{\vskip 2.5pt}

\multirow{3}{*}{$\alpha=\hat{\alpha}_n, \beta=\hat{\beta}_n$} & $n=4000$ &-1.576 &-0.546 &0.140 &0.798 &1.880 &0.944 &0.014 &0.104  \\ \noalign{\vskip 2.5pt} 
 & $n=5000$ &-1.426 & -0.565 &0.094 &0.785 &1.747 &0.956 &0.009 &0.089  \\ \noalign{\vskip 2.5pt}
& $n=6000$ &-1.595 &-0.614 &0.079 &0.764 &1.882 &0.953 &0.007 & 0.085  \\ \noalign{\vskip 2.5pt}  \hline \noalign{\vskip 2.5pt}

  \multicolumn{10}{c}{$ER=0.6,   \nu=1.5$}  \\ \noalign{\vskip 2.5pt} \noalign{\vskip 2.5pt}
 \multicolumn{1}{c}{$\bftheta$}\\ \cline{1-1} \noalign{\vskip 2.5pt}   \noalign{\vskip 2.5pt}

\multirow{3}{*}{$\alpha=\alpha_0, \beta=\beta_0$} & $n=4000$ &-1.567 &-0.624 &-0.010 &0.689 &1.764 &0.952 &0.103 & 2.513  \\ \noalign{\vskip 2.5pt} 
 & $n=5000$ &-1.469 &-0.633 &0.005 &0.734 &1.620 &0.958 & 0.083 & 2.200   \\ \noalign{\vskip 2.5pt}
& $n=6000$ &-1.729 &-0.614 &0.016 &0.592 &1.646 &0.953 & -0.013 &2.027  \\ \noalign{\vskip 2.5pt}   \noalign{\vskip 2.5pt}

\multirow{3}{*}{$\alpha=\alpha_0, \beta=\sqrt{0.5}\beta_0$} & $n=4000$ &1.226 &2.227 &2.885 &3.622 &4.748 &0.215 &7.351 &7.840  \\ \noalign{\vskip 2.5pt} 
 & $n=5000$ &1.056 &2.005 & 2.772 & 3.469 & 4.427 & 0.257 & 6.145 & 6.586  \\ \noalign{\vskip 2.5pt}
& $n=6000$ &0.725 &1.861 &2.602 &3.180 & 4.350 & 0.296 &5.225 & 5.657  \\ \noalign{\vskip 2.5pt}   \noalign{\vskip 2.5pt}

\multirow{3}{*}{$\alpha=\alpha_0, \beta=\sqrt{2}\beta_0$} & $n=4000$ &-2.801 &-1.876 &-1.283 &-0.593 &0.433 &0.749 &-3.103 & 3.948  \\ \noalign{\vskip 2.5pt} 
 & $n=5000$ &-2.590 & -1.829& -1.185 & -0.497 & 0.385 & 0.765 &-2.612 & 3.379  \\ \noalign{\vskip 2.5pt}
& $n=6000$ &-2.807 &-1.771 &-1.128 &-0.565 &0.493 &0.779 & -2.350 & 3.073 \\ \noalign{\vskip 2.5pt}   \noalign{\vskip 2.5pt}

\multirow{3}{*}{$\alpha=\alpha_0, \beta=\hat{\beta}_n$} & $n=4000$ &-1.613 &-0.649 &0.093 &0.790 &1.923 &0.928 &0.223 &2.701  \\ \noalign{\vskip 2.5pt} 
 & $n=5000$ &-1.417 &-0.648 &0.046 &0.818 &1.821 & 0.957 & 0.197 & 2.332  \\ \noalign{\vskip 2.5pt}
& $n=6000$ & -1.623& -0.655 &0.056 &0.691 & 1.793 & 0.954 & 0.071 & 2.111  \\ \noalign{\vskip 2.5pt}   \noalign{\vskip 2.5pt}

\multirow{3}{*}{$\alpha=\hat{\alpha}_n, \beta=\hat{\beta}_n$} & $n=4000$ &-1.558 &-0.598 &0.095 &0.811 &1.977 &0.921 &0.301 & 2.759  \\ \noalign{\vskip 2.5pt} 
 & $n=5000$  &-1.543 &-0.594 & 0.040 & 0.798 &1.827 & 0.950 & 0.186 & 2.342 \\ \noalign{\vskip 2.5pt}
& $n=6000$ &-1.604 &-0.569 &0.090 &0.763 &1.827 &0.946 & 0.198 & 2.132  \\  \noalign{\vskip 2.5pt}   \hline \noalign{\vskip 2.5pt}

\multicolumn{10}{c}{$ER=0.9,  \nu=0.5$}  \\ \noalign{\vskip 2.5pt}  \noalign{\vskip 2.5pt} 
 \multicolumn{1}{c}{$\bftheta$}\\ \cline{1-1} \noalign{\vskip 2.5pt}   \noalign{\vskip 2.5pt}

\multirow{3}{*}{$\alpha=\alpha_0, \beta=\beta_0$} & $n=4000$ &-1.574 &-0.594 &-0.026 &0.666 &1.776 &0.952 &0.002 &0.066 \\ \noalign{\vskip 2.5pt}
& $n=5000$ & -1.458&-0.664 &-0.052 &0.638 &1.547 &0.962 & -0.001 &0.057  \\ \noalign{\vskip 2.5pt} 
& $n=6000$ &-1.742 &-0.548 &0.145 &0.809 &1.784 &0.942 &0.005 & 0.055  \\ \noalign{\vskip 2.5pt}   \noalign{\vskip 2.5pt}

\multirow{3}{*}{$\alpha=\alpha_0, \beta=\sqrt{0.5}\beta_0$} & $n=4000$ &-1.319 &-0.321 &0.254 &0.930 &2.042 &0.938 &0.020 & 0.069  \\ \noalign{\vskip 2.5pt}
 & $n=5000$&-1.220 & -0.411 &0.185 &0.904 &1.817 &0.962 & 0.014 & 0.059  \\ \noalign{\vskip 2.5pt} 
& $n=6000$ &-1.567 &-0.366 & 0.385&1.022 &2.052 &0.925 & 0.017 & 0.058  \\ \noalign{\vskip 2.5pt}   \noalign{\vskip 2.5pt}

\multirow{3}{*}{$\alpha=\alpha_0, \beta=\sqrt{2}\beta_0$} & $n=4000$ &-1.704 &-0.723 &-0.158 &0.524 &1.635 &0.950 & -0.007 &0.066 \\ \noalign{\vskip 2.5pt}
 & $n=5000$ &-1.579 &-0.786 &-0.173 &0.517 &1.424 &0.963 & -0.007 & 0.057  \\ \noalign{\vskip 2.5pt} 
& $n=6000$ &-1.862 &-0.653 &0.024 &0.700 &1.646 &0.950 &-0.001 & 0.054  \\ \noalign{\vskip 2.5pt}   \noalign{\vskip 2.5pt}

\multirow{3}{*}{$\alpha=\alpha_0, \beta=\hat{\beta}_n$} & $n=4000$ &-1.586 &-0.578 &0.022 &0.695 &1.799 & 0.948&0.005 & 0.068  \\ \noalign{\vskip 2.5pt}
 & $n=5000$ & -1.440&-0.615 &-0.006 &0.696 &1.648 &0.959 & 0.001 & 0.058  \\ \noalign{\vskip 2.5pt} 
& $n=6000$ &-1.653 & -0.450&0.244 &0.880 &1.646 &0.930 &0.009 & 0.055 \\ \noalign{\vskip 2.5pt}  \noalign{\vskip 2.5pt}

\multirow{3}{*}{$\alpha=\hat{\alpha}_n, \beta=\hat{\beta}_n$} & $n=4000$ &-1.572 &-0.550 &0.123 &0.789 &1.862 &0.950 &0.008 &0.068  \\ \noalign{\vskip 2.5pt} 
 & $n=5000$ &-1.378 &-0.541 & 0.097 &0.798 & 1.782 & 0.957 & 0.007 & 0.059  \\ \noalign{\vskip 2.5pt}
& $n=6000$ &-1.659 &-0.595 &0.055 &0.727 &1.786 &0.954 & 0.003 & 0.055  \\ \noalign{\vskip 2.5pt}  \hline \noalign{\vskip 2.5pt}

\multicolumn{10}{c}{$ER=0.9,   \nu=1.5$}  \\ \noalign{\vskip 2.5pt} \noalign{\vskip 2.5pt}
 \multicolumn{1}{c}{$\bftheta$}\\ \cline{1-1} \noalign{\vskip 2.5pt}   \noalign{\vskip 2.5pt}

\multirow{3}{*}{$\alpha=\alpha_0, \beta=\beta_0$} & $n=4000$ &-1.589 &-0.595 &-0.015 &0.700 &1.748 &0.955 &0.028 &0.744  \\ \noalign{\vskip 2.5pt}
 & $n=5000$ &-1.454 &-0.668 &-0.029 &0.673 &1.531 &0.966 &-0.006 & 0.638  \\ \noalign{\vskip 2.5pt} 
& $n=6000$ &-1.701 &-0.671 &-0.098 &0.572 &1.747 &0.940 & -0.031 &0.652  \\ \noalign{\vskip 2.5pt}   \noalign{\vskip 2.5pt}

\multirow{3}{*}{$\alpha=\alpha_0, \beta=\sqrt{0.5}\beta_0$} & $n=4000$ &-0.026 &0.961 &1.579 & 2.307 &3.369 &0.675 &1.205 & 1.434  \\ \noalign{\vskip 2.5pt}
& $n=5000$ &-0.117 &0.805 & 1.443 & 2.165& 3.003 & 3.705 & 0.978 & 1.183  \\ \noalign{\vskip 2.5pt} 
& $n=6000$ &-0.371 &0.705 &1.327 &1.957 & 3.189 & 0.764 &0.813 & 1.037  \\ \noalign{\vskip 2.5pt}   \noalign{\vskip 2.5pt}

\multirow{3}{*}{$\alpha=\alpha_0, \beta=\sqrt{2}\beta_0$} & $n=4000$ &-2.288 &-1.288 &-0.722 & -0.002&1.009 &0.886 &-0.498 & 0.886  \\ \noalign{\vskip 2.5pt}
 & $n=5000$ & -2.098& -1.348 & -0.706 & -0.010 & 0.858 & 0.916 & -0.447 & 0.771 \\ \noalign{\vskip 2.5pt} 
& $n=6000$ &-2.312 &-1.309 &-0.728 &-0.050 &1.071 & 0.887 & -0.411 & 0.742  \\ \noalign{\vskip 2.5pt}   \noalign{\vskip 2.5pt}

\multirow{3}{*}{$\alpha=\alpha_0, \beta=\hat{\beta}_n$} & $n=4000$ &-1.684 &-0.599 &0.092 &0.751 &1.805 &0.934 &0.058 & 0.780  \\ \noalign{\vskip 2.5pt}
 & $n=5000$ &-1.468 & -0.698 & -0.003 & 0.696 & 1.641 & 0.966 & 0.007 & 0.658 \\ \noalign{\vskip 2.5pt} 
& $n=6000$ &-1.670 &-0.725 &-0.068 &0.660 &1.775 &0.931 & -0.023 & 0.644  \\ \noalign{\vskip 2.5pt} \noalign{\vskip 2.5pt}

\multirow{3}{*}{$\alpha=\hat{\alpha}_n, \beta=\hat{\beta}_n$} & $n=4000$ &-1.532 &-0.611 &0.100 &0.820 &1.903 &0.934 &0.080 & 0.781   \\ \noalign{\vskip 2.5pt} 
 & $n=5000$ &-1.422 & -0.584& 0.050 & 0.746 & 1.745 & 0.959 & 0.049 & 0.664  \\ \noalign{\vskip 2.5pt}
& $n=6000$ &-1.498 &-0.544 &0.068 &0.810 &1.843 &0.950 &0.042 & 0.618  \\ 

\noalign{\vskip 1.5pt} \bottomrule
   \end{tabular}%
   }}

   \label{table: mle0}
\end{table}
\clearpage%
}

\afterpage{%
\clearpage%
\thispagestyle{empty}

\begin{table}[htbp]
\centering
   \caption{Percentiles of $\xi$ and CVG, bias, and RMSE of $\hat{c}_n({\bftheta})$ when $\alpha_0=5$.}
   
   
  {\resizebox{1.0\textwidth}{!}{%
  \setlength{\tabcolsep}{2.26em}
   \begin{tabular}{ l c c c c  c c  c  c c} 
   \toprule \noalign{\vskip 1.5pt} 
	\multicolumn{2}{c}{Settings} 	& 5\%  & 25\% & 50\% & 75\%  & 95\%  & CVG & bias & RMSE  \\  \noalign{\vskip 1.5pt}  
	 \midrule \noalign{\vskip 1.5pt} 
\multicolumn{2}{c}{$\mathcal{N}(0, 1)$}  & -1.6449 & -0.6749 & 0 & 0.6749 & 1.6449 & 0.95 & 0 & \\ \noalign{\vskip 1.5pt} \hline \noalign{\vskip 2.5pt}

\multicolumn{10}{c}{$ER=0.6,  \nu=0.5$}  \\ \noalign{\vskip 2.5pt}  \noalign{\vskip 2.5pt} 
 \multicolumn{1}{c}{$\bftheta$}\\ \cline{1-1} \noalign{\vskip 2.5pt}   \noalign{\vskip 2.5pt}

\multirow{3}{*}{$\alpha=\alpha_0, \beta=\beta_0$} & $n=4000$ &-1.612  &-0.659  &-0.049  &0.655  &1.661  &0.954  & 0.000 & 0.085 \\ \noalign{\vskip 2.5pt} 
 & $n=5000$ & -1.468 &-0.636  & -0.027 & 0.685  & 1.683  & 0.961 & 0.002 & 0.074 \\ \noalign{\vskip 2.5pt}
& $n=6000$ & -1.633 &-0.560  &0.079  &0.723  &1.751  &0.940  &0.003 & 0.070 \\ \noalign{\vskip 2.5pt}   \noalign{\vskip 2.5pt}

\multirow{3}{*}{$\alpha=\alpha_0, \beta=\sqrt{0.5}\beta_0$} & $n=4000$ &-1.273  &-0.289  &0.319  &1.010  &2.038  &0.942  &0.030 &0.091 \\ \noalign{\vskip 2.5pt} 
 & $n=5000$ &-1.183  & -0.311 & 0.298  &1.033  &2.054  &0.942  &0.027 & 0.079 \\ \noalign{\vskip 2.5pt}
& $n=6000$ &-1.261  & -0.245 &0.401  &1.023  &2.070  & 0.932 & 0.025 &0.074 \\ \noalign{\vskip 2.5pt}   \noalign{\vskip 2.5pt}

\multirow{3}{*}{$\alpha=\alpha_0, \beta=\sqrt{2}\beta_0$} & $n=4000$ & -1.757 & -0.829 &-0.222  &0.479  &1.477  &0.945  &-0.015 & 0.086 \\ \noalign{\vskip 2.5pt} 
 & $n=5000$ &-1.616  &-0.793  &-0.191  &-0.500  &-1.520  & 0.958 & -0.011 & 0.074 \\ \noalign{\vskip 2.5pt}
& $n=6000$ & -1.787 &-0.709  &-0.071  & 0.556 & 1.587 &0.936  & -0.007 & 0.070 \\ \noalign{\vskip 2.5pt}   \noalign{\vskip 2.5pt}

\multirow{3}{*}{$\alpha=\alpha_0, \beta=\hat{\beta}_n$} & $n=4000$ & -1.607 &-0.633  & 0.000 &0.690  &1.705  & 0.951 &0.003 & 0.087 \\ \noalign{\vskip 2.5pt} 
 & $n=5000$ &-1.434  &-0.591  &0.030  &0.719  &1.766  &0.953  & 0.005 & 0.075 \\ \noalign{\vskip 2.5pt}
& $n=6000$ & -1.609 &-0.564  &0.094  &0.758  &1.822  & 0.930 & 0.006 & 0.071 \\ \noalign{\vskip 2.5pt}  \noalign{\vskip 2.5pt}

\multirow{3}{*}{$\alpha=\hat{\alpha}_n, \beta=\hat{\beta}_n$} & $n=4000$ &-1.556  &-0.538  & 0.142 & 0.813 &1.884  &0.948  &0.012 &0.089 \\ \noalign{\vskip 2.5pt} 
 & $n=5000$ &-1.378  &-0.514  &0.116  &0.835  &1.762  &0.958  &0.010 &0.076 \\ \noalign{\vskip 2.5pt}
& $n=6000$ &-1.530  &-0.557  &0.100  &0.758  &1.832  &0.950  & 0.008 &0.070 \\ \noalign{\vskip 2.5pt}  \hline \noalign{\vskip 2.5pt}

  \multicolumn{10}{c}{$ER=0.6,   \nu=1.5$}  \\ \noalign{\vskip 2.5pt} \noalign{\vskip 2.5pt}
 \multicolumn{1}{c}{$\bftheta$}\\ \cline{1-1} \noalign{\vskip 2.5pt}   \noalign{\vskip 2.5pt}

\multirow{3}{*}{$\alpha=\alpha_0, \beta=\beta_0$} & $n=4000$ &-1.611  &-0.651  & -0.015 &0.676  &1.748  &0.957  &0.026 & 1.179 \\ \noalign{\vskip 2.5pt} 
 & $n=5000$ &-1.409  & -0.606  &0.052  &0.727  & 1.705 & 0.958 & 0.069 &1.031  \\ \noalign{\vskip 2.5pt}
& $n=6000$ &-1.346  &-0.495  &0.146  &0.763  &1.660  &0.950  & 0.114 &0.930 \\ \noalign{\vskip 2.5pt}   \noalign{\vskip 2.5pt}

\multirow{3}{*}{$\alpha=\alpha_0, \beta=\sqrt{0.5}\beta_0$} & $n=4000$ &0.153  &1.146  &1.783  &2.469  &3.601  & 0.618 & 2.132 & 2.463 \\ \noalign{\vskip 2.5pt} 
 & $n=5000$ & 0.109 & 1.070  &1.720  & 2.432 &3.475  &0.628  & 1.818 & 2.113 \\ \noalign{\vskip 2.5pt}
& $n=6000$ & 0.053 &1.018  &1.675  &2.311  &3.351  &0.653  & 1.623 & 1.886 \\ \noalign{\vskip 2.5pt}   \noalign{\vskip 2.5pt}

\multirow{3}{*}{$\alpha=\alpha_0, \beta=\sqrt{2}\beta_0$} & $n=4000$ &-2.308  &-1.383  &-0.802  &-0.088  &0.934  & 0.875 &-0.877 & 1.456  \\ \noalign{\vskip 2.5pt} 
 & $n=5000$ &-2.093  &-1.323  & -0.675  & 0.004  &0.964  &0.909  & -0.685 & 1.224 \\ \noalign{\vskip 2.5pt}
& $n=6000$ &-1.985  &-1.186  & -0.521 & 0.083 & 0.993  & 0.929  & -0.538 & 1.059\\ \noalign{\vskip 2.5pt}   \noalign{\vskip 2.5pt}

\multirow{3}{*}{$\alpha=\alpha_0, \beta=\hat{\beta}_n$} & $n=4000$ &-1.664  & -0.641  & 0.069  &0.757  &1.802  &0.951  &0.070 & 1.239 \\ \noalign{\vskip 2.5pt} 
 & $n=5000$ &-1.362  & -0.615 & 0.080  & 0.792  & 1.788  & 0.956  &0.105 & 1.074 \\ \noalign{\vskip 2.5pt}
& $n=6000$ &-1.368  &-0.588  &0.177  &0.809  &1.812  & 0.950 & 0.129 & 0.966 \\ \noalign{\vskip 2.5pt}   \noalign{\vskip 2.5pt}

\multirow{3}{*}{$\alpha=\hat{\alpha}_n, \beta=\hat{\beta}_n$} & $n=4000$ &-1.528  &-0569  &0.138 &0.845 &1.937 &0.929 &0.181 &1.263 \\ \noalign{\vskip 2.5pt} 
 & $n=5000$  &-1.342  &-0.522  & 0.105 &0.838  &1.822  &0.956  &0.150 &1.065 \\ \noalign{\vskip 2.5pt}
& $n=6000$ & -1.468 & -0.518 &0.102  & 0.842 &1.920  &0.937  &0.142 & 1.003 \\  \noalign{\vskip 2.5pt}   \hline \noalign{\vskip 2.5pt}

\multicolumn{10}{c}{$ER=0.9,  \nu=0.5$}  \\ \noalign{\vskip 2.5pt}  \noalign{\vskip 2.5pt} 
 \multicolumn{1}{c}{$\bftheta$}\\ \cline{1-1} \noalign{\vskip 2.5pt}   \noalign{\vskip 2.5pt}

\multirow{3}{*}{$\alpha=\alpha_0, \beta=\beta_0$}  & $n=4000$ &-1.567  &-0.602  &-0.013  &0.681  &1.759  & 0.955  & 0.002 & 0.057 \\ \noalign{\vskip 2.5pt}
& $n=5000$ & -1.482  & -0.664  & -0.034  &0.643  &1.563 &0.961  & 0.000 & 0.049 \\ \noalign{\vskip 2.5pt} 
& $n=6000$ & -1.513 & -0.615 &0.077  &0.637  & 1.745 & 0.953  & 0.002 &0.045 \\ \noalign{\vskip 2.5pt}   \noalign{\vskip 2.5pt}

\multirow{3}{*}{$\alpha=\alpha_0, \beta=\sqrt{0.5}\beta_0$} & $n=4000$ &-1.377  &-0.423  &0.179  &0.871  &1.956  &0.947  &0.013 & 0.058 \\ \noalign{\vskip 2.5pt}
 & $n=5000$ & -1.333 & -0.478 & 0.153 & 0.835 & 1.743  &0.961  & 0.009 & 0.050 \\ \noalign{\vskip 2.5pt} 
& $n=6000$ &-1.357  & -0.468 &0.233  &0.800  &1.894  &0.947  &0.009 &0.046 \\ \noalign{\vskip 2.5pt}   \noalign{\vskip 2.5pt}

\multirow{3}{*}{$\alpha=\alpha_0, \beta=\sqrt{2}\beta_0$} & $n=4000$ &-1.663  &-0.708  &-0.109  &0.587  &1.646  & 0.957 &-0.004 &0.057 \\ \noalign{\vskip 2.5pt}
 & $n=5000$ &-1.554  & -0.765  & -0.121  &0.556  &1.465  &0.958  & -0.005 & 0.049 \\ \noalign{\vskip 2.5pt} 
& $n=6000$ &-1.605  &-0.692  &-0.006  &0.559  &1.655  &0.942  &-0.002 &0.045 \\ \noalign{\vskip 2.5pt}   \noalign{\vskip 2.5pt}

\multirow{3}{*}{$\alpha=\alpha_0, \beta=\hat{\beta}_n$} & $n=4000$ &-1.571  &-0.593  &0.016  &0.716  &1.744  &0.951  &0.004 & 0.058 \\ \noalign{\vskip 2.5pt}
 & $n=5000$ &-1.424  & -0.653 & -0.007  & 0.669  & 1.638  & 0.962  & 0.001 & 0.049 \\ \noalign{\vskip 2.5pt} 
& $n=6000$ &-1.462  &-0.572  &0.095  &0.704  &1.803  &0.949  &0.003 &0.045 \\ \noalign{\vskip 2.5pt}  \noalign{\vskip 2.5pt}

\multirow{3}{*}{$\alpha=\hat{\alpha}_n, \beta=\hat{\beta}_n$} & $n=4000$ &-1.574  & -0.539 &0.106  &0.803  &1.853  &0.952  & 0.007 & 0.059 \\ \noalign{\vskip 2.5pt} 
 & $n=5000$ &-1.327  & -0.537 & 0.134 & 0.810  & 1.778  & 0.960  & 0.007 & 0.049 \\ \noalign{\vskip 2.5pt}
& $n=6000$ &-1.606  &-0.601  &0.058  &0.713  & 1.756 & 0.956  & 0.003 & 0.046 \\ \noalign{\vskip 2.5pt}  \hline \noalign{\vskip 2.5pt}

\multicolumn{10}{c}{$ER=0.9,   \nu=1.5$}  \\ \noalign{\vskip 2.5pt} \noalign{\vskip 2.5pt}
 \multicolumn{1}{c}{$\bftheta$}\\ \cline{1-1} \noalign{\vskip 2.5pt}   \noalign{\vskip 2.5pt}

\multirow{3}{*}{$\alpha=\alpha_0, \beta=\beta_0$} & $n=4000$ & -1.574 &-0.616  &-0.015  &0.661  &1.764  &0.955  & 0.009 & 0.345 \\ \noalign{\vskip 2.5pt}
 & $n=5000$ &-1.414  & -0.617  & 0.015  &0.741  &1.612  &0.974  & 0.017 & 0.298 \\ \noalign{\vskip 2.5pt} 
& $n=6000$ &-1.743  & -0.591 &0.066 &0.677 &1.795  &0.933  &0.008  & 0.292 \\ \noalign{\vskip 2.5pt}   \noalign{\vskip 2.5pt}

\multirow{3}{*}{$\alpha=\alpha_0, \beta=\sqrt{0.5}\beta_0$}  & $n=4000$ &-0.614  &0.325  & 0.969 &1.631  &2.762  &0.837  &0.351 & 0.498 \\ \noalign{\vskip 2.5pt}
& $n=5000$ &-0.587  &0.274  &0.936  &1.652  & 2.593  &0.857  & 0.299 & 0.428 \\ \noalign{\vskip 2.5pt} 
& $n=6000$ & -0.978 &0.198  &0.889  &1.566  & 2.643 &0.870  &0.248 &0.389 \\ \noalign{\vskip 2.5pt}   \noalign{\vskip 2.5pt}

\multirow{3}{*}{$\alpha=\alpha_0, \beta=\sqrt{2}\beta_0$} & $n=4000$ &-1.988  &-1.042  & -0.448 & 0.234  & 1.285  & 0.933  & -0.139 & 0.369 \\ \noalign{\vskip 2.5pt}
 & $n=5000$ & -1.784 &-1.013  & -0.373 & 0.330 &1.207  &0.956  &-0.105  & 0.313\\ \noalign{\vskip 2.5pt} 
& $n=6000$ &-2.094  &-0.968  &-0.301  &0.312  &1.420 &0.925  &-0.097 &0.306 \\ \noalign{\vskip 2.5pt}   \noalign{\vskip 2.5pt}

\multirow{3}{*}{$\alpha=\alpha_0, \beta=\hat{\beta}_n$} & $n=4000$ &-1.590  & -0.609 & 0.029  & 0.720 & 1.791  &  0.943& 0.021 & 0.356 \\ \noalign{\vskip 2.5pt}
 & $n=5000$ & -1.364 &-0.605  &0.055  &0.779  &1.677  & 0.967 & 0.025 & 0.303 \\ \noalign{\vskip 2.5pt} 
& $n=6000$ &-1.648  &-0.608  &0.107  &0.709  &1.823  &0.947  &0.015 &0.296 \\ \noalign{\vskip 2.5pt} \noalign{\vskip 2.5pt}

\multirow{3}{*}{$\alpha=\hat{\alpha}_n, \beta=\hat{\beta}_n$} & $n=4000$ & -1.533 & -0.549 &0.113 &0.849  &1.836  & 0.940  & 0.047  & 0.369  \\ \noalign{\vskip 2.5pt} 
 & $n=5000$ &-1.387  & -0.526 &0.111  & 0.798  &1.733  & 0.961  & 0.039 & 0.307 \\ \noalign{\vskip 2.5pt}
& $n=6000$ &-1.352  &-0.534  &0.111  &0.765  &1.755  &0.956  &0.039 & 0.281 \\ 

\noalign{\vskip 1.5pt} \bottomrule
   \end{tabular}%
   }}

   \label{table: mle1}
\end{table}
\clearpage%
}

\newpage
\section{Additional Simulation Examples} \label{app: numerical examples}

The results in the three cases in Section~\ref{sec:numerical} of the main manuscript are based $n=2000$ observations. In Section~\ref{app: simulation with different ER} of the Supplementary Material, we provide results on exactly the same simulation settings with $n=500$ and 1000. Similar conclusions can be drawn there. In addition, we also investigate the predictive performance when the covariance of the underlying true process is a product of individual covariance functions in Section~\ref{app: UQ simulation} of the Supplementary Material. The examples there show significant improvement of the CH class over the Mat\'ern class and the GC class. In all these simulation examples, we found that the CH class is quite flexible in terms of capturing both the smoothness and the tail behavior. No matter which covariance structure (the Mat\'ern class or the GC class) the true underlying process is generated from, the CH class is able to capture the underlying true covariance structure with satisfactory performance as implied by our theoretical developments. In contrast, the Mat\'ern class is not able to capture the underlying true covariance structure with polynomially decaying dependence and the GC class is not able to capture the underlying true covariance structure with different degrees of smoothness behaviors. Below are the detailed results. 

\subsection{Predictive Performance with Different Sample Sizes} \label{app: simulation with different ER}
In this section, we use the same simulation settings as in Section \ref{sec:numerical} but with $n=500$ and $1000$ observations for parameter estimation. The simulation setup here is the same as the one considered in Section~\ref{sec:numerical}. For $n=500$ observations, the results are shown in Figure~\ref{fig: simulation setting under case 1, n=500} for Case 1, Figure~\ref{fig: simulation setting under case 2, n=500} for Case 2, and Figure~\ref{fig: simulation setting under case 3, n=500} for Case 3. For $n=1000$ observations, the results are shown in Figure~\ref{fig: simulation setting under case 1, n=1000} for Case 1, Figure~\ref{fig: simulation setting under case 2, n=1000} for Case 2, and Figure~\ref{fig: simulation setting under case 3, n=1000} for Case 3. To conclude, the CH class is very flexible since it can allow different smoothness behaviors in the same way as the Mat\'ern class and can allow different degrees of tail behaviors that can capture the one in the GC class.

\begin{figure}[htbp] 
\begin{subfigure}{.35\textwidth}
  \centering
\makebox[\textwidth][c]{ \includegraphics[width=1.0\linewidth, height=0.18\textheight]{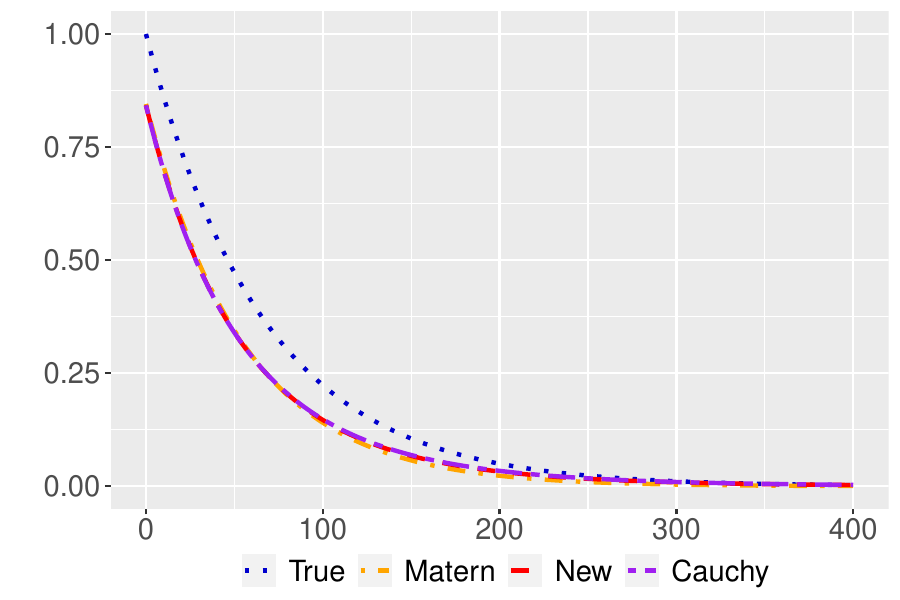}}
\end{subfigure}%
\begin{subfigure}{.65\textwidth}
  \centering
\makebox[\textwidth][c]{ \includegraphics[width=1.0\linewidth, height=0.18\textheight]{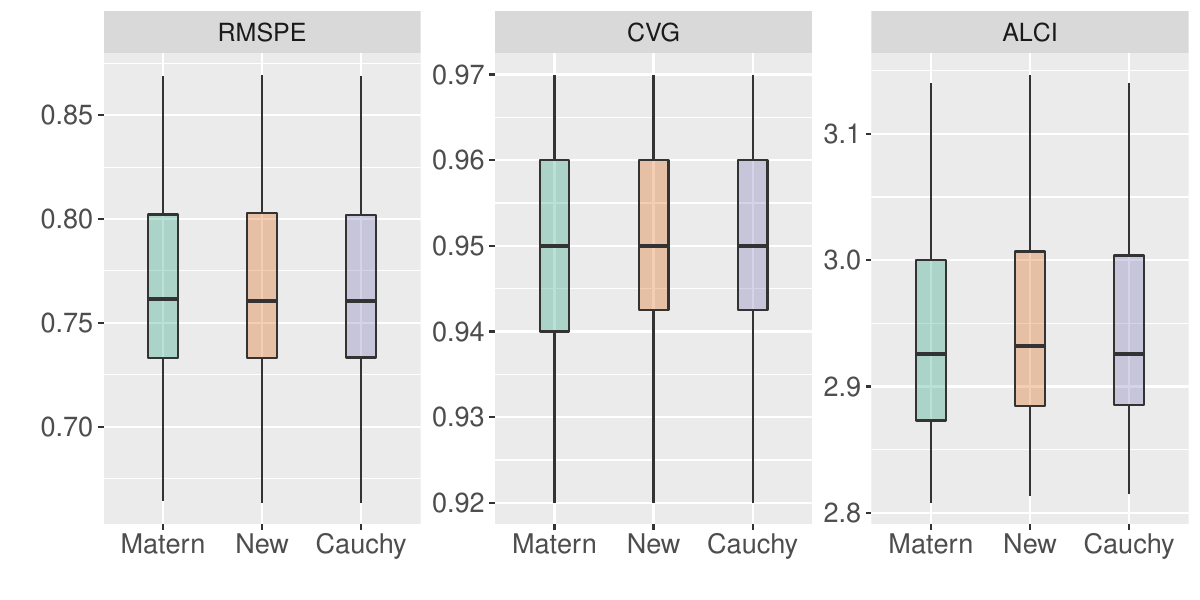}}
\end{subfigure}
\caption*{$\nu=0.5,  ER=200$}

\begin{subfigure}{.35\textwidth}
  \centering
\makebox[\textwidth][c]{ \includegraphics[width=1.0\linewidth, height=0.18\textheight]{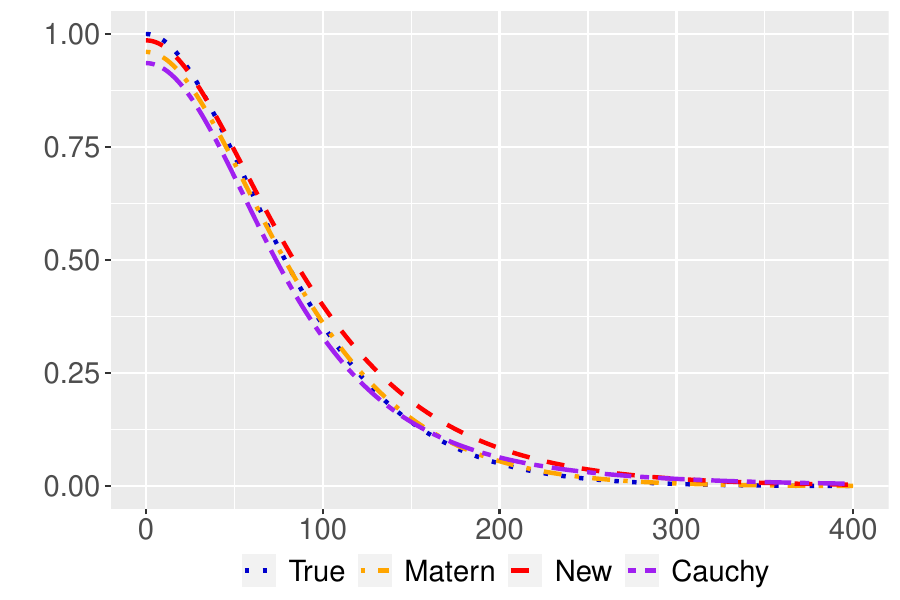}}
\end{subfigure}%
\begin{subfigure}{.65\textwidth}
  \centering
\makebox[\textwidth][c]{ \includegraphics[width=1.0\linewidth, height=0.18\textheight]{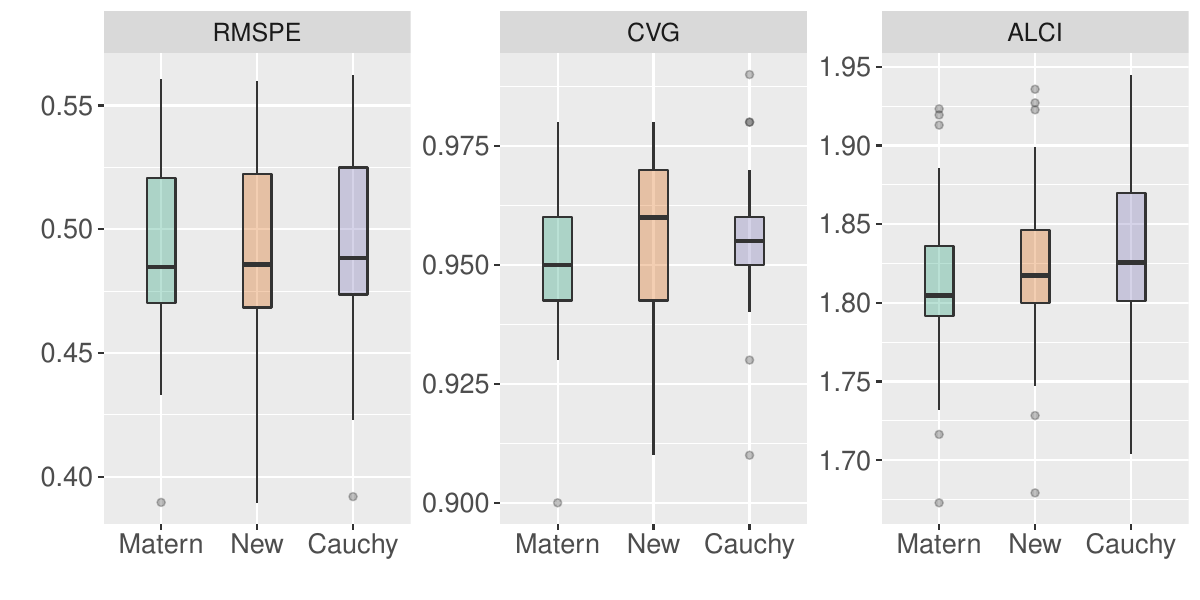}}
\end{subfigure}
\caption*{$\nu=2.5,  ER=200$}

\begin{subfigure}{.35\textwidth}
  \centering
\makebox[\textwidth][c]{ \includegraphics[width=1.0\linewidth, height=0.18\textheight]{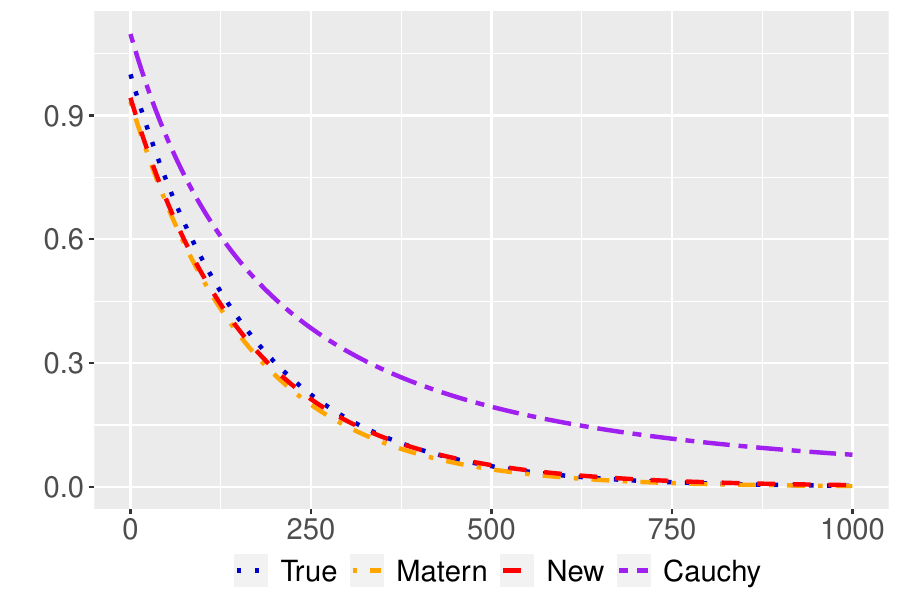}}
\end{subfigure}%
\begin{subfigure}{.65\textwidth}
  \centering
\makebox[\textwidth][c]{ \includegraphics[width=1.0\linewidth, height=0.18\textheight]{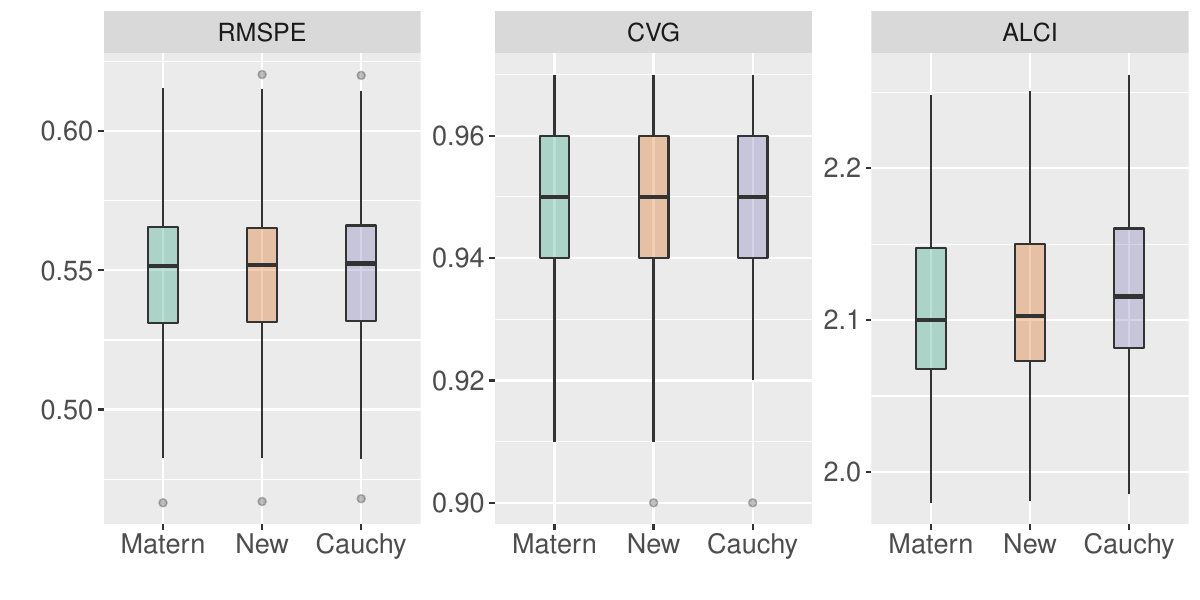}}
\end{subfigure}
\caption*{$\nu=0.5,  ER=500$}

\begin{subfigure}{.35\textwidth}
  \centering
\makebox[\textwidth][c]{ \includegraphics[width=1.0\linewidth, height=0.18\textheight]{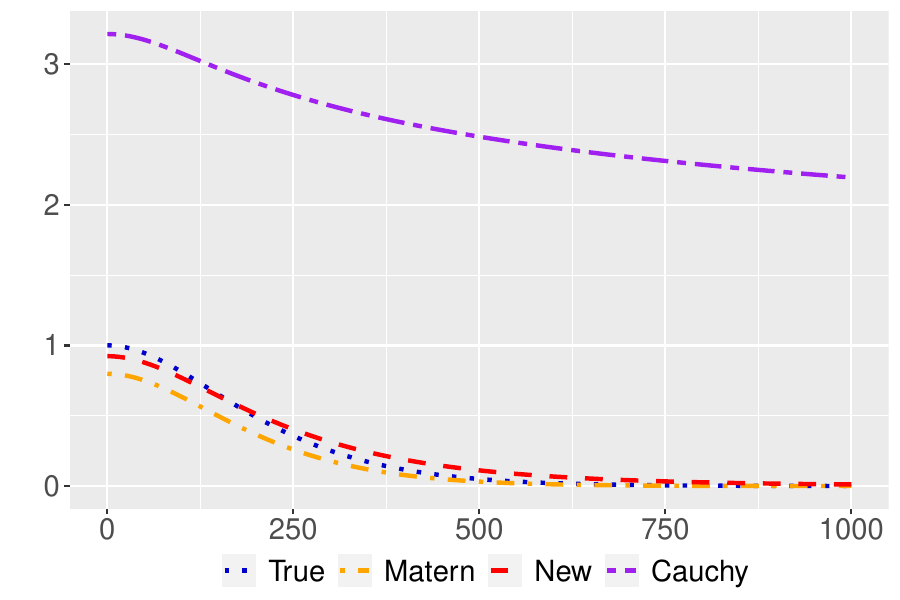}}
\end{subfigure}%
\begin{subfigure}{.65\textwidth}
  \centering
\makebox[\textwidth][c]{ \includegraphics[width=1.0\linewidth, height=0.18\textheight]{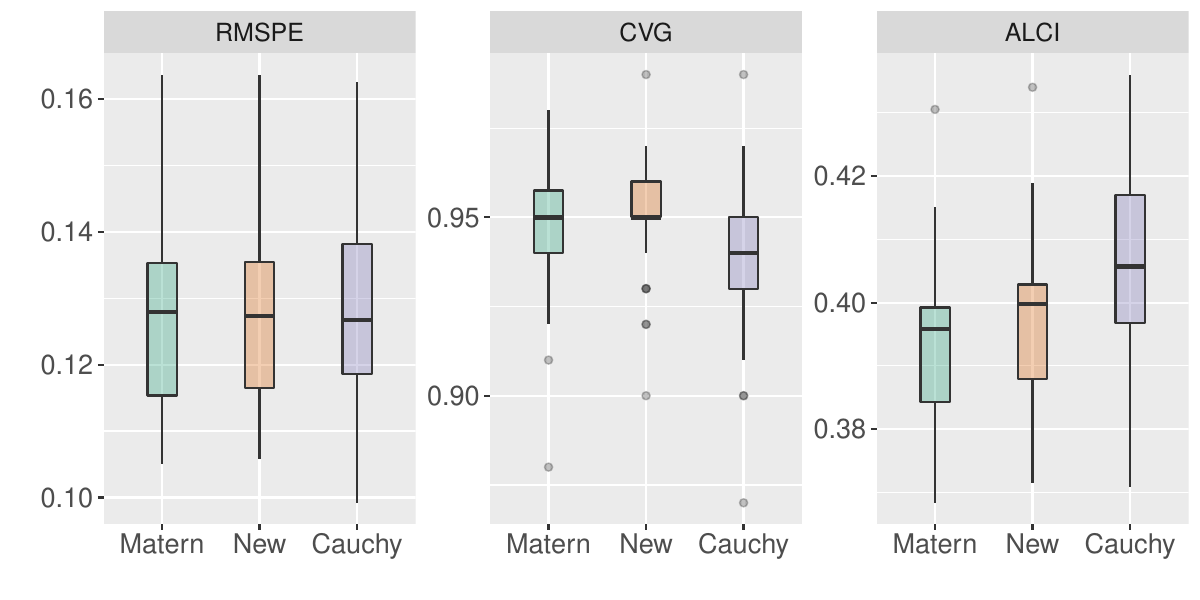}}
\end{subfigure}
\caption*{$\nu=2.5,  ER=500$}

\caption{Case 1: Comparison of predictive performance and estimated covariance structures when the true covariance is the Mat\'ern class with 500 observations. The predictive performance is evaluated at 10-by-10 regular grids in the square domain. These figures summarize the predictive measures based on RMSPE, CVG and ALCI under 30 simulated realizations.}
\label{fig: simulation setting under case 1, n=500}
\end{figure}

\begin{figure}[htbp] 
\begin{subfigure}{.35\textwidth}
  \centering
\makebox[\textwidth][c]{ \includegraphics[width=1.0\linewidth, height=0.18\textheight]{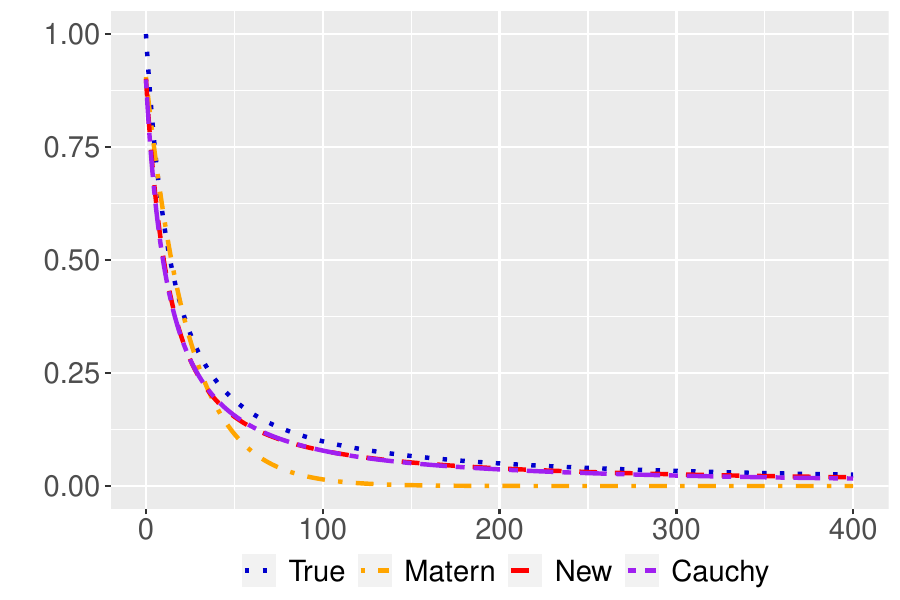}}
\end{subfigure}%
\begin{subfigure}{.65\textwidth}
  \centering
\makebox[\textwidth][c]{ \includegraphics[width=1.0\linewidth, height=0.18\textheight]{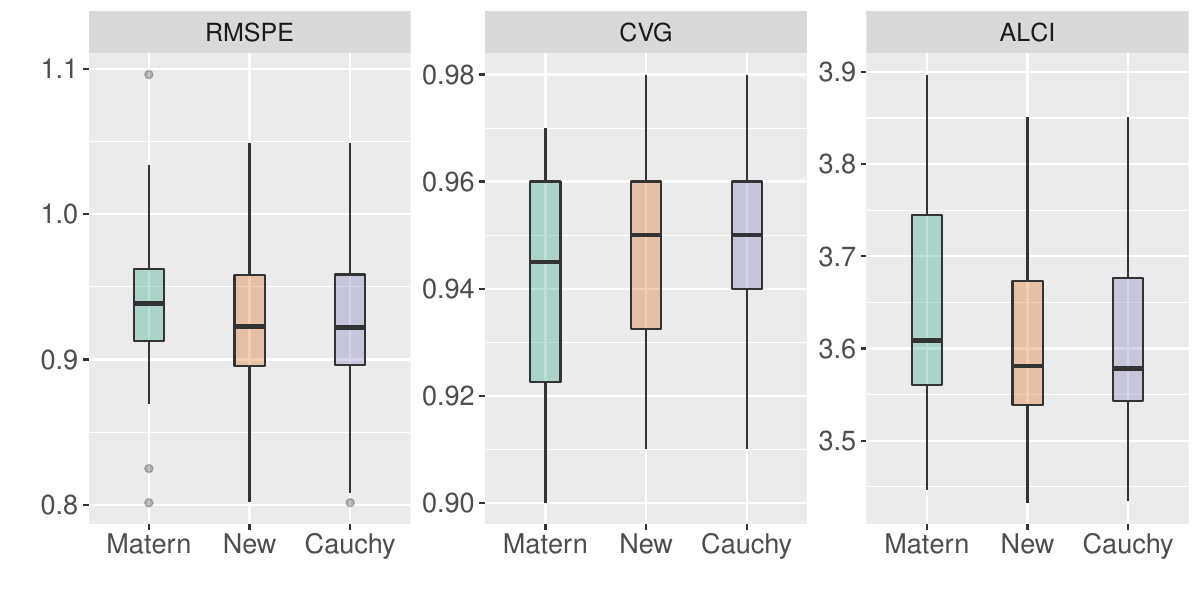}}
\end{subfigure}
\caption*{$\nu=0.5,  ER=200$}

\begin{subfigure}{.35\textwidth}
  \centering
\makebox[\textwidth][c]{ \includegraphics[width=1.0\linewidth, height=0.18\textheight]{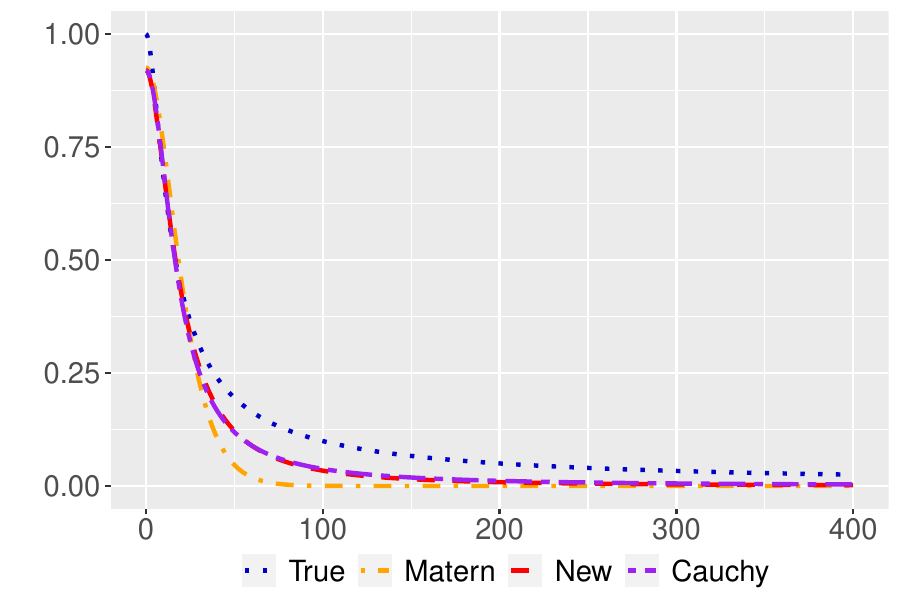}}
\end{subfigure}%
\begin{subfigure}{.65\textwidth}
  \centering
\makebox[\textwidth][c]{ \includegraphics[width=1.0\linewidth, height=0.18\textheight]{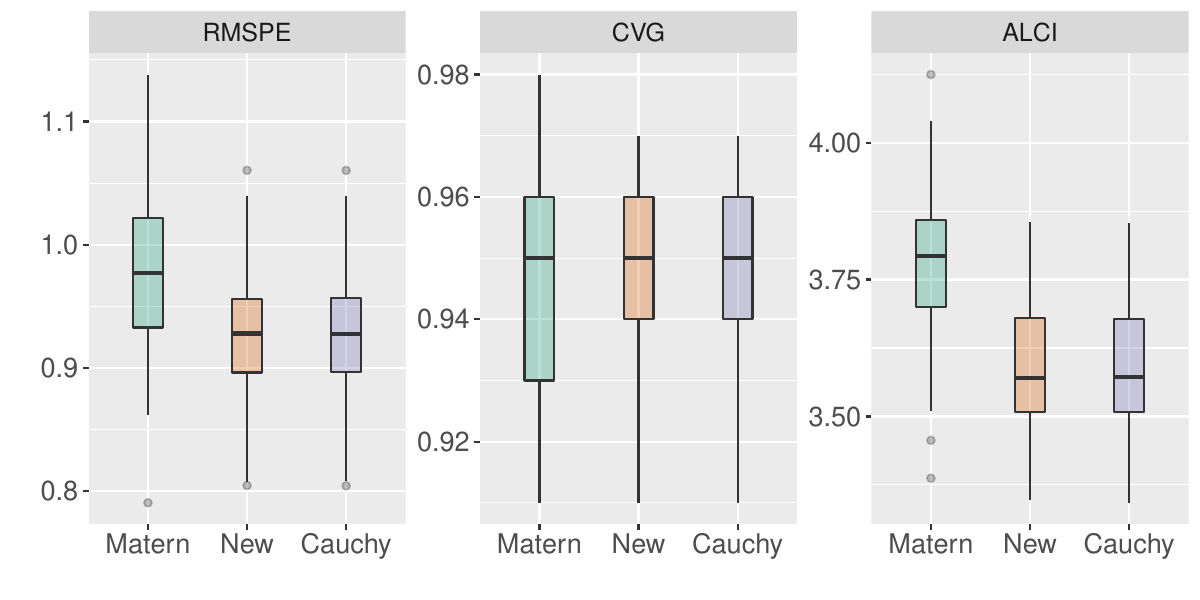}}
\end{subfigure}
\caption*{$\nu=2.5,  ER=200$}

\begin{subfigure}{.35\textwidth}
  \centering
\makebox[\textwidth][c]{ \includegraphics[width=1.0\linewidth, height=0.18\textheight]{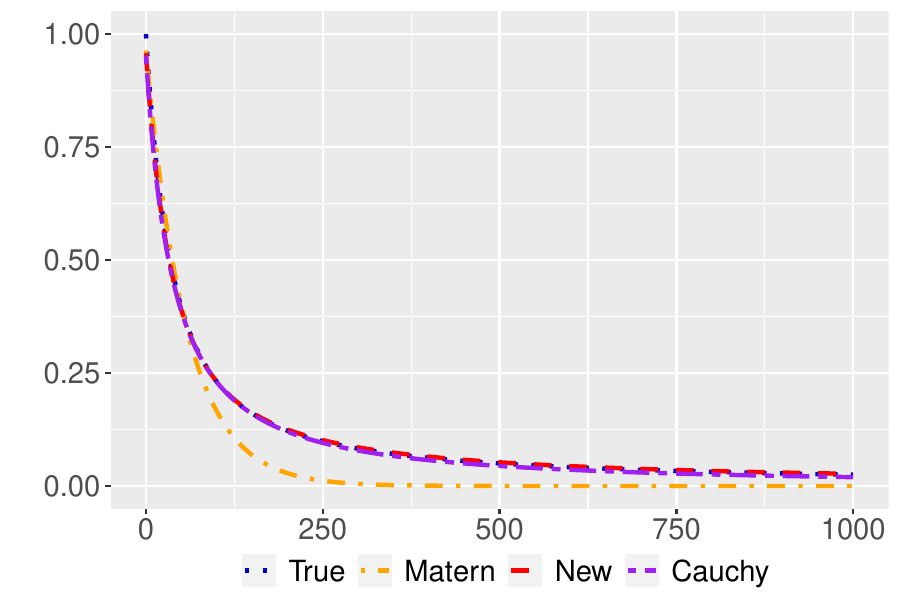}}
\end{subfigure}%
\begin{subfigure}{.65\textwidth}
  \centering
\makebox[\textwidth][c]{ \includegraphics[width=1.0\linewidth, height=0.18\textheight]{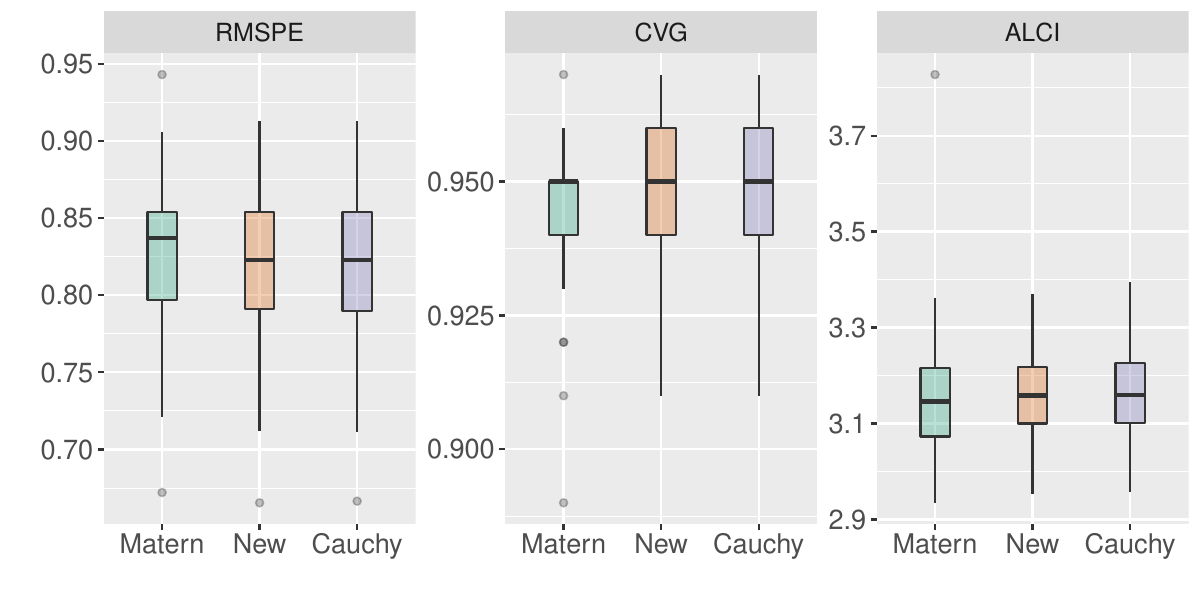}}
\end{subfigure}
\caption*{$\nu=0.5,  ER=500$}

\begin{subfigure}{.35\textwidth}
  \centering
\makebox[\textwidth][c]{ \includegraphics[width=1.0\linewidth, height=0.18\textheight]{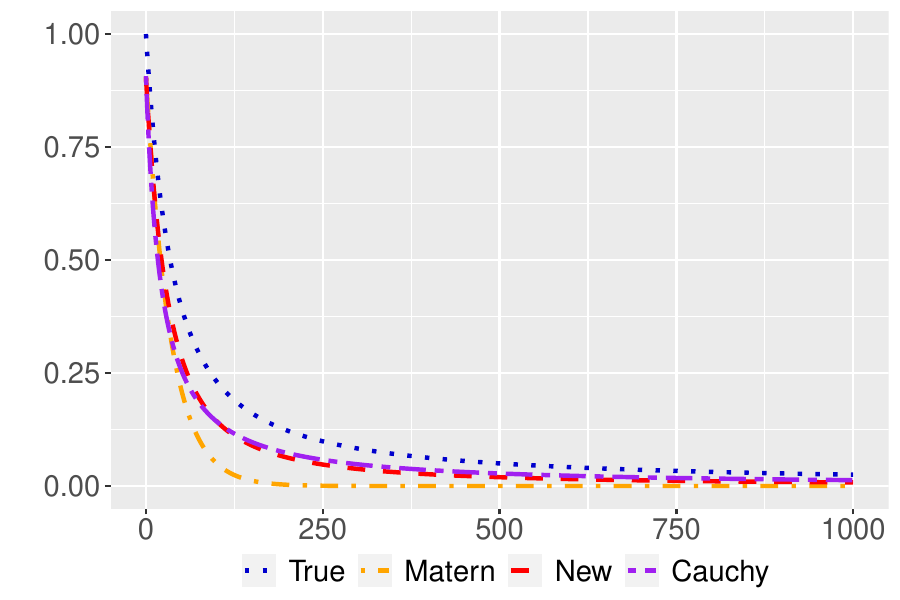}}
\end{subfigure}%
\begin{subfigure}{.65\textwidth}
  \centering
\makebox[\textwidth][c]{ \includegraphics[width=1.0\linewidth, height=0.18\textheight]{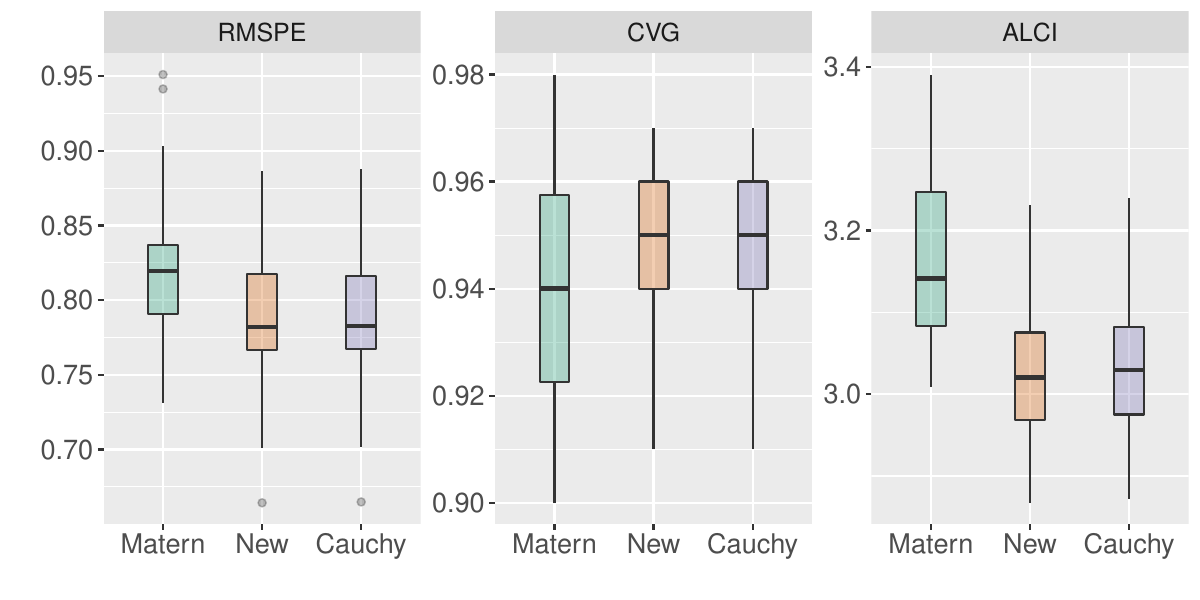}}
\end{subfigure}
\caption*{$\nu=2.5,  ER=500$}

\caption{Case 2: Comparison of predictive performance and estimated covariance structures when the true covariance is the CH class with 500 observations. The predictive performance is evaluated at 10-by-10 regular grids in the square domain. These figures summarize the predictive measures based on RMSPE, CVG and ALCI under 30 simulated realizations.}
\label{fig: simulation setting under case 2, n=500}
\end{figure}

\begin{figure}[htbp] 
\begin{subfigure}{.35\textwidth}
  \centering
\makebox[\textwidth][c]{ \includegraphics[width=1.0\linewidth, height=0.18\textheight]{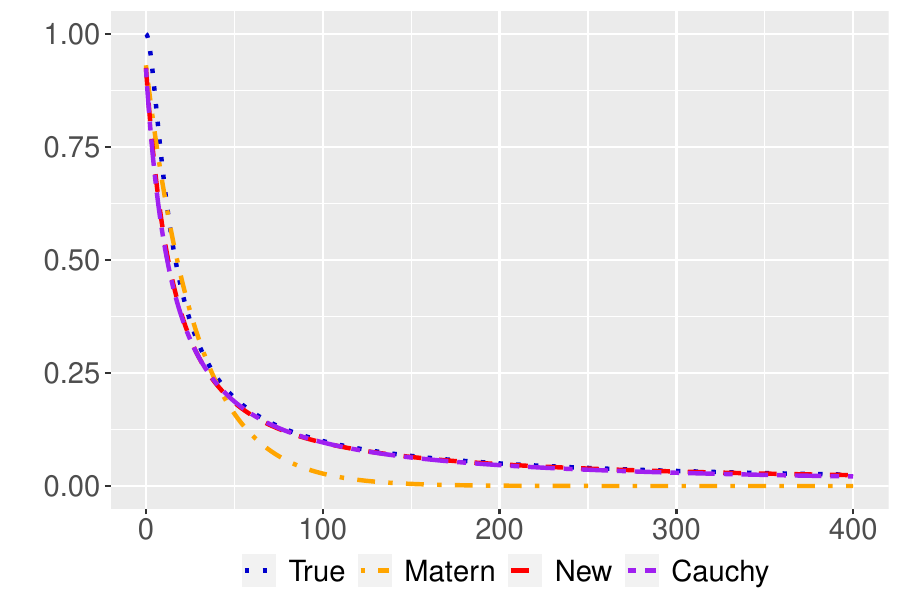}}
\end{subfigure}%
\begin{subfigure}{.65\textwidth}
  \centering
\makebox[\textwidth][c]{ \includegraphics[width=1.0\linewidth, height=0.18\textheight]{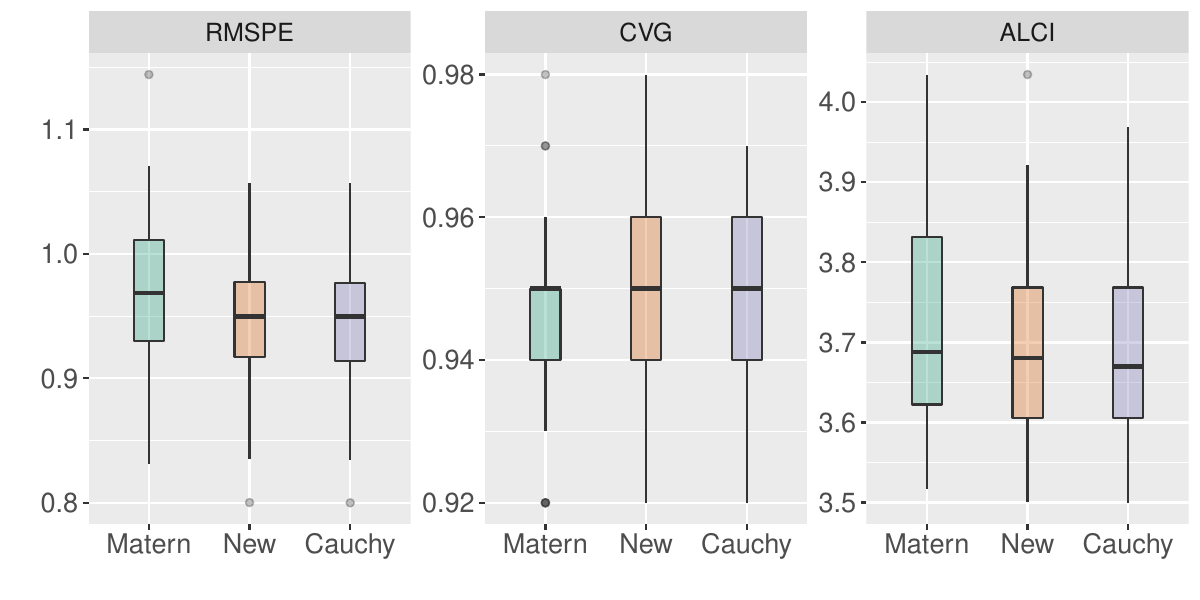}}
\end{subfigure}
\caption*{$\delta=1,  ER=200$}

\begin{subfigure}{.35\textwidth}
  \centering
\makebox[\textwidth][c]{ \includegraphics[width=1.0\linewidth, height=0.18\textheight]{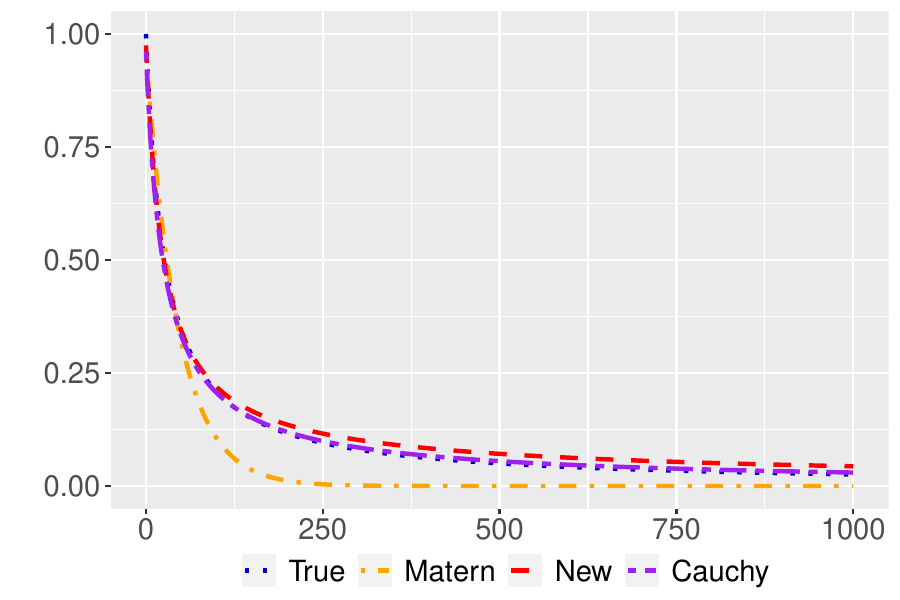}}
\end{subfigure}%
\begin{subfigure}{.65\textwidth}
  \centering
\makebox[\textwidth][c]{ \includegraphics[width=1.0\linewidth, height=0.18\textheight]{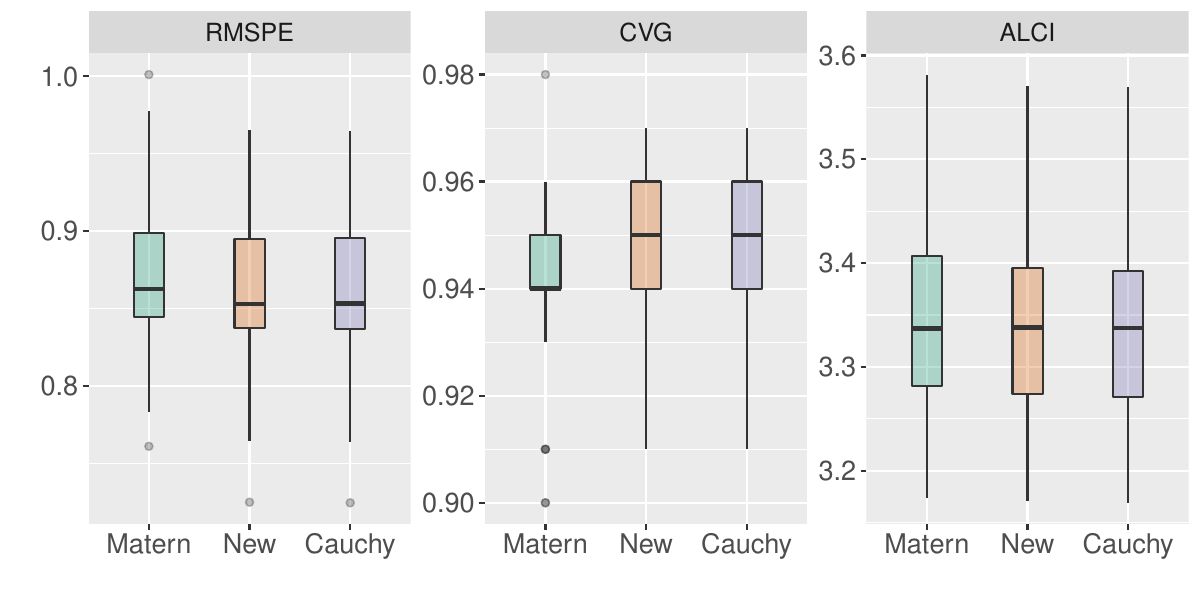}}
\end{subfigure}
\caption*{$\delta=1,  ER=500$}

\caption{Case 3: Comparison of predictive performance and estimated covariance structures when the true covariance is the GC class with 500 observations. The predictive performance is evaluated at 10-by-10 regular grids in the square domain. These figures summarize the predictive measures based on RMSPE, CVG and ALCI under 30 simulated realizations.}
\label{fig: simulation setting under case 3, n=500}
\end{figure}

\begin{figure}[htbp] 
\begin{subfigure}{.35\textwidth}
  \centering
\makebox[\textwidth][c]{ \includegraphics[width=1.0\linewidth, height=0.18\textheight]{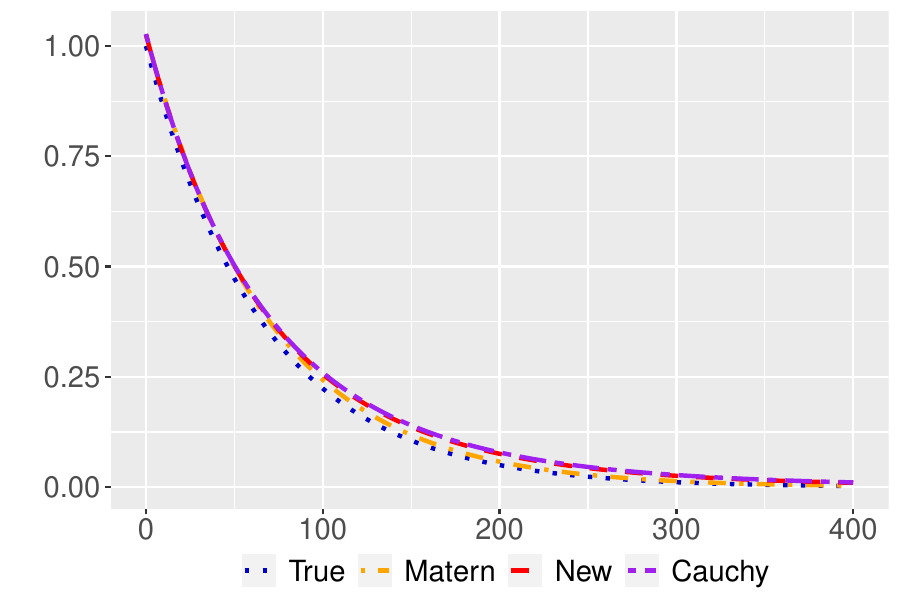}}
\end{subfigure}%
\begin{subfigure}{.65\textwidth}
  \centering
\makebox[\textwidth][c]{ \includegraphics[width=1.0\linewidth, height=0.18\textheight]{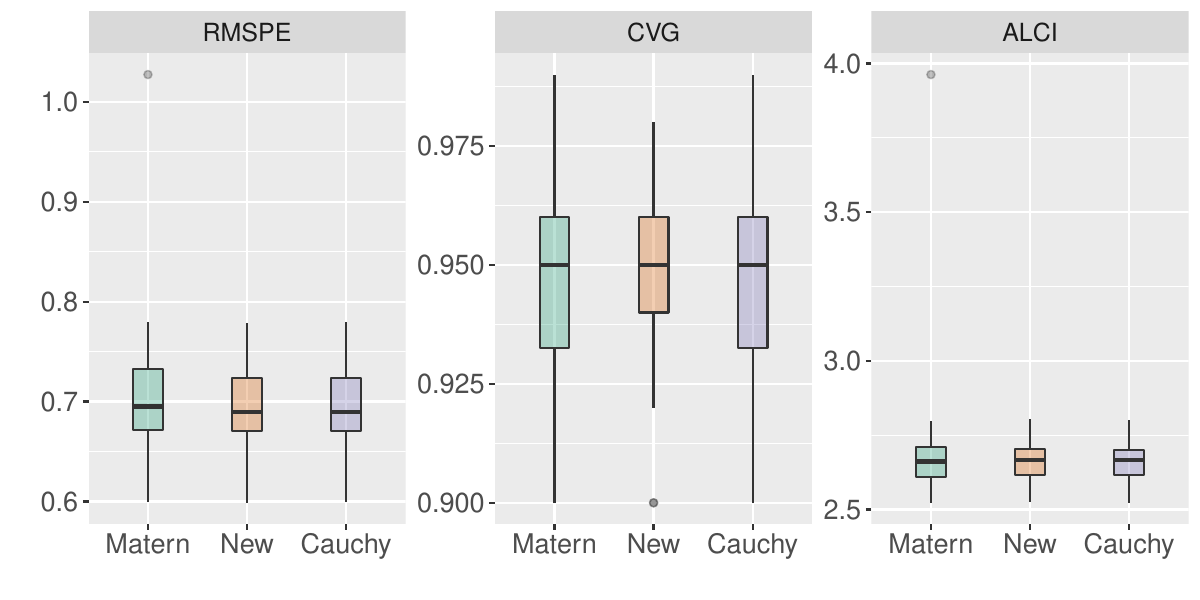}}
\end{subfigure}
\caption*{$\nu=0.5,  ER=200$}

\begin{subfigure}{.35\textwidth}
  \centering
\makebox[\textwidth][c]{ \includegraphics[width=1.0\linewidth, height=0.18\textheight]{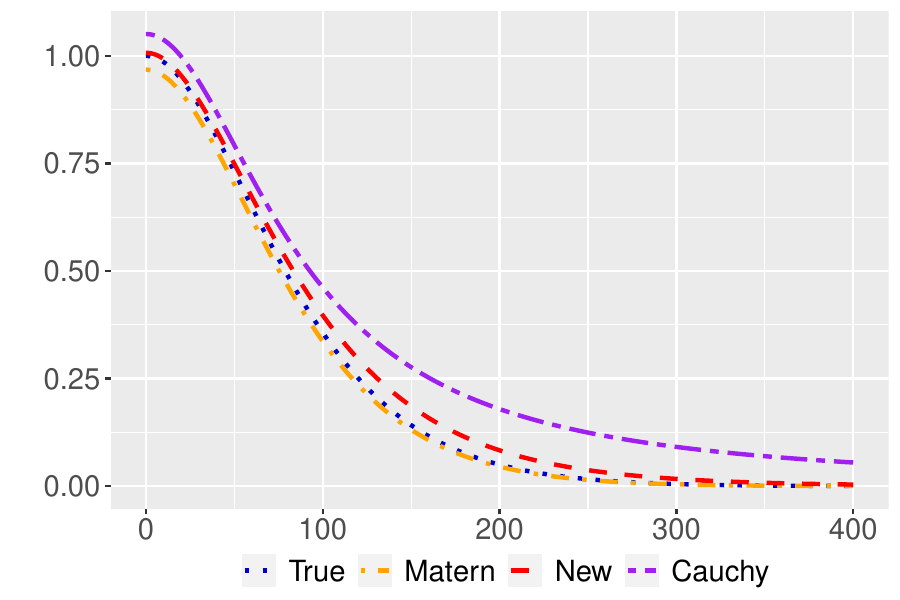}}
\end{subfigure}%
\begin{subfigure}{.65\textwidth}
  \centering
\makebox[\textwidth][c]{ \includegraphics[width=1.0\linewidth, height=0.18\textheight]{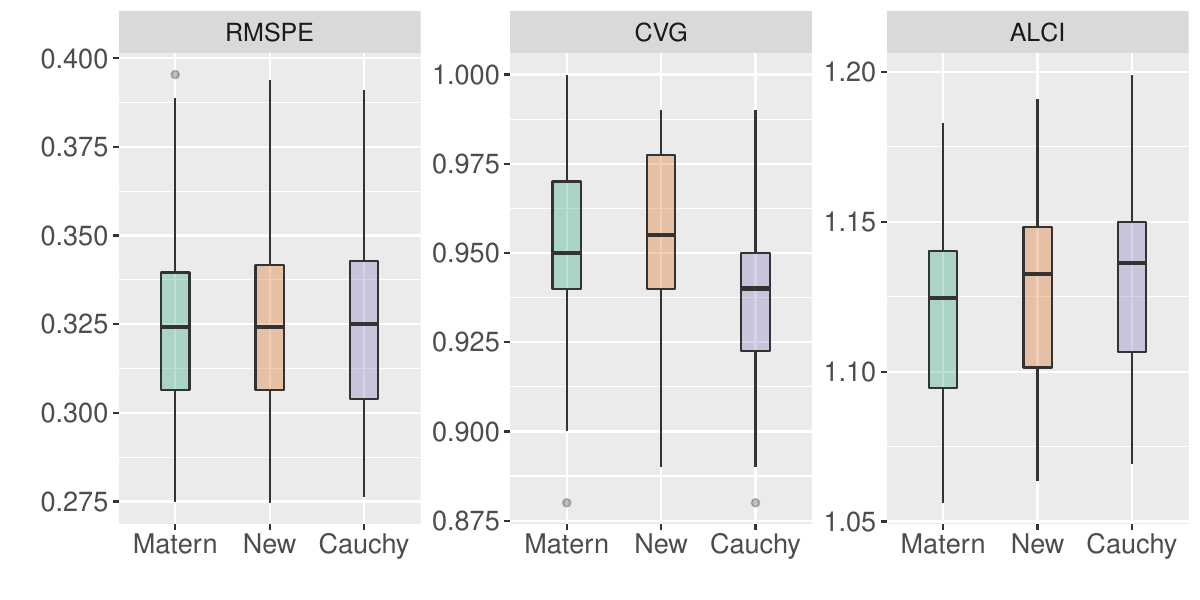}}
\end{subfigure}
\caption*{$\nu=2.5,  ER=200$}

\begin{subfigure}{.35\textwidth}
  \centering
\makebox[\textwidth][c]{ \includegraphics[width=1.0\linewidth, height=0.18\textheight]{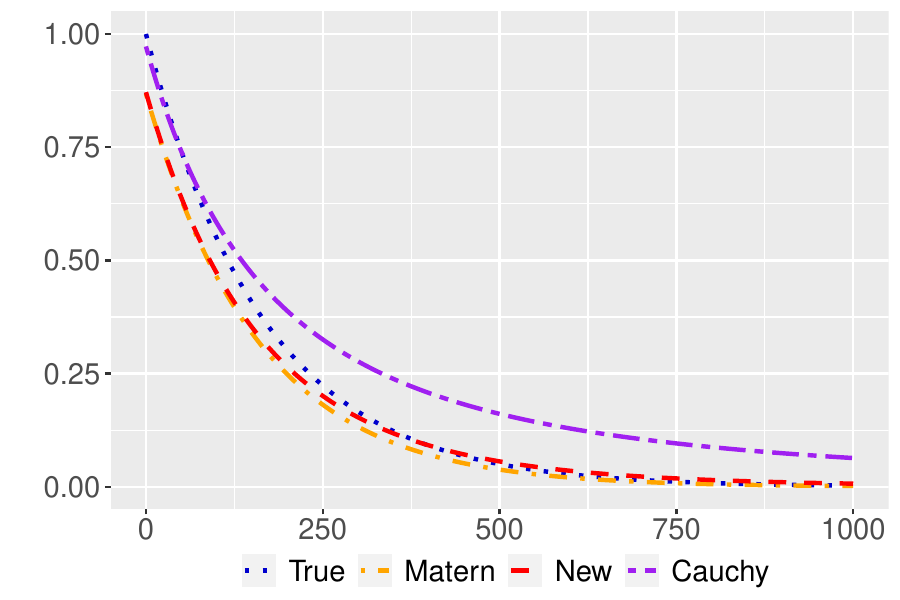}}
\end{subfigure}%
\begin{subfigure}{.65\textwidth}
  \centering
\makebox[\textwidth][c]{ \includegraphics[width=1.0\linewidth, height=0.18\textheight]{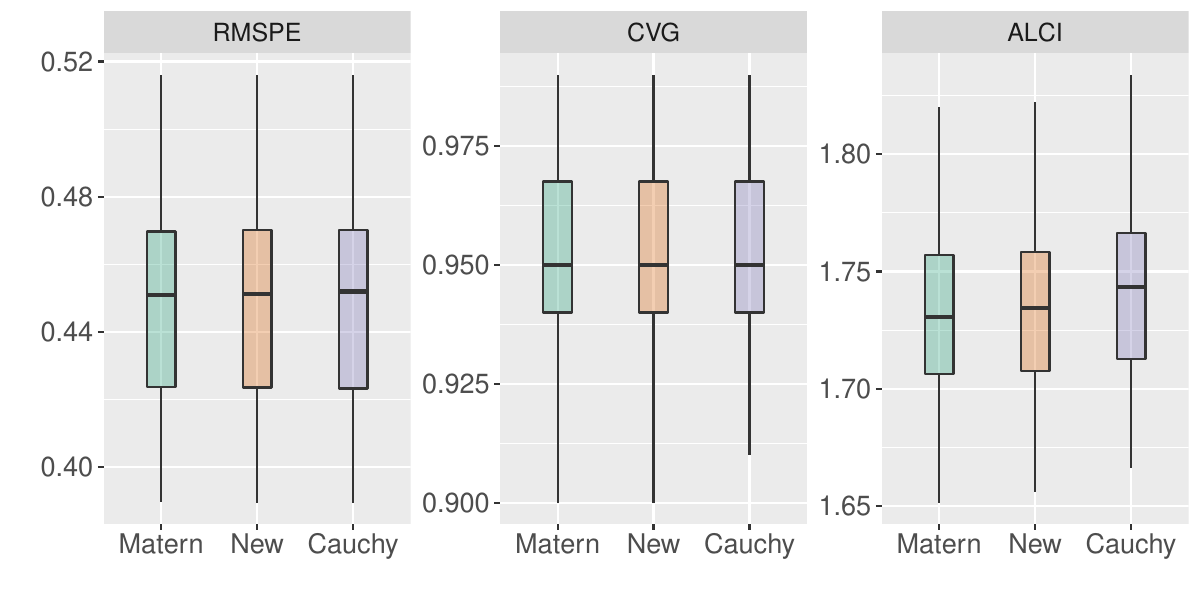}}
\end{subfigure}
\caption*{$\nu=0.5,  ER=500$}

\begin{subfigure}{.35\textwidth}
  \centering
\makebox[\textwidth][c]{ \includegraphics[width=1.0\linewidth, height=0.18\textheight]{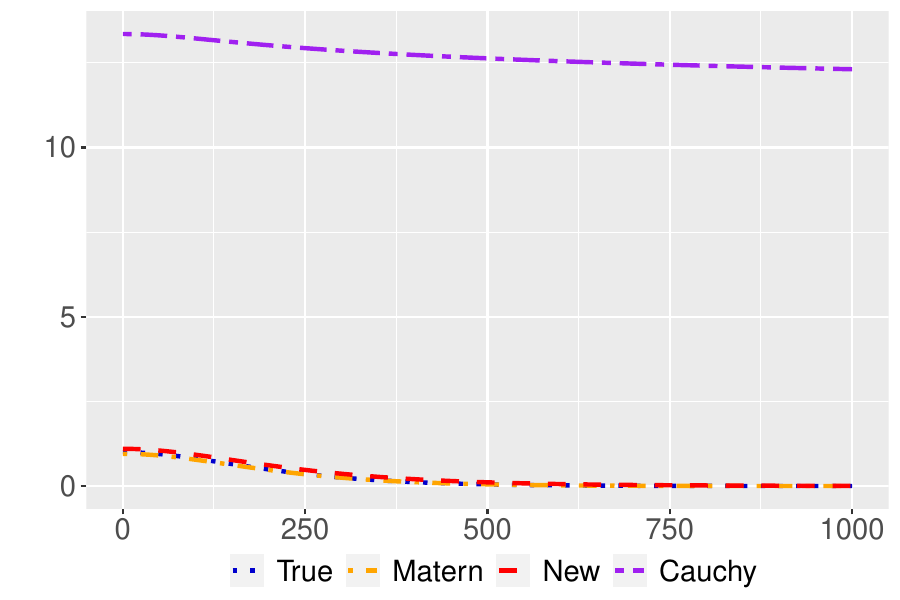}}
\end{subfigure}%
\begin{subfigure}{.65\textwidth}
  \centering
\makebox[\textwidth][c]{ \includegraphics[width=1.0\linewidth, height=0.18\textheight]{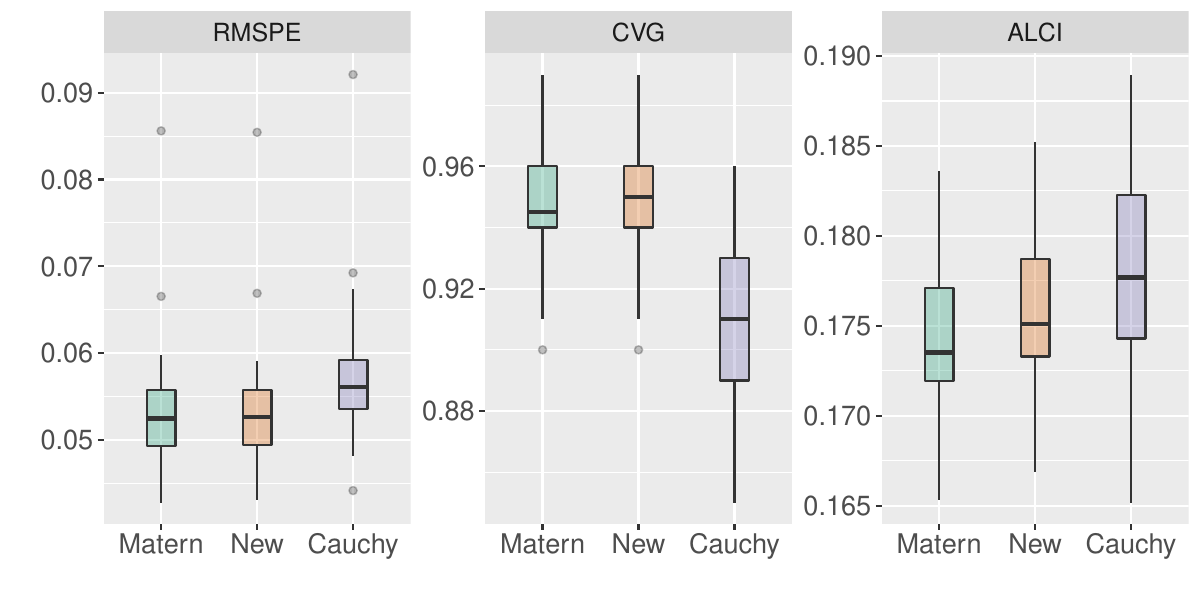}}
\end{subfigure}
\caption*{$\nu=2.5,  ER=500$}

\caption{Case 1: Comparison of predictive performance and estimated covariance structures when the true covariance is the Mat\'ern class with 1000 observations. The predictive performance is evaluated at 10-by-10 regular grids in the square domain. These figures summarize the predictive measures based on RMSPE, CVG and ALCI under 30 simulated realizations.}
\label{fig: simulation setting under case 1, n=1000}
\end{figure}

\begin{figure}[htbp] 
\begin{subfigure}{.35\textwidth}
  \centering
\makebox[\textwidth][c]{ \includegraphics[width=1.0\linewidth, height=0.18\textheight]{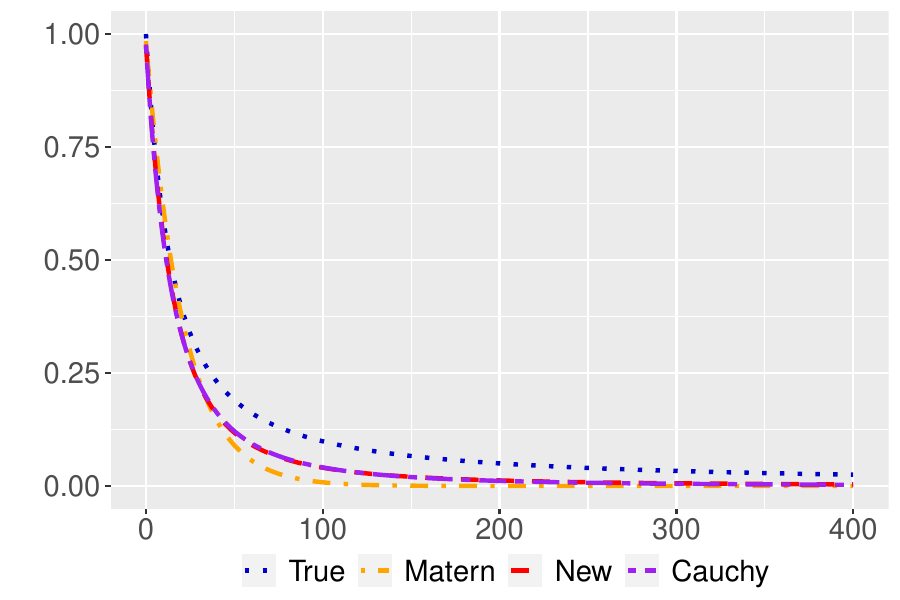}}
\end{subfigure}%
\begin{subfigure}{.65\textwidth}
  \centering
\makebox[\textwidth][c]{ \includegraphics[width=1.0\linewidth, height=0.18\textheight]{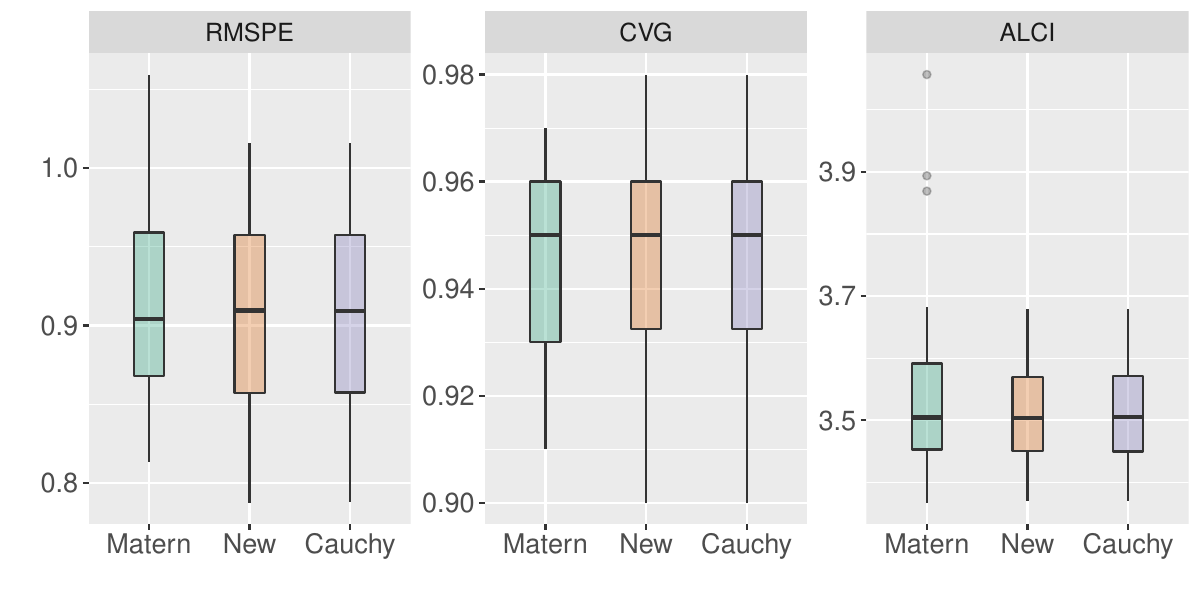}}
\end{subfigure}
\caption*{$\nu=0.5,  ER=200$}

\begin{subfigure}{.35\textwidth}
  \centering
\makebox[\textwidth][c]{ \includegraphics[width=1.0\linewidth, height=0.18\textheight]{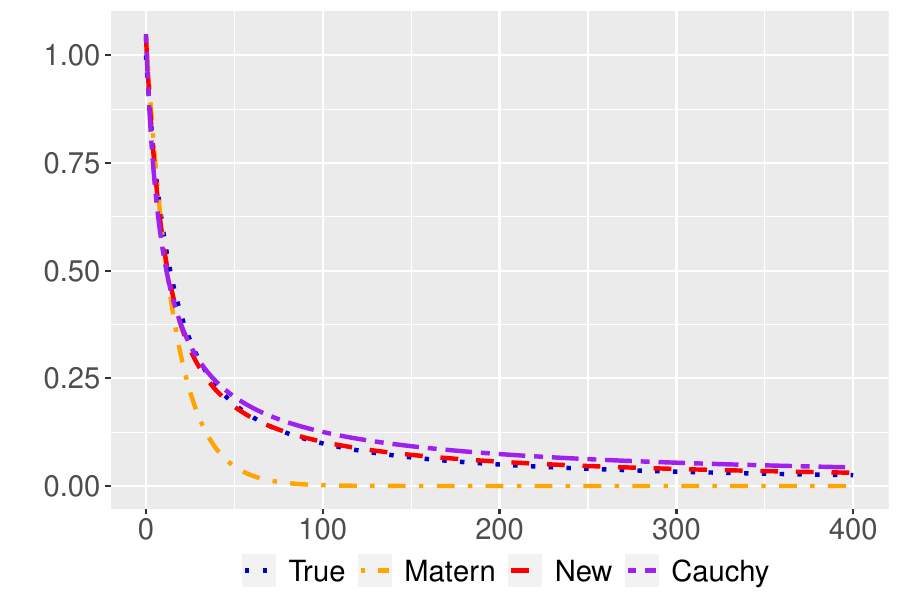}}
\end{subfigure}%
\begin{subfigure}{.65\textwidth}
  \centering
\makebox[\textwidth][c]{ \includegraphics[width=1.0\linewidth, height=0.18\textheight]{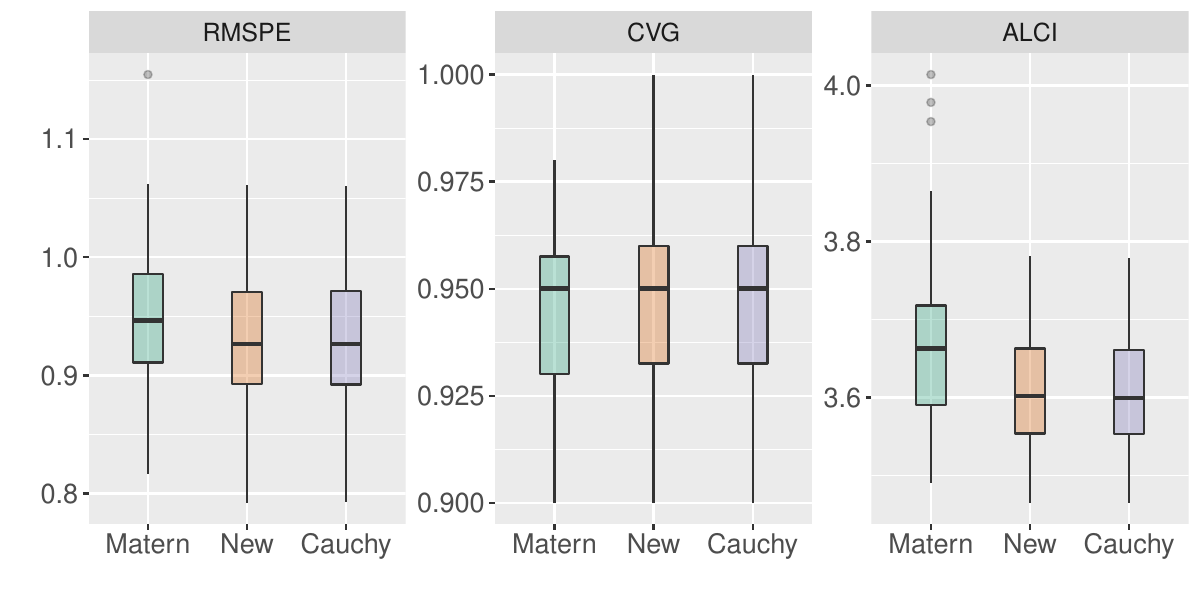}}
\end{subfigure}
\caption*{$\nu=2.5,  ER=200$}

\begin{subfigure}{.35\textwidth}
  \centering
\makebox[\textwidth][c]{ \includegraphics[width=1.0\linewidth, height=0.18\textheight]{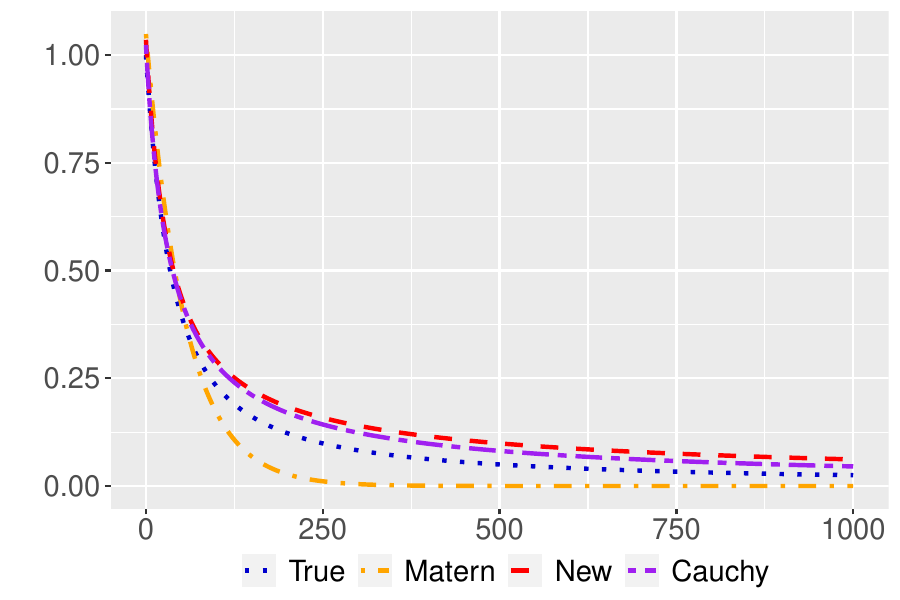}}
\end{subfigure}%
\begin{subfigure}{.65\textwidth}
  \centering
\makebox[\textwidth][c]{ \includegraphics[width=1.0\linewidth, height=0.18\textheight]{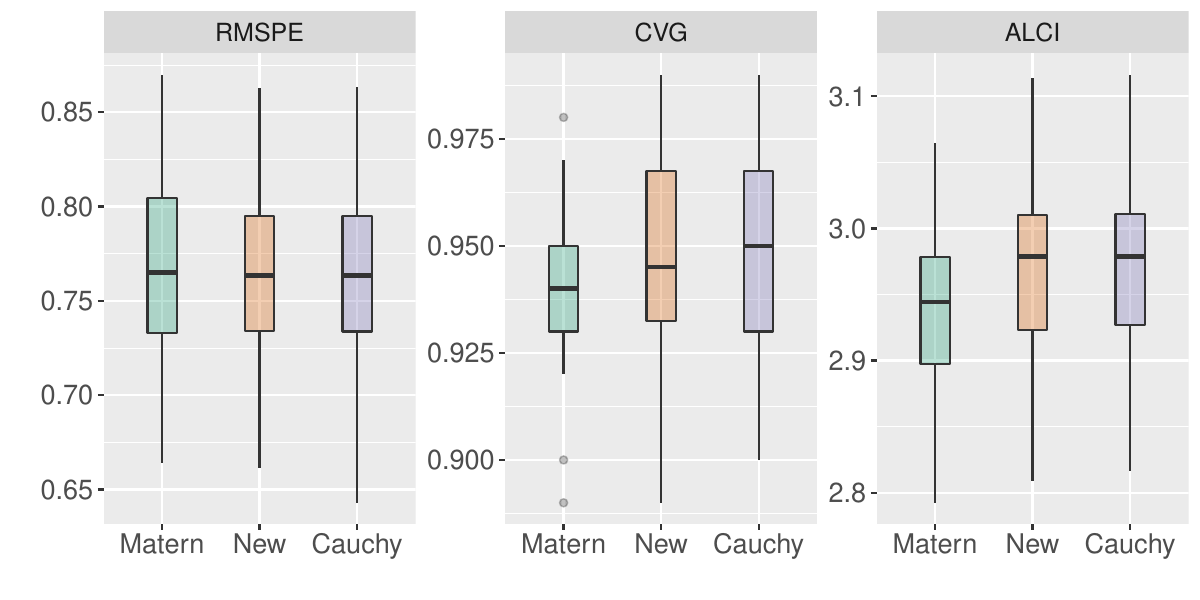}}
\end{subfigure}
\caption*{$\nu=0.5,  ER=500$}

\begin{subfigure}{.35\textwidth}
  \centering
\makebox[\textwidth][c]{ \includegraphics[width=1.0\linewidth, height=0.18\textheight]{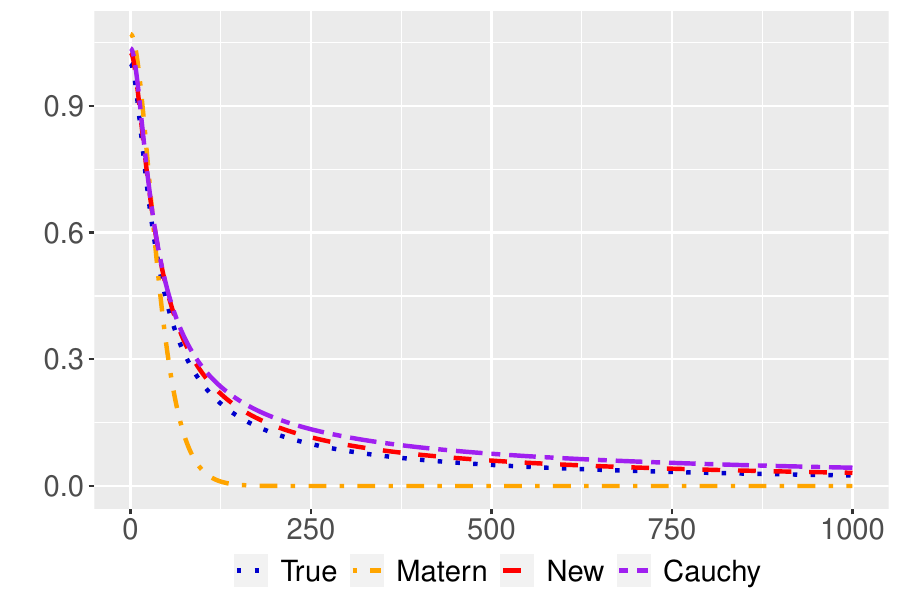}}
\end{subfigure}%
\begin{subfigure}{.65\textwidth}
  \centering
\makebox[\textwidth][c]{ \includegraphics[width=1.0\linewidth, height=0.18\textheight]{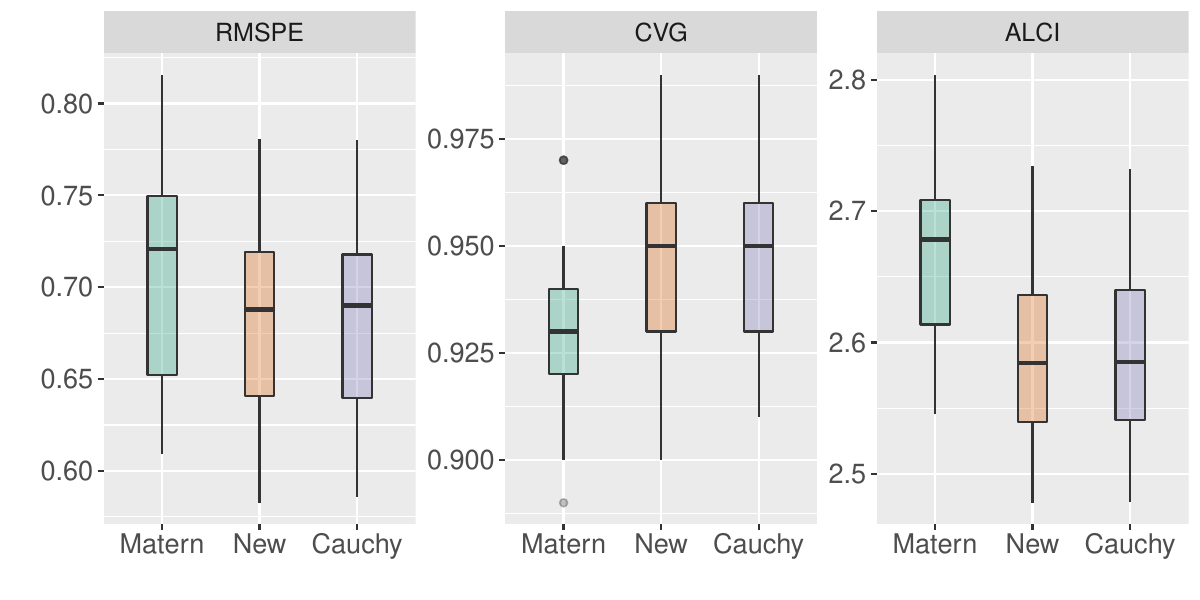}}
\end{subfigure}
\caption*{$\nu=2.5,  ER=500$}

\caption{Case 2: Comparison of predictive performance and estimated covariance structures when the true covariance is the CH class with 1000 observations. The predictive performance is evaluated at 10-by-10 regular grids in the square domain. These figures summarize the predictive measures based on RMSPE, CVG and ALCI under 30 simulated realizations.}
\label{fig: simulation setting under case 2, n=1000}
\end{figure}

\begin{figure}[htbp] 
\begin{subfigure}{.35\textwidth}
  \centering
\makebox[\textwidth][c]{ \includegraphics[width=1.0\linewidth, height=0.18\textheight]{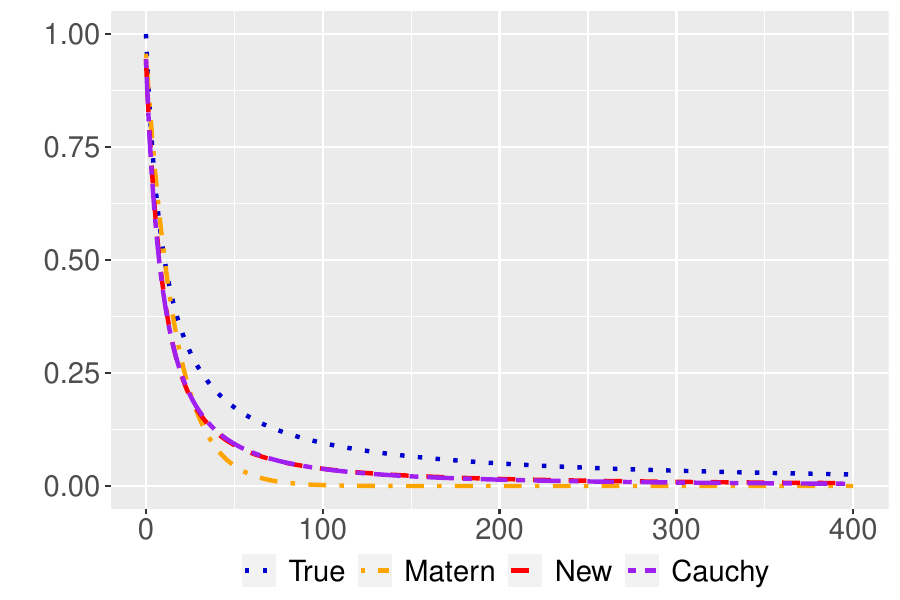}}
\end{subfigure}%
\begin{subfigure}{.65\textwidth}
  \centering
\makebox[\textwidth][c]{ \includegraphics[width=1.0\linewidth, height=0.18\textheight]{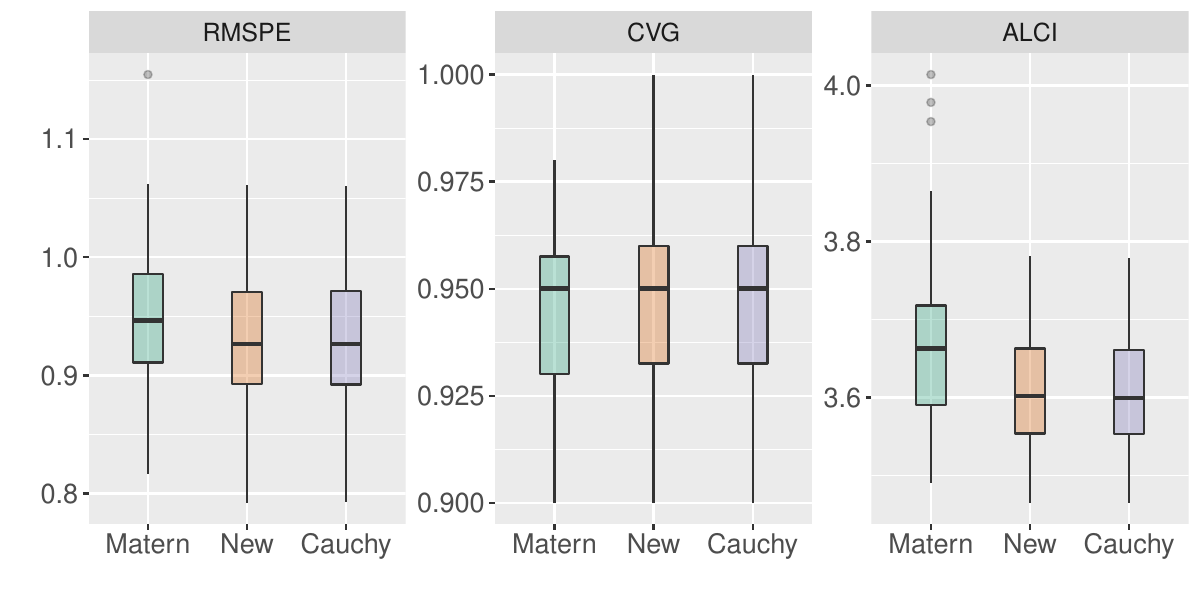}}
\end{subfigure}
\caption*{$\delta=1,  ER=200$}

\begin{subfigure}{.35\textwidth}
  \centering
\makebox[\textwidth][c]{ \includegraphics[width=1.0\linewidth, height=0.18\textheight]{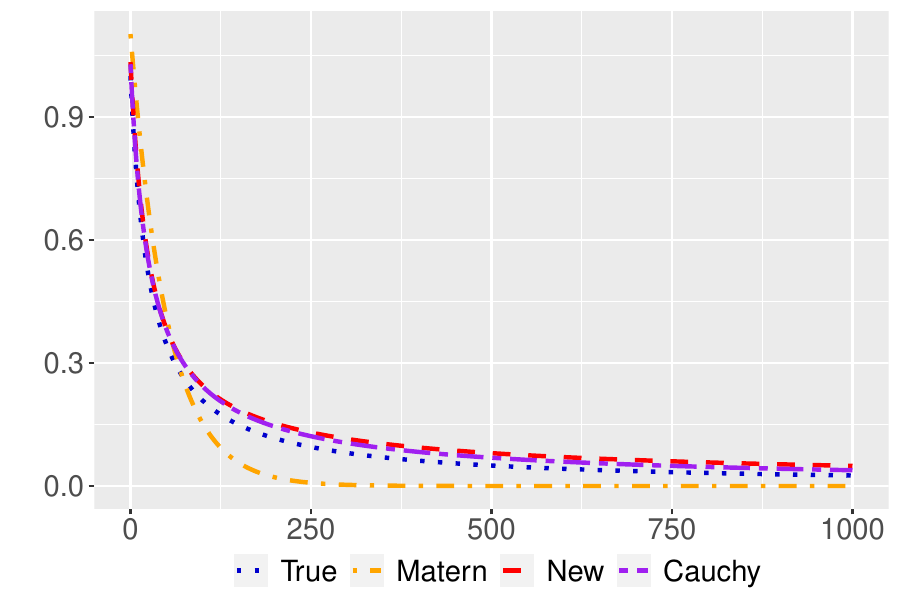}}
\end{subfigure}%
\begin{subfigure}{.65\textwidth}
  \centering
\makebox[\textwidth][c]{ \includegraphics[width=1.0\linewidth, height=0.18\textheight]{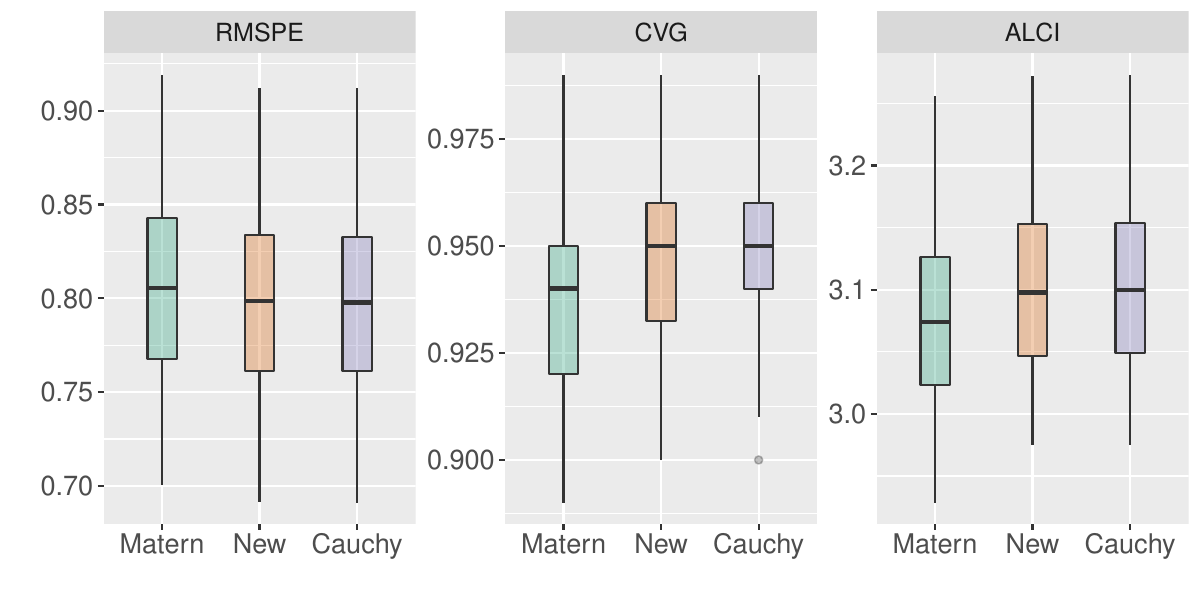}}
\end{subfigure}
\caption*{$\delta=1,  ER=500$}

\caption{Case 3: Comparison of predictive performance and estimated covariance structures when the true covariance is the GC class with 1000 observations. The predictive performance is evaluated at 10-by-10 regular grids in the square domain. These figures summarize the predictive measures based on RMSPE, CVG and ALCI under 30 simulated realizations.}
\label{fig: simulation setting under case 3, n=1000}
\end{figure}

\clearpage

\subsection{Simulation with a Tensor Product of Covariance Functions} \label{app: UQ simulation}
In this section, we study the predictive performance of the CH class with a product form, i.e., $r(\|\bfs-\bfu\|) = \prod_{i=1}^dR(|s_i - u_i|; \bftheta_i)$, where $R(\cdot; \bftheta_i)$ is an isotropic covariance function with parameter $\bftheta_i$. This product form of covariance functions allows different properties  along different coordinate directions (or input space) and has been widely used in uncertainty quantification and machine learning. 

We simulate the true processes under the Mat\'ern class and the CH class with effective range fixed at 200 and 500. For the smoothness parameter, we consider $\nu=0.5, 2.5$. The tail decay parameter in the CH class is chosen to be 0.5.  As each dimension has a different range parameter or scale parameter, we choose these parameters in each dimension such that their correlation will be $0.5^{1/2}$ at distance 200 and 500. This will guarantee the overall effective range will be 200 and 500, respectively.  For each simulation setting, the true process is simulated at $n=100, 500, 1000$ locations. The GC class has a smoothness parameter that is specified as in Section~\ref{sec:numerical}. The prediction locations are the same as those in Section~\ref{sec:numerical}.

In the first case where the true process has a product of Mat\'ern covariance functions, the prediction results under the Mat\'ern class, the CH class and the GC class are shown in panels from (a) to (f) of Figure~\ref{fig: UQ simulation}. As expected, the Mat\'ern class and the CH class yield indistinguishable predictive performance in terms of RMSPE, CVG, and ALCI. However, the GC class has much worse performance than the other two covariance classes. In the second case where the true process has a product of CH functions, the prediction results under these three covariance classes are shown in panels from (g) to (l) of Figure~\ref{fig: UQ simulation}. As expected, the CH class yields much better prediction results than the Mat\'ern class, since the Mat\'ern class has an exponentially decaying tail that is not able to capture the tail behavior in the CH class. It is worth noting that the GC class yields much worse predictive performance than the other two covariance classes. This is quite different from the situation when the true process does not have a product covariance form.

\newpage
\begin{figure}[htbp] 
\begin{subfigure}{.333\textwidth}
  \centering
\makebox[\textwidth][c]{ \includegraphics[width=1.0\linewidth, height=0.18\textheight]{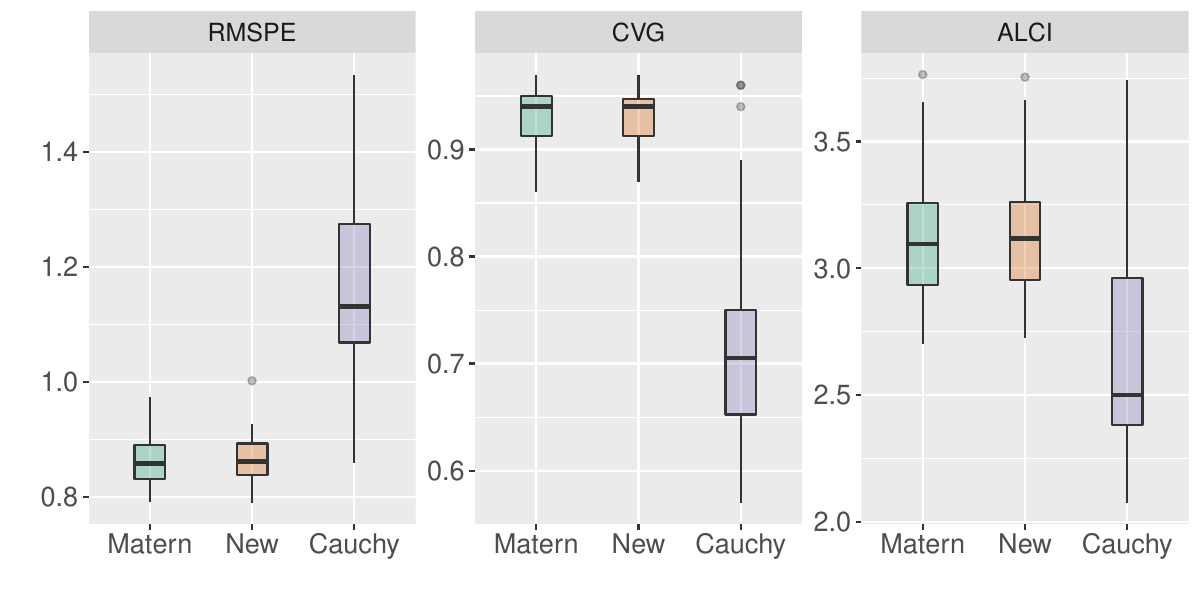}}
\caption{Mat\'ern: $\nu=0.5, n=100$}
\end{subfigure}%
\begin{subfigure}{.333\textwidth}
  \centering
\makebox[\textwidth][c]{ \includegraphics[width=1.0\linewidth, height=0.18\textheight]{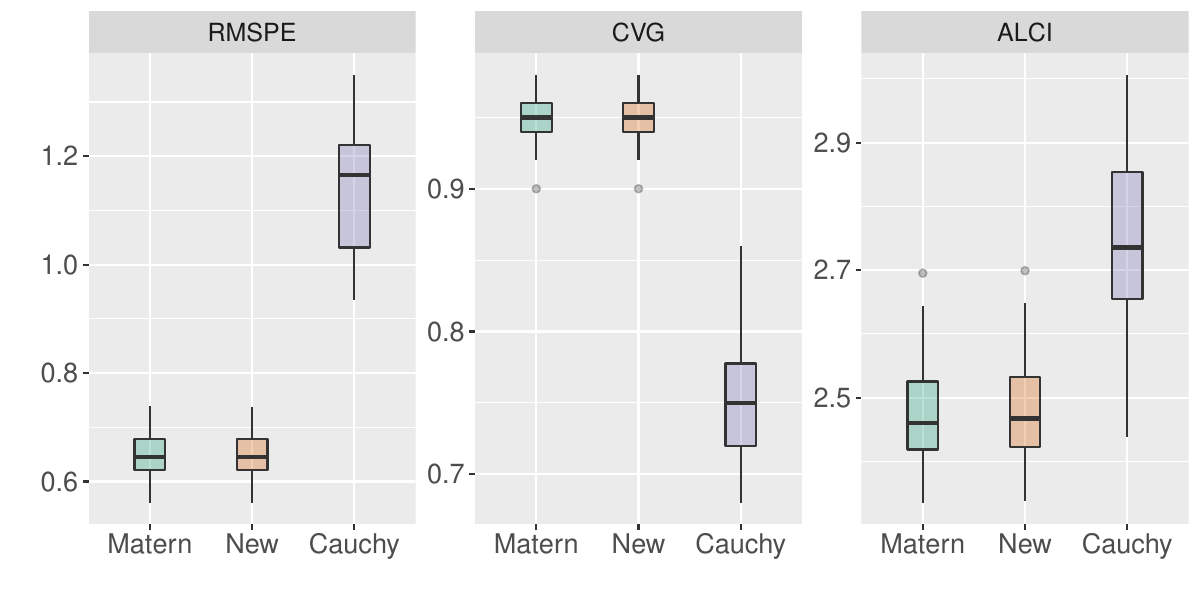}}
\caption{Mat\'ern: $\nu=0.5, n=500$}
\end{subfigure}%
\begin{subfigure}{.333\textwidth}
  \centering
\makebox[\textwidth][c]{ \includegraphics[width=1.0\linewidth, height=0.18\textheight]{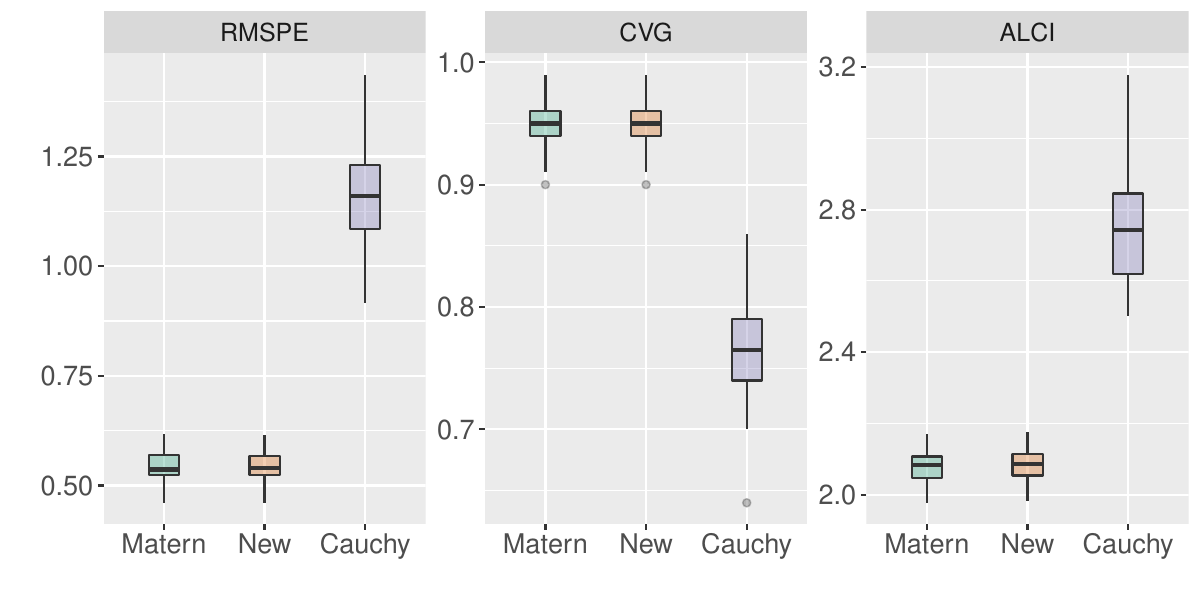}}
\caption{Mat\'ern: $\nu=0.5,  n=1000$}
\end{subfigure}

\begin{subfigure}{.333\textwidth}
  \centering
\makebox[\textwidth][c]{ \includegraphics[width=1.0\linewidth, height=0.18\textheight]{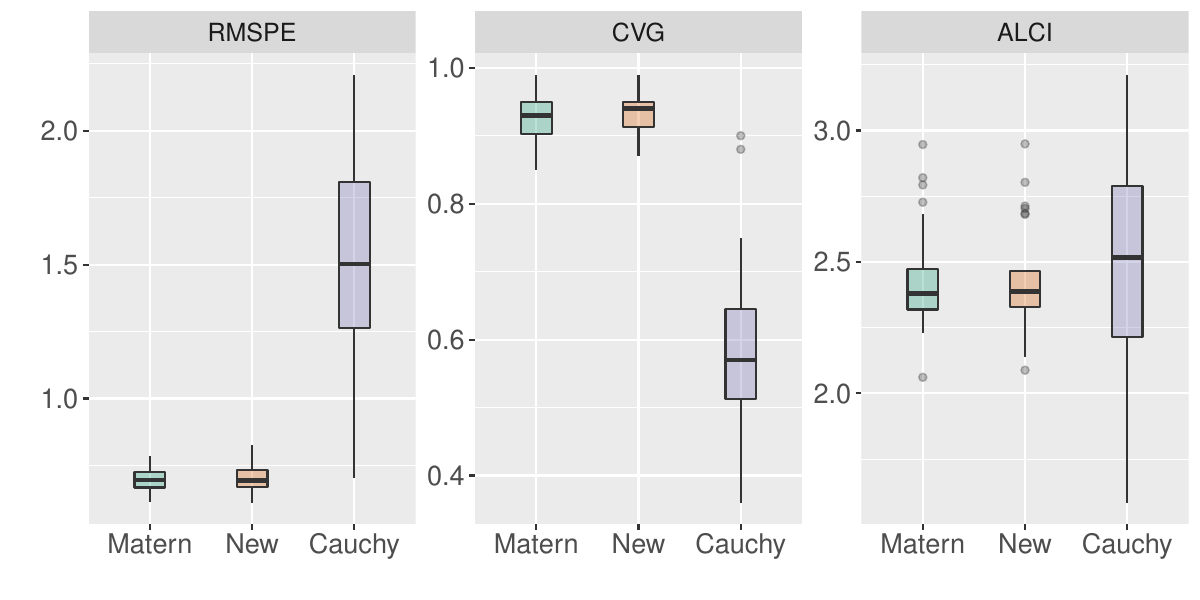}}
\caption{Mat\'ern: $\nu=2.5, n=100$}
\end{subfigure}%
\begin{subfigure}{.333\textwidth}
  \centering
\makebox[\textwidth][c]{ \includegraphics[width=1.0\linewidth, height=0.18\textheight]{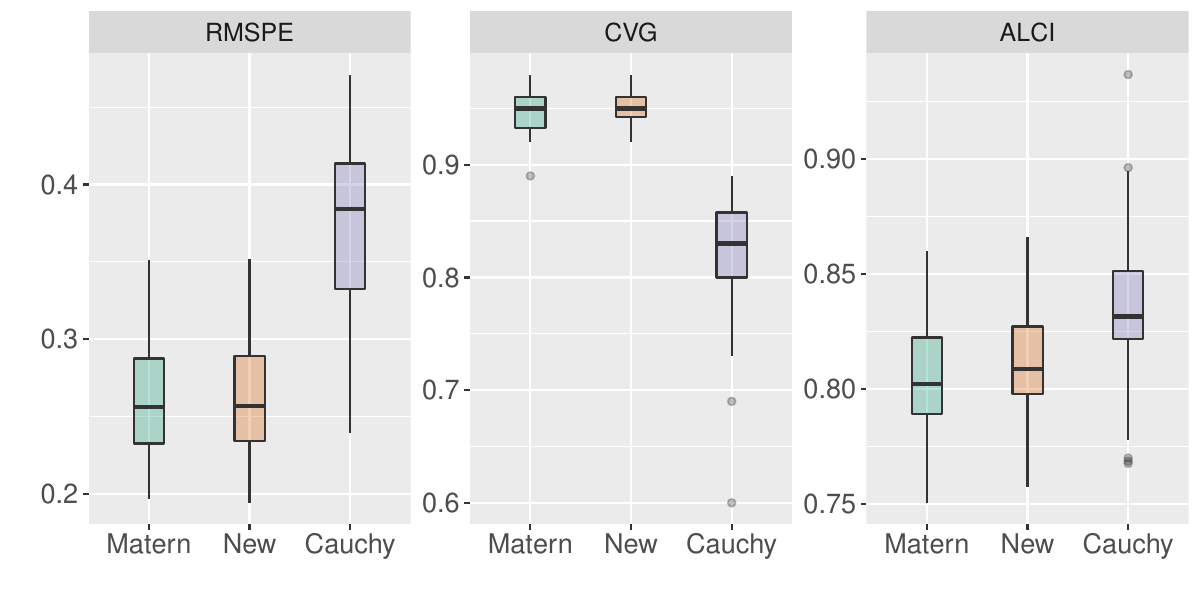}}
\caption{Mat\'ern: $\nu=2.5, n=500$}
\end{subfigure}%
\begin{subfigure}{.333\textwidth}
  \centering
\makebox[\textwidth][c]{ \includegraphics[width=1.0\linewidth, height=0.18\textheight]{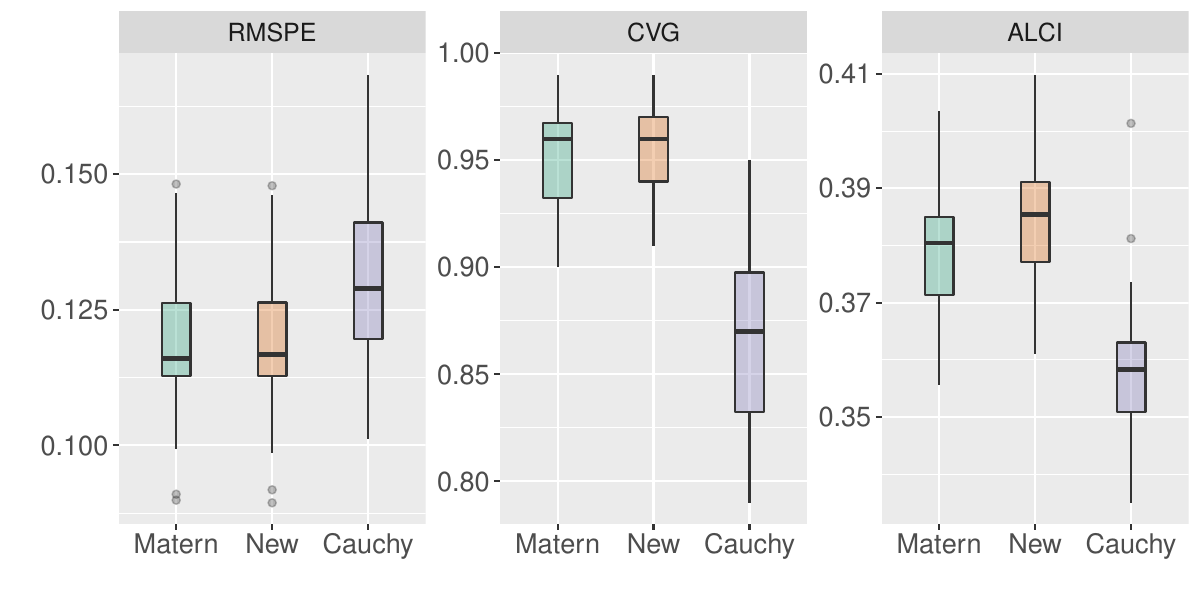}}
\caption{Mat\'ern: $\nu=2.5,  n=1000$}
\end{subfigure}

\begin{subfigure}{.333\textwidth}
  \centering
\makebox[\textwidth][c]{ \includegraphics[width=1.0\linewidth, height=0.18\textheight]{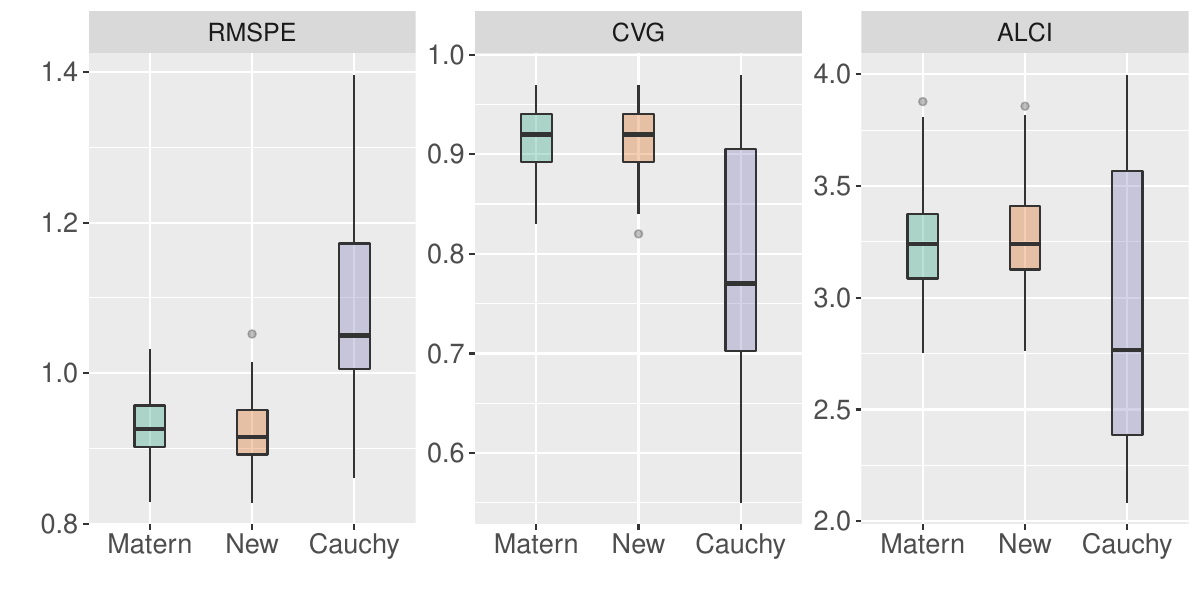}}
\caption{New: $\nu=0.5,  n=100$}
\end{subfigure}%
\begin{subfigure}{.33\textwidth}
  \centering
\makebox[\textwidth][c]{ \includegraphics[width=1.0\linewidth, height=0.18\textheight]{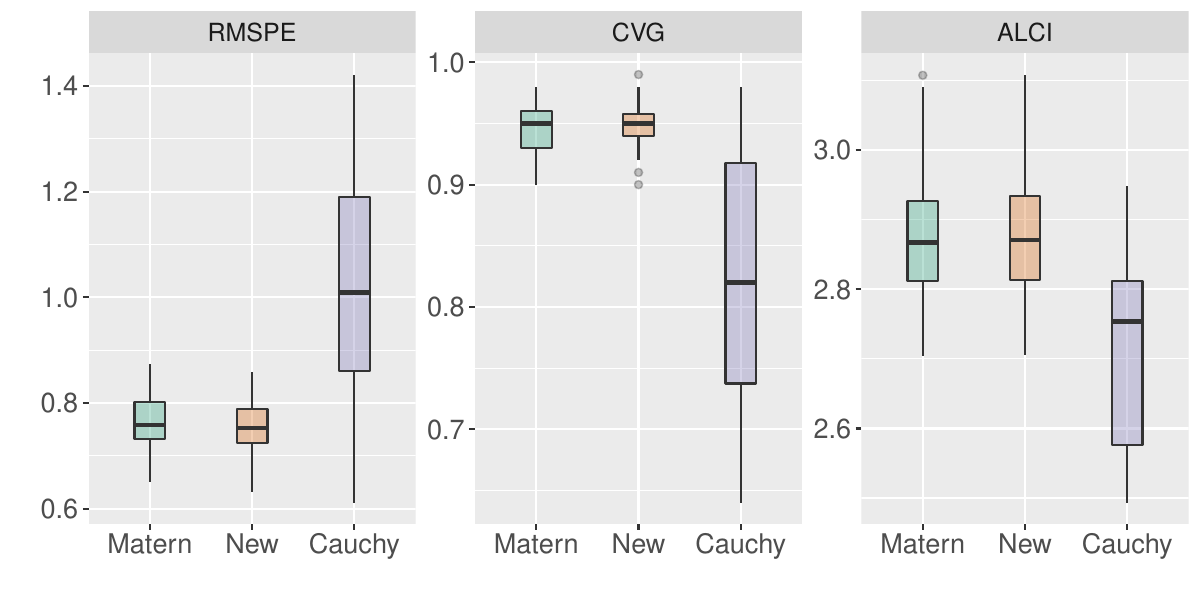}}
\caption{New: $\nu=0.5,  n=500$}
\end{subfigure}%
\begin{subfigure}{.333\textwidth}
  \centering
\makebox[\textwidth][c]{ \includegraphics[width=1.0\linewidth, height=0.18\textheight]{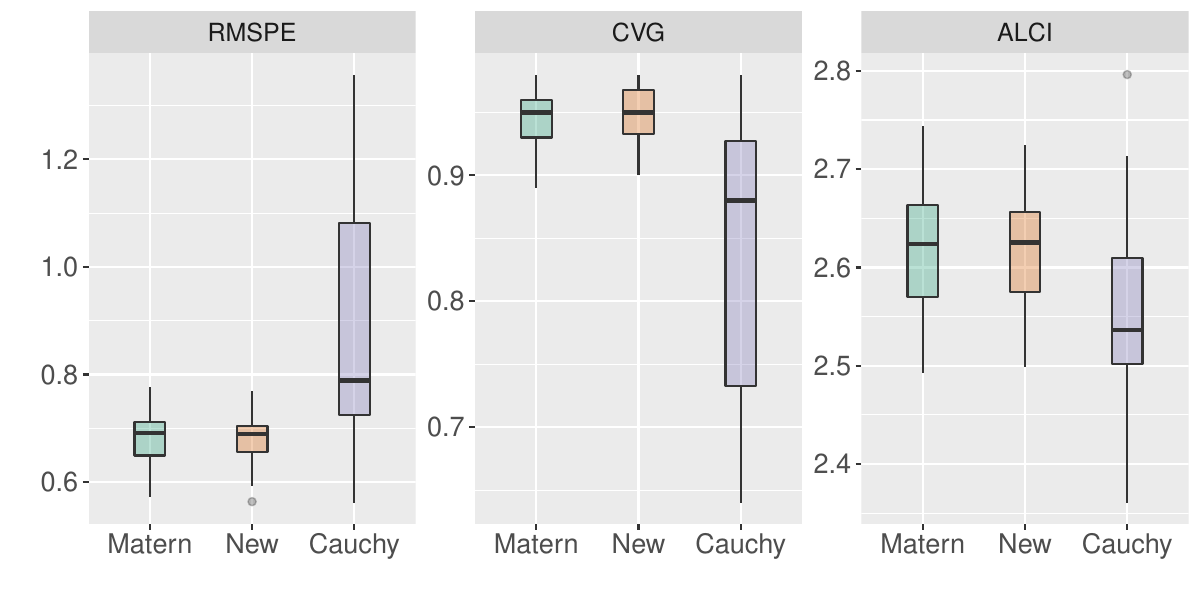}}
\caption{New: $\nu=0.5,  n=1000$}
\end{subfigure}

\begin{subfigure}{.333\textwidth}
  \centering
\makebox[\textwidth][c]{ \includegraphics[width=1.0\linewidth, height=0.18\textheight]{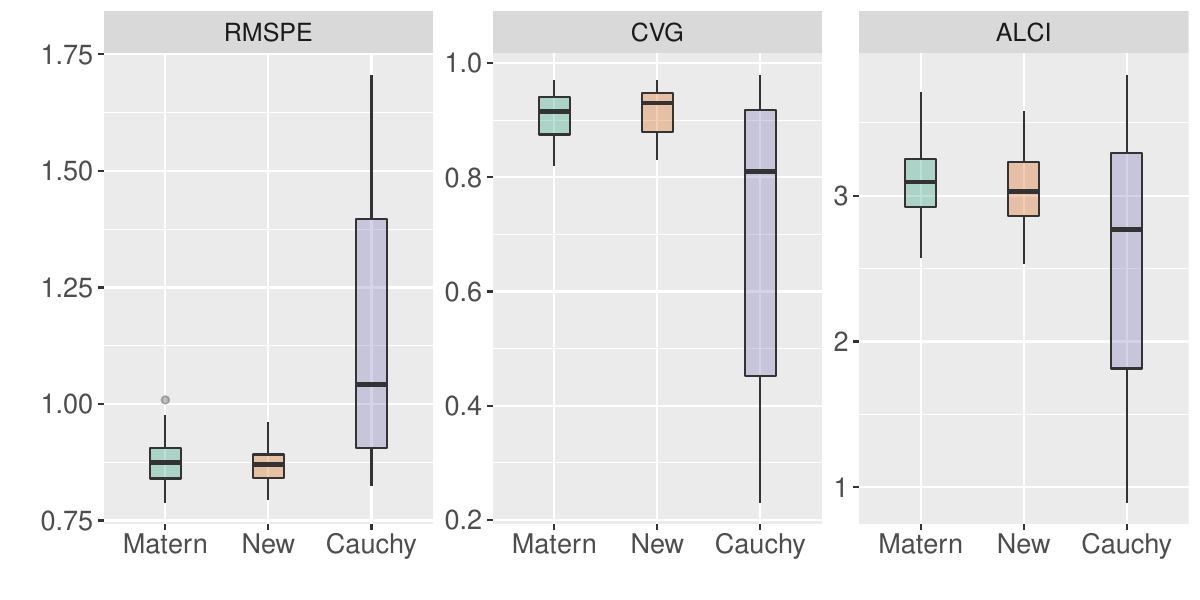}}
\caption{New: $\nu=2.5,  n=100$}
\end{subfigure}%
\begin{subfigure}{.33\textwidth}
  \centering
\makebox[\textwidth][c]{ \includegraphics[width=1.0\linewidth, height=0.18\textheight]{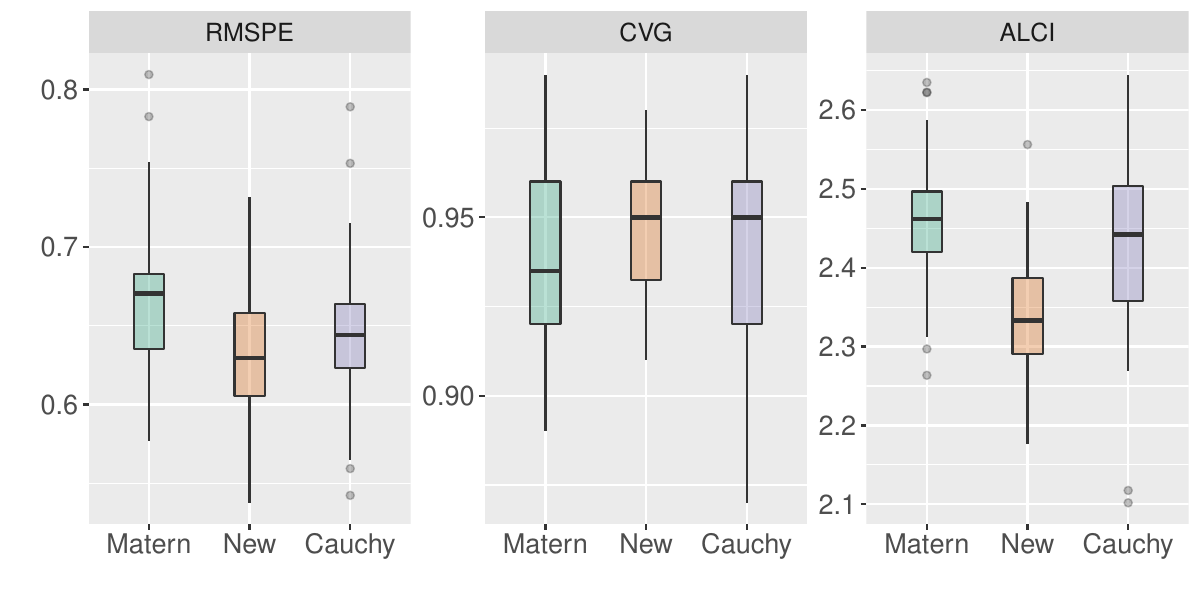}}
\caption{New: $\nu=2.5,  n=500$}
\end{subfigure}%
\begin{subfigure}{.333\textwidth}
  \centering
\makebox[\textwidth][c]{ \includegraphics[width=1.0\linewidth, height=0.18\textheight]{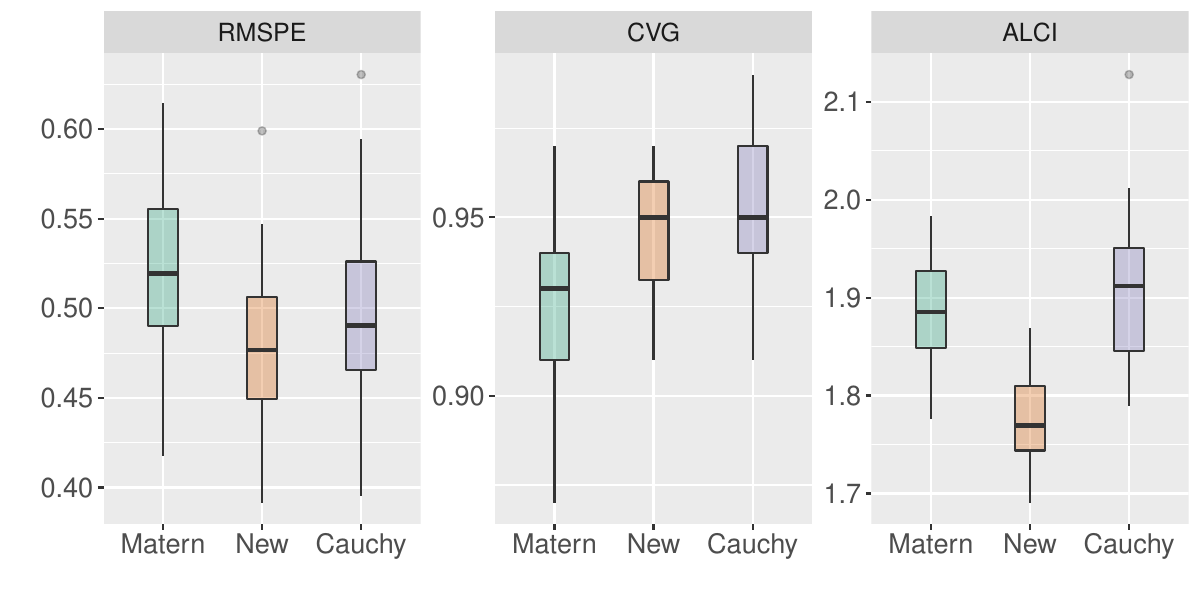}}
\caption{New: $\nu=2.5,  n=1000$}
\end{subfigure}
\caption{Predictive performance over 10-by-10 regular grids under three covariance classes when the true process has a product form of covariance structures. The predictive performance is studied under different smoothness parameters, effective ranges and number of observation locations.}
\label{fig: UQ simulation}
\end{figure}

\clearpage

\newpage
\section{Additional Numerical Results} \label{app: numerical results}
This section contains parameter estimation results and figures referenced in Section~\ref{sec:real}. Figure~\ref{fig: isotropy} shows the directional semivariograms for the OCO2 data. Table~\ref{table: CV for OCO2 data parameter estimation} shows the estimated parameters under the Mat\'ern covariance model and the CH covariance in the cross-validation study of Section~\ref{sec:real}. Figure~\ref{fig: comparison for OCO2 data} compares kriging predictors and associated standard errors under the CH class and Mat\'ern class in the cross-validation study of Section~\ref{sec:real}.


\begin{figure}[htbp]
\makebox[\textwidth][c]{\includegraphics[width=1.0\textwidth, height=0.5\textheight]{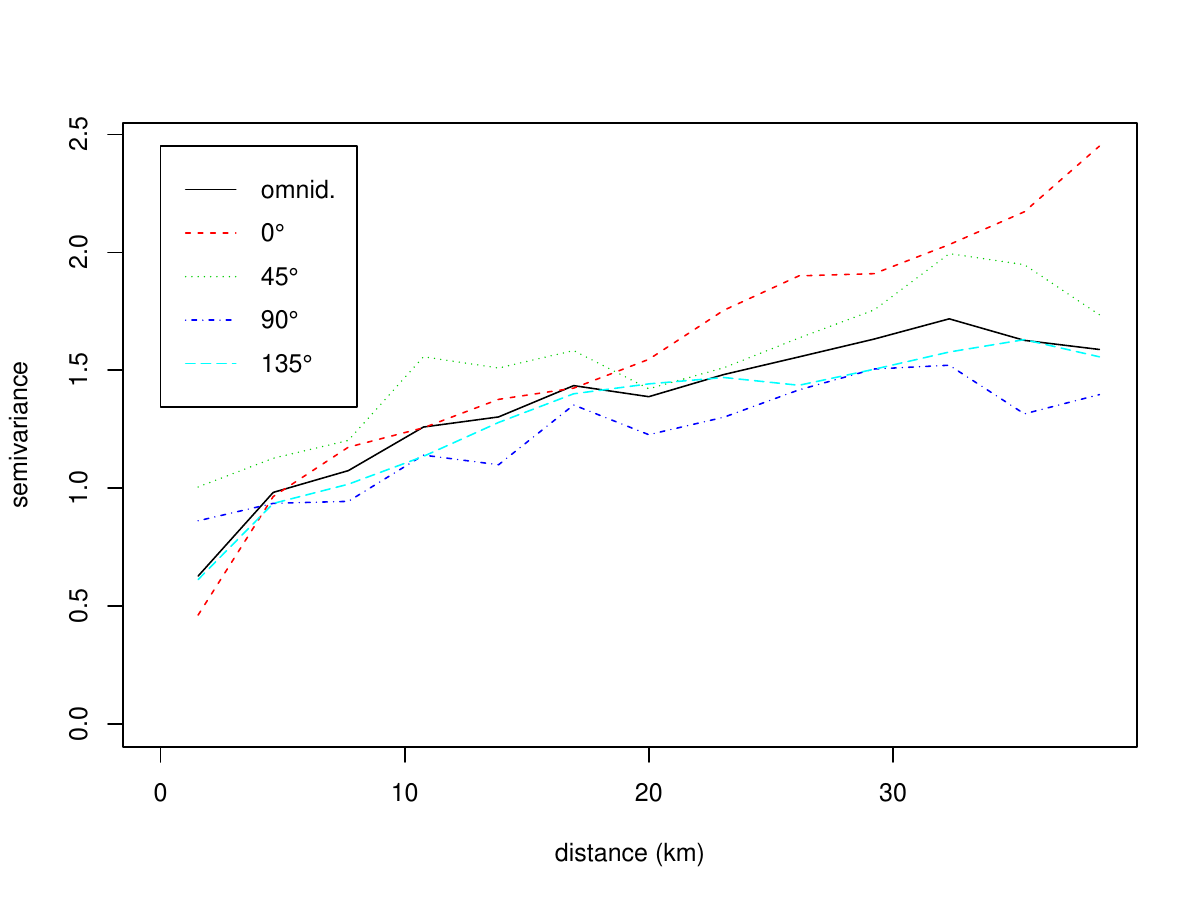}}
\caption{Graphical assessments of isotropy in the OCO2 data. The directional semivariograms do not appear to exhibit differences, indicating that the assumption of an isotropic covariance function is likely to be true.}
\label{fig: isotropy}
\end{figure}


\begin{table}[htbp]
\centering
   \caption{Cross-validation results based on the Mat\'ern covariance model and the CH covariance model. }
  {\resizebox{1.0\textwidth}{!}{%
  \setlength{\tabcolsep}{3.8em}
   \begin{tabular}{ l c c c  c c} 
   \toprule \noalign{\vskip 1.5pt} 
		& \multicolumn{2}{c}{Mat\'ern class}   &\multicolumn{2}{c}{CH class}  \\  \noalign{\vskip 1.5pt}  
	 \noalign{\vskip 1.5pt} 
 & $\nu = 0.5$ &   $\nu = 1.5$ & $\nu = 0.5$ & $\nu=1.5$ \\ \noalign{\vskip 1.5pt}  \noalign{\vskip 3.5pt}
$b$             &411.1  &411.1   &411.0   & 411.0  \\ \noalign{\vskip 1.5pt}  \noalign{\vskip 2.5pt}
$\sigma^2$ & 1.679  &1.439     &1.750   &1.585  \\ \noalign{\vskip 1.5pt}  \noalign{\vskip 2.5pt}
$\phi$ & 160.5        &104.1    & ---   & ---  \\ \noalign{\vskip 1.5pt}  \noalign{\vskip 2.5pt}
$\alpha$ & ---           & ---          &0.381   &0.353  \\ \noalign{\vskip 1.5pt}  \noalign{\vskip 2.5pt}
$\beta$   & ---           & ---         &80.17   & 58.65 \\ 
\noalign{\vskip 1.5pt} \bottomrule
   \end{tabular}%
   }}
   \label{table: CV for OCO2 data parameter estimation}
\end{table}

\begin{figure}[htbp] 

\begin{subfigure}{.5\textwidth}
  \centering
\makebox[\textwidth][c]{ \includegraphics[width=1.0\linewidth, height=0.20\textheight]{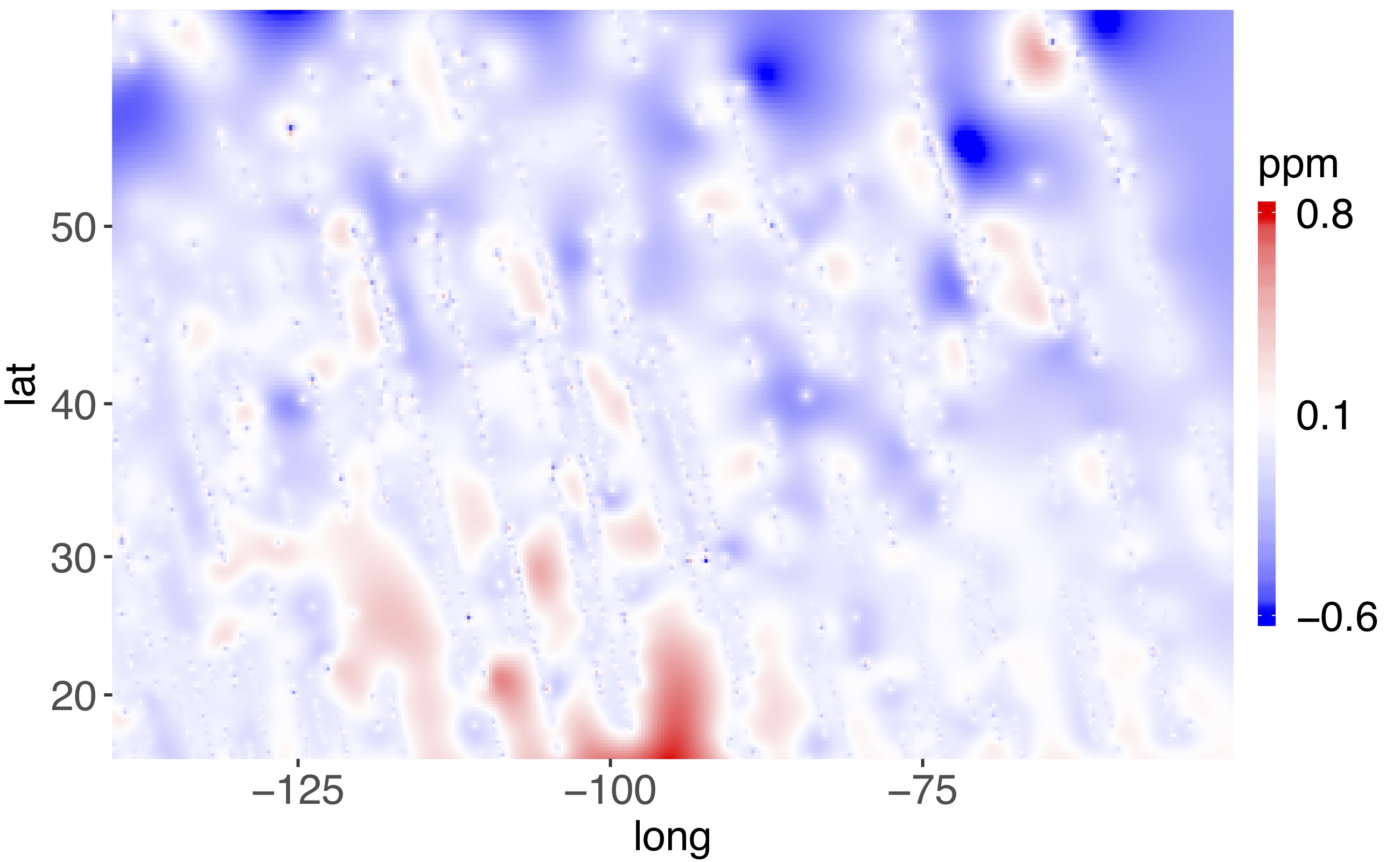}}
\caption{Difference of kriging predictors}
\end{subfigure}%
\begin{subfigure}{.5\textwidth}
\makebox[\textwidth][c]{ \includegraphics[width=1\linewidth, height=0.20\textheight]{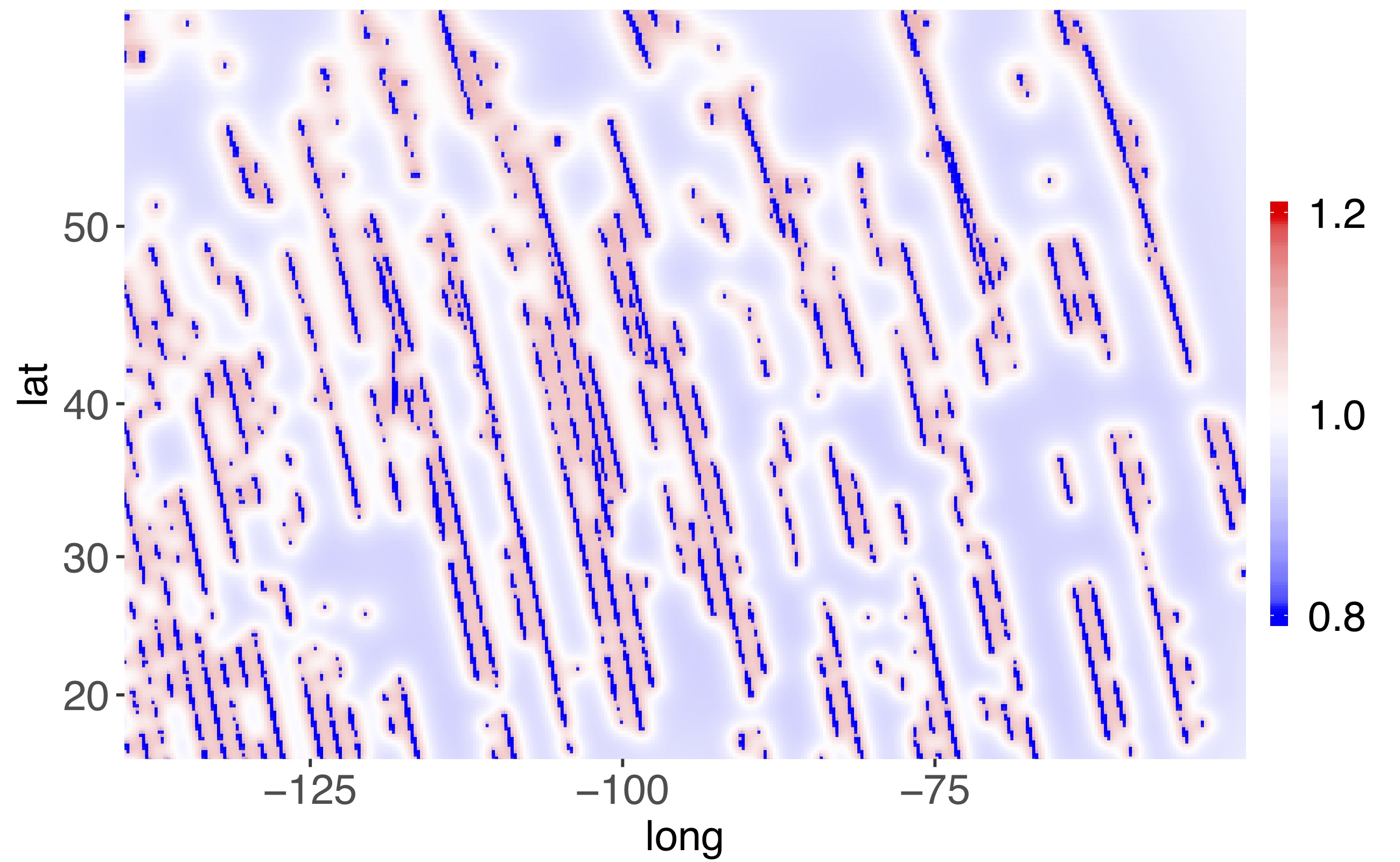}}
\caption{Ratio of kriging standard errors.}
\end{subfigure}

\caption{Comparison of kriging predictions under the Mat\'ern class and the CH class. The left panel shows the difference between kriging predictors under the CH class and those under the Mat\'ern class. The right panel shows the ratio of kriging standard errors under the CH class to those under the Mat\'ern class.}
\label{fig: comparison for OCO2 data}
\end{figure}

\end{document}